\documentclass[a4paper,reqno]{amsart}
\pdfoutput=1

\usepackage{amsmath,amssymb,amsthm,mathtools}
\usepackage{mathrsfs}
\usepackage{stmaryrd}
\usepackage{graphicx, xcolor}
\usepackage[all]{xy}
\usepackage{color}
\usepackage{tikz}
\usepackage{ifthen}
\usepackage{enumerate}
\usepackage{enumitem}
\usepackage{cite}
\usepackage{comment}
\usetikzlibrary{arrows, calc}

\makeatletter
\newcommand{\oset}[3][0ex]{%
  \mathrel{\mathop{#3}\limits^{
    \vbox to#1{\kern-5\ex@
    \hbox{$\scriptstyle#2$}\vss}}}}
\makeatother

\makeatletter
\newcommand{\loweroset}[3][0ex]{%
  \mathrel{\mathop{#3}\limits^{
    \vbox to#1{\kern-3\ex@
    \hbox{$\scriptstyle#2$}\vss}}}}
\makeatother

\setlist[enumerate]{leftmargin=18pt}
\numberwithin{equation}{section}

\makeatletter
\renewcommand{\p@enumii}{\arabic{enumi})(}
\renewcommand{\p@enumiii}{\arabic{enumi})(\roman{enumii}-}
\makeatother

\theoremstyle{definition}
\newtheorem{thm}{Theorem}[section]
\newtheorem*{thm*}{Theorem}
\newtheorem{prop}[thm]{Proposition}
\newtheorem{defi-prop}[thm]{Definition-Proposition}
\newtheorem{lem}[thm]{Lemma}
\newtheorem{cor}[thm]{Corollary}
\newtheorem{defi}[thm]{Definition}
\newtheorem{rmk}[thm]{Remark}
\newtheorem{ex}[thm]{Example}

\newtheorem{nota}[thm]{Notation}

\newtheorem{cla}[thm]{Claim}
\newtheorem{condi}[thm]{Condition}
\newtheorem*{ack}{Acknowledgements}

\newcommand{\ul}[1]{\underline{#1}}

\newcommand{\wt}[1]{\widetilde{#1}}
\newcommand{\wh}[1]{\widehat{#1}}
\newcommand{\xrar}[1]{\xrightarrow{#1}}
\newcommand{\xdrar}[1]{\oset{#1}{\dashrightarrow}}
\newcommand{\lxdrar}[1]{\loweroset{#1}{\dashrightarrow}}
\newcommand{\resp}[1]{(resp.\hspace{2pt}#1)}
\newcommand{\bracket}[1]{\langle #1 \rangle}
\newcommand{\lpp}[1]{{}^{\perp}#1}
\newcommand{\rpp}[1]{#1^{\perp}}
\newcommand{\dpp}[1]{\lpp{\rpp{#1}}}

\newcommand{\msize}[2]{\scalebox{#1}{$#2$}}

\newcommand{\bbF}{\mathbb{F}}
\newcommand{\bbE}{\mathbb{E}}

\newcommand{\bbS}{\mathbb{S}}

\renewcommand{\AA}{\mathcal{A}}

\newcommand{\CC}{\mathcal{C}}
\newcommand{\DD}{\mathcal{D}}
\newcommand{\EE}{\mathcal{E}}

\newcommand{\II}{\mathcal{I}}
\newcommand{\JJ}{\mathcal{J}}

\newcommand{\MM}{\mathcal{M}}
\newcommand{\NN}{\mathcal{N}}

\newcommand{\PP}{\mathcal{P}}

\newcommand{\RR}{\mathcal{R}}
\renewcommand{\SS}{\mathcal{S}}
\newcommand{\TT}{\mathcal{T}}
\newcommand{\UU}{\mathcal{U}}
\newcommand{\VV}{\mathcal{V}}

\newcommand{\XX}{\mathcal{X}}
\newcommand{\YY}{\mathcal{Y}}
\newcommand{\ZZ}{\mathcal{Z}}

\newcommand{\fraks}{\mathfrak{s}}

\DeclareMathOperator{\Proj}{Proj}
\DeclareMathOperator{\Inj}{Inj}

\DeclareMathOperator{\add}{\mathsf{add}\hspace{-0.8pt}}

\DeclareMathOperator{\Ab}{\mathsf{Ab}}
\DeclareMathOperator{\Set}{\mathsf{Set}}

\DeclareMathOperator{\Ob}{\mathrm{Ob}}

\DeclareMathOperator{\id}{id}

\DeclareMathOperator{\Cone}{Cone \hspace{0.5pt}}
\DeclareMathOperator{\CoCone}{CoCone \hspace{0.5pt}}
\DeclareMathOperator{\per}{\mathsf{per}}
\DeclareMathOperator{\thick}{\mathsf{thick}}
\DeclareMathOperator{\sfD}{\mathsf{D}}
\DeclareMathOperator{\Db}{\sfD^b}

\newcommand{\op}{\mathrm{op}}
\newcommand{\pr}{\prime}
\newcommand{\rar}{\rightarrow}
\newcommand{\drar}{\dashrightarrow}
\newcommand{\car}{\circlearrowright}

\newcommand{\vsim}{\rotatebox[origin=c]{90}{{\scriptsize{$\sim$}}}}
\newcommand{\bcap}[1]{ \raise2pt\hbox{$\displaystyle\bigcap_{#1}$}}
\newcommand{\bcup}[1]{ \raise0pt\hbox{$\displaystyle\bigcup_{#1}$}}
\newcommand{\bopl}[1]{ \raise0pt\hbox{$\displaystyle\bigoplus_{#1}$}}

\begin{document}

\title{
Triangulated structures induced by mutations
}
\author{Ryota Iitsuka}
\email{}
\begin{abstract}
In representation theory of algebras, there exist two types of mutation pairs: 
rigid subcategories by Iyama-Yoshino 
and orthogonal collections by Coelho Sim\~oes-Pauksztello.
It is known that such mutation pairs induce triangulated categories, however,
these facts have been proved in different ways.
In this paper, we introduce the concept of ``premutation triples'', which is a simultaneous generalization of two different types of mutation pairs as well as concentric twin cotorsion pairs.
We present two main theorems concerning mutation triples.
The first theorem is that premutation triples induce pretriangulated categories. 
The second one is that pretriangulated categories induced by mutation triples, which are premutation triples satisfying an additional condition (MT4), become triangulated categories.
\end{abstract}

\maketitle
\vspace{-15pt}
\tableofcontents

\newpage
\section{Introduction}
The notion of ``mutation'' plays important roles in representation theory and related fields. 
Roughly speaking, mutation is an operation to obtain new objects from old ones, 
usually considered in triangulated categories, exact categories, or extriangulated categories \cite{IY08, AI12, AIR14, GNP23}. 
There are many studies about mutation of tilting objects (for example, APR-tilting \cite{ASS06}), silting objects\cite{AI12}, cluster-tilting objects\cite{IY08, BMRRT06} and support $\tau$-tilting objects\cite{AIR14}. 
They are respectively called tilting mutations, silting mutations, cluster-tilting mutations and support $\tau$-tilting mutations.
There are some mutations which are considered in more generalized situations \cite{LZ13, ZZ18}.
In many cases, we can study characters of certain objects (silting objects and so on) by mutating them \cite{AIR14,BMRRT06,IY08}.
We collected some results on mutations of rigid subcategories in Appendix \ref{Rigid mutation pairs} and \ref{Triangulated structures induced by rigid mutation pairs}.

On the other hand, we may consider not only mutations of rigid subcategories (called ``rigid mutations'' here) 
but also those of orthogonal collections (called ``orthogonal mutations'' here).
For example, simple-minded collections \cite{KY14} and simple-minded systems \cite{SP20,IJ23,Sim17, Dug15,SPP22}.
We also collected some results in orthogonal mutations in Appendix \ref{Orthogonal mutation pairs} and 
\ref{Triangulated structures induced by orthogonal mutation pairs}.

What is more interesting is that both rigid and orthogonal mutations in a triangulated category induce another smaller triangulated category \cite{IY08, BMRRT06, AI12, SP20}.
Furthermore, it is also known that some mutation-like concepts induce triangulated categories.
For example, a Frobenius extriangulated category induces a triangulated category 
whose shift functor is exactly a cosyzygy functor, 
which can be seen as a special case of rigid mutations in extriangulated categories \cite{NP19}.
For another example, a concentric twin cotorsion pair \cite{Nak18, NP19, LN19} in triangulated category 
induces a pretriangulated category \cite{BR07} 
and induces a triangulated category with some conditions \cite{Nak18}.
We review extriangulated categories and pretriangulated categories in section \ref{Structures associated with additive category}.
(For details on concentric twin cotorsion pairs, see Appendix \ref{ccTCP}.)

However, the proofs showing that they induce triangulated structures are independent of all four cases: rigid mutations, orthogonal mutations, Frobenius extriangulated categories and concentric twin cotorsion pairs \cite{IY08, IY18, Jin23, SP20, NP19, Nak18}.
So our goal is to understand these triangulated structures within the same framework.
In other words, we consider a simultaneous generalization of all four cases.

In section \ref{pretri}, we introduce the new concept of ``premutation triples'',
which is the framework we wanted to explain induced triangulated structures.
Then we collect elementary results of premutation triples.

In the following definition, the concept of ``strongly functorially finite'' is defined in Definition \ref{defi_strongly-finite} and the extriangulated categories $(\CC, \bbE^{\II}, \fraks^{\II})$ and $(\CC, \bbE_{\II}, \fraks_{\II})$ are defined in Example \ref{ex_relative-ET}.

\begin{defi} (Condition \ref{MT1-2}, \ref{MT3} and Definition \ref{mutation_triple})
\label{defi_first}
Let $(\CC, \bbE, \fraks)$ be an extriangulated category and $(\SS, \ZZ, \VV)$ be a triplet of subcategories of $\CC$.
$(\SS, \ZZ, \VV)$ is called \emph{premutation triple} if it satisfies the following conditions.
	\begin{itemize}[leftmargin=40pt]
	\item[(MT1)] $\SS \cap \ZZ = \ZZ \cap \VV$, denoted by $\II$, and 
	$\II$ is strongly functorially finite in $\ZZ$.
	\item[(MT2)] 
		\begin{enumerate}[label=(\roman*)]
		\item $\bbE^{\II}(\SS, \ZZ) = 0$ and $\bbE_{\II}(\SS, \CoCone_{\bbE_{\II}}(\II, \ZZ)) = 0$. 
		\item $\bbE_{\II}(\ZZ, \VV) = 0$ and $\bbE^{\II}(\Cone_{\bbE^{\II}}(\ZZ, \II), \VV) = 0$.
		\end{enumerate}
	\item[(MT3)] 
		\begin{enumerate}[label=(\roman*)]
		\item $\Cone_{\bbE^{\II}}(\ZZ, \ZZ) \subset \CoCone_{\bbE^{\II}}(\ZZ, \SS)$.
		\item $\CoCone_{\bbE_{\II}}(\ZZ, \ZZ) \subset \Cone_{\bbE_{\II}}(\VV, \ZZ)$.
		\item $\SS$ and $\ZZ$ are closed under extensions in $(\CC, \bbE^{\II}, \fraks^{\II})$ and $\VV$ and 
		$\ZZ$ are closed under extensions in $(\CC, \bbE_{\II}, \fraks_{\II})$.
		\end{enumerate}
	\end{itemize}
\end{defi}

In the last of this section,
We show the first main theorems below, which is a generalization of the results in \cite{Nak18}.

\begin{thm} (Theorem \ref{main_thm1})
Let $(\SS, \ZZ, \VV)$ be a premutation triple. Then $\ZZ/[\II]$ has a pretriangulated structure.
\end{thm}

In section \ref{triangulated}, we collect sufficient conditions for mutation triples to induce a triangulated category.
We consider two cases.
The former one requires an additional condition (MT4), but it is not necessary that $\CC$ is a triangulated category.

\begin{thm} (Theorem \ref{main_thm2}, Remark \ref{MT4}) 
Let $(\SS, \ZZ, \VV)$ be a premutation triple. 
$(\, \cdot \,)^-$ and $(\, \cdot \,)^+$ are defined in Proposition \ref{prop_+_and_-}.
We consider the following new condition (MT4).
	\begin{itemize}[leftmargin=40pt]
	\item[(MT4)] $(\ZZ\bracket{1}/[\II])^- = (\ZZ\bracket{-1}/[\II])^+$
	\end{itemize}
$(\SS, \ZZ, \VV)$ is called a \emph{mutation triple} if it satisfies (MT4).
If $(\SS, \ZZ, \VV)$ is a mutation triple, then $\ZZ/[\II]$ has a triangulated structure.
\end{thm}

The latter one is the result in \cite{Nak18}, which can be applied to mutation triples defined  by concentric twin cotorsion pairs in triangulated category with specific conditions: Hovey and heart-equivalent.
We show that (MT4) follows from these conditions.
Therefore, the latter case is a special case of the former one.

In section \ref{redMT}, we introduce another triplet of subcategories, named reducible triple.
We consider the following different version of (MT3) and (MT4) to define reducible triples.
The extriangulated category $(\CC, \bbE^{\II}_{\II}, \fraks^{\II}_{\II})$ is also defined in Example \ref{ex_relative-ET}.
\begin{itemize}[leftmargin=40pt]
	\item[(RT3)]
		\begin{enumerate}[label=(\roman*)]
		\item $\Cone_{\bbE^{\II}}(\ZZ, \ZZ) \subset \CoCone_{\bbE^{\II}_{\II}}(\ZZ, \SS)$.
		\item $\CoCone_{\bbE_{\II}}(\ZZ, \ZZ) \subset \Cone_{\bbE^{\II}_{\II}}(\VV, \ZZ)$.
		\item $\SS$, $\ZZ$ and $\VV$ are closed under extensions in $(\CC, \bbE^{\II}_{\II}, \fraks^{\II}_{\II})$.
		\end{enumerate}
	\item[(RT4)]
		\begin{enumerate}[label=(\roman*)]
		\item $\II$ is strongly contravariantly finite in $\SS$.
		\item $\II$ is strongly covariantly finite in $\VV$.
		\item $\CoCone_{\bbE_{\II}}(\II, \SS) = \Cone_{\bbE^{\II}}(\VV, \II)$, denoted by $\RR$.
		\end{enumerate}
\end{itemize}
Then we define a reducible triple as a triplet of subcategories satisfying 
(MT1), (MT2), (RT3) and (RT4).

Reducible triples have the following nice property.
\begin{thm} (Theorem \ref{main_thm3})
Let $(\SS, \ZZ, \VV)$ be a reducible triple.
Let $\EE$ be an extension closed subcategory in $(\CC, \bbE^{\II}_{\II}, \fraks^{\II}_{\II})$ containing $\RR$.
Then $(\EE \cap \ZZ)/[\II]$ is an extension closed in $\ZZ/[\II]$.
\end{thm}

As an application of this theorem, 
we may consider restricting mutations to extension closed subcategory $\EE$ in $(\CC, \bbE^{\II}_{\II}, \fraks^{\II}_{\II})$, which induces mutations in the extriangulated category $(\EE \cap \ZZ)/[\II]$.
Mutations of 2-term silting complexes \cite{AIR14} are one of these examples.

Another advantage of introducing reducible triples is that 
we may define mutations of collections in $\CC$, 
which is a simultaneous generalization of cluster-tilting mutations, silting mutations, 
mutations of simple-minded systems and mutations of simple-minded collections.

\begin{defi} (Definition \ref{defi_mu-rMT})
Let $(\SS, \ZZ, \VV)$ be a reducible triple and $\RR^{\pr}$ be a collection in $\Ob(\CC)$ whose extension closure in $(\CC, \bbE^{\II}_{\II}, \fraks^{\II}_{\II})$ is $\RR$.
Assume that $\XX \subset \Ob(\ZZ)$ be a collection containing $\RR^{\pr}$.
We denote $\XX \setminus \RR^{\pr}$ by $\XX_{\RR^{\pr}}$.
	\begin{enumerate}
	\item We define \emph{right $\RR^{\pr}$-mutation} of $\XX$ as 
	$\RR^{\pr} \cup \Sigma \XX_{\RR^{\pr}}$, which is denoted by $\mu^-_{\RR^{\pr}}(\XX)$.
	\item We define \emph{left $\RR^{\pr}$-mutation} of $\XX$ as 
	$\RR^{\pr} \cup \Omega \XX_{\RR^{\pr}}$, which is denoted by $\mu^+_{\RR^{\pr}}(\XX)$.
	\end{enumerate}
\end{defi}

Throughout this thesis, let $k$ be a field and $\CC$ be a skeletally small additive category, thus the isomorphism class of $\Ob(\CC)$ is a set. 
If $\CC$ is an extriangulated or triangulated, we denote the extension closure in $\CC$ by $\bracket{\cdot}$.
We denote the category of abelian groups \resp{sets} by $\Ab$ \resp{$\Set$}.

We also assume that all subcategories are additive, full and closed under isomorphisms.
We do not always assume that all subcategories are closed under direct summands, 
so we denote the smallest subcategory containing $\DD$ and closed under direct summands by $\add \DD$ for a subcategory $\DD$.

We recall the concept of approximations.

\begin{defi}
Let $\II$ and $\ZZ$ be subcategories of $\CC$ and let $X \in \CC$.
	\begin{enumerate}
	\item A morphism $a \colon I_X \to X$ in $\CC$ is \emph{$\II$-epic}
	if $\CC(I, a) \colon \CC(I, I_X) \to \CC(I, X)$ is surjective for any $I$ in $\II$.
	\item A morphism $b \colon I_X \to X$ in $\CC$ is a \emph{right $\II$-approximation} of $X$ 
	if $I_X \in \II$ and $b$ is $\II$-epic.
	\item $\II$ is \emph{contravariantly finite} in $\ZZ$ if any $Z$ in $\ZZ$ has a right $\II$-approximation.
	\end{enumerate}

Dually, we define \emph{$\II$-monic}, a \emph{left $\II$-approximation} of $X$ and \emph{covariantly finite} in $\ZZ$. 
$\II$ is called \emph{functorially finite} in $\ZZ$ if $\II$ is both covariantly finite and contravariantly finite in $\ZZ$. 
\end{defi}

\begin{ack}
The author would like to thank H. Nakaoka and Professor Michael Wemyss for valuable suggestions to improve this paper.
\end{ack}

\section{Structures associated with additive category} \label{Structures associated with additive category}

\subsection{Extriangulated categories}
First, we start this section from the definition of extriangulated categories \cite{NP19}. 

\begin{defi} \cite[Definition 2.7, 2.8]{NP19}
	\begin{enumerate}
	\item For $X,Y \in \CC$, we denote the collection of three-term sequences whose first-term is $X$ 
	and third-term is $Y$ by $\wt{\EE}(Y,X)$ (note the order of $X$ and $Y$). 
	Then we introduce an equivalence relation $\sim$ in $\wt{\EE}(Y,X)$ as follows.

	For $\mathbf{E} = (X \xrar{x} E \xrar{y} Y), 
	\mathbf{E}^{\pr} = (X \xrar{x^{\pr}} E^{\pr} \xrar{y^{\pr}} Y)$ in $\wt{\EE}(Y,X)$,
	\[
	\mathbf{E} \sim \mathbf{E}^{\pr} \iff 
	\begin{aligned}
	\text{There } &\text{exists an isomorphism $e \colon E \to E^{\pr}$} \\
	&\text{such that } x^{\pr} = ex \text{ and } y = y^{\pr}e.
	\end{aligned}
	\]
	We denote $\wt{\EE}(Y,X) /\sim$ by $\EE(Y,X)$. 
	
	\item For $X,Y \in \CC$, we denote as 
	$0 = (X \xrar{\msize{0.6}{\begin{bmatrix} 1 \\ 0  \end{bmatrix}}} X \oplus Y 
	\xrar{\msize{0.6}{\begin{bmatrix} 0 & 1 \end{bmatrix}}} Y)$
	in $\EE(Y,X)$.
	
	\item For $(X \xrar{a} E \xrar{b} Y)$ in $\EE(Y,X)$ 
	and $(X^{\pr} \xrar{a^{\pr}} E^{\pr} \xrar{b^{\pr}} Y^{\pr})$ in $\EE(Y^{\pr},X^{\pr})$,
	 $(X \xrar{a} E \xrar{b} Y) \oplus (X^{\pr} \xrar{a^{\pr}} E^{\pr} \xrar{b^{\pr}} Y^{\pr})$ is defined by 
	 $(X\oplus X^{\pr} \xrar{a\oplus a^{\pr}} E\oplus E^{\pr} \xrar{b \oplus b^{\pr}} Y\oplus Y^{\pr})$
	  in $\EE(Y\oplus Y^{\pr}, X\oplus X^{\pr})$.
	\end{enumerate}
\end{defi}

\begin{rmk} \cite[Definition 2.1-2.3, Remark 2.2]{NP19}
Let $\bbE \colon \CC^{\op} \times \CC \to \Ab$ be an additive bifunctor and 
$X,X^{\pr},Y,Y^{\pr},Z \in \CC$.
\begin{enumerate} 
	\item An element $\delta \in \bbE(X,Y)$ is called \emph{$\bbE$-extension}.
	
	\item Let $a \colon X \to Y$ and $b \colon Y \to Z$ be morphisms in $\CC$, 
	we can define the following natural transformations.
		\[
		\bbE(b,-) \colon \bbE(Z, -) \Rightarrow \bbE(Y, -)
		\]
		\[
		\bbE(-,a) \colon \bbE(-,X) \Rightarrow \bbE(-,Y)
		\]
	\item There exists the following isomorphism.
	\[
	\bbE(X \oplus Y, X^{\pr} \oplus Y^{\pr}) \cong \bbE(X, X^{\pr}) \oplus \bbE(X, Y^{\pr}) 
	\oplus \bbE(Y, X^{\pr}) \oplus \bbE(Y,Y^{\pr})
	\]
	Then we define $\delta \oplus \delta^{\pr}$ in left-hand side as the element which corresponds to $(\delta, 0, 0, \delta^{\pr})$ in right-hand side by the above isomorphism.
	
	\item Let $\delta \in \bbE(Z,X)$.
	We write $\bbE(b,X)(\delta), \bbE(Z,a)(\delta)$ as $\delta b, a \delta$ respectively.
	
	\item In the rest of this paper, we sometimes regard $\bbE$-extensions as ``morphisms'' in $\CC$, that is, 
	we interpret $\delta b$ as a ``composition'' of $(Y \xrar{b} Z \lxdrar{\delta} X)$ and $a \delta$ as a ``composition'' of  $(Z \lxdrar{\delta} X \xrar{a} Y)$. Then we can consider commutative diagrams with $\bbE$-extensions by this notation. 
	
	For example, let $\delta \in \bbE(X,Y), \delta^{\pr} \in \bbE(X^{\pr}, Y^{\pr}), 
	x \in \CC(X,X^{\pr}), y \in \CC(Y, Y^{\pr})$, then
	\[
	y \delta = \delta^{\pr} x \text{ is expressed as }
	\xy
	(0,8)*+{X}="1";
	(0,-8)*+{X^{\pr}}="2";
	(16,8)*+{Y}="3";
	(16,-8)*+{Y^{\pr}}="4";
	(20,-8)*+{\vphantom{X}.}="piri";
	{\ar@{-->}^{\delta} "1";"3"};
	{\ar^{x} "1";"2"};
	{\ar@{-->}^{\delta ^{\pr}} "2";"4"};
	{\ar^{y} "3";"4"};
	{\ar@{}|{\car} "1";"4"};
	\endxy
	\]
\end{enumerate}
\end{rmk}

\begin{defi} \cite[Definition 2.4, 2.5]{NP19}
\begin{enumerate}
\item Let $\bbE \colon \CC^{\op} \times \CC \to \Ab$ be an additive bifunctor. 
$\fraks$ is called a \emph{realization} of $\bbE$ if $\fraks$ satisfies the following conditions.
	\begin{enumerate}[label=(\roman*)]
	\item $\fraks$ is a collection of correspondence $\{ \fraks_{X,Y} \colon \bbE(X,Y) \to \EE(X,Y)\}_{X,Y \in \CC}$.
	We often denote $\fraks_{X,Y}$ as $\fraks$ if there is no confusion.
	\item For $\delta \in \bbE(Y,X), \delta^{\pr} \in \bbE(Y^{\pr}, X^{\pr})$, let $\fraks(\delta) = (X \xrar{x} E \xrar{y} Y), \fraks(\delta^{\pr}) = (X^{\pr} \xrar{x^{\pr}} E^{\pr} \xrar{y^{\pr}} Y^{\pr})$. Then for any commutative diagrams in $\CC$,
	\[
	\xy
	(0,8)*+{X}="1";
	(0,-8)*+{X^{\pr}}="2";
	(16,8)*+{E}="3";
	(16,-8)*+{E^{\pr}}="4";
	(32,8)*+{Y}="5";
	(32,-8)*+{Y^{\pr}}="6";
	(48,8)*+{X}="7";
	(48,-8)*+{X^{\pr}}="8";
	{\ar^{x} "1";"3"};
	{\ar^{x^{\pr}} "2";"4"};
	{\ar^{y} "3";"5"};
	{\ar^{y^{\pr}} "4";"6"};
	{\ar^{a} "1";"2"};
	{\ar^{b} "5";"6"};
	{\ar^{a} "7";"8"};
	{\ar@{-->}^{\delta} "5";"7"};
	{\ar@{-->}^{\delta^{\pr}} "6";"8"};
	{\ar@{}|\car "5";"8"};
	\endxy
	\]
	there exists a morphism $e \colon E \to E^{\pr}$ which makes the following diagram commutative.
	\[
	\xy
	(0,8)*+{X}="1";
	(0,-8)*+{X^{\pr}}="2";
	(16,8)*+{E}="3";
	(16,-8)*+{E^{\pr}}="4";
	(32,8)*+{Y}="5";
	(32,-8)*+{Y^{\pr}}="6";
	(48,8)*+{X}="7";
	(48,-8)*+{X^{\pr}}="8";
	{\ar^{x} "1";"3"};
	{\ar^{x^{\pr}} "2";"4"};
	{\ar^{y} "3";"5"};
	{\ar^{y^{\pr}} "4";"6"};
	{\ar^{a} "1";"2"};
	{\ar^{^{\exists}e} "3";"4"};
	{\ar^{b} "5";"6"};
	{\ar^{a} "7";"8"};
	{\ar@{-->}^{\delta} "5";"7"};
	{\ar@{-->}^{\delta^{\pr}} "6";"8"};
	{\ar@{}|\car "1";"4"};
	{\ar@{}|\car "3";"6"};
	{\ar@{}|\car "5";"8"};
	\endxy
	\]
\end{enumerate}
\item Let $\fraks$ be a realization of $\bbE$. $\fraks$ is \emph{additive} if it satisfies the following conditions.
\begin{enumerate}[label=(\roman*)]
	\item For any $X,Y \in \CC$, $\fraks(0) = (X \xrar{\msize{0.5}{\begin{bmatrix} 1 \\ 0 \\ \end{bmatrix}}} X \oplus Y \xrar{\msize{0.5}{[0 \ 1]}}Y)$, that is, $\fraks$ maps $0$ in $\bbE(Y,X)$ to $0$ in $\EE(Y,X)$.
	\item For any $\delta \in \bbE(X,Y), \delta^{\pr} \in \bbE(X^{\pr}, Y^{\pr})$, $\fraks(\delta \oplus \delta^{\pr}) = \fraks(\delta) \oplus \fraks(\delta^{\pr})$.
\end{enumerate}
\end{enumerate}
\end{defi}

\begin{rmk} \cite[Definition 2.15, 2.19]{NP19}
\begin{enumerate}
	\item Let $\delta \in \bbE(Y,X)$ and $\fraks(\delta) = (X \xrar{x} E \xrar{y} Y)$. This sequence $(X \xrar{x} E \xrar{y} Y)$ is called $\fraks$-\emph{conflation}. We often call it \emph{conflation} if there is no confusion.
	\item The left morphism of a conflation is called an $\fraks$-\emph{inflation} and 
	the right one is called an $\fraks$-\emph{deflation}. 	
	\item A pair $(\delta, \fraks(\delta))$ is called an \emph{$\fraks$-triangle} and it is denoted by
	\[
	X \xrar{x} E \xrar{y} Y \xdrar{\delta} X
	\quad \text{or} \quad
	Y \xdrar{\delta} X \xrar{x} E \xrar{y} Y.
	\]
\end{enumerate}
\end{rmk}

\begin{defi} \cite[Definition 2.12]{NP19}

A triplet ($\CC$, $\bbE$, $\fraks$) is called an \emph{extriangulated category}, or \emph{ET category} if the triplet satisfies the following conditions.

\begin{itemize}[leftmargin=45pt]
	\item[(ET1)] $\bbE \colon \CC^{\op} \times \CC \to \Ab$ is an additive bifunctor.
	\item[(ET2)] $\fraks$ is an additive realization of $\bbE$.
	\item[(ET3)] Let $\delta \in \bbE(Y,X), \delta^{\pr} \in \bbE(Y^{\pr}, X^{\pr})$. For $\fraks(\delta) = (X \xrar{x} E \xrar{y} Y), \fraks(\delta^{\pr}) = (X^{\pr} \xrar{x^{\pr}} E^{\pr} \xrar{y^{\pr}} Y^{\pr})$ and any diagram in $\CC$,
	\[
	\xy
	(0,8)*+{X}="1";
	(0,-8)*+{X^{\pr}}="2";
	(16,8)*+{E}="3";
	(16,-8)*+{E^{\pr}}="4";
	(32,8)*+{Y}="5";
	(32,-8)*+{Y^{\pr}}="6";
	(48,8)*+{X}="7";
	(48,-8)*+{X^{\pr}}="8";
	{\ar^{x} "1";"3"};
	{\ar^{x^{\pr}} "2";"4"};
	{\ar^{y} "3";"5"};
	{\ar^{y^{\pr}} "4";"6"};
	{\ar^{a} "1";"2"};
	{\ar^{e} "3";"4"};
	{\ar^{a} "7";"8"};
	{\ar@{-->}^{\delta} "5";"7"};
	{\ar@{-->}^{\delta^{\pr}} "6";"8"};
	{\ar@{}|\car "1";"4"};
	\endxy
	\]
	there exists a morphism $b \colon Y \to Y^{\pr}$ which makes the following diagram commutative.
	\[
	\xy
	(0,8)*+{X}="1";
	(0,-8)*+{X^{\pr}}="2";
	(16,8)*+{E}="3";
	(16,-8)*+{E^{\pr}}="4";
	(32,8)*+{Y}="5";
	(32,-8)*+{Y^{\pr}}="6";
	(48,8)*+{X}="7";
	(48,-8)*+{X^{\pr}}="8";
	{\ar^{x} "1";"3"};
	{\ar^{x^{\pr}} "2";"4"};
	{\ar^{y} "3";"5"};
	{\ar^{y^{\pr}} "4";"6"};
	{\ar^{a} "1";"2"};
	{\ar^{e} "3";"4"};
	{\ar^{b} "5";"6"};
	{\ar^{a} "7";"8"};
	{\ar@{-->}^{\delta} "5";"7"};
	{\ar@{-->}^{\delta^{\pr}} "6";"8"};
	{\ar@{}|\car "1";"4"};
	{\ar@{}|\car "3";"6"};
	{\ar@{}|\car "5";"8"};
	\endxy
	\]
	\item[(ET3)$^{\op}$] Dual of $\mathrm{(ET3)}$.
	\item[(ET4)] For $\delta \in \bbE(C,X), \epsilon \in \bbE(D,Y)$, $\fraks(\delta) = (X \xrar{x} Y \xrar{x^{\pr}} C),
	 \fraks(\epsilon) = (Y \xrar{y} Z \xrar{y^{\pr}} D)$, there exist $\fraks$-triangles
	$X \xrar{z} Z \xrar{z^{\pr}} E \xdrar{\delta^{\pr}} X$ and $C \xrar{c} E \xrar{d} D \xdrar{\epsilon^{\pr}} C$
	which make the following diagram commutative.
	\[
	\xy
	(0,8)*+{X}="01";
	(0,-8)*+{X}="02";
	(16,24)*+{D}="10";
	(16,8)*+{Y}="11";
	(16,-8)*+{Z}="12";
	(16,-24)*+{D}="13";
	(32,24)*+{D}="20";
	(32,8)*+{C}="21";
	(32,-8)*+{E}="22";
	(32,-24)*+{D}="23";
	(48,8)*+{X}="31";
	(48,-8)*+{X}="32";
	(48,-24)*+{Y}="33";
	{\ar^{x} "01";"11"};
	{\ar^{x} "32";"33"};
	{\ar^{x^{\pr}} "11";"21"};
	{\ar^{y} "11";"12"};
	{\ar^{y^{\pr}} "12";"13"};
	{\ar^{z} "02";"12"};
	{\ar^{z^{\pr}} "12";"22"};
	{\ar^{c} "21";"22"};
	{\ar^{d} "22";"23"};
	{\ar@{=} "01";"02"};
	{\ar@{=} "10";"20"};
	{\ar@{=} "13";"23"};
	{\ar@{=} "31";"32"};
	{\ar@{-->}^{\epsilon} "10";"11"};
	{\ar@{-->}^{\epsilon} "23";"33"};
	{\ar@{-->}^{\epsilon^{\pr}} "20";"21"};
	{\ar@{-->}^{\delta} "21";"31"};
	{\ar@{-->}^{\delta^{\pr}} "22";"32"};
	{\ar@{}|\car "01";"12"};
	{\ar@{}|\car "10";"21"};
	{\ar@{}|\car "11";"22"};
	{\ar@{}|\car "12";"23"};
	{\ar@{}|\car "21";"32"};
	{\ar@{}|\car "22";"33"};
	\endxy
	\]
	\item[(ET4)$^{\op}$] Dual of $\mathrm{(ET4)}$.
\end{itemize}
\end{defi}

\begin{rmk} \cite[Corollary 3.12]{NP19} \label{long_exact_seq}
Let $(\CC, \bbE, \fraks)$ be an ET category and $(X \xrar{x} Y \xrar{y} Z  \lxdrar{\delta} X)$ be an $\fraks$-triangle. Then the following are exact sequences.
\[
\CC(Z,-) \xrar{\msize{0.6}{\CC(y,-)}} \CC(Y,-) \xrar{\msize{0.6}{\CC(x,-)}} \CC(X,-) 
\xrar{- \circ \delta} 
\bbE(Z,-) \xrar{\msize{0.6}{\bbE(y,-)}} \bbE(Y,-) \xrar{\msize{0.6}{\bbE(x,-)}} \bbE(X,-)
\]
\[
\CC(-,X) \xrar{\msize{0.6}{\CC(-,x)}} \CC(-,Y) \xrar{\msize{0.6}{\CC(-,y)}} \CC(-,Z) \xrar{\delta \circ -}
\bbE(-,X) \xrar{\msize{0.6}{\bbE(-,x)}} \bbE(-,Y) \xrar{\msize{0.6}{\bbE(-,y)}} \bbE(-,Z)
\]
\end{rmk}

The following proposition is often used in this paper.

\begin{prop} \cite[Proposition 3.15]{NP19} (Shifted octahedrons) \label{shifted_octahedrons}
Let $X_i \xrar{x_i} Y_i \xrar{y_i} Z \lxdrar{\delta_i} X_i$ be an $\fraks$-triangle for $i =1,2$.
Then there exist $\fraks$-triangles
$X_2 \xrar{v_1} W \xrar{w_1} Y_1 \lxdrar{\epsilon_1} X_2$ and $X_1 \xrar{v_2} W \xrar{w_2} Y_2 \lxdrar{\epsilon_2} X_1$ which make the following diagram commutative.
\[
\xy
(0,24)*+{}="11";
(16,24)*+{X_2}="12";
(32,24)*+{X_2}="13";
(48,24)*+{}="14";
(0,8)*+{X_1}="21";
(16,8)*+{W}="22";
(32,8)*+{Y_2}="23";
(48,8)*+{X_1}="24";
(0,-8)*+{X_1}="31";
(16,-8)*+{Y_1}="32";
(32,-8)*+{Z}="33";
(48,-8)*+{X_1}="34";
(0,-24)*+{}="41";
(16,-24)*+{X_2}="42";
(32,-24)*+{X_2}="43";
(48,-24)*+{W}="44";
{\ar@{=} "12";"13"};
{\ar^{v_2} "21";"22"};
{\ar^{w_2} "22";"23"};
{\ar@{-->}^{\epsilon_2} "23";"24"};
{\ar^{x_1} "31";"32"};
{\ar^{y_1} "32";"33"};
{\ar@{-->}^{\delta_1} "33";"34"};
{\ar@{=} "42";"43"};
{\ar^{-v_1} "43";"44"};
{\ar@{=} "21";"31"};
{\ar^{v_1} "12";"22"};
{\ar^{w_1} "22";"32"};
{\ar@{-->}^{\epsilon_1} "32";"42"};
{\ar^{x_2} "13";"23"};
{\ar^{y_2} "23";"33"};
{\ar@{-->}^{\delta_2} "33";"43"};
{\ar@{=} "24";"34"};
{\ar^{v_2} "34";"44"};
{\ar@{}|\car "12";"23"};
{\ar@{}|\car "21";"32"};
{\ar@{}|\car "22";"33"};
{\ar@{}|\car "23";"34"};
{\ar@{}|\car "32";"43"};
{\ar@{}|\car "33";"44"};
\endxy
\]
\end{prop}
\begin{proof}
See \cite[Proposition 3.15]{NP19}.
\end{proof}

The following definitions of projective objects and injective objects are analogies of exact category.

\begin{defi} \cite{NP19}
Let $(\CC, \bbE, \fraks)$ be an ET category.
	\begin{enumerate}
	\item We define a subcategory of $\CC$, $\mathrm{Proj}_{\bbE} \CC$ as $\{ X \in \CC \mid \bbE(X,\CC) = 0\}$. An object in $\Proj_{\bbE} \CC$ is called a \emph{projective object}. 
	\item $\CC$ \emph{has enough projectives} if, for any $X$ in $\CC$, there exists a conflation $X^{\pr} \rar P \rar X$ with $P \in \Proj_{\bbE} \CC $.
	\item For subcategories $\XX, \YY$ of $\CC$, we define the following three subcategories. \vspace{-5pt}
		\[
		\XX \ast \YY = 
		\left\{
		E \in \CC \ \middle| \
		\begin{aligned}
			\text{There }&\text{exists an } \fraks \text{-conflation } \\
			&X \rar E \rar Y \text{ where } X \in \XX, Y \in \YY 
		\end{aligned}
		\right\}
		\]
		\[
		\Cone_{\bbE}(\XX, \YY) =
		\left\{ 
		Z \in \CC \ \middle| \
		\begin{aligned}
			\text{There }&\text{exists an } \fraks \text{-conflation } \\
			&X \rar Y \rar Z \text{ where } X \in \XX, Y \in \YY 
		\end{aligned}
		\right\}
		\]
		\[
		\CoCone_{\bbE}(\XX, \YY) =
		\left\{ 
		Z^{\pr} \in \CC \ \middle| \
		\begin{aligned}
			\text{There }&\text{exists an } \fraks \text{-conflation } \\
			& Z^{\pr} \rar X \rar Y  \text{ where } X \in \XX, Y \in \YY 
		\end{aligned}
		\right\}
		\]
	\end{enumerate}
	
We denote $\Proj_{\bbE} \CC$ by $\Proj \CC$ when there is no confusion. 
Dually, we define $\Inj_{\bbE} \CC$ and \emph{enough injectives}. 
$\CC$ is called \emph{Frobenius} if $\Proj \CC = \Inj \, \CC$ and $\CC$ has enough projectives and enough injectives.
\end{defi}

\begin{ex}	\cite[Corollary 3.18, Proposition 3.22]{NP19}
	\begin{enumerate}
	\item An exact category is an ET category whose inflations are monomorphic and deflations are epimorphic. 
	In this situation, $\fraks$-conflations are exactly conflations in the exact structure.
	\item A triangulated category is exactly a Frobenius ET category with $\Proj \CC = 0$.
	\end{enumerate}
\end{ex}

There are some ways to obtain a new ET category from old one. First case is a generalized statement of Happel's theorem \cite{Hap88}.

\begin{prop} \cite[Proposition 3.30]{NP19}
Let $\CC$ be an ET category and let $\II \subset \Proj \CC \cap \Inj \CC$. 
Then $\CC/[\II]$ also becomes an ET category.
\end{prop}
\begin{proof}
See \cite[Proposition 3.30]{NP19}.
\end{proof}

\begin{cor} \cite[Corollary 7.4, Remark 7.5]{NP19}
Let $\CC$ be an Frobenius ET category. 
Then $\CC/[\Proj \CC]$ becomes a triangulated category.
\end{cor}
\begin{proof}
See \cite[Corollary 7.4, Remark 7.5]{NP19}.
\end{proof}

Next way is to restrict the bifunctor $\bbE$ and the realization $\fraks$ to an extension-closed subcategory. We start from the definition of ``extension-closed'' subcategories.

\begin{defi} \cite[Definition 2.17]{NP19}
Let $\CC$ be an ET category and $\ZZ$ be a subcategory. 
$\ZZ$ is called \emph{extension-closed} if, for any conflation $A \rar B \rar C$ where $A,C$ in $\ZZ$, 
then $B$ is also in $\ZZ$.
\end{defi}

\begin{lem} \cite[Remark 2.18]{NP19}
Let $\CC$ be an ET category and $\ZZ$ is an extension-closed subcategory. Then $\ZZ$ has an ET structure defined by restricting $\bbE$ and $\fraks$ to $\ZZ$.
\end{lem}

Last way is to restrict the bifunctor to a closed subfunctor.
See also \cite[p649]{DRSSK99} for the following definitions in exact categories.

\begin{defi} \cite[Definition 3.7]{HLN21}
Let $(\CC, \bbE, \fraks)$ be an ET category.
	\begin{enumerate}
	\item A functor $\bbF \colon \CC^{\op} \times \CC \to \Set$ is called a \emph{subfunctor} of $\bbE$ 
	if it satisfies the following conditions.
		\begin{enumerate}[label=(\roman*)]
		\item For any $X,Y \in \CC$, $\bbF(X,Y)$ is a subset of $\bbE(X,Y)$.
		\item For any morphism $x \colon X^{\pr} \to X$ and $y \colon Y \to Y^{\pr}$, $\bbF(x,y) = \bbE(x,y)|_{\bbF(X,Y)}$.
		\end{enumerate}
	Then we denote $\bbF \subset \bbE$.
	\item A subfunctor $\bbF \subset \bbE$ is called \emph{additive} if $\bbF$ is an additive bifunctor.
	\end{enumerate}
\end{defi}

\begin{defi} \cite[Definition 3.8]{HLN21}
Let $(\CC, \bbE, \fraks)$ be an ET category and $\bbF$ be an additive subfunctor of $\bbE$.
We define $\fraks |_{\bbF}$ by restriction of $\fraks$ onto $\bbF$.
\end{defi}

\begin{prop} \cite[Proposition 3.16]{HLN21} \cite[Proposition 1.4]{DRSSK99}
Let $(\CC, \bbE, \fraks)$ be an ET category and $\bbF$ be an additive subfunctor of $\bbE$. 
Then the following are equivalent.
	\begin{enumerate}[label=(\roman*)]
	\item $\fraks |_{\bbF}$-inflations are closed under composition.
	\item $\fraks |_{\bbF}$-deflations are closed under composition.
	\item $(\CC, \bbF, \fraks |_{\bbF})$ satisfies (ET4).
	\item $(\CC, \bbF, \fraks |_{\bbF})$ satisfies (ET4)$^{\op}$.
	\item $(\CC, \bbF, \fraks |_{\bbF})$ is an ET category.
	\end{enumerate}
In this case, $\bbF$ is called \emph{closed}.
\end{prop}

The following are examples of closed subfunctors defined in \cite{HLN21}.

\begin{ex} \cite[Definition 3.18, Proposition 3.19]{HLN21} \label{ex_relative-ET}
Let $(\CC, \bbE, \fraks)$ be an ET category and $\II$ be a subcategory of $\CC$. 
	\begin{enumerate}
	\item We define a closed subfunctor $\bbE_{\II}$ of $\bbE$ as follows.
	\[
	\bbE_{\II}(C,A) = \{ \delta \in \bbE(C,A) \mid \delta \circ - \colon \CC(\II, C) \to \bbE(\II, A); \text{zero morphism}\}.
	\]
	\item We define a closed subfunctor $\bbE^{\II}$ of $\bbE$ as follows.
	\[
	\bbE^{\II}(C,A) = \{ \delta \in \bbE(C,A) \mid - \circ \delta \colon \CC(A,\II) \to \bbE(C,\II); \text{zero morphism}\}.
	\]
	\item We define a closed subfunctor $\bbE^{\II}_{\II}$ of $\bbE$ as follows.
	\[
	\bbE^{\II}_{\II}(C,A) = \bbE^{\II}(C,A) \cap \bbE_{\II}(C,A).
	\]
	\end{enumerate}
We denote $\fraks|_{\bbE_{\II}}$ \resp{$\fraks|_{\bbE^{\II}}$, $\fraks|_{\bbE^{\II}_{\II}}$} by $\fraks_{\II}$ \resp{$\fraks^{\II}$, $\fraks^{\II}_{\II}$}. 
\end{ex}

\begin{rmk}
In this paper, the ET structures defined by $(\CC,$ $\bbE_{\II},$ $\fraks_{\II})$ and $(\CC,$ $\bbE^{\II},$ $\fraks^{\II})$ are called \emph{relative extriangulated structure}, or more simply \emph{relative structure}.
On the other hand, in \cite[Section 2]{FGPPP24}, all extriangulated substructures are called relative extriangulated structure.
\end{rmk}

\begin{rmk}
Let $(\CC, \bbE, \fraks)$ be an ET category and $\II$ be a subcategory of $\CC$. 
Then $\II \subset \Proj_{\bbE_{\II}} \CC$ and $\II \subset \Inj_{\bbE^{\II}} \CC$.
This follows from the long exact sequences in Remark \ref{long_exact_seq}.
\end{rmk}

Now, we consider the approximation theory in ET categories.
In the rest of this subsection, we fix an ET category $(\CC, \bbE, \fraks)$.
We start from a reformulation of extensions in relative structures by using $\II$-epic and $\II$-monic morphisms in ET categories for a subcategory $\II$. 

\begin{lem} \cite[Proposition 3.2]{Ara24} \label{epic_monic}
Let $\II$ be a subcategory of $\CC$. Let $X \xrar{x} Y \xrar{y} Z \lxdrar{\delta} X$ be an $\fraks$-triangle.
	\begin{enumerate}
	\item $\delta \in \bbE_{\II}(Z,X) \iff y$ is $\II$-epic.
	\item $\delta \in \bbE^{\II}(Z,X) \iff x$ is $\II$-monic.
	\end{enumerate}
\end{lem}

\begin{proof}
From Remark \ref{long_exact_seq}, this follows from definitions of $\II$-epic and $\II$-monic.
\end{proof}

\begin{defi} \cite[Definition 3.21]{ZZ18} \label{defi_strongly-finite}
Let $\II, \XX$ be subcategories of $\CC$ where $\II$ is closed under direct summands. 
	\begin{enumerate}
	\item $\II$ is \emph{strongly contravariantly finite} in $\XX$ with respect to $(\CC, \bbE, \fraks)$ if, for any $X \in \XX$, there exists an $\fraks$-deflation $I_X \xrar{g} X$ where $g$ is a right $\II$-approximation.
	\item $\II$ is \emph{strongly covariantly finite} in $\XX$ with respect to $(\CC, \bbE, \fraks)$ if, for any $X \in \XX$, there exists an $\fraks$-inflation $X \xrar{f} I^X$ where $f$ is a left $\II$-approximation.
	\item $\II$ is \emph{strongly functorially finite} in $\XX$ with respect to $(\CC, \bbE, \fraks)$ if,
	$\II$ is both strongly covariantly finite and strongly contravariantly finite in $\XX$. 
	\end{enumerate}
\end{defi}

\begin{rmk}
We do not assume that $\II$ is contained in $\XX$ in Definition \ref{defi_strongly-finite}.
\end{rmk}

\begin{ex}
Assume that $\CC$ has enough projectives \resp{injectives}. 
Then $\Proj \CC$ \resp{$\Inj \, \CC$} is strongly contravariantly \resp{covariantly} finite in $\CC$.
\end{ex}

The following lemma is an ET version of Lemma \ref{lem_bracket_rigidver}.

\begin{lem} \cite[Lemma 3.5]{Ara24} \label{def_of_bracket}
Let $\II \subset \XX$ be subcategories of $\CC$.
	\begin{enumerate}
	\item Assume that $\II$ is strongly covariantly finite in $\XX$.
		\begin{enumerate}[leftmargin=15pt]
		\item For $X \in \XX$, there exists an inflation $i^X \colon X \to I^X$ which is a left $\II$-approximation of $X$.
		Then we obtain the following $\fraks^{\II}$-triangle.
		\[
		X \xrar{i^X} I^{X} \xrar{p^X} X\bracket{1} \xdrar{\lambda^X} X
		\]
		\item For a morphism $x \colon X \to X^{\pr}$ in $\XX$, 
		we define $x\bracket{1} \colon X\bracket{1} \to X^{\pr}\bracket{1}$ as a morphism in $\CC$ 
		which makes the following diagram in $\CC$ commutative.
		\[
		\xy
		(0,8)*+{X \bracket{1}}="1";
		(0,-8)*+{X^{\pr} \bracket{1}}="2";
		(16,8)*+{X}="3";
		(16,-8)*+{X^{\pr}}="4";
		{\ar@{-->}^-{\lambda^X} "1";"3"};		
		{\ar@{-->}^-{\lambda^{X^{\pr}}} "2";"4"};
		{\ar^{x} "3";"4"};
		{\ar_{x \bracket{1}} "1";"2"};
		{\ar@{}|\car "1";"4"};
		\endxy
		\]
		\end{enumerate}
	Then $\bracket{1}$ induces an additive functor $\bracket{1} \colon \XX/[\II] \to \CC/[\II]$.
	Moreover, $\bracket{1}$ is unique up to natural isomorphisms.
	\item Assume that $\II$ is strongly contravariantly finite in $\XX$.
		\begin{enumerate}[leftmargin=15pt]
		\item For $X \in \XX$, there exists a deflation $p_X \colon I_X \to X$ which is a right $\II$-approximation of $X$.
		Then we obtain the following $\fraks_{\II}$-triangle.
		\[
		X \xdrar{\lambda_X}  X\bracket{-1} \xrar{i_X} I_{X} \xrar{p_X} X
		\]
		\item For a morphism $x \colon X \to X^{\pr}$ in $\XX$, we define $x\bracket{-1} \colon X\bracket{-1} \to X^{\pr}\bracket{-1}$ as a morphism in $\CC$ which makes the following diagram commutative.
		\[
		\xy
		(16,8)*+{X \bracket{-1}}="3";
		(16,-8)*+{X^{\pr} \bracket{-1}}="4";
		(0,8)*+{X}="1";
		(0,-8)*+{X^{\pr}}="2";
		{\ar@{-->}^-{\lambda_X} "1";"3"};		
		{\ar@{-->}^-{\lambda_{X^{\pr}}} "2";"4"};
		{\ar^{x  \bracket{-1}} "3";"4"};
		{\ar_{x} "1";"2"};
		{\ar@{}|\car "1";"4"};
		\endxy
		\]
		\end{enumerate}
	Then $\bracket{-1}$ induces an additive functor $\bracket{-1} \colon \XX/[\II] \to \CC/[\II]$.
	Moreover, $\bracket{-1}$ is unique up to natural isomorphisms.
	\end{enumerate}
\end{lem}
\begin{proof}
We prove only (1).
First, for each $X$ in $\CC$, choose an $\fraks$-triangle $X \xrar{i^X} I^{X} \xrar{p^X} X \bracket{1} \xdrar{\lambda^X} X$ where $i^X$ is a left $\II$-approximation. 
For any morphism $x \colon X \to X^{\pr}$ in $\XX$, since $i^X$ is a left $\II$-approximation and (ET3), there exist morphisms $y$ and $i$ which make the following diagram commutative.
\[
\xy
	(0,8)*+{X}="1";
	(0,-8)*+{X^{\pr}}="2";
	(16,8)*+{I^{X}}="3";
	(16,-8)*+{I^{X^{\pr}}}="4";
	(32,8)*+{X \bracket{1}}="5";
	(32,-8)*+{X^{\pr} \bracket{1}}="6";
	(48,8)*+{X}="7";
	(48,-8)*+{X^{\pr}}="8";
	{\ar^{i^X} "1";"3"};
	{\ar^{i^{X^{\pr}}} "2";"4"};
	{\ar^{p^X} "3";"5"};
	{\ar^{p^{X^{\pr}}} "4";"6"};
	{\ar^{x} "1";"2"};
	{\ar^{i} "3";"4"};
	{\ar^{y} "5";"6"};
	{\ar^{x} "7";"8"};
	{\ar@{-->}^{\lambda^X} "5";"7"};
	{\ar@{-->}^{\lambda^{X^{\pr}}} "6";"8"};
	{\ar@{}|\car "1";"4"};
	{\ar@{}|\car "3";"6"};
	{\ar@{}|\car "5";"8"};
\endxy
\]
Then we define $x\bracket{1}$ by $y$.
Assume that morphisms $y,y^{\pr} \colon X \bracket{1} \to X^{\pr} \bracket{1}$ in $\CC$ satisfy
$x \lambda^X = \lambda^{X^{\pr}} y = \lambda^{X^{\pr}} y^{\pr}$.
Then $\lambda^{X^{\pr}} (y-y^{\pr}) =0$. So, $y-y^{\pr} = 0$ in $\CC/[\II]$ and $\bracket{1} \colon \XX \to \CC/[\II]$ is well-defined. 
This induces a functor $\bracket{1} \colon \XX/[\II] \to \CC/[\II]$ because 
$\lambda^{X^{\pr}} x\bracket{1} = x \lambda^X = 0$ for $x \in [\II]$.
Therefore, $y$ is uniquely determined by $x$ up to $[\II]$ and $\bracket{1}$ is an additive functor. 

Uniqueness of $\bracket{1}$ up to natural isomorphisms follows from the diagram below and $n^Xm^X = \id_{X\bracket{1}}$ in $\CC/[\II]$ 
where $X \xrar{j^X} J^X \xrar{q^X} X \bracket{1}^{\pr} \xdrar{{\lambda^{\pr}}^X} X$ is another
$\fraks$-triangle with a left $\II$-approximation $j^X$.
\[
\xy
	(0,8)*+{X}="1";
	(0,-8)*+{X}="2";
	(0,-24)*+{X}="9";
	(16,8)*+{I^X}="3";
	(16,-8)*+{J^X}="4";
	(16,-24)*+{I^{X}}="10";
	(32,8)*+{X \bracket{1}}="5";
	(32,-8)*+{X \bracket{1}^{\pr}}="6";
	(32,-24)*+{X \bracket{1}}="11";
	(48,8)*+{X}="7";
	(48,-8)*+{X}="8";
	(48,-24)*+{X}="12";
	{\ar^{i^X} "1";"3"};
	{\ar^{i^X} "9";"10"};
	{\ar^{j^X} "2";"4"};
	{\ar^{p^X} "3";"5"};
	{\ar^{p^X} "10";"11"};
	{\ar^{q^X} "4";"6"};
	{\ar@{=}^{\id} "1";"2"};
	{\ar^{i} "3";"4"};
	{\ar^{m^X} "5";"6"};
	{\ar@{=}^{\id} "7";"8"};
	{\ar@{=}^{\id} "2";"9"};
	{\ar^{j} "4";"10"};
	{\ar^{n^X} "6";"11"};
	{\ar@{=}^{\id} "8";"12"};
	{\ar@{-->}^{\lambda^X} "5";"7"};
	{\ar@{-->}^{\lambda^X} "11";"12"};
	{\ar@{-->}^{{\lambda^{\pr}}^X} "6";"8"};
	{\ar@{}|\car "1";"4"};
	{\ar@{}|\car "3";"6"};
	{\ar@{}|\car "5";"8"};
	{\ar@{}|\car "2";"10"};
	{\ar@{}|\car "4";"11"};
	{\ar@{}|\car "6";"12"};
\endxy
\]
\end{proof}

\begin{rmk} \label{rmk_natiso_bracket}
In the proof of Lemma \ref{def_of_bracket}, we used the following commutative diagram.
\begin{align}
\xy
	(0,8)*+{X}="1";
	(0,-8)*+{X}="2";
	(16,8)*+{I^{X}}="3";
	(16,-8)*+{J^X}="4";
	(32,8)*+{X \bracket{1}}="5";
	(32,-8)*+{X \bracket{1}^{\pr}}="6";
	(48,8)*+{X}="7";
	(48,-8)*+{X}="8";
	{\ar^{i^X} "1";"3"};
	{\ar^{j^X} "2";"4"};
	{\ar^{p^X} "3";"5"};
	{\ar^{q^X} "4";"6"};
	{\ar@{=}^{\id} "1";"2"};
	{\ar^{i} "3";"4"};
	{\ar^{m^X} "5";"6"};
	{\ar@{=}^{\id} "7";"8"};
	{\ar@{-->}^{\lambda^X} "5";"7"};
	{\ar@{-->}^{{\lambda^{\pr}}^X} "6";"8"};
	{\ar@{}|\car "1";"4"};
	{\ar@{}|\car "3";"6"};
	{\ar@{}|\car "5";"8"};
\endxy \label{diag_natural-iso-bracket}
\end{align}
Then we denote $m^X$ in $\CC/[\II]$ (that is $\ul{m^X}$ in Notation \ref{nota_ul}) by $\mu_X$ because $\mu$ induces natural isomorphism $\bracket{1} \Rightarrow \bracket{1}^{\pr}$ 
where $\bracket{1}, \bracket{1}^{\pr} \colon \XX/[\II] \to \CC/[\II]$.
\end{rmk}

\begin{nota}
Let $\II$ and $\XX$ be subcategories of $\CC$.
	\begin{enumerate}
	\item We denote $\Cone_{\bbE^{\II}}(\XX, \II)$ by $\XX \bracket{1}_{\II}$.
	In particular, $\II \subset \XX \bracket{1}_{\II}$.
	\item We denote $\CoCone_{\bbE_{\II}}(\II, \XX)$ by $\XX \bracket{-1}_{\II}$.
	In particular, $\II \subset \XX \bracket{-1}_{\II}$.
	\end{enumerate}
If there is no confusion, we often drop $\II$ of $\bracket{1}_{\II}$ and $\bracket{-1}_{\II}$.

$\{ X\bracket{1} \mid X \in \XX \}$ in Lemma \ref{def_of_bracket} and 
$\XX \bracket{1}$ are same up to isomorphisms in $\CC/[\II]$.

Note that we can define $\XX \bracket{1}$ \resp{$\XX \bracket{-1}$}
even if $\II$ is not strongly covariantly \resp{contravariantly} finite in $\XX$.
\end{nota}

\begin{lem} \cite[Proposition 3.2]{Ara24}
Let $\II$ be a subcategory of $\CC$.
	\begin{enumerate}
	\item If $\II$ is strongly contravariantly finite in $\CC$, 
	then $\II = \Proj_{\bbE_{\II}} \!\CC$ and $(\CC, \bbE_{\II}, \fraks_{\II})$ has enough projectives.
	\item If $\II$ is strongly covariantly finite in $\CC$, 
	then $\II = \Inj_{\bbE^{\II}} \!\CC$ and $(\CC, \bbE^{\II}, \fraks^{\II})$ has enough injectives.
	\end{enumerate}
\end{lem}
\begin{proof}
We only prove (2). 
Let $\II^{\pr} = \Inj_{\bbE^{\II}} \CC$. By definition of relative structure, $\II \subset \II^{\pr}$. 
On the other hand, for any $I^{\pr} \in \II^{\pr}$, 
there exists an $\fraks^{\II}$-triangle $I^{\pr} \rar I \rar Z \drar I^{\pr}$ 
where $I \in \II$ since $\II$ is strongly covariantly finite in $\CC$.  
Since $I^{\pr} \in \II^{\pr}$ and $\II$ is closed under direct summands, 
the above $\fraks$-triangle splits and $I^{\pr} \in \II$.
\end{proof}

ET categories are different from triangulated categories because not every morphism has a cone or cocone. However, we may sometimes replace any morphisms by inflations \resp{deflations} up to ideal quotient in the following meanings.

\begin{lem} \cite[Proposition 1.20]{LN19} \label{inflation}
Take a morphism $f \colon X \to X^{\pr}$. Let $X \xrar{x} E \xrar{y} Y \xdrar{\delta} X$ be an $\fraks$-triangle and 
$X^{\pr} \xrar{x^{\pr}} E^{\prime} \xrar{y^{\pr}} Y \xdrar{f \delta} X^{\pr}$ be a realization of $f \delta$. 
Then there exists a morphism $g \colon E \to E^{\prime}$ which satisfies the following two conditions.
	\begin{enumerate}[label=(\roman*)]
	\item $g$ makes the following diagram commutative.
	\[
	\xy
	(0,8)*+{X}="00";
	(0,-8)*+{X^{\prime}}="01";
	(16,8)*+{E}="10";
	(16,-8)*+{E^{\prime}}="11";
	(32,8)*+{Y}="20";
	(32,-8)*+{Y}="21";
	(48,8)*+{X}="30";
	(48,-8)*+{X^{\prime}}="31";
	{\ar@{->}^{x} "00";"10"};
	{\ar@{->}^{y} "10";"20"};
	{\ar@{-->}^{\delta} "20";"30"};
	{\ar@{->}^{x^{\prime}} "01";"11"};
	{\ar@{->}^{y^{\prime}} "11";"21"};
	{\ar@{-->}^{f\delta} "21";"31"};
	{\ar@{->}^{f} "00";"01"};
	{\ar@{->}^{g} "10";"11"};
	{\ar@{=} "20";"21"};
	{\ar@{->}^{f} "30";"31"};
	{\ar@{}|\circlearrowright "00";"11"};
	{\ar@{}|\circlearrowright "10";"21"};
	{\ar@{}|\circlearrowright "20";"31"};
	\endxy
	\]
	\item
	$
	X \xrightarrow{\msize{0.5}{\begin{bmatrix} f \\ x \\ \end{bmatrix}}} X^{\prime} \oplus E \xrar{\msize{0.5}{[-x^{\pr} \ g]}} E^{\pr} \xdrar{\delta y^{\pr}} X
	$ \
	is an $\fraks$-triangle, in particular, ${\msize{0.8}{\begin{bmatrix} f \\ x \\ \end{bmatrix}}}$ is an inflation.
	\end{enumerate}
\end{lem}
\begin{proof}
See \cite[Proposition 1.20]{LN19}.
\end{proof}

The following statement is also useful.

\begin{lem} \cite[Proposition 3.17]{NP19}  \label{happel_diagram}
Let 
$X \xrar{a} Y \xrar{b^{\pr}} Z \lxdrar{\delta_1} X$,
$ X \xrar{c} Z^{\pr} \xrar{d} Z^{\pr\pr} \lxdrar{\delta_2} X$ and
$X^{\pr} \xrar{a^{\pr}} Y \xrar{b} Z^{\pr} \lxdrar{\delta_3} X^{\pr}$
be $\fraks$-triangles where $c = ba$.
Then there exists an $\fraks$-triangle $X^{\pr} \xrar{f} Z \xrar{g} Z^{\pr\pr} \lxdrar{\delta} X^{\pr}$ 
which makes the following diagram commutative.
\[
\xy
(0,24)*+{}="11";
(0,8)*+{X}="21";
(0,-8)*+{X}="31";
(0,-24)*+{}="41";
(16,24)*+{X^{\pr}}="12";
(16,8)*+{Y}="22";
(16,-8)*+{Z^{\pr}}="32";
(16,-24)*+{X^{\pr}}="42";
(32,24)*+{X^{\pr}}="13";
(32,8)*+{Z}="23";
(32,-8)*+{Z^{\pr\pr}}="33";
(32,-24)*+{X^{\pr}}="43";
(48,24)*+{}="14";
(48,8)*+{X}="24";
(48,-8)*+{X}="34";
(48,-24)*+{Y}="44";
{\ar@{=} "12";"13"};
{\ar^-{a} "21";"22"};
{\ar^-{b^{\pr}} "22";"23"};
{\ar@{-->}^{\delta_1} "23";"24"};
{\ar^{c} "31";"32"};
{\ar^{d} "32";"33"};
{\ar@{-->}^{\delta_2} "33";"34"};
{\ar@{=} "42";"43"};
{\ar^{-a^{\pr}} "43";"44"};
{\ar@{=} "21";"31"};
{\ar^{a^{\pr}} "12";"22"};
{\ar^{b} "22";"32"};
{\ar@{-->}^{\delta_3} "32";"42"};
{\ar^{f} "13";"23"};
{\ar^{g} "23";"33"};
{\ar@{-->}^{\delta} "33";"43"};
{\ar@{=} "24";"34"};
{\ar^{a} "34";"44"};
{\ar@{}|\car "12";"23"};
{\ar@{}|\car "21";"32"};
{\ar@{}|\car "22";"33"};
{\ar@{}|\car "23";"34"};
{\ar@{}|\car "32";"43"};
{\ar@{}|\car "33";"44"};
\endxy
\]
\end{lem}
\begin{proof}
This is a dual statement of  \cite[Proposition 3.17]{NP19}.
\end{proof}

\begin{cor} \label{uniqueness_of_cone}
Let $\II$ be a strongly covariantly finite subcategory in $\CC$ and $a \colon X \to Y$ be a morphism in $\CC$. 
Take $\fraks^{\II}$-triangles $X \xrar{i^X} I^X \xrar{p^X} X\bracket{1} \xdrar{\lambda^X} X$
and $X \xrar{j^X} J^X \xrar{q^X} X\bracket{1}^{\pr} \xdrar{{\lambda^{\pr}}^X} X \vphantom{\Big(}$ 
where $i^X$ and $j^X$ are left $\II$-approximations.

There exist the following $\fraks^{\II}$-triangles from Lemma \ref{inflation}.
	\begin{gather*}
	X \xrar{\msize{0.5}{\begin{bmatrix} a \\ i^X \\ \end{bmatrix}}} Y \oplus I^X \xrar{\wt{b}} 
	C^a \xdrar{\wt{\delta}} X \\
	X \xrar{\msize{0.5}{\begin{bmatrix} a \\ j^X \\ \end{bmatrix}}} Y \oplus J^X \xrar{\wt{b^{\pr}}} 
	C^{a^{\pr}} \xdrar{\wt{\delta^{\pr}}} X \\
	X \xrar{\msize{0.5}{\begin{bmatrix} a \\ i^X \\ j^X \\ \end{bmatrix}}} Y \oplus I^X \oplus J^X \xrar{\wt{b^{\pr\pr}}} 
	C^{a^{\pr\pr}} \xdrar{\wt{\delta^{\pr\pr}}} X
	\end{gather*}
Then $C^{a^{\pr\pr}} \cong C^a \oplus J^X \cong C^{a^{\pr}} \oplus I^X$. 
In particular, $C^a$ is uniquely determined up to isomorphisms in $\CC/[\II]$ 
and does not depend on choices of $\fraks^{\II}$-triangle $X \xrar{i^X} I^X \xrar{p^X} X\bracket{1} \xdrar{\lambda^X} \vphantom{\Big(}X$.
\end{cor}
\begin{proof}
From Lemma \ref{happel_diagram}, there exists a commutative diagram.
\[
\xy
(-8,24)*+{}="11";
(-8,8)*+{X}="21";
(-8,-8)*+{X}="31";
(-8,-24)*+{}="41";
(16,24)*+{J^X}="12";
(16,8)*+{Y\! \oplus \! I^X \! \oplus \!J^X}="22";
(16,-8)*+{Y \! \oplus \!I^X}="32";
(16,-24)*+{J^X}="42";
(40,24)*+{J^X}="13";
(40,8)*+{C^{a^{\pr\pr}}}="23";
(40,-8)*+{C^a}="33";
(40,-24)*+{J^X}="43";
(64,24)*+{}="14";
(64,8)*+{X}="24";
(64,-8)*+{X}="34";
{\ar@{=} "12";"13"};
{\ar^-{\msize{0.5}{\begin{bmatrix} a \\ i^X \\ j^X \end{bmatrix}}} "21";"22"};
{\ar^-{\wt{b^{\pr\pr}}} "22";"23"};
{\ar@{-->}^{\wt{\delta^{\pr\pr}}} "23";"24"};
{\ar^-{\msize{0.5}{\begin{bmatrix} a \\ i^X \end{bmatrix}}} "31";"32"};
{\ar^-{\wt{b}} "32";"33"};
{\ar@{-->}^{\wt{\delta}} "33";"34"};
{\ar@{=} "42";"43"};
{\ar@{=} "21";"31"};
{\ar^-{} "12";"22"};
{\ar^-{} "22";"32"};
{\ar@{-->}^{0} "32";"42"};
{\ar^-{} "13";"23"};
{\ar^-{} "23";"33"};
{\ar@{-->}^-{\gamma} "33";"43"};
{\ar@{=} "24";"34"};
%
{\ar@{}|\car "12";"23"};
{\ar@{}|\car "21";"32"};
{\ar@{}|\car "22";"33"};
{\ar@{}|\car "23";"34"};
{\ar@{}|\car "32";"43"};
\endxy
\]
Since $\bbE^{\II}(C^a, J^X) = 0$, $C^{a^{\pr\pr}} \cong C^a \oplus J^X$. One can show $C^{a^{\pr\pr}} \cong C^{a^{\pr}} \oplus I^X$ in the same way.
\end{proof}

\begin{cor} \cite[Corollary 3.7]{Nak18} \label{cor_iso-in-ul}
Let $\II \subset \ZZ$ be subcategories in $\CC$ and assume that $\II$ is strongly functorially finite in $\ZZ$.
If a morphism $f \colon X \to Y$ in $\ZZ$ is an isomorphism in $\CC/[\II]$, then 
there exist $I,J \in \II$ and an isomorphism $\msize{0.8}{\begin{bmatrix} f & j \\ i & k \end{bmatrix}} \colon X \oplus J \to Y \oplus I$.
\end{cor}
\begin{proof}
Take an inflation $i^X \colon X \to I^X$ which is also a left $\II$-approximation.
Since $f$ is an isomorphism in $\CC/[\II]$, $\msize{0.8}{\begin{bmatrix} f \\ i^X \end{bmatrix}} \colon X \to Y \oplus I^X$ is both an inflation and a section.
We denote $\msize{0.8}{\begin{bmatrix} f \\ i^X \end{bmatrix}}$ by $\wt{f}$.
Next, there exists a deflation $p \colon I_Y \to Y \oplus I^X$ which is also a right $\II$-approximation.
Since $\wt{f}$ is also an isomorphism in $\CC/[\II]$, $\msize{0.8}{\begin{bmatrix} \wt{f} & p \end{bmatrix}} \colon X \oplus I_Y \to Y \oplus I^X$ is both an deflation and a retraction.
From Lemma \ref{happel_diagram}, there exists the following commutative diagram in $\CC$.
\[
\xy
(0,8)*+{X}="21";
(0,-8)*+{X}="31";
(16,24)*+{J^{\pr}}="12";
(16,8)*+{X \oplus I_Y}="22";
(16,-8)*+{Y \oplus I^X}="32";
(16,-24)*+{J^{\pr}}="42";
(32,24)*+{J^{\pr}}="13";
(32,8)*+{I_Y}="23";
(32,-8)*+{J}="33";
(32,-24)*+{J^{\pr}}="43";
(48,8)*+{X}="24";
(48,-8)*+{X}="34";
{\ar@{=} "12";"13"};
{\ar^-{\msize{0.6}{\begin{bmatrix} 1 \\ 0 \end{bmatrix}}} "21";"22"};
{\ar^-{\msize{0.6}{\begin{bmatrix} 0 & 1 \end{bmatrix}}} "22";"23"};
{\ar@{-->}^{0} "23";"24"};
{\ar^-{\wt{f}} "31";"32"};
{\ar^-{} "32";"33"};
{\ar@{-->}^{0} "33";"34"};
{\ar@{=} "42";"43"};
{\ar@{=} "21";"31"};
{\ar^{} "12";"22"};
{\ar^{\msize{0.6}{\begin{bmatrix} \wt{f} & p \end{bmatrix}}} "22";"32"};
{\ar@{-->}^{0} "32";"42"};
{\ar^{} "13";"23"};
{\ar^{} "23";"33"};
{\ar@{-->}^{0} "33";"43"};
{\ar@{=} "24";"34"};
{\ar@{}|\car "12";"23"};
{\ar@{}|\car "21";"32"};
{\ar@{}|{\phantom{XX}\car} "22";"33"};
{\ar@{}|\car "23";"34"};
{\ar@{}|\car "32";"43"};
\endxy
\]

Because $\II$ is closed under direct summands, $J \in \II$.
By taking a section $\msize{0.8}{\begin{bmatrix} j \\ k \end{bmatrix}} \colon J \to Y \oplus I^X$,
$\msize{0.8}{\begin{bmatrix} f & j \\ i^X & k \end{bmatrix}}$ induces an isomorphism.
\end{proof}

At the end of this subsection, we add the following lemma which is used in section \ref{triangulated}.

\begin{lem}\label{bracket_equiv}
Let $\II \subset \XX$ be subcategories of $\CC$ and suppose that $\II$ is strongly covariantly finite in $\XX$.
Assume that $\bbE^{\II}(\II, \XX) = 0$, then $\bracket{1} \colon \XX/[\II] \to \XX\bracket{1}/[\II]$ is an equivalence.
In particular, an $\fraks^{\II}$-triangle $X \rar I^X \rar X\bracket{1} \xdrar{\lambda^X} X\vphantom{\bigg(}$ is an $\fraks^{\II}_{\II}$-triangle.
\end{lem}
\begin{proof}
Since $\bracket{1} \colon \XX/[\II] \to \XX\bracket{1}/[\II]$ is essentially surjective by definition and one can directly prove $\bracket{1}$ is full,
we only show that $\bracket{1}$ is faithful.
Let $x \colon X \to X^{\pr}$ be a morphism in $\XX$ and take $\fraks^{\II}$-triangles 
$X \rar I^X \rar X\bracket{1} \xdrar{\lambda^X} X\vphantom{\Big(}$ and 
$X^{\pr} \rar I^{X^{\pr}} \rar X^{\pr} \bracket{1} \xdrar{\lambda^{X^{\pr}}} X^{\pr} \vphantom{\Big(}$.
Then we obtain a morphism $x\bracket{1} \colon X\bracket{1} \to X^{\pr}\bracket{1}$ where $x \lambda^X = \lambda^{X^{\pr}} x\bracket{1}$. 
If $x\bracket{1} =0$ in $\CC/[\II]$, $x \lambda^X =0$ from $\bbE^{\II}(\II, \XX) =0$. 
Thus, $x$ factors through $I^X$.
\end{proof}

\subsection{Pretriangulated categories} 

In this subsection, we introduce right triangulated categories, left triangulated categories and pretriangulated categories in \cite{BR07}.
First, we define right triangles and left triangles like as distinguished triangles in triangulated categories.

\begin{defi} \cite[II.1]{BR07}
Let $\Sigma, \Omega \colon \CC \to \CC$ be additive (endo)functors.
	\begin{enumerate}
	\item We define a category $\mathsf{RTri}$ as follows.
		\begin{enumerate}[leftmargin=20pt]
		\item Objects are sequences in $\CC$ of the form $ X \xrar{f} Y \xrar{g} Z \xrar{h} \Sigma X$.
		\item Morphisms are triplets $(x, y, z)$ which make the following diagram commutative.
		\[
		\xy
		(-30,8)*+{}="0";
		(0,8)*+{X_1}="1";
		(0,-8)*+{X_2}="2";
		(16,8)*+{Y_1}="3";
		(16,-8)*+{Y_2}="4";
		(32,8)*+{Z_1}="5";
		(32,-8)*+{Z_2}="6";
		(48,8)*+{\Sigma X_1}="7";
		(48,-8)*+{\Sigma X_2}="8";
		{\ar^{f_1} "1";"3"};
		{\ar^{g_1} "3";"5"};	
		{\ar^{h_1} "5";"7"};
		{\ar^{f_2} "2";"4"};
		{\ar^{g_2} "4";"6"};
		{\ar^{h_2} "6";"8"};
		{\ar^{x} "1";"2"};
		{\ar^{y} "3";"4"};
		{\ar^{z} "5";"6"};
		{\ar^{\Sigma x} "7";"8"};
		{\ar@{}|\car "1";"4"};
		{\ar@{}|\car "3";"6"};
		{\ar@{}|\car "5";"8"};
		\endxy
		\]
		\end{enumerate}
	\item We define a category $\mathsf{LTri}$ as follows.
		\begin{enumerate}[leftmargin=20pt]
		\item Objects are sequences in $\CC$ of the form $ \Omega Z \xrar{h^{\pr}} X \xrar{f} Y \xrar{g} Z$.
		\item Morphisms are triplets $(x, y, z)$ which satisfy the following commutative diagram.
		\[
		\xy
		(-30,8)*+{}="0";
		(0,8)*+{\Omega Z_1}="1";
		(0,-8)*+{\Omega Z_2}="2";
		(16,8)*+{X_1}="3";
		(16,-8)*+{X_2}="4";
		(32,8)*+{Y_1}="5";
		(32,-8)*+{Y_2}="6";
		(48,8)*+{Z_1}="7";
		(48,-8)*+{Z_2}="8";
		{\ar^{h_1^{\pr}} "1";"3"};
		{\ar^{f_1} "3";"5"};	
		{\ar^{g_1} "5";"7"};
		{\ar^{h_2^{\pr}} "2";"4"};
		{\ar^{f_2} "4";"6"};
		{\ar^{g_2} "6";"8"};
		{\ar^{\Omega z} "1";"2"};
		{\ar^{x} "3";"4"};
		{\ar^{y} "5";"6"};
		{\ar^{z} "7";"8"};
		{\ar@{}|\car "1";"4"};
		{\ar@{}|\car "3";"6"};
		{\ar@{}|\car "5";"8"};
		\endxy
		\]
		\end{enumerate}
	\end{enumerate}
\end{defi}

Now, let us define right triangulated categories and left triangulated categories.

\begin{defi} \cite[II.1]{BR07}
	\begin{enumerate}
	\item A \emph{right triangulated category} is a triplet $(\CC, \Sigma, \nabla)$ where
		\begin{enumerate}
		\item $\Sigma \colon \CC \to \CC$ be an additive functor.
		\item $\nabla$ is a full subcategory of $\mathsf{RTri}$.
		\item $(\CC, \Sigma, \nabla)$ satisfies all of the axioms of a triangulated category except that $\Sigma$ is not necessarily an equivalence.
		\end{enumerate}
	\item A \emph{left triangulated category} is a triplet $(\CC, \Omega, \Delta)$ where
		\begin{enumerate}
		\item $\Omega \colon \CC \to \CC$ be an additive functor.
		\item $\Delta$ is a full subcategory of $\mathsf{LTri}$.
		\item $(\CC, \Omega, \Delta)$ satisfies all of the axioms of a triangulated category except that $\Omega$ is not necessarily an equivalence.
		\end{enumerate}
	\end{enumerate}
\end{defi}

\begin{rmk} \label{axioms_of_RTri}
For the convenience of the reader, we list the axioms of right triangulated categories $(\CC, \Sigma, \nabla)$ below.
	\begin{itemize}[leftmargin=40pt]
	\item[(rTR0)] $\nabla$ is closed under isomorphisms.
	\item[(rTR1)]
		\begin{enumerate}[label=(\roman*)]
		\item For any $X \in \CC$, the sequence $X \xrar{\text{id}} X \rar 0 \rar \Sigma X$ is in $\nabla$.
		\item For any morphism $f \colon X \to Y$, there exists the sequence $X \xrar{f} Y \rar Z \rar \Sigma X$ in $\nabla$.
		\end{enumerate}
	\item[(rTR2)] Let $X \xrar{f} Y \xrar{g} Z \xrar{h} \Sigma X$ be a sequence in $\nabla$, 
	then $Y \xrar{g} Z \xrar{h} \Sigma X \xrar{- \Sigma f} \Sigma Y$ is also in $\nabla$.
	\item[(rTR3)] Assume that there exists a commutative diagram where each row is in $\nabla$.
	\[
	\xy
	(0,8)*+{X_1}="11";
	(16,8)*+{Y_1}="12";
	(32,8)*+{Z_1}="13";
	(48,8)*+{\Sigma X_1}="14";
	(0,-8)*+{X_2}="21";
	(16,-8)*+{Y_2}="22";
	(32,-8)*+{Z_2}="23";
	(48,-8)*+{\Sigma X_2}="24";
	{\ar^{f_1} "11";"12"};
	{\ar^{g_1} "12";"13"};	
	{\ar^{h_1} "13";"14"};
	{\ar^{f_2} "21";"22"};
	{\ar^{g_2} "22";"23"};
	{\ar^{h_2} "23";"24"};
	{\ar^{x} "11";"21"};
	{\ar^{y} "12";"22"};
	{\ar^{\Sigma x} "14";"24"};
	{\ar@{}|\car "11";"22"};
	\endxy
	\]
	Then there exists a morphism $z \colon Z_1 \to Z_2$ which makes the following diagram commutative.
	\[
	\xy
	(0,8)*+{X_1}="11";
	(16,8)*+{Y_1}="12";
	(32,8)*+{Z_1}="13";
	(48,8)*+{\Sigma X_1}="14";
	(0,-8)*+{X_2}="21";
	(16,-8)*+{Y_2}="22";
	(32,-8)*+{Z_2}="23";
	(48,-8)*+{\Sigma X_2}="24";
	{\ar^{f_1} "11";"12"};
	{\ar^{g_1} "12";"13"};	
	{\ar^{h_1} "13";"14"};
	{\ar^{f_2} "21";"22"};
	{\ar^{g_2} "22";"23"};
	{\ar^{h_2} "23";"24"};
	{\ar^{x} "11";"21"};
	{\ar^{y} "12";"22"};
	{\ar^{z} "13";"23"};
	{\ar^{\Sigma x} "14";"24"};
	{\ar@{}|\car "11";"22"};
	{\ar@{}|\car "12";"23"};
	{\ar@{}|\car "13";"24"};	
	\endxy
	\]
	\item[(rTR4)]
	Assume that $a^{\pr\pr} = a^{\pr} a$ and $X \xrar{a} Y \xrar{b} C \xrar{c} \Sigma X$, 
	$Y \xrar{a^{\pr}} Z \xrar{b^{\pr}} D \xrar{c^{\pr}} \Sigma Y$, 
	$X \xrar{a^{\pr\pr}} Z \xrar{b^{\pr\pr}} E \xrar{c^{\pr\pr}} \Sigma X$ are in $\nabla$.
	Then there exists $C \xrar{s} E \xrar{t} D \xrar{u} \Sigma C$ in $\nabla$ which makes the following diagram commutative.
	\[
	\xy
	(0,24)*+{X}="11";
	(16,24)*+{Y}="12";
	(32,24)*+{C}="13";
	(48,24)*+{\Sigma X}="14";
	(0,8)*+{X}="21";
	(16,8)*+{Z}="22";
	(32,8)*+{E}="23";
	(48,8)*+{\Sigma X}="24";
	(0,-8)*+{}="31";
	(16,-8)*+{D}="32";
	(32,-8)*+{D}="33";
	(48,-8)*+{\Sigma Y}="34";
	(0,-24)*+{}="41";
	(16,-24)*+{\Sigma Y}="42";
	(32,-24)*+{\Sigma C}="43";
	(48,-24)*+{}="44";
	{\ar^{\Sigma b} "42";"43"};
	{\ar^{a} "11";"12"};
	{\ar^{b} "12";"13"};
	{\ar^{c} "13";"14"};
	{\ar^{a^{\pr\pr}} "21";"22"};
	{\ar^{b^{\pr\pr}} "22";"23"};
	{\ar^{c^{\pr\pr}} "23";"24"};
	{\ar@{=} "32";"33"};
	{\ar^{c^{\pr}} "33";"34"};
	{\ar@{=} "11";"21"};
	{\ar^{a^{\pr}} "12";"22"};
	{\ar^{b^{\pr}} "22";"32"};
	{\ar^{c^{\pr}} "32";"42"};
	{\ar^{s} "13";"23"};
	{\ar^{t} "23";"33"};
	{\ar^{u} "33";"43"};
	{\ar@{=} "14";"24"};
	{\ar^{\Sigma a} "24";"34"};
	{\ar@{}|\car "12";"23"};
	{\ar@{}|\car "11";"22"};
	{\ar@{}|\car "22";"33"};
	{\ar@{}|\car "23";"34"};
	{\ar@{}|\car "32";"43"};
	{\ar@{}|\car "13";"24"};
	\endxy
	\]
	\end{itemize}
\end{rmk}

\begin{rmk} \label{rmk_facts_for_rtri}
Assume that $(\CC, \Sigma, \nabla)$ satisfies from (rTR0) to (rTR3). 
Then one can show the following statements like triangulated categories.
	\begin{enumerate}
	\item For any right triangle $X \xrar{f} Y \xrar{g} Z \xrar{h} \Sigma X$,
	\[
	\CC(-,X) \xrar{f \circ -} \CC(-,Y) \xrar{g \circ -} \CC(-,Z)
	\]
	is exact.
	\item \label{rmk_uniqueness_of_cone}
	Assume that there exists a commutative diagram where each row is in $\nabla$ and $x,y$ are isomorphisms.
	\[
	\xy
	(0,8)*+{X_1}="11";
	(16,8)*+{Y_1}="12";
	(32,8)*+{Z_1}="13";
	(48,8)*+{\Sigma X_1}="14";
	(0,-8)*+{X_2}="21";
	(16,-8)*+{Y_2}="22";
	(32,-8)*+{Z_2}="23";
	(48,-8)*+{\Sigma X_2}="24";
	{\ar^{f_1} "11";"12"};
	{\ar^{g_1} "12";"13"};	
	{\ar^{h_1} "13";"14"};
	{\ar^{f_2} "21";"22"};
	{\ar^{g_2} "22";"23"};
	{\ar^{h_2} "23";"24"};
	{\ar^{x}_{\vsim} "11";"21"};
	{\ar^{y}_{\vsim} "12";"22"};
	{\ar^{\Sigma x}_{\vsim} "14";"24"};
	{\ar@{}|\car "11";"22"};
	\endxy
	\]
	Then there exists an isomorphism $z \colon Z_1 \to Z_2$ which makes the following diagram commutative.
	\[
	\xy
	(0,8)*+{X_1}="11";
	(16,8)*+{Y_1}="12";
	(32,8)*+{Z_1}="13";
	(48,8)*+{\Sigma X_1}="14";
	(0,-8)*+{X_2}="21";
	(16,-8)*+{Y_2}="22";
	(32,-8)*+{Z_2}="23";
	(48,-8)*+{\Sigma X_2}="24";
	{\ar^{f_1} "11";"12"};
	{\ar^{g_1} "12";"13"};	
	{\ar^{h_1} "13";"14"};
	{\ar^{f_2} "21";"22"};
	{\ar^{g_2} "22";"23"};
	{\ar^{h_2} "23";"24"};
	{\ar^{x}_{\vsim} "11";"21"};
	{\ar^{y}_{\vsim} "12";"22"};
	{\ar^{z}_{\vsim} "13";"23"};
	{\ar^{\Sigma x}_{\vsim} "14";"24"};
	{\ar@{}|\car "11";"22"};
	{\ar@{}|\car "12";"23"};
	{\ar@{}|\car "13";"24"};	
	\endxy
	\]
	\end{enumerate}
\end{rmk}

Finally, we define pretriangulated categories.

\begin{defi} \cite[II.1]{BR07} \label{defi_pretri-in-BR}
	$(\CC, \Sigma, \Omega, \nabla, \Delta)$ is a \emph{pretriangulated category} if it satisfies the following conditions.
	\begin{enumerate}[label=(\roman*)]
	\item $(\Sigma, \Omega)$ is an adjoint pair of additive endofunctors $\Sigma, \Omega \colon \CC \to \CC$.
	\\ Now, let $\alpha$ be a unit and $\beta$ be a counit.
	\item $(\CC, \Sigma, \nabla)$ is a right triangulated category.
	\item $(\CC, \Omega, \Delta)$ is a left triangulated category.
	\item For any commutative diagrams in $\CC$
	\[
	\xy
	(0,8)*+{X_1}="1";
	(0,-8)*+{\Omega Z_2}="2";
	(14,8)*+{Y_1}="3";
	(14,-8)*+{X_2}="4";
	(28,8)*+{Z_1}="5";
	(28,-8)*+{Y_2}="6";
	(42,8)*+{\Sigma X_1}="7";
	(42,-8)*+{Z_2}="8";
	{\ar^{f_1} "1";"3"};
	{\ar^{g_1} "3";"5"};
	{\ar^{h_1} "5";"7"};
	{\ar^{h_2^{\pr}} "2";"4"};
	{\ar^{f_2} "4";"6"};
	{\ar^{g_2} "6";"8"};
	{\ar^{s} "1";"2"};
	{\ar^{t} "3";"4"};
	{\ar^{\beta_{Z_2} \circ \Sigma s} "7";"8"};
	{\ar@{}|\car "1";"4"};
	\endxy
	\xy
	(0,8)*+{X_1}="1";
	(0,-8)*+{\Omega Z_2}="2";
	(14,8)*+{Y_1}="3";
	(14,-8)*+{X_2}="4";
	(28,8)*+{Z_1}="5";
	(28,-8)*+{Y_2}="6";
	(42,8)*+{\Sigma X_1}="7";
	(42,-8)*+{Z_2}="8";
	{\ar^{f_1} "1";"3"};
	{\ar^{g_1} "3";"5"};
	{\ar^{h_1} "5";"7"};
	{\ar^{h_2^{\pr}} "2";"4"};
	{\ar^{f_2} "4";"6"};
	{\ar^{g_2} "6";"8"};
	{\ar_{u^{\pr}} "7";"8"};
	{\ar_{t^{\pr}} "5";"6"};
	{\ar_{\Omega u^{\pr} \circ \alpha_{X_1}} "1";"2"};
	%
	{\ar@{}|\car "5";"8"};
	\endxy
	\]
	where $X_1 \xrar{f_1} Y_1 \xrar{g_1} Z_1 \xrar{h_1} \Sigma X_1$ is a right triangle and
	$\Omega Z_2 \xrar{h_2^{\pr}} X_2 \xrar{f_2} Y_2 \xrar{g_2} Z_2$ is a left triangle,
	then there exist morphisms $u \colon Z_1 \to Y_2$ and $s^{\pr} \colon Y_1 \to X_2$ 
	which make following diagrams commutative.
	\[
	\xy
	(0,8)*+{X_1}="1";
	(0,-8)*+{\Omega Z_2}="2";
	(14,8)*+{Y_1}="3";
	(14,-8)*+{X_2}="4";
	(28,8)*+{Z_1}="5";
	(28,-8)*+{Y_2}="6";
	(42,8)*+{\Sigma X_1}="7";
	(42,-8)*+{Z_2}="8";
	{\ar^{f_1} "1";"3"};
	{\ar^{g_1} "3";"5"};
	{\ar^{h_1} "5";"7"};
	{\ar^{h_2^{\pr}} "2";"4"};
	{\ar^{f_2} "4";"6"};
	{\ar^{g_2} "6";"8"};
	{\ar^{s} "1";"2"};
	{\ar^{t} "3";"4"};
	{\ar^{u} "5";"6"};
	{\ar^{\beta_{Z_2} \circ \Sigma s} "7";"8"};
	{\ar@{}|\car "1";"4"};
	{\ar@{}|\car "3";"6"};
	{\ar@{}|\car "5";"8"};
	\endxy
	\xy
	(0,8)*+{X_1}="1";
	(0,-8)*+{\Omega Z_2}="2";
	(14,8)*+{Y_1}="3";
	(14,-8)*+{X_2}="4";
	(28,8)*+{Z_1}="5";
	(28,-8)*+{Y_2}="6";
	(42,8)*+{\Sigma X_1}="7";
	(42,-8)*+{Z_2}="8";
	{\ar^{f_1} "1";"3"};
	{\ar^{g_1} "3";"5"};
	{\ar^{h_1} "5";"7"};
	{\ar^{h_2^{\pr}} "2";"4"};
	{\ar^{f_2} "4";"6"};
	{\ar^{g_2} "6";"8"};
	{\ar_{u^{\pr}} "7";"8"};
	{\ar_{t^{\pr}} "5";"6"};
	{\ar_{s^{\pr}} "3";"4"};
	{\ar_{\Omega u^{\pr} \circ \alpha_{X_1}} "1";"2"};
	{\ar@{}|\car "1";"4"};
	{\ar@{}|\car "3";"6"};
	{\ar@{}|\car "5";"8"};
	\endxy
	\]
	\end{enumerate}
\end{defi}

\begin{ex} \cite[II.1]{BR07}
	\begin{enumerate}
	\item A triangulated category $(\CC, [1], \triangle)$ is a pretriangulated category $(\CC, [1], [-1],$ 
	$\triangle, \triangle)$.
	\item Assume that $\AA$ is an abelian category. 
	Let $\mathsf{REx}$ be the collection of right exact sequences and 
	$\mathsf{LEx}$ be the collection of left exact sequences. 
	Then $(\AA,$ $0,$ $0,$ $\mathsf{REx},$ $\mathsf{LEx})$ is a pretriangulated category.
	\end{enumerate}
\end{ex}

\section{Pretriangulated structures induced by premutation triples} \label{pretri}

We fix an ET category $(\CC, \bbE, \fraks)$ in this section.

\subsection{The condition (MT1) and (MT2)}
In the following definition, we use notations which are compatible with previous section and \cite{Nak18}.

\begin{condi} \label{MT1-2}
Let $\SS, \VV, \ZZ$ be subcategories of $\CC$.
We denote $\SS \cap \ZZ, \VV \cap \ZZ$ by $\II, \JJ$, respectively.
We consider the following two conditions.
	\begin{itemize}[leftmargin=40pt]
	\item[(rMT1)] $\II$ is strongly contravariantly finite in $\ZZ$.
	\item[(rMT2)] $\bbE^{\II}(\SS, \ZZ) = 0$ and $\bbE_{\II}(\SS, \ZZ \bracket{-1}) = 0$.
	\end{itemize}

Dually, we also consider the following two conditions.
	\begin{itemize}[leftmargin=40pt]
	\item[(lMT1)] $\JJ$ is strongly covariantly finite in $\ZZ$.
	\item[(lMT2)] $\bbE_{\JJ}(\ZZ, \VV) = 0$ and $\bbE^{\JJ}(\ZZ\bracket{1}, \VV) = 0$.
	\end{itemize}

For convenience, we also use the following conditions.
	\begin{itemize}[leftmargin=40pt]
	\item[(MT1)] $\II = \JJ$, (rMT1) and (lMT1).
	\item[(MT2)] (rMT2) and (lMT2).
	\end{itemize}
\end{condi}

\begin{rmk} \label{rmk_MT2_third_vanish}
	\begin{enumerate}
	\item By definition of strongly contravariantly \resp{covariantly} finite, 
	we assume that $\II$ \resp{$\JJ$} is closed under direct summands. 
	On the other hand, recall that we do not always assume that $\SS, \ZZ$ and $\VV$ are closed under direct summands.
	\item \label{rmk_MT2_third_vanish_2}
	$\Cone_{\bbE^{\II}}(\ZZ,\ZZ)$ $=$ $\Cone_{\bbE^{\II}_{\II}}(\ZZ,\ZZ)$ if (rMT2) holds.
	Dually, $\CoCone_{\bbE_{\II}}(\ZZ,\ZZ)$ $=$ $\CoCone_{\bbE^{\II}_{\II}}(\ZZ,\ZZ)$ if (lMT2) holds.
	\item From Lemma \ref{bracket_equiv},
	$\bbE^{\II}(\SS, \ZZ\bracket{-1}) =0$ and $\bbE_{\JJ}(\ZZ\bracket{1}, \VV) =0$ under (MT1) and (MT2).
	That is because $\delta \in \bbE^{\II}(S, Z^{\pr})$ where $S \in \SS, Z^{\pr} \in \ZZ\bracket{-1}$ factors through $\lambda \in \bbE^{\II}_{\II}(Z, Z^{\pr})$ where $Z \in \ZZ$.
	We can show $\bbE_{\JJ}(\ZZ\bracket{1}, \VV) =0$ in the same way.
	\item From Lemma \ref{bracket_equiv}, 
	$\bbE_{\II}(\SS, \ZZ) =0$ and $\bbE^{\JJ}(\ZZ, \VV)=0$ under (MT1) and (MT2).
	That is because $\delta \in \bbE_{\II}(S, Z)$ where $S \in \SS, Z \in \ZZ$ factors through 
	$\lambda^Z \in \bbE^{\II}_{\II}(Z\bracket{1}, Z)$ from 
	$\bbE_{\II}(\SS, \II) \subset \bbE_{\II}(\SS, \ZZ\bracket{-1}) =0$.
	We can show $\bbE^{\JJ}(\ZZ, \VV)=0$ in the same way.
	\end{enumerate}
\end{rmk}

The definitions of rigid mutation pairs and orthogonal mutation pairs are in Appendix \ref{tri_str_by_MP}.
The definition of concentric twin cotorsion pairs is in Appendix \ref{ccTCP}.

\begin{ex} \phantom{X} \label{example_oneside_mut}
	\begin{enumerate}
	\item \cite{NP19} \label{example_oneside_mut_1}
	Assume that $(\CC, \bbE, \fraks)$ is Frobenius with $\PP = \Proj_{\bbE} \CC$. 
	Then $(\PP, \CC, \PP)$ satisfies (MT1) and (MT2).
	\item \cite{IY08} \label{example_oneside_mut_2}
	Assume that $\CC$ is a triangulated category.
	Let $\II$ be a functorially finite rigid subcategory of $\CC$ 
	and $(\XX, \YY)$ be a rigid $\II$-mutation pair.
	We define $\ZZ = \XX \cap \YY$. 
	Then $(\II, \ZZ, \II)$ satisfies (MT1) and (MT2).
	\item \cite{SP20} \label{example_oneside_mut_3}
	Assume that $\CC$ is a Hom-finite Krull-Schmidt triangulated $k$-category with a Serre functor $\bbS$.
	Let $\MM$ be a collection of $\Ob(\CC)$ which satisfies the condition (SP1) in Condition \ref{SP_condi1}
	and $(\XX, \YY)$ be an orthogonal $\bracket{\MM}$-mutation pair.
	We define $\ZZ = \XX \cap \YY$. 
	Then $(\bracket{\MM [1]}, \ZZ, \bracket{\MM [-1]})$ satisfies (MT1) and (MT2).
	\item \cite{Nak18} \label{example_oneside_mut_4}
	Let $((\SS, \TT), (\UU, \VV))$ be a concentric twin cotorsion pairs.
	We define $\ZZ = \TT \cap \UU$. Then $(\SS, \ZZ, \VV)$ satisfies (MT1) and (MT2).
	\end{enumerate}
\end{ex}
\begin{proof}
(1) (MT1) follows from the definition of Frobenius ET categories.
Since $\PP$ is both projective and injective, 
$\bbE = \bbE^{\PP} = \bbE_{\PP}$ and $\bbE(\PP,-) = 0, \bbE(-,\PP) = 0$.
Thus, (MT2) holds.

(2) By definition of rigid $\II$-mutation pairs, $\II$ is functorially finite in $\ZZ$. In particular, (MT1) holds.
$\ZZ \subset \lpp{\II[1]} \cap \rpp{\II[-1]}$ also follows from definition of rigid $\II$-mutation pairs.
Note that $\ZZ\bracket{1} \subset \YY, \ZZ\bracket{-1} \subset \XX$.
Thus, $\bbE(\II, \ZZ) = 0, \bbE(\II, \ZZ\bracket{-1}) = 0, \bbE(\ZZ, \II) = 0$ and $\bbE(\ZZ\bracket{1}, \II) = 0$.
In particular, (MT2) holds.

(3) Note that $\II = \JJ = 0$ since $(\XX, \YY)$ is an orthogonal mutation pair, in particular,
$\bbE = \bbE^{\II} = \bbE_{\II}$ and $\bracket{1} = [1]$, then (MT1) hold.
(MT2) follows from $\ZZ \subset \dpp{\MM} \cap {}^{\perp} \MM[-1] \cap \MM[1]^{\perp}$.

(4) $\II = \JJ$ since $((\SS, \TT), (\UU, \VV))$ is concentric. 
By definition of twin cotorsion pairs, $\bbE(\SS, \ZZ) = 0$ and $\bbE(\ZZ, \VV) = 0$.
From $\ZZ \bracket{1} \subset \UU$ and $\ZZ \bracket{-1} \subset \TT$,
$\bbE(\SS, \ZZ\bracket{-1}) = 0$ and $\bbE(\ZZ\bracket{1}, \VV) = 0$.
Thus, we obtain (MT2).
(MT1) is direct from the definition of twin cotorsion pairs (for details, see \cite{LN19}).
\end{proof}

From Example \ref{example_oneside_mut}(\ref{example_oneside_mut_4}), 
we may consider a triplet $(\SS, \ZZ, \VV)$ satisfying (MT1) and (MT2) as a generalization of concentric twin cotorsion pairs.
We define new subcategories of $\CC$, $\wt{\UU}$ and $\wt{\TT}$, 
which are denoted by $\CC^-$ and $\CC^+$ in Definition \ref{defi_plus-minus_ccTCP}, respectively.
We use these notations because $\UU \subset \wt{\UU} = \CC^-$ and $\TT \subset \wt{\TT} = \CC^+$ 
hold for any concentric twin cotorsion pair $((\SS, \TT), (\UU, \VV))$.

\begin{nota} \label{nota_wtU}
	\begin{enumerate}
	\item Let $(\SS, \ZZ)$ be a pair of subcategories which satisfies (rMT1) and (rMT2). 
	We define $\widetilde{\UU}$ as $\CoCone_{\bbE^{\II}}(\ZZ, \SS)$.
	\item Let $(\ZZ, \VV)$ be a pair of subcategories which satisfies (lMT1) and (lMT2).
	We define $\widetilde{\TT}$ as $\Cone_{\bbE_{\JJ}}(\VV, \ZZ)$.
	\end{enumerate}
\end{nota}

\begin{nota} \label{nota_ul}
	\begin{enumerate}
	\item Let $(\SS, \ZZ)$ be a pair of subcategories which satisfies (rMT1) and (rMT2). 
	We denote $\XX/[\II]$ by $\ul{\XX}$ for a subcategory $\XX$ containing $\II$.
	\item Let $(\ZZ, \VV)$ be a pair of subcategories which satisfies (lMT1) and (lMT2). 
	We denote $\YY/[\JJ]$ by $\ul{\YY}$ for a subcategory $\YY$ containing $\JJ$.
	\item For a morphism $x$ in $\ZZ$, we denote $\ul{x\bracket{1}}$ by $\ul{x}\bracket{1}$.
	\end{enumerate}
\end{nota}

\begin{lem} \label{adjoint_pair} \cite[Fact 2.1, Definition 3.10]{Nak18}
	\begin{enumerate}
	\item Let $(\SS,\ZZ)$ be a pair of subcategories which satisfies (rMT1) and (rMT2).
	For $U \in \wt{\UU}$, take an $\fraks^{\II}$-triangle $U \xrar{h^U} Z^U \xrar{g^U} S \xdrar{} U$. 
	Then $- \circ \ul{h^U} \colon \ul{\ZZ}(Z^U\!, \ZZ) \to \ul{\wt{\UU}}(U, \ZZ)$ is a natural isomorphism.
	In particular, there exists an additive functor $\sigma \colon \ul{\wt{\UU}} \to \ul{\ZZ} \, ; U \mapsto Z^U$ 
	which is a left adjoint of the inclusion functor $\iota_{\sigma} \colon \ul{\ZZ} \to \ul{\wt{\UU}}$.
	\label{adjoint_pair_sigma}
	\item Let $(\ZZ, \VV)$ be a pair of subcategories which satisfies (lMT1) and (lMT2).
	For $T \in \wt{\TT}$, take an $\fraks_{\II}$-triangle $T \xdrar{} V \xrar{} Z_T \xrar{h_T} T$. 
	Then $\ul{h_T} \circ - \colon \ul{\ZZ}(\ZZ, Z_T) \to \ul{\wt{\TT}}(\ZZ, T)$ is a natural isomorphism.
	In particular, there exists an additive functor $\omega \colon \ul{\wt{\TT}} \to \ul{\ZZ} \, ; T \mapsto Z_T$ 
	which is a right adjoint of the inclusion functor $\iota_{\omega} \colon \ul{\ZZ} \to \ul{\wt{\TT}}$.
	\label{adjoint_pair_omega}
	\end{enumerate}
\end{lem}
\begin{proof}
We only prove (1). We denote $h^U, g^U$ by $h, g$ respectively.
Since $- \circ \ul{h} \colon \ul{\ZZ}(Z^U\!, \ZZ) \to \ul{\wt{\UU}}(U, \ZZ)$ is clearly well-defined and functorial,
we only have to show that this is bijective for any $Z^{\pr} \in \ZZ$.
Take a morphism $z \colon Z^U \to Z^{\pr}$ with $\ul{zh} = 0$. 
There exists an $\fraks_{\II}$-triangle $Z^{\pr} \bracket{-1} \xrar{} I \xrar{p_{Z^{\pr}}} Z^{\pr} \xdrar{\lambda_{Z^{\pr}}} Z^{\pr} \bracket{-1}$ because $\II$ is strongly contravariantly finite in $\ZZ$.
We denote $p_{Z^{\pr}}, \lambda_{Z^{\pr}}$ by $p, \lambda$, respectively.
Since $\lambda$ is an $\bbE_{\II}$-extension and $zh \in [\II]$, there exists a morphism 
$a \colon U \to I$ with $zh = p a$. 
Since $\bbE^{\II}(\SS, \ZZ) = 0$, there exists a morphism $a^{\pr} \colon Z^U \to I$ such that $a = a^{\pr} h$. 
Then $0 = zh - pa = (z- p a^{\pr})h$. 
Thus, there exists a morphism $z^{\pr} \colon S \to Z^{\pr}$ where $z - p a^{\pr} = z^{\pr} g$.
From $\bbE_{\II}(\SS, \ZZ \bracket{-1}) = 0$, then $\lambda z^{\pr} = 0$ and there exists a morphism 
$z^{\pr\pr} \colon S \to I$ where $z^{\pr} = p z^{\pr\pr}$.
Therefore $\ul{z} = \ul{pa^{\pr} + z^{\pr}g} = \ul{p(a^{\pr} + z^{\pr\pr}g)} = 0$, that is, $-\circ \ul{h}$ is injective.
On the other hand, $- \circ \ul{h}$ is surjective because $\bbE^{\II}(\SS,\ZZ) = 0$.
\end{proof}

\begin{rmk} \label{rmk_adj_of_sigma}
	\begin{enumerate}
	\item \label{rmk_adj_of_sigma_1}
	From Lemma \ref{adjoint_pair}(\ref{adjoint_pair_sigma}), for a morphism $z \colon Z^U \to Z$ in $\ZZ$,
	\[
	\ul{z} = 0 \text{ in } \ul{\ZZ} \iff \ul{z h^U} = 0 \text{ in } \wt{\ul{\UU}} \, .
	\]
	\item From the above proof, for any morphism $u \colon U \to Z^{\pr}$ where $U \in \wt{\UU}$ and $Z^{\pr} \in \ZZ$, 
	there exists a morphism $z \colon Z^U \to Z^{\pr}$ which satisfies $u = zh^U$ in $\wt{\UU}$ 
	and such $z$ is unique up to $[\II]$.
	\label{uniq_sigma}
	\item  \label{rmk_adj_of_sigma_3}
	We can construct $\sigma \colon \wt{\ul{\UU}} \to \ul{\ZZ}$ directly like $\bracket{1}$.
		\begin{enumerate}
		\item For $U \in \wt{\UU}$, there exists an $\fraks^{\II}$-triangle
		\[
		U \xrar{h^U} \sigma U \xrar{g^U} S \xdrar{\rho^U} U
		\]
		where $\sigma U \in \ZZ$ and $S \in \SS$.
		\item For a morphism $u \colon U_1 \to U_2$ in $\wt{\UU}$, there exist $\fraks^{\II}$-triangles 
		$U_1 \xrar{h^{U_1}} \sigma U_1 \xrar{g^{U_1}} S_1 \xdrar{\rho^{U_1}} U_1$ and
		$U_2 \xrar{h^{U_2}} \sigma U_2 \xrar{g^{U_2}} S_2 \xdrar{\rho^{U_2}} U_2$
		where $S_1, S_2 \in \SS$.
		We denote $\sigma U_1, \sigma U_2$ by $Z_1, Z_2$, respectively. 
		Then there exists a unique morphism $z \colon Z_1 \to Z_2$ in $\CC$ up to $[\II]$ which makes the following diagram commutative from (\ref{uniq_sigma}).
		\[
		\xy
		(16,8)*+{U_1}="12";
		(32,8)*+{Z_1}="13";
		(16,-8)*+{U_2}="22";
		(32,-8)*+{Z_2}="23";
		{\ar@{->}^{h^{U_1}} "12";"13"};	
		{\ar@{->}^{h^{U_2}} "22";"23"};
		{\ar^{u} "12";"22"};
		{\ar^{z} "13";"23"};
		{\ar@{}|\car "12";"23"};
		\endxy
		\]
		 We define $\sigma (u)$ as $z$.
		\end{enumerate}
	Then $\sigma $ induces an additive functor $\sigma \colon \ul{\wt{\UU}} \to \ul{\ZZ}$. 
	From uniqueness of left adjoint functor up to natural isomorphisms, 
	$\sigma$ does not depend on the choices of $\fraks^{\II}$-triangle
	$U \xrar{h^U} Z^U \rar S \drar U$ in Lemma \ref{adjoint_pair}
	up to natural isomorphisms.
	\end{enumerate}
\end{rmk}

\begin{nota}
	\begin{enumerate}
	\item For a morphism $u$ in $\wt{\UU}$, we denote $\ul{\sigma(u)}$ by $\sigma(\ul{u})$.
	\item We denote the unit of the adjoint pair $(\sigma, \iota_{\sigma})$ by $\eta$.
	\item We denote the counit of the adjoint pair $(\iota_{\omega}, \omega)$ by $\varepsilon$.
	\end{enumerate}
\end{nota}

\subsection{Definition of premutation triples}

\begin{condi} \label{MT3}
We consider the following conditions.
\begin{enumerate}
\item Assume (rMT1) and (rMT2).
	\begin{itemize}[leftmargin=40pt]
	\item[(rMT3)]
		\begin{enumerate}[label=(\roman*)]
		\item $\Cone_{\bbE^{\II}}(\ZZ,\ZZ) \subset \wt{\UU}$. \label{RM3_1}
		\item $\SS$ and $\ZZ$ are closed under extensions in $(\CC, \bbE^{\II}, \fraks^{\II})$.
		\end{enumerate}
	\end{itemize}
\item Assume (lMT1) and (lMT2).
	\begin{itemize}[leftmargin=40pt]
	\item[(lMT3)]
		\begin{enumerate}[label=(\roman*)]
		\item $\CoCone_{\bbE_{\JJ}}(\ZZ,\ZZ) \subset \wt{\TT}$. \label{LM3_1}
		\item $\ZZ$ and $\VV$ are closed under extensions in $(\CC, \bbE_{\JJ}, \fraks_{\JJ})$.
		\end{enumerate}
	\end{itemize}
\item Assume (MT1) and (MT2).
	\begin{itemize}[leftmargin=40pt]
	\item[(MT3)] (rMT3) and (lMT3).
	\end{itemize}
\end{enumerate}
\end{condi}

\begin{rmk}
If (rMT2) holds, then $\Cone_{\bbE^{\II}}(\ZZ, \ZZ)$ $=$ $\Cone_{\bbE^{\II}_{\II}}(\ZZ, \ZZ)$.
On the other hand, $\CoCone_{\bbE^{\II}_{\II}}(\ZZ, \SS) \subsetneq \CoCone_{\bbE^{\II}}(\ZZ, \SS) = \wt{\UU}$ in general.
In Example \ref{ex_HoveyTCP_ARquiv}(\ref{ex_HoveyTCP_ARquiv_1}),
 $Y \in \CoCone_{\bbE^{\II}}(\ZZ, \SS) = \UU$ but $Y \notin \CoCone_{\bbE^{\II}_{\II}}(\ZZ, \SS)$.
That is because there exists a triangle $Y \xrar{f} I_1 \xrar{g} I_2 \rar Y[1]$ 
where $f$ is $\II$-monic but $g$ is not $\II$-epic.

The condition $\Cone_{\bbE^{\II}}(\ZZ, \ZZ) \subset \CoCone_{\bbE^{\II}_{\II}}(\ZZ, \SS)$ is contained in Condition \ref{condi_rMT3}.
\end{rmk}

\begin{defi} \label{mutation_triple}
Let $\SS, \ZZ, \VV$ be subcategories of $\CC$.
	\begin{enumerate}
	\item $(\SS, \ZZ)$ is a \emph{right mutation double} if $(\SS, \ZZ, \II)$ satisfies (MT1), (rMT2) and (rMT3).
	\item $(\ZZ, \VV)$ is a \emph{left mutation double} if $(\JJ, \ZZ, \VV)$ satisfies (MT1), (lMT2) and (lMT3).
	\item $(\SS, \ZZ,\VV)$ is a \emph{premutation triple} if it satisfies from (MT1) to (MT3).
	\end{enumerate}
\end{defi}

\begin{rmk}
	\begin{enumerate}
	\item Note that both right and left mutation doubles are required to satisfy (MT1).
	That is because right \resp{left} mutation doubles are required to have
	the additive functor $\bracket{1}$ \resp{$\bracket{-1}$}
	so that we may define $\Sigma$ \resp{$\Omega$} in Definition \ref{def_mutations}.
	\item A triplet $(\SS, \ZZ, \VV)$ which satisfies (MT1), (MT2) and (MT3), that is
	$(\SS, \ZZ)$ is a right mutation double and $(\ZZ, \VV)$ is a left mutation double with $\II = \JJ$, is \emph{not} mutation triple.
	(We define mutation triples in Section \ref{triangulated}.) 
	Later, we check that right mutation doubles \resp{left mutation doubles, premutation triples, mutation triples}
	induce right triangulated \resp{left triangulated, pretriangulated, triangulated} categories.
	\end{enumerate}
\end{rmk}

From the following examples, Frobenius ET categories, rigid mutation pairs, orthogonal mutation pairs and concentric twin cotorsion pairs are examples of premutation triples.

\begin{ex} \phantom{X}\label{ex_MT}
	\begin{enumerate}
	\item \label{ex_MT_enough}
	Assume that $\CC$ has enough projectives and $\Proj \CC$ is strongly \emph{functorially} finite in $\CC$.
	Then $(\Proj \CC, \CC)$ is a right mutation double.
	Dually, assume that $\CC$ has enough injectives and $\Inj \CC$ is strongly \emph{functorially} finite in $\CC$.
	Then $(\CC, \Inj \CC)$ is a left mutation double.
	\item \cite{NP19} \label{ex_MT_Frob}
	Assume that $\CC$ is Frobenius with $\PP = \Proj \CC$. 
	Then $(\PP, \CC, \PP)$ is a premutation triple. 
	More generally, for any strongly functorially finite subcategory $\XX$ in $\CC$, 
	$(\XX, \CC, \XX)$ in $(\CC, \bbE^{\XX}_{\XX}, \fraks^{\II}_{\II})$ is a premutation triple.
	\item \cite{IY08} \label{ex_MT_IY}
	In the case of Example \ref{example_oneside_mut}(\ref{example_oneside_mut_2}), 
	we additionally assume that $\XX = \YY$ and (IY3) in Condition \ref{IY_condi}.
	Then $(\II, \ZZ, \II)$ is a premutation triple.
	\item \cite{SP20} \label{ex_MT_SP}
	In the case of Example \ref{example_oneside_mut}(\ref{example_oneside_mut_3}),
	we additionally assume that $\XX = \YY$ and (SP3) in Condition \ref{SP_condi2}.
	Then $(\bracket{\MM[1]}, \ZZ, \bracket{\MM[-1]})$ is a premutation triple.
	\item \cite{Nak18} \label{ex_MT_CTP}
	In the case of Example \ref{example_oneside_mut}(\ref{example_oneside_mut_4}), 
	$(\SS, \ZZ, \VV)$ is a premutation triple.
	\end{enumerate}
\end{ex}
\begin{proof}
(\ref{ex_MT_enough}) 
We only prove when $\CC$ has enough projectives.
From assumptions, (MT1) holds. 
(rMT2) follows from $\bbE(\PP, -) = 0$.
(rMT3) is also true because $\wt{\UU} = \CC$ and $\PP$ is closed under extensions in $(\CC, \bbE, \fraks)$.
(\ref{ex_MT_Frob}) is direct from (\ref{ex_MT_enough}).
(\ref{ex_MT_IY}) follows from Lemma \ref{epic_monic} and \cite[Lemma 4.3]{IY08}. 
(\ref{ex_MT_SP}) is by definition of (SP3).
(\ref{ex_MT_CTP}) follows from Lemma \ref{conic} and the definition of twin cotorsion pairs.
\end{proof}

Because of the condition (rMT3)\ref{RM3_1} and (lMT3)\ref{LM3_1}, we can finally define mutation functors.

\begin{defi} \label{def_mutations}
	\begin{enumerate}
	\item Let $(\SS, \ZZ)$ be a right mutation double.
		We define an additive functor $\Sigma = \sigma \circ \bracket{1} \colon \ul{\ZZ} \to \ul{\ZZ}$,
		called a \emph{right mutation functor}.
	\item Let $(\ZZ, \VV)$ be a left mutation double.
		We define an additive functor $\Omega = \omega \circ \bracket{-1} \colon \ul{\ZZ} \to \ul{\ZZ}$,
		called a \emph{left mutation functor}.
	\end{enumerate}
\end{defi}

\begin{ex}
	\begin{enumerate}
	\item In the case of Example \ref{ex_MT}(\ref{ex_MT_Frob}),
	a right \resp{left} mutation functor is exactly a cosyzygy \resp{syzygy} functor.
	\item In the case of Example \ref{ex_MT}(\ref{ex_MT_IY}), 
	a right \resp{left} mutation functor is exactly right \resp{left} mutation functor 
	in Lemma \ref{lem_bracket_rigidver}.
	\item In the case of Example \ref{ex_MT}(\ref{ex_MT_SP}),
	a right \resp{left} mutation functor is exactly right \resp{left} mutation functor
	in Definition \ref{sigma_and_omega_inSP}.
	\item In the case of Example \ref{ex_MT}(\ref{ex_MT_CTP}),
	a right \resp{left} mutation functor is exactly right \resp{left} mutation functor 
	in Definition \ref{defi_sigma-omega_CTP}.
	\end{enumerate}
\end{ex}

In the last part of this section, we collect some lemmas we use later.

\begin{lem} \label{adjoint_sigma_omega} \cite[Proposition 4.3]{Nak18}
Let $(\SS, \ZZ, \VV)$ be a premutation triple.
Then $(\Sigma, \Omega)$ is an adjoint pair.
\end{lem}
\begin{proof}
From Lemma \ref{adjoint_pair}, we only have to show that there exists a bifunctorial isomorphism
$\Phi \colon \ul{\wt{\UU}}(Z\bracket{1}, Z^{\pr}) \xrar{\sim} \wt{\ul{\TT}}(Z, Z^{\pr}\bracket{-1})$ for any $Z,Z^{\pr} \in \ZZ$.
For a morphism $z \colon Z\bracket{1} \to Z^{\pr}$, 
we take an $\fraks^{\II}$-triangle $Z \xrar{i^Z} I^Z \xrar{p^Z} Z\bracket{1} \xdrar{\lambda^Z} Z$ and
an $\fraks_{\II}$-triangle $Z^{\pr}\bracket{-1} \xrar{i_{Z^{\pr}}} I_{Z^{\pr}} \xrar{p_{Z^{\pr}}} Z^{\pr} \xdrar{\lambda_{Z^{\pr}}} Z^{\pr}\bracket{-1}$.
Then there exists a morphism $z^{\pr} \colon Z \to Z^{\pr} \bracket{-1}$ which makes the following diagram commutative since $p_{Z^{\pr}}$ is a right $\II$-approximation.
\[
\xy
(0,8)*+{Z}="11";
(0,-8)*+{Z^{\pr} \bracket{-1}}="21";
(16,8)*+{I^Z}="12";
(16,-8)*+{I_{Z^{\pr}}}="22";
(32,8)*+{Z\bracket{1}}="13";
(32,-8)*+{Z^{\pr}}="23";
(48,8)*+{Z}="14";
(48,-8)*+{Z^{\pr} \bracket{-1}}="24";
{\ar^{i^Z} "11";"12"};
{\ar^{p^Z} "12";"13"};	
{\ar@{-->}^{\lambda^{Z}} "13";"14"};
{\ar^{i_{Z^{\pr}}} "21";"22"};
{\ar^{-p_{Z^{\pr}}} "22";"23"};
{\ar@{-->}^{-\lambda_{Z^{\pr}}} "23";"24"};
{\ar^{z^{\pr}} "11";"21"};
{\ar^{} "12";"22"};
{\ar^{z} "13";"23"};
{\ar^{z^{\pr}} "14";"24"};
{\ar@{}|\car "11";"22"};
{\ar@{}|\car "12";"23"};
{\ar@{}|\car "13";"24"};
\endxy
\]
 
Then we define $\Phi(\ul{z}) = \ul{z^{\pr}}$.
This correspondence is well-defined and injective from the commutativity of the right most square.
In particular, $z^{\pr}$ is unique up to $[\II]$ which satisfies $z^{\pr} \lambda^Z = - \lambda_{Z^{\pr}} z$.
This correspondence is also surjective since $i^Z$ is a left $\II$-approximation.
Finally, this is bifunctorial because $z^{\pr}$ is uniquely determined by the commutativity of right square.
\end{proof}

\begin{nota}
We denote the unit \resp{counit} of $(\Sigma, \Omega)$ by $\alpha$ \resp{$\beta$}.
\end{nota}

\begin{rmk} \label{rmk_adjoint}
From Lemma \ref{adjoint_pair} and \ref{adjoint_sigma_omega}, we have the following isomorphisms for $Z, Z^{\pr} \in \ZZ$.
\begin{align}
\ul{\ZZ}(\Sigma Z, Z^{\pr}) \xrar{- \circ \ul{h^{Z\bracket{1}}}} 
\ul{\wt{\UU}}(Z \bracket{1}, Z^{\pr}) \xrar{\phantom{x} \Phi \phantom{x}}
\ul{\wt{\TT}}(Z, Z^{\pr}\bracket{-1}) \xleftarrow{\ul{h_{Z^{\pr}\bracket{-1}}} \circ -}
\ul{\ZZ}(Z, \Omega Z^{\pr})
\label{rmk1}
\end{align}

Then for a morphism $f \colon Z \to \Omega Z^{\pr}$ in $\ZZ$, the corresponding morphism 
$f^{\pr} \colon \Sigma Z \to Z^{\pr}$ in \eqref{rmk1} is uniquely determined by the following commutative diagram
up to $[\II]$.
\begin{align}
\xy
(-10,-12)*+{}="10";
(0,8)*+{Z}="11";
(16,8)*+{Z\bracket{1}}="12";
(32,8)*+{\Sigma Z}="13";
(-16,-8)*+{\Omega Z^{\pr}}="20";
(0,-8)*+{Z^{\pr} \bracket{-1}}="21";
(16,-8)*+{Z^{\pr}}="22";
(26,12)*+{}="23";
{\ar@{-->}_{\lambda^Z} "12";"11"};
{\ar@{->}^{h^{Z\bracket{1}}} "12";"13"};
{\ar@{->}_{h_{Z^{\pr}\bracket{-1}}} "20";"21"};
{\ar@{-->}^-{-\lambda_{Z^{\pr}}} "22";"21"};
{\ar_{f} "11";"20"};
{\ar^{} "11";"21"};
{\ar_{} "12";"22"};
{\ar^{f^{\pr}} "13";"22"};
{\ar@{}|\car "11";"10"};
{\ar@{}|\car "11";"22"};
{\ar@{}|\car "23";"22"};
\endxy
\label{rmk2}
\end{align}

Note that $\ul{f^{\pr}} = \beta_{Z^{\pr}} \circ \Sigma(\ul{f})$ and $\ul{f} = \Omega(\ul{f^{\pr}}) \circ \alpha_Z$.
\end{rmk}

\begin{rmk}
We assume that $\CC$ is a triangulated category and 
let $(\SS, \ZZ, \VV)$ be a premutation triple induced by a concentric twin cotorsion pair.
The negative sign in the proof of Lemma \ref{adjoint_sigma_omega} comes from the following isomorphic correspondence.
\[
\begin{array}{ccc}
\bbE(X\bracket{1}, X) & \xrar{\sim} & \CC(X\bracket{1}, X[1]) \\[3pt]
\lambda^X & \mapsto & l^X
\end{array}
\]
\[
\begin{array}{ccccc}
\bbE(Y, Y\bracket{-1}) & \xrar{\sim} & \CC(Y[-1], Y\bracket{-1}) & \xrar{[1]} & \CC(Y, Y\bracket{-1}[1]) \\[3pt]
 \lambda_Y & \mapsto & -l_Y[-1] & \mapsto & -l_Y
\end{array}
\]
Then $\Phi$ in Lemma \ref{adjoint_sigma_omega} is defined by the following commutative diagram in $\ul{\CC}$.
\[
\xy
(0,8)*+{Z\bracket{1}}="11";
(24,8)*+{Z[1]}="12";
(0,-8)*+{Z^{\pr}}="21";
(24,-8)*+{Z^{\pr}\bracket{-1}[1]}="22";
{\ar^{l^Z} "11";"12"};
{\ar^-{l_{Z^{\pr}}} "21";"22"};
{\ar^{z} "11";"21"};
{\ar^{(\Phi(z))[1]} "12";"22"};
{\ar@{}|\car "11";"22"};
\endxy
\]

This correspondence is used in \cite[Definition 4.1]{Nak18}.
\end{rmk}

\begin{lem} \label{isomorphism} \cite[Lemma 3.11]{Nak18}
Let $(\SS, \ZZ)$ be a right mutation double.
Assume that there exists an $\fraks^{\II}$-triangle $U \xrar{u} U^{\pr} \rar S \drar U$
 where $U,U^{\pr} \in \wt{\UU}$.
Then $\sigma(\ul{u}) \colon \sigma U \to \sigma U^{\pr}$ is an isomorphism. 
\end{lem}
\begin{proof}
By definition of right mutation doubles, there exists an $\fraks^{\II}$-triangle $S^{\pr} \drar U^{\pr} \rar Z^{\pr} \rar S^{\pr}$ where $Z^{\pr} \in \ZZ$ and $S^{\pr} \in \SS$. Since $\SS$ is closed under extensions in $(\CC, \bbE^{\II}, \fraks^{\II})$,
there exists the following commutative diagram in $\CC$ where $S^{\pr\pr} \in \SS$. 
\[
\xy
(0,8)*+{U}="21";
(0,-8)*+{U}="31";
(16,8)*+{U^{\pr}}="22";
(16,-8)*+{Z^{\pr}}="32";
(16,-24)*+{S^{\pr}}="42";
(32,8)*+{S}="23";
(32,-8)*+{S^{\pr\pr}}="33";
(32,-24)*+{S^{\pr}}="43";
(48,8)*+{U}="24";
(48,-8)*+{U}="34";
{\ar^{u} "21";"22"};
{\ar^{} "22";"23"};
{\ar@{-->}^{} "23";"24"};
{\ar^{} "31";"32"};
{\ar^{} "32";"33"};
{\ar@{-->}^{} "33";"34"};
{\ar@{=} "42";"43"};
{\ar@{=} "21";"31"};
{\ar^{} "22";"32"};
{\ar^{} "32";"42"};
{\ar^{} "23";"33"};
{\ar^{} "33";"43"};
{\ar@{=} "24";"34"};
%
{\ar@{}|\car "21";"32"};
{\ar@{}|\car "22";"33"};
{\ar@{}|\car "23";"34"};
{\ar@{}|\car "32";"43"};
\endxy
\]
Thus, we obtain an $\fraks^{\II}$-triangle $U \rar Z^{\pr} \rar S^{\pr\pr} \drar U$.
From uniqueness of $\sigma$, $\sigma(\ul{u})$ is an isomorphism.
\end{proof}

\subsection{Right triangles induced by a right mutation double}
In this subsection, we consider right triangulated \resp{left triangulated, pretriangulated} structures induced by right mutation doubles \resp{left mutation doubles, premutation triples}.
We assume that $(\SS, \ZZ)$ is a right mutation double.

First, we fix the following $\fraks^{\II}$-triangles to define $\bracket{1}$ and $\sigma$.

\begin{enumerate}
\item For $X \in \ZZ$, there exists the following $\fraks^{\II}$-triangle (and also an $\fraks_{\II}$-triangle) 
where $I^X \in \II$ and we fix it.
	\[
	X \xrar{i^X} I^X \xrar{p^X} X\bracket{1} \xdrar{\lambda^X} X
	\]
	Then we define $i^X, p^X, \lambda^X$ by the above fixed $\fraks^{\II}$-triangle.
\item For $U \in \wt{\UU}$, there exists the following $\fraks^{\II}$-triangle (this is not an $\fraks_{\II}$-triangle in general)
where $\sigma U \in \ZZ, S^U \in \SS$ and we fix it.
	\[
	U \xrar{h^U} \sigma U \xrar{g^U} S^U \xdrar{\rho^U} U
	\]
	Then we define $h^U, g^U, \rho^U$ by the above fixed $\fraks^{\II}$-triangle.
	For $Z \in \ZZ$, we always take $\sigma Z$ and $h^Z$ so that $\sigma Z = Z$ and $h^Z = \mathrm{id}_Z$.
\end{enumerate}
We remind readers that $\bracket{1} \colon \ul{\ZZ} \to \ul{\wt{\UU}}$ and $\sigma \colon \ul{\wt{\UU}} \to \ul{\ZZ}$ 
do not depend on the choices of the above $\fraks^{\II}$-triangles up to natural isomorphisms.

\begin{rmk} \label{rmk_sigma_of_h}
Let $U \in \wt{\UU}$. From Remark \ref{rmk_adj_of_sigma}(\ref{rmk_adj_of_sigma_3}),
$\sigma(h^U)$ is one of the morphisms which makes the following diagram in $\CC$ commutative.
Note that $h^{\sigma U} = \mathrm{id}_{\sigma U}$.
\[
\xy
(0,8)*+{U}="11";
(16,8)*+{\sigma U}="12";
(0,-8)*+{\sigma U}="21";
(16,-8)*+{\sigma U}="22";
{\ar^{h^U} "11";"12"};
{\ar^{h^{\sigma U}} "21";"22"};
{\ar^{h^U} "11";"21"};
{\ar^{\sigma(h^U)} "12";"22"};
{\ar@{}|\car "11";"22"};
\endxy
\]
Then $(\sigma(h^U) - \id_{\sigma U}) h^U = 0$.
From Remark \ref{rmk_adj_of_sigma}(\ref{rmk_adj_of_sigma_1}),
$\sigma(\ul{h^U}) = \ul{\mathrm{id}_{\sigma U}}$.
\end{rmk}

\begin{nota} \label{notation_right_tri}
	\begin{enumerate}
	\item Let $a \colon X \to Y$ be a morphism in $\ZZ$. 
	From Lemma \ref{inflation}, there exists the following $\fraks^{\II}$-triangle (this is also an $\fraks_{\II}$-triangle from Remark \ref{rmk_MT2_third_vanish}(\ref{rmk_MT2_third_vanish_2}))
	and we fix it.
	\vspace{-3pt}
	\[
	X \xrar{\msize{0.6}{\begin{bmatrix} a \\ i^X \\ \end{bmatrix}}} Y \oplus I^X \xrar{\wt{b}} C^{a} \xdrar{\wt{\delta}} X
	\]
	Then we define $\wt{b}, C^a$ and $\wt{\delta}$ by the above $\fraks^{\II}_{\II}$-triangle. 
	We also define $\wt{a} = \msize{0.8}{\begin{bmatrix} a \\ i^X \\ \end{bmatrix}}$ and $b$ as the composition of $Y \xrar{\msize{0.6}{\begin{bmatrix} 1 \\ 0 \end{bmatrix}}} Y \oplus I^X \xrar{\wt{b}} C^a$.
	\item \label{notation_right_tri_2}
	From Lemma \ref{happel_diagram}, there exists the following commutative diagram in $(\CC, \bbE^{\II}_{\II}, \fraks^{\II}_{\II})$ and we fix it.
	\begin{align}
	\xy
	(0,24)*+{}="11";
	(0,8)*+{X}="21";
	(0,-8)*+{X}="31";
	(0,-24)*+{}="41";
	(16,24)*+{Y}="12";
	(16,8)*+{Y \! \oplus \! I^X}="22";
	(16,-8)*+{I^X}="32";
	(16,-24)*+{Y}="42";
	(32,24)*+{Y}="13";
	(32,8)*+{C^a}="23";
	(32,-8)*+{X \bracket{1}}="33";
	(32,-24)*+{Y}="43";
	(48,24)*+{}="14";
	(48,8)*+{X}="24";
	(48,-8)*+{X}="34";
	(48,-24)*+{Y \! \oplus \! I^X}="44";
	{\ar@{=} "12";"13"};
	{\ar^-{\wt{a}} "21";"22"};
	{\ar^-{\wt{b}} "22";"23"};
	{\ar@{-->}^{\wt{\delta}} "23";"24"};
	{\ar^{i^X} "31";"32"};
	{\ar^{p^X} "32";"33"};
	{\ar@{-->}^{\lambda^X} "33";"34"};
	{\ar@{=} "42";"43"};
	{\ar^{\msize{0.6}{\begin{bmatrix} -1 \\ 0 \\ \end{bmatrix}}} "43";"44"};
	{\ar@{=} "21";"31"};
	{\ar^{\msize{0.6}{\begin{bmatrix} 1 \\ 0 \\ \end{bmatrix}}} "12";"22"};
	{\ar^{\msize{0.6}{[0 \ 1]}} "22";"32"};
	{\ar@{-->}^{0} "32";"42"};
	{\ar^{b} "13";"23"};
	{\ar^{c^a} "23";"33"};
	{\ar@{-->}^{\gamma^a} "33";"43"};
	{\ar@{=} "24";"34"};
	{\ar^{\wt{a}} "34";"44"};
	{\ar@{}|\car "12";"23"};
	{\ar@{}|\car "21";"32"};
	{\ar@{}|\car "22";"33"};
	{\ar@{}|\car "23";"34"};
	{\ar@{}|\car "32";"43"};
	{\ar@{}^\car "33";"44"};
	\endxy
	\label{THE_shifted_oct}
	\end{align}
	Then we define $c^a, \gamma^a$ by the above diagram. We often drop ``$a$'' if there is no confusion.
	\end{enumerate}
\end{nota}

There exists the following sequence in $\ul{\ZZ}$.
\[
X \xrar{\, \ul{\wt{a}} \,} Y \oplus I^X \xrar{\, \ul{h^{C^a} \wt{b}} \,} \sigma C^a \xrar{\sigma({\ul{c}})} \Sigma X
\label{diag:right_tri}
\]

We show that the sequence \eqref{diag:right_tri} in Notation \ref{notation_right_tri} does not depend on choices of morphism $c$ and $\fraks^{\II}_{\II}$-triangle $X \rar I^X \rar X \bracket{1} \xdrar{\lambda^X \vphantom{X^{X^x}}} X$.

\begin{lem}
	\begin{enumerate}
	\item The morphism $c$ in Notation \ref{notation_right_tri}(\ref{notation_right_tri_2}) 
	is uniquely determined  in $\ul{\ZZ}$.
	\item The sequence \eqref{diag:right_tri} in Notation \ref{notation_right_tri} does not depend on the choices of $\bracket{1}$, up to isomorphisms of sequences in $\ul{\ZZ}$.
	\end{enumerate}
\end{lem}
\begin{proof}
(1) If morphisms $c, c^{\pr} \colon C^a \to X\bracket{1}$ satisfies $\wt{\delta} = \lambda c = \lambda c^{\pr}$, 
then $c-c^{\pr}$ factors through $I^X$.

(2) Take another $\fraks^{\II}_{\II}$-triangle $X \xrar{j^X} J^X \xrar{q^X} X\bracket{1}^{\pr} \xdrar{{\lambda^{\pr}}^X} X$ where $j^X$ is a left $\II$-approximation. 
Let $\mu \colon \bracket{1} \Rightarrow \bracket{1}^{\pr}$ be a natural isomorphism defined in Remark \ref{rmk_natiso_bracket} and $\Sigma^{\pr} = \sigma \circ \bracket{1}^{\pr}$.
Then $\lambda^X = {\lambda^{\pr}}^X m^X$ holds.
From Corollary \ref{uniqueness_of_cone}, there exists the following commutative diagrams
where $\ul{f_i}$ and $\ul{g_i}$ are isomorphisms for $i =1,2$. 
\begin{align}
\xy
(-16,16)*+{X}="11";
(8,16)*+{Y \! \oplus \! I^X}="12";
(32,16)*+{C^a}="13";
(48,16)*+{X}="14";
(-16,0)*+{X}="21";
(8,0)*+{Y \! \oplus \! I^X \! \oplus \! J^X}="22";
(32,0)*+{C^{a^{\pr\pr}}}="23";
(48,0)*+{X}="24";
(-16,-16)*+{X}="31";
(8,-16)*+{Y \! \oplus \! J^X}="32";
(32,-16)*+{C^{a^{\pr}}}="33";
(48,-16)*+{X}="34";
{\ar^-{\wt{a}} "11";"12"};
{\ar^-{\wt{b}} "12";"13"};	
{\ar@{-->}^-{\wt{\delta}} "13";"14"};
{\ar^-{\wt{{a}^{\pr\pr}}} "21";"22"};
{\ar^-{\wt{{b}^{\pr\pr}}} "22";"23"};
{\ar@{-->}^-{\wt{{\delta}^{\pr\pr}}} "23";"24"};
{\ar^-{\wt{{a}^{\pr}}} "31";"32"};
{\ar^-{\wt{{b}^{\pr}}} "32";"33"};
{\ar@{-->}^-{\wt{{\delta}^{\pr}}} "33";"34"};
{\ar@{=}^{} "11";"21"};
{\ar_{f_1} "22";"12"};
{\ar_{g_1} "23";"13"};
{\ar@{=}^{} "14";"24"};
{\ar@{=}^{} "21";"31"};
{\ar^{f_2} "22";"32"};
{\ar^{g_2} "23";"33"};
{\ar@{=}^{} "24";"34"};
{\ar@{}|\car "11";"22"};
{\ar@{}|\car "12";"23"};
{\ar@{}|\car "13";"24"};
{\ar@{}|\car "21";"32"};
{\ar@{}|\car "22";"33"};
{\ar@{}|\car "23";"34"};
\endxy
\qquad
\xy
(32,16)*+{C^a}="13";
(48,16)*+{\sigma C^a}="14";
(32,0)*+{C^{a^{\pr\pr}}}="23";
(48,0)*+{\sigma {C^{a^{\pr\pr}}}}="24";
(32,-16)*+{C^{a^{\pr}}}="33";
(48,-16)*+{\sigma{C^{a^{\pr}}}}="34";
{\ar@{->}^-{\ul{h^{C^a}}} "13";"14"};
{\ar@{->}^-{\ul{h^{C^{a^{\pr\pr}}}}} "23";"24"};
{\ar@{->}^-{\ul{h^{C^{a^{\pr}}}}} "33";"34"};
{\ar@{<-}^{\ul{g_1}} "13";"23"};
{\ar@{<-}^{\sigma(\ul{g_1})} "14";"24"};
{\ar^{\ul{g_2}} "23";"33"};
{\ar@{->}^{\sigma(\ul{g_2})} "24";"34"};
{\ar@{}|\car "13";"24"};
{\ar@{}|\car "23";"34"};
\endxy \label{diag:uniqueness_of_righttri}
\end{align}
Let $c \colon C^a \to X\bracket{1}$ and $c^{\pr} \colon C^{a^{\pr}} \to X\bracket{1}^{\pr}$ be the morphisms
defined in Notation \ref{notation_right_tri}(\ref{notation_right_tri_2}), 
then ${\lambda^{\pr}}^X m^X c g_1 = \wt{\delta} g_1 = \wt{\delta^{\pr}} g_2 = {\lambda^{\pr}}^X c^{\pr} g_2$
from \eqref{diag_natural-iso-bracket}, \eqref{THE_shifted_oct} and \eqref{diag:uniqueness_of_righttri}.
Thus, 
$\ul{m^X c} = \ul{c^{\pr} g_2} \ul{{g_1}^{-1}}$. 
Applying $\sigma$, $\sigma(\ul{m^X}) \sigma(\ul{c}) = \sigma(\ul{c^{\pr}}) \sigma(\ul{g_2 {g_1}^{-1}})$.
Therefore, the following commutative diagram exists in $\ul{\CC}$.
\[
\xy
(0,8)*+{X}="11";
(16,8)*+{Y \!\oplus\! I^X}="12";
(32,8)*+{C^a}="13";
(48,8)*+{\Sigma X}="14";
(0,-8)*+{X}="21";
(16,-8)*+{Y \!\oplus\! J^X}="22";
(32,-8)*+{C^{a^{\pr}}}="23";
(48,-8)*+{\Sigma^{\pr}X}="24";
{\ar^-{\ul{\wt{a}}} "11";"12"};
{\ar^-{\ul{h^{C^a} \wt{b}}} "12";"13"};	
{\ar@{->}^-{\sigma(\ul{c})} "13";"14"};
{\ar^-{\ul{\wt{a^{\pr}}}} "21";"22"};
{\ar^-{\ul{h^{C^{a^{\pr}}} \wt{b^{\pr}}}} "22";"23"};
{\ar@{->}^-{\sigma(\ul{c^{\pr}})} "23";"24"};
{\ar@{=}^{} "11";"21"};
{\ar^{\msize{0.6}{\ul{f_2 {f_1}\!^{-1}}}}_{\vsim} "12";"22"};
{\ar^{\msize{0.6}{\sigma(\ul{g_2 {g_1}\!^{-1}})}}_{\vsim} "13";"23"};
{\ar^{\sigma(\ul{m^X})}_{\vsim} "14";"24"};
{\ar@{}|\car "11";"22"};
{\ar@{}|\car@<12pt> "12";"23"};
{\ar@{}|\car@<12pt> "13";"24"};
\endxy
\]
\end{proof}

\begin{lem} \label{previous_lem}
Let $a \colon X \to Y$ be a morphism in $\ZZ$ and $X \xrar{a^{\pr}} Y \xrar{b^{\pr}} U^{\pr} \xdrar{\delta^{\pr}} X$ 
be an $\fraks^{\II}_{\II}$-triangle where $\ul{a} = \ul{a^{\pr}}$.
Then there exists a morphism $s \colon \sigma C^a \to \sigma U^{\pr}$ which makes the following diagram in $\ul{\ZZ}$ commutative where $\ul{s}$ is an isomorphism.
\[
\xy
(0,8)*+{X}="11";
(16,8)*+{Y}="12";
(32,8)*+{\sigma C^a}="13";
(48,8)*+{\Sigma X}="14";
(0,-8)*+{X}="21";
(16,-8)*+{Y}="22";
(32,-8)*+{\sigma U^{\pr}}="23";
(48,-8)*+{\Sigma X}="24";
{\ar^{\ul{a}} "11";"12"};
{\ar^{\ul{h^{C^a} b}} "12";"13"};	
{\ar@{->}^{\sigma(\ul{c})} "13";"14"};
{\ar^{\ul{a^{\pr}}} "21";"22"};
{\ar^{\ul{h^{C^{a^{\pr}}}b^{\pr}}} "22";"23"};
{\ar@{->}^{\sigma(\ul{c^{\pr})}} "23";"24"};
{\ar@{=}^{} "11";"21"};
{\ar@{=}^{} "12";"22"};
{\ar^{\ul{s}}_{\vsim} "13";"23"};
{\ar@{=}^{} "14";"24"};
{\ar@{}|\car "11";"22"};
{\ar@{}|\car "12";"23"};
{\ar@{}|\car "13";"24"};
\endxy
\]

In particular, the isomorphism class of the sequence \eqref{diag:right_tri} in $\ZZ$ does not depend on the choices of $a$ up to $[\II]$.
\end{lem}

Before we show the above statement, we prove the following claim.

\begin{cla} \label{induced_right_tri}
Assume that there exists a commutative diagram in $\CC$ with two $\fraks^{\II}_{\II}$-triangles
$X_1 \xrar{a_1} Y_1 \xrar{b_1} U_1 \lxdrar{\delta_1} X_1$ and 
$X_2 \xrar{a_2} Y_2 \xrar{b_2} U_2 \lxdrar{\delta_2} X_2$ where $X_i,Y_j \in \ZZ$ for $1 \leq i,j \leq 2$.
\[
\xy
(0,8)*+{X_1}="11";
(16,8)*+{Y_1}="12";
(32,8)*+{U_1}="13";
(48,8)*+{X_1}="14";
(0,-8)*+{X_2}="21";
(16,-8)*+{Y_2}="22";
(32,-8)*+{U_2}="23";
(48,-8)*+{X_2}="24";
{\ar^{a_1} "11";"12"};
{\ar^{b_1} "12";"13"};	
{\ar@{-->}^{\delta_1} "13";"14"};
{\ar^{a_2} "21";"22"};
{\ar^{b_2} "22";"23"};
{\ar@{-->}^{\delta_2} "23";"24"};
{\ar^{x} "11";"21"};
{\ar^{y} "12";"22"};
{\ar^{u} "13";"23"};
{\ar^{x} "14";"24"};
{\ar@{}|\car "11";"22"};
{\ar@{}|\car "12";"23"};
{\ar@{}|\car "13";"24"};
\endxy
\]
This induces the following commutative diagram in $\ul{\ZZ}$.
\[
\xy
(0,8)*+{X_1}="11";
(16,8)*+{Y_1}="12";
(32,8)*+{\sigma U_1}="13";
(48,8)*+{\Sigma X_1}="14";
(0,-8)*+{X_2}="21";
(16,-8)*+{Y_2}="22";
(32,-8)*+{\sigma U_2}="23";
(48,-8)*+{\Sigma X_2}="24";
{\ar^{\ul{a_1}} "11";"12"};
{\ar^{\ul{h^{U_1} b_1}} "12";"13"};	
{\ar@{->}^{\sigma{(\ul{c_1})}} "13";"14"};
{\ar^{\ul{a_2}} "21";"22"};
{\ar^{\ul{h^{U_2} b_2}} "22";"23"};
{\ar@{->}^{\sigma{(\ul{c_2})}} "23";"24"};
{\ar^{\ul{x}} "11";"21"};
{\ar^{\ul{y}} "12";"22"};
{\ar^{\sigma(\ul{u})} "13";"23"};
{\ar^{\Sigma \ul{x}} "14";"24"};
{\ar@{}|\car "11";"22"};
{\ar@{}|\car "12";"23"};
{\ar@{}|{\phantom{XX}\car} "13";"24"};
\endxy
\]
\end{cla}
\begin{proof}
First, we recall the following commutative diagrams \eqref{diag:two-brackets} in $\CC$.

Then $\lambda^{X_2} x\bracket{1} c_1 = x \lambda^{X_1} c_1 = x \delta_1 = \lambda^{X_2} c_2 u$.
Thus, $x\bracket{1} c_1 - c_2 u$ factors through $I^{X_2}$ and $\ul{x}\bracket{1} \ul{c_1} = \ul{c_2 u}$.
Applying $\sigma$, $\Sigma(\ul{x}) \sigma(\ul{c_1}) = \sigma(\ul{c_2}) \sigma(\ul{u})$.
Therefore, the claim holds.
\begin{align}
\xy
(0,24)*+{X_1}="11";
(16,24)*+{I^{X_1}}="12";
(32,24)*+{X_1\bracket{1}}="13";
(48,24)*+{X_1}="14";
(0,8)*+{X_1}="21";
(16,8)*+{Y_1}="22";
(32,8)*+{U_1}="23";
(48,8)*+{X_1}="24";
(0,-8)*+{X_2}="31";
(16,-8)*+{Y_2}="32";
(32,-8)*+{U_2}="33";
(48,-8)*+{X_2}="34";
(0,-24)*+{X_2}="41";
(16,-24)*+{I^{X_2}}="42";
(32,-24)*+{X_2\bracket{1}}="43";
(48,-24)*+{X_2}="44";
{\ar^{i^{X_1}} "11";"12"};
{\ar^{p^{X_1}} "12";"13"};	
{\ar@{-->}^{\lambda^{X_1}} "13";"14"};
{\ar^{a_1} "21";"22"};
{\ar^{b_1} "22";"23"};
{\ar@{-->}^{\delta_1} "23";"24"};
{\ar^{a_2} "31";"32"};
{\ar^{b_2} "32";"33"};
{\ar@{-->}^{\delta_2} "33";"34"};
{\ar^{i^{X_2}} "41";"42"};
{\ar^{p^{X_2}} "42";"43"};
{\ar@{-->}^{\lambda^{X_2}} "43";"44"};
{\ar@{=}^{} "11";"21"};
{\ar@{<-}^{} "12";"22"};
{\ar@{<-}^{c_1} "13";"23"};
{\ar@{=}^{} "14";"24"};
{\ar^{x} "21";"31"};
{\ar^{y} "22";"32"};
{\ar^{u} "23";"33"};
{\ar^{x} "24";"34"};
{\ar@{=}^{} "31";"41"};
{\ar^{} "32";"42"};
{\ar^{c_2} "33";"43"};
{\ar@{=}^{} "34";"44"};
{\ar@{}|\car "11";"22"};
{\ar@{}|\car "12";"23"};
{\ar@{}|\car "13";"24"};
{\ar@{}|\car "21";"32"};
{\ar@{}|\car "22";"33"};
{\ar@{}|\car "23";"34"};
{\ar@{}|\car "31";"42"};
{\ar@{}|\car "32";"43"};
{\ar@{}|\car "33";"44"};
\endxy
\qquad
\xy
(0,8)*+{X_1}="11";
(16,8)*+{I^{X_1}}="12";
(32,8)*+{X_1\bracket{1}}="13";
(48,8)*+{X_1}="14";
(0,-8)*+{X_2}="21";
(16,-8)*+{I^{X_2}}="22";
(32,-8)*+{X_2\bracket{1}}="23";
(48,-8)*+{X_2}="24";
{\ar^{i^{X_1}} "11";"12"};
{\ar^{p^{X_1}} "12";"13"};	
{\ar@{-->}^{\lambda^{X_1}} "13";"14"};
{\ar^{i^{X_2}} "21";"22"};
{\ar^{p^{X_2}} "22";"23"};
{\ar@{-->}^{\lambda^{X_2}} "23";"24"};
{\ar^{x} "11";"21"};
{\ar^{} "12";"22"};
{\ar^{x\bracket{1}} "13";"23"};
{\ar^{x} "14";"24"};
{\ar@{}|\car "11";"22"};
{\ar@{}|\car "12";"23"};
{\ar@{}|{\phantom{XX}\car} "13";"24"};
\endxy
\label{diag:two-brackets}
\end{align}
\end{proof}

\noindent {\textbf{Proof of Lemma \ref{previous_lem}.}}
Since $i^X \colon X \to I^X$ is a left $\II$-approximation, $a-a^{\pr}$ factors through $i^X$.
Let $k^X \colon I^X \to Y$ be a morphism where  $a-a^{\pr} = k^X i^X$. 
There exists an $\fraks$-triangle
$I^X \xrar{\msize{0.6}{\begin{bmatrix} k^X \\ 1  \end{bmatrix}}} Y \oplus I^X 
\xrar{\msize{0.6}{\begin{bmatrix} 1 & -k^X  \end{bmatrix}}} Y \lxdrar{0} I^X$.
Then we obtain the following commutative diagram from Lemma \ref{happel_diagram}.
\[
\xy
(0,24)*+{}="11";
(16,24)*+{I^X}="12";
(32,24)*+{I^X}="13";
(48,24)*+{}="14";
(0,8)*+{X}="21";
(16,8)*+{Y \!\oplus\! I^X}="22";
(32,8)*+{C^a}="23";
(48,8)*+{X}="24";
(0,-8)*+{X}="31";
(16,-8)*+{Y}="32";
(32,-8)*+{U^{\pr}}="33";
(48,-8)*+{X}="34";
(0,-24)*+{}="41";
(16,-24)*+{I^X}="42";
(32,-24)*+{I^X}="43";
{\ar@{=} "12";"13"};
{\ar^-{\wt{a}} "21";"22"};
{\ar^-{\wt{b}} "22";"23"};
{\ar@{-->}^{\wt{\delta}} "23";"24"};
{\ar^{a^{\pr}} "31";"32"};
{\ar^{b^{\pr}} "32";"33"};
{\ar@{-->}^{\delta^{\pr}} "33";"34"};
{\ar@{=} "42";"43"};
{\ar@{=} "21";"31"};
{\ar^{\msize{0.6}{\begin{bmatrix} k^X \\ 1  \end{bmatrix}}} "12";"22"};
{\ar^{\msize{0.6}{\begin{bmatrix} 1 & -k^X  \end{bmatrix}}} "22";"32"};
{\ar@{-->}^{0} "32";"42"};
{\ar^{} "13";"23"};
{\ar^{u} "23";"33"};
{\ar@{-->}^{0} "33";"43"};
{\ar@{=} "24";"34"};
{\ar@{}|\car "12";"23"};
{\ar@{}|\car "21";"32"};
{\ar@{}|\car@<12pt> "22";"33"};
{\ar@{}|\car "23";"34"};
{\ar@{}|\car "32";"43"};
\endxy
\]
From the above claim, we only have to prove $\sigma(\ul{u})$ is an isomorphism.
However, this is clear because $\ul{u}$ is isomorphic from the above diagram.
$\qed$

\hspace{8pt}

Therefore, the sequence \eqref{diag:right_tri} does not depend on the choices of $a$ up to $[\II]$. 
Thus, the following definition makes sense.

\begin{defi}
Let $\ul{a} \colon X \to Y$ be a morphism in $\ul{\ZZ}$. 
Then there exists the following commutative diagram in $\ul{\ZZ}$.
That is because 
$\ul{h^{C^a}\wt{b}}$ 
$=$ $\ul{h^{C^a}\wt{b} \msize{0.8}{\begin{bmatrix} 1 & 0 \\ 0 & 0 \end{bmatrix}}}$
$=$ $\ul{h^{C^a} b \msize{0.8}{\begin{bmatrix} 1 & 0 \end{bmatrix}}}$.
\[
\xy
(0,8)*+{X}="11";
(16,8)*+{Y\oplus I^X}="12";
(40,8)*+{\sigma C^a}="13";
(56,8)*+{\Sigma X}="14";
(0,-8)*+{X}="21";
(16,-8)*+{Y}="22";
(40,-8)*+{\sigma C^a}="23";
(56,-8)*+{\Sigma X}="24";
{\ar^-{\ul{\wt{a}}} "11";"12"};
{\ar^{\ul{h^{C^a}\wt{b}}} "12";"13"};
{\ar^{\sigma(\ul{c})} "13";"14"};
{\ar^{\ul{a}} "21";"22"};
{\ar^{\ul{h^{C^a}b}} "22";"23"};
{\ar^{\sigma(\ul{c})} "23";"24"};
{\ar@{=} "11";"21"};
{\ar^{\msize{0.6}{\begin{bmatrix} 1 & 0 \end{bmatrix}}}_{\vsim} "12";"22"};
{\ar@{=} "13";"23"};
{\ar@{=} "14";"24"};
{\ar@{}|\car "11";"22"};
{\ar@{}|\car "12";"23"};
{\ar@{}|\car "13";"24"};
\endxy
\]
Then the following sequence in $\ul{\ZZ}$ is unique up to isomorphisms.
It is called the \emph{standard right triangle} of $\ul{a}$.
\begin{align}
X \xrar{\, \ul{a} \,} Y \xrar{\, \ul{h^{C^a} b} \,} \sigma C^a \xrar{\sigma(\ul{c})} \Sigma X
\end{align}
We define
\[
\nabla = 
\left(
\begin{array}{ll}
\text{sequences} &\text{in } \ul{\ZZ} \text{ isomorphic to one in } \\
&\{ X \xrar{\ul{a}} Y \xrar{\ul{h^{C^a}b}} \sigma C^a \xrar{\sigma(\ul{c})} \Sigma X \mid a \text{ is a morphism in } \ZZ \}
\end{array}
\right)
\]
and a sequence in $\nabla$ is called a \emph{right triangle} in $\ul{\ZZ}$.
\end{defi}

\begin{ex} \label{ex_righttri}
The diagram \eqref{THE_shifted_oct} is induced by $a = \mathrm{id}_Z$ for $Z \in \ZZ$.
\[
	\xy
	(0,24)*+{}="11";
	(0,8)*+{Z}="21";
	(0,-8)*+{Z}="31";
	(0,-24)*+{}="41";
	(16,24)*+{Z}="12";
	(16,8)*+{Z \! \oplus \! I^Z}="22";
	(16,-8)*+{I^Z}="32";
	(16,-24)*+{Z}="42";
	(32,24)*+{Z}="13";
	(32,8)*+{I^Z}="23";
	(32,-8)*+{Z \bracket{1}}="33";
	(32,-24)*+{Z}="43";
	(48,24)*+{}="14";
	(48,8)*+{Z}="24";
	(48,-8)*+{Z}="34";
	(48,-24)*+{Z \! \oplus \! I^Z}="44";
	{\ar@{=} "12";"13"};
	{\ar^-{\msize{0.6}{\begin{bmatrix} 1 \\ i^Z  \end{bmatrix}}} "21";"22"};
	{\ar^-{\msize{0.6}{\begin{bmatrix} -i^Z & 1  \end{bmatrix}}} "22";"23"};
	{\ar@{-->}^{0} "23";"24"};
	{\ar^{i^Z} "31";"32"};
	{\ar^{p^Z} "32";"33"};
	{\ar@{-->}^{\lambda^Z} "33";"34"};
	{\ar@{=} "42";"43"};
	{\ar^{\msize{0.6}{\begin{bmatrix} -1 \\ 0 \\ \end{bmatrix}}} "43";"44"};
	{\ar@{=} "21";"31"};
	{\ar^{\msize{0.6}{\begin{bmatrix} 1 \\ 0 \\ \end{bmatrix}}} "12";"22"};
	{\ar^{\msize{0.6}{[0 \ 1]}} "22";"32"};
	{\ar@{-->}^{0} "32";"42"};
	{\ar^{-i^Z} "13";"23"};
	{\ar^{p^Z} "23";"33"};
	{\ar@{-->}^{-\lambda^Z} "33";"43"};
	{\ar@{=} "24";"34"};
	{\ar^{\msize{0.6}{\begin{bmatrix} 1 \\ i^Z  \end{bmatrix}}} "34";"44"};
	{\ar@{}|\car "12";"23"};
	{\ar@{}|\car "21";"32"};
	{\ar@{}|\car "22";"33"};
	{\ar@{}|\car "23";"34"};
	{\ar@{}|\car "32";"43"};
	{\ar@{}|\car@<10pt> "33";"44"};
	\endxy
\]
Thus, $Z \xrar{\ul{\mathrm{id}_Z}} Z \xrar{\ul{0}} 0 \xrar{\ul{0}} \Sigma Z$ is a right triangle.
\end{ex}

\subsection{Right triangulated structures induced by right mutation doubles}

We assume that $(\SS, \ZZ)$ is a right mutation double.
Now, we check the triplet $(\ul{\ZZ}, \Sigma, \nabla)$ satisfies the axioms of right triangulated category.

\begin{lem} \label{proof_of_RT0-1}
The triplet $(\ul{\ZZ}, \Sigma, \nabla)$ satisfies (rTR0) and (rTR1) in Remark \ref{axioms_of_RTri}.
\end{lem}
\begin{proof}
(rTR0) is by definition of $\nabla$.
(rTR1)(i) follows from Example \ref{ex_righttri}.
Since $\II$ is strongly covariantly finite in $\ZZ$, (rTR1)(ii) holds from Lemma \ref{inflation}.
\end{proof}

\begin{lem} \label{proof_of_RT2}
$(\ul{\ZZ}, \Sigma, \nabla)$ satisfies (rTR2) in Remark \ref{axioms_of_RTri}.
\end{lem}
\begin{proof}
We only have to show for the standard right triangles 
$X \xrar{\ul{a}} Y \xrar{\ul{h^{C^a} b}} \sigma C^a \xrar{\sigma({\ul{c})}} \Sigma X$.
In the rest of this proof, we denote $\sigma C^a$ by $Z$.
Recall that both $b \colon Y \to C^a$ and $h^{C^a} \colon C^a \to Z$ are $\fraks^{\II}$-inflations. 
Thus $h^{C^a} b$, now denoted by $b^{\pr}$, is also an $\fraks^{\II}$-inflation and 
there exists an $\fraks^{\II}$-triangle $Y \xrar{b^{\pr}} Z \xrar{c^{\pr}} U \xdrar{\delta^{\pr}} Y$
(this is also an $\fraks_{\II}$-triangle since $Y \in \ul{\ZZ}$).
We define a morphism $a^{\pr} \colon U \to Y \bracket{1}$ by the following commutative diagram.
\begin{align}
\xy
(0,8)*+{Y}="11";
(16,8)*+{Z}="12";
(32,8)*+{U}="13";
(48,8)*+{Y}="14";
(0,-8)*+{Y}="21";
(16,-8)*+{I^Y}="22";
(32,-8)*+{Y \bracket{1}}="23";
(48,-8)*+{Y}="24";
{\ar^{b^{\pr}} "11";"12"};
{\ar^{c^{\pr}} "12";"13"};	
{\ar@{-->}^{\delta^{\pr}} "13";"14"};
{\ar^{i^Y} "21";"22"};
{\ar^{p^Y} "22";"23"};
{\ar@{-->}^{\lambda^{Y}} "23";"24"};
{\ar@{=}^{} "11";"21"};
{\ar^{} "12";"22"};
{\ar^{a^{\pr}} "13";"23"};
{\ar@{=}^{} "14";"24"};
{\ar@{}|\car "11";"22"};
{\ar@{}|\car "12";"23"};
{\ar@{}|\car "13";"24"};
\endxy \label{diag:pullback_7-2}
\end{align}
This $\fraks^{\II}_{\II}$-triangle induces the following right triangle.
\[
Y \xrar{\ul{b^{\pr}}} Z \xrar{\ul{h^U c^{\pr}}} \sigma U \xrar{\sigma(\ul{a^{\pr}})} \Sigma Y
\]

It is enough to show that the above sequence is isomorphic to 
$Y \xrar{\ul{b^{\pr}}} Z \xrar{\sigma(\ul{c})} \Sigma X \xrar{-\Sigma \ul{a}} \Sigma Y$.
From (ET4), we obtain a morphism $u \colon X\bracket{1} \to U$ 
which is defined by the following left commutative diagram.
We also obtain the following right commutative diagram.

\begin{align}
\xy
(0,24)*+{}="11";
(16,24)*+{S^{C^a}}="12";
(32,24)*+{S^{C^a}}="13";
(48,24)*+{}="14";
(0,8)*+{Y}="21";
(16,8)*+{C^a}="22";
(32,8)*+{X\bracket{1}}="23";
(48,8)*+{Y}="24";
(0,-8)*+{Y}="31";
(16,-8)*+{Z}="32";
(32,-8)*+{U}="33";
(48,-8)*+{Y}="34";
(0,-24)*+{}="41";
(16,-24)*+{S^{C^a}}="42";
(32,-24)*+{S^{C^a}}="43";
{\ar@{=} "12";"13"};
{\ar^{b} "21";"22"};
{\ar^{c} "22";"23"};
{\ar@{-->}^{\gamma} "23";"24"};
{\ar^{b^{\pr}} "31";"32"};
{\ar^{c^{\pr}} "32";"33"};
{\ar@{-->}^{\delta^{\pr}} "33";"34"};
{\ar@{=} "42";"43"};
{\ar@{=} "21";"31"};
{\ar@{-->}^{} "12";"22"};
{\ar^{h^{C^a}} "22";"32"};
{\ar^{} "32";"42"};
{\ar@{-->}^{} "13";"23"};
{\ar^{u} "23";"33"};
{\ar^{} "33";"43"};
{\ar@{=} "24";"34"};
{\ar@{}|\car "12";"23"};
{\ar@{}|\car "21";"32"};
{\ar@{}|\car "22";"33"};
{\ar@{}|\car "23";"34"};
{\ar@{}|\car "32";"43"};
\endxy
\qquad
\xy
(0,16)*+{C^a}="11";
(16,16)*+{Z}="12";
(0,0)*+{X\bracket{1}}="21";
(16,0)*+{\Sigma X}="22";
(0,-16)*+{U}="31";
(16,-16)*+{\sigma U}="32";
{\ar^{h^{C^a}} "11";"12"};
{\ar^{h^{X\bracket{1}}} "21";"22"};	
{\ar^{h^U} "31";"32"};
{\ar^{c} "11";"21"};
{\ar^{u} "21";"31"};
{\ar^{\sigma (c)} "12";"22"};
{\ar^{\sigma (u)} "22";"32"};
{\ar@{}|\car "11";"22"};
{\ar@{}|\car "21";"32"};
\endxy \label{diag:rotated-righttri}
\end{align}
Then $\sigma(u) \sigma(c) h^{C^a} = h^U u c = h^U c^{\pr} h^{C^a}$. 
From Remark \ref{rmk_adj_of_sigma}(\ref{rmk_adj_of_sigma_1}), $\sigma(\ul{u}) \sigma(\ul{c}) = \ul{h^U c^{\pr}}$.

On the other hand, from \eqref{THE_shifted_oct}, \eqref{diag:pullback_7-2} and \eqref{diag:rotated-righttri}, 
$a \lambda^X = -\gamma = -\delta^{\pr} u = -\lambda^Y a^{\pr} u$.
Thus, there exists the following commutative diagram.
\[
\xy
(0,8)*+{X}="11";
(16,8)*+{I^X}="12";
(32,8)*+{X\bracket{1}}="13";
(48,8)*+{X}="14";
(0,-8)*+{Y}="21";
(16,-8)*+{I^Y}="22";
(32,-8)*+{Y\bracket{1}}="23";
(48,-8)*+{Y}="24";
{\ar^{i^X} "11";"12"};
{\ar^{p^X} "12";"13"};	
{\ar@{-->}^{\lambda^X} "13";"14"};
{\ar^{i^Y} "21";"22"};
{\ar^{p^Y} "22";"23"};
{\ar@{-->}^{\lambda^Y} "23";"24"};
{\ar^{a} "11";"21"};
{\ar^{} "12";"22"};
{\ar^{-a^{\pr} u} "13";"23"};
{\ar^{a} "14";"24"};
{\ar@{}|\car "11";"22"};
{\ar@{}|\car "12";"23"};
{\ar@{}^\car "13";"24"};
\endxy
\]
So, $\ul{a}\bracket{1} = - \ul{a^{\pr}u}$ and $\Sigma \ul{a} = -\sigma(\ul{a^{\pr}}) \sigma(\ul{u})$.
From Lemma \ref{isomorphism}, $\sigma(\ul{u})$ is an isomorphism. 
Therefore, there exists the following commutative diagram we wanted.
\[
\xy
(0,8)*+{Y}="11";
(16,8)*+{Z}="12";
(32,8)*+{\Sigma X}="13";
(48,8)*+{\Sigma Y}="14";
(0,-8)*+{Y}="21";
(16,-8)*+{Z}="22";
(32,-8)*+{\sigma U}="23";
(48,-8)*+{\Sigma Y}="24";
{\ar^{\ul{b^{\pr}}} "11";"12"};
{\ar^{\sigma(\ul{c})} "12";"13"};	
{\ar^{- \Sigma\ul{a}} "13";"14"};
{\ar^{\ul{b^{\pr}}} "21";"22"};
{\ar^{\ul{h^U c^{\pr}}} "22";"23"};
{\ar^{\sigma(\ul{a^{\pr}})} "23";"24"};
{\ar@{=} "11";"21"};
{\ar@{=} "12";"22"};
{\ar^{\sigma(\ul{u})}_{\vsim} "13";"23"};
{\ar@{=} "14";"24"};
{\ar@{}|\car "11";"22"};
{\ar@{}|\car "12";"23"};
{\ar@{}|{\phantom{XX}\car} "13";"24"};
\endxy
\]
\end{proof}

\begin{lem} \label{proof_of_RT3}
$(\ul{\ZZ}, \Sigma, \nabla)$ satisfies (rTR3) in Remark \ref{axioms_of_RTri}, 
that is, the following statement holds.

Assume that there exists the following commutative diagram where each row is a right triangle.
	\[
	\xy
	(0,8)*+{X_1}="11";
	(20,8)*+{Y_1}="12";
	(40,8)*+{\sigma C^{a_1}}="13";
	(60,8)*+{\Sigma X_1}="14";
	(0,-8)*+{X_2}="21";
	(20,-8)*+{Y_2}="22";
	(40,-8)*+{\sigma C^{a_2}}="23";
	(60,-8)*+{\Sigma X_2}="24";
	{\ar^{\ul{a_1}} "11";"12"};
	{\ar^{\ul{h^{C^{a_1}} b_1}} "12";"13"};	
	{\ar^{\sigma(\ul{c_1})} "13";"14"};
	{\ar^{\ul{a_2}} "21";"22"};
	{\ar^{\ul{h^{C^{a_2}} b_2}} "22";"23"};
	{\ar^{\sigma(\ul{c_2})} "23";"24"};
	{\ar^{\ul{x}} "11";"21"};
	{\ar^{\ul{y}} "12";"22"};
	{\ar^{\Sigma \ul{x}} "14";"24"};
	{\ar@{}|\car "11";"22"};
	\endxy
	\]
	Then there exists a morphism $z \colon \sigma C^{a_1} \to \sigma C^{a_2}$ 
	which makes the following diagram 
	commutative.
	\[
	\xy
	(0,8)*+{X_1}="11";
	(20,8)*+{Y_1}="12";
	(40,8)*+{\sigma C^{a_1}}="13";
	(60,8)*+{\Sigma X_1}="14";
	(0,-8)*+{X_2}="21";
	(20,-8)*+{Y_2}="22";
	(40,-8)*+{\sigma C^{a_2}}="23";
	(60,-8)*+{\Sigma X_2}="24";
	{\ar^{\ul{a_1}} "11";"12"};
	{\ar^{\ul{h^{C^{a_1}} b_1}} "12";"13"};	
	{\ar^{\sigma(\ul{c_1})} "13";"14"};
	{\ar^{\ul{a_2}} "21";"22"};
	{\ar^{\ul{h^{C^{a_2}} b_2}} "22";"23"};
	{\ar^{\sigma(\ul{c_2})} "23";"24"};
	{\ar^{\ul{x}} "11";"21"};
	{\ar^{\ul{y}} "12";"22"};
	{\ar^{\ul{z}} "13";"23"};
	{\ar^{\Sigma \ul{x}} "14";"24"};
	{\ar@{}|\car "11";"22"};
	{\ar@{}|\car "12";"23"};
	{\ar@{}|\car "13";"24"};
	\endxy
	\]
\end{lem}
\begin{proof}
From $\ul{y a_1} = \ul{a_2 x}$ and $i^{X_1}$ is a left $\II$-approximation, 
there exist morphisms $j_1 \colon I^{X_1} \to Y_2$ and $j_2 \colon I^{X_1} \to I^{X_2}$
where $a_2 x - y a_1 = j_1 i^{X_1}$ and $i^{X_2} x = j_2 i^{X_1}$. 
Then $\msize{0.8}{\begin{bmatrix} y & j_1 \\ 0 & j_2  \end{bmatrix}} \wt{a_1} = 
\msize{0.8}{\begin{bmatrix} y a_1 + j_1 i^{X_1} \\ j_2 i^{X_1} \end{bmatrix}} = \wt{a_2} x$.
Thus, there exists the following commutative diagram in $\CC$.
\[
\xy
(0,8)*+{X_1}="11";
(16,8)*+{Y_1 \!\oplus\! I^{X_1}}="12";
(32,8)*+{C^{a_1}}="13";
(48,8)*+{X_1}="14";
(0,-8)*+{X_2}="21";
(16,-8)*+{Y_2 \!\oplus\! I^{X_2}}="22";
(32,-8)*+{C^{a_2}}="23";
(48,-8)*+{X_2}="24";
{\ar^-{\wt{a_1}} "11";"12"};
{\ar^-{\wt{b_1}} "12";"13"};	
{\ar@{-->}^{\wt{\delta_1}} "13";"14"};
{\ar^-{\wt{a_2}} "21";"22"};
{\ar^-{\wt{b_2}} "22";"23"};
{\ar@{-->}^{\wt{\delta_2}} "23";"24"};
{\ar^{x} "11";"21"};
{\ar^{\msize{0.6}{\begin{bmatrix} y & j_1 \\ 0 & j_2 \end{bmatrix}}} "12";"22"};
{\ar^{} "13";"23"};
{\ar^{x} "14";"24"};
{\ar@{}|\car "11";"22"};
{\ar@{}|{\phantom{Xx}\car} "12";"23"};
{\ar@{}|\car "13";"24"};
\endxy
\]
The statement follows from Lemma \ref{previous_lem}.
\end{proof}

To show $(\ul{\ZZ}, \Sigma, \nabla)$ satisfies (rTR4) in Remark \ref{axioms_of_RTri}, 
we need some preparations.

\begin{lem} \label{lem_for_RT4}
Let $U_1 \xrar{a} U_2 \xrar{b} U_3 \xdrar{\delta} U_1$ be an $\fraks^{\II}$-triangle where 
$U_1, U_2, U_3 \in \Cone_{\bbE^{\II}}(\ZZ, \ZZ)$. Then there exist the following three $\fraks^{\II}$-triangles.
\[
\begin{array}{l}
	S_2 \drar U_2 \xrar{z_2} Z_2 \rar S_2 \\
	S_3 \drar U_3 \xrar{u_3} {U_3}^{\pr} \rar S_3 \\
	\sigma U_1 \xrar{x} Z_2 \xrar{y} {U_3}^{\pr} \lxdrar{\epsilon} Z_1
\end{array}
\]
where $S_2, S_3 \in \SS$, $Z_2 \in \ZZ$, ${U_3}^{\pr} \in \Cone_{\bbE^{\II}}(\ZZ, \ZZ)$ and make the following diagram commutative.
\[
\xy
(0,8)*+{U_1}="11";
(16,8)*+{U_2}="12";
(32,8)*+{U_3}="13";
(48,8)*+{U_1}="14";
(0,-8)*+{\sigma U_1}="21";
(16,-8)*+{Z_2}="22";
(32,-8)*+{{U_3}^{\pr}}="23";
(48,-8)*+{\sigma U_1}="24";
{\ar^{a} "11";"12"};
{\ar^{b} "12";"13"};	
{\ar@{-->}^{\delta} "13";"14"};
{\ar^{x} "21";"22"};
{\ar^{y} "22";"23"};
{\ar@{-->}^{\epsilon} "23";"24"};
{\ar^{h^{U_1}} "11";"21"};
{\ar^{z_2} "12";"22"};
{\ar^{u_3} "13";"23"};
{\ar^{h^{U_1}} "14";"24"};
{\ar@{}|\car "11";"22"};
{\ar@{}|\car "12";"23"};
{\ar@{}|\car "13";"24"};
\endxy
\]
Moreover, there exists the following commutative diagram.
\[
\xy
(0,8)*+{\sigma U_1}="11";
(20,8)*+{\sigma U_2}="12";
(40,8)*+{\sigma U_3}="13";
(60,8)*+{\Sigma(\sigma U_1)}="14";
(0,-8)*+{\sigma U_1}="21";
(20,-8)*+{Z_2}="22";
(40,-8)*+{\sigma {U_3}^{\pr}}="23";
(60,-8)*+{\Sigma(\sigma U_1)}="24";
{\ar^-{\sigma(\ul{a})} "11";"12"};
{\ar^-{\sigma(\ul{b})} "12";"13"};	
{\ar^-{\sigma(\ul{c^x u_3})} "13";"14"};
{\ar^{\ul{x}} "21";"22"};
{\ar^{\ul{h^{{U_3}^{\pr}} y}} "22";"23"};
{\ar^{\sigma(\ul{c^x})} "23";"24"};
{\ar@{=}^{} "11";"21"};
{\ar^{\sigma(\ul{z_2})} "12";"22"};
{\ar^{\sigma(\ul{u_3})} "13";"23"};
{\ar@{=}^{} "14";"24"};
{\ar@{}|\car "11";"22"};
{\ar@{}|\car "12";"23"};
{\ar@{}|{\phantom{Xx}\car} "13";"24"};
\endxy
\]

In particular,
\[
\sigma U_1 \xrar{\sigma(\ul{a})} \sigma U_2 \xrar{\sigma(\ul{b})} \sigma U_3 \xrar{\sigma(\ul{c^x u_3})} \Sigma(\sigma U_1)
\]
is a right triangle.
\end{lem}
\begin{proof}
First, $\ZZ \!\ast\! (\Cone(\ZZ, \ZZ)) = \Cone(\ZZ,\ZZ)$ in $(\CC, \bbE^{\II}, \fraks^{\II})$ from Proposition \ref{shifted_octahedrons}.
From the dual of Proposition \ref{shifted_octahedrons} with $\fraks^{\II}$-triangles 
$U_1 \xrar{a} U_2 \xrar{b} U_3 \lxdrar{\delta} U_3$ and 
$U_1 \xrar{h^{U_1}} \sigma U_1 \rar S^{U_1} \xdrar{\rho^{U_1}} U_1$,
 we obtain the following commutative diagram.
\[
\xy
(0,24)*+{}="11";
(16,24)*+{U_3}="12";
(32,24)*+{U_3}="13";
(48,24)*+{}="14";
(0,8)*+{S^{U_1}}="21";
(16,8)*+{U_1}="22";
(32,8)*+{\sigma U_1}="23";
(48,8)*+{S^{U_1}}="24";
(0,-8)*+{S^{U_1}}="31";
(16,-8)*+{U_2}="32";
(32,-8)*+{{U_2}^{\pr}}="33";
(48,-8)*+{S^{U_1}}="34";
(0,-24)*+{}="41";
(16,-24)*+{U_3}="42";
(32,-24)*+{U_3}="43";
{\ar@{=} "12";"13"};
{\ar@{-->}^{\rho^{U_1}} "21";"22"};
{\ar^{h^{U_1}} "22";"23"};
{\ar@{->}^{} "23";"24"};
{\ar@{-->}^{} "31";"32"};
{\ar^{u_2} "32";"33"};
{\ar@{->}^{} "33";"34"};
{\ar@{=} "42";"43"};
{\ar@{=} "21";"31"};
{\ar@{-->}^{\delta} "12";"22"};
{\ar^{a} "22";"32"};
{\ar@{->}^{b} "32";"42"};
{\ar@{-->}^{h^{U_1} \delta} "13";"23"};
{\ar^{{u_2}^{\pr}} "23";"33"};
{\ar@{->}^{{v_2}^{\pr}} "33";"43"};
{\ar@{=} "24";"34"};
%
{\ar@{}|\car "12";"23"};
{\ar@{}|\car "21";"32"};
{\ar@{}|\car "22";"33"};
{\ar@{}|\car "23";"34"};
{\ar@{}|\car "32";"43"};
\endxy
\]
Note that ${U_2}^{\pr} \in \Cone_{\bbE_{\II}}(\ZZ, \ZZ)$ by the first remark.

Next, from (ET4) with $\fraks^{\II}$-triangles
$\sigma U_1 \xrar{{u_2}^{\pr}} {U_2}^{\pr} \xrar{{v_2}^{\pr}} U_3 \xdrar{h^{U_1} \delta} \sigma U_1$ and
${U_2}^{\pr} \xrar{h^{{U_2}^{\pr}}} \sigma {U_2}^{\pr} \rar S^{{U_2}^{\pr}} \xdrar{\rho^{{U_2}^{\pr}}} {U_2}^{\pr} \vphantom{\Big(}$,
there exists the following commutative diagram where ${U_3}^{\pr} \in \Cone_{\bbE^{\II}}(\ZZ, \ZZ)$.

\[
\xy
(0,24)*+{}="11";
(16,24)*+{\sigma U_1}="12";
(32,24)*+{\sigma U_1}="13";
(48,24)*+{}="14";
(0,8)*+{S^{{U_2}^{\pr}}}="21";
(16,8)*+{{U_2}^{\pr}}="22";
(32,8)*+{\sigma {U_2}^{\pr}}="23";
(48,8)*+{S^{{U_2}^{\pr}}}="24";
(0,-8)*+{S^{{U_2}^{\pr}}}="31";
(16,-8)*+{U_3}="32";
(32,-8)*+{{U_3}^{\pr}}="33";
(48,-8)*+{S^{{U_2}^{\pr}}}="34";
(0,-24)*+{}="41";
(16,-24)*+{\sigma U_1}="42";
(32,-24)*+{\sigma U_1}="43";
{\ar@{=} "12";"13"};
{\ar@{-->}^{} "21";"22"};
{\ar^{h^{{U_2}^{\pr}}} "22";"23"};
{\ar@{->}^{} "23";"24"};
{\ar@{-->}^{} "31";"32"};
{\ar^{u_3} "32";"33"};
{\ar@{->}^{} "33";"34"};
{\ar@{=} "42";"43"};
{\ar@{=} "21";"31"};
{\ar@{->}^{{u_2}^{\pr}} "12";"22"};
{\ar^{{v_2}^{\pr}} "22";"32"};
{\ar@{-->}^{h^{U_1} \delta} "32";"42"};
{\ar@{->}^{x} "13";"23"};
{\ar^{y} "23";"33"};
{\ar@{-->}^{\epsilon} "33";"43"};
{\ar@{=} "24";"34"};
%
{\ar@{}|\car "12";"23"};
{\ar@{}|\car "21";"32"};
{\ar@{}|\car "22";"33"};
{\ar@{}|\car "23";"34"};
{\ar@{}|\car@<10pt> "32";"43"};
\endxy
\]

Since ${U_3}^{\pr} \in \Cone_{\bbE^{\II}}(\ZZ, \ZZ)$, we obtain an $\fraks^{\II}$-triangle
${U_3}^{\pr} \xrar{h^{{U_3}^{\pr}}} \sigma {U_3}^{\pr} \rar S^{{U_3}^{\pr}} \xdrar{\rho^{{U_3}^{\pr}}} {U_3}^{\pr}$.
From (ET4), there exist the following commutative diagrams.
\[
\xy
(0,24)*+{}="11";
(16,24)*+{S^{{U_2}^{\pr}}}="12";
(32,24)*+{S^{{U_2}^{\pr}}}="13";
(48,24)*+{}="14";
(0,8)*+{U_2}="21";
(16,8)*+{{U_2}^{\pr}}="22";
(32,8)*+{S^{U_1}}="23";
(48,8)*+{U_2}="24";
(0,-8)*+{U_2}="31";
(16,-8)*+{\sigma {U_2}^{\pr}}="32";
(32,-8)*+{S_2}="33";
(48,-8)*+{U_2}="34";
(0,-24)*+{}="41";
(16,-24)*+{S^{{U_2}^{\pr}}}="42";
(32,-24)*+{S^{{U_2}^{\pr}}}="43";
{\ar@{=} "12";"13"};
{\ar^{u_2} "21";"22"};
{\ar^{} "22";"23"};
{\ar@{-->}^{} "23";"24"};
{\ar^{z_2} "31";"32"};
{\ar^{} "32";"33"};
{\ar@{-->}^{} "33";"34"};
{\ar@{=} "42";"43"};
{\ar@{=} "21";"31"};
{\ar@{-->}^{} "12";"22"};
{\ar^{h^{{U_2}^{\pr}}} "22";"32"};
{\ar^{} "32";"42"};
{\ar@{-->}^{} "13";"23"};
{\ar^{} "23";"33"};
{\ar^{} "33";"43"};
{\ar@{=} "24";"34"};
{\ar@{}|\car "12";"23"};
{\ar@{}|\car "21";"32"};
{\ar@{}|{\phantom{Xx}\car} "22";"33"};
{\ar@{}|\car "23";"34"};
{\ar@{}|\car "32";"43"};
\endxy
\qquad
\xy
(0,24)*+{}="11";
(16,24)*+{S^{{U_3}^{\pr}}}="12";
(32,24)*+{S^{{U_3}^{\pr}}}="13";
(48,24)*+{}="14";
(0,8)*+{U_3}="21";
(16,8)*+{{U_3}^{\pr}}="22";
(32,8)*+{S^{{U_2}^{\pr}}}="23";
(48,8)*+{U_3}="24";
(0,-8)*+{U_3}="31";
(16,-8)*+{\sigma {U_3}^{\pr}}="32";
(32,-8)*+{S_3}="33";
(48,-8)*+{U_3}="34";
(0,-24)*+{}="41";
(16,-24)*+{S^{{U_3}^{\pr}}}="42";
(32,-24)*+{S^{{U_3}^{\pr}}}="43";
{\ar@{=} "12";"13"};
{\ar^{u_3} "21";"22"};
{\ar^{} "22";"23"};
{\ar@{-->}^{} "23";"24"};
{\ar^{z_3} "31";"32"};
{\ar^{} "32";"33"};
{\ar@{-->}^{} "33";"34"};
{\ar@{=} "42";"43"};
{\ar@{=} "21";"31"};
{\ar@{-->}^{} "12";"22"};
{\ar^{h^{{U_3}^{\pr}}} "22";"32"};
{\ar^{} "32";"42"};
{\ar@{-->}^{} "13";"23"};
{\ar^{} "23";"33"};
{\ar^{} "33";"43"};
{\ar@{=} "24";"34"};
{\ar@{}|\car "12";"23"};
{\ar@{}|\car "21";"32"};
{\ar@{}|{\phantom{Xx}\car} "22";"33"};
{\ar@{}|\car "23";"34"};
{\ar@{}|\car "32";"43"};
\endxy
\]
We denote $h^{{U_2}^{\pr}} u_2, h^{{U_3}^{\pr}} u_3$ by $z_2, z_3$, respectively.
Note that $S_2, S_3 \in \SS$ since $\SS$ is closed under extensions in $(\CC, \bbE^{\II}, \fraks^{\II})$.

Then we finally obtain three $\fraks^{\II}$-triangles
\[
\begin{array}{l}
	S_2 \drar U_2 \xrar{z_2} \sigma {U_2}^{\pr} \rar S_2 \\
	S^{{U_2^{\pr}}} \drar U_3 \xrar{u_3} {U_3}^{\pr} \rar S^{{U_2^{\pr}}} \\
	\sigma U_1 \xrar{x} \sigma {U_2}^{\pr} \xrar{y} {U_3}^{\pr} \lxdrar{\epsilon} \sigma U_1
\end{array}
\]
and the following commutative diagram in $\CC$.
\[
\xy
(0,8)*+{U_1}="11";
(16,8)*+{U_2}="12";
(32,8)*+{U_3}="13";
(48,8)*+{U_1}="14";
(0,-8)*+{\sigma U_1}="21";
(16,-8)*+{\sigma {U_2}^{\pr}}="22";
(32,-8)*+{{U_3}^{\pr}}="23";
(48,-8)*+{\sigma U_1}="24";
{\ar^{a} "11";"12"};
{\ar^{b} "12";"13"};	
{\ar@{-->}^{\delta} "13";"14"};
{\ar^{x} "21";"22"};
{\ar^{y} "22";"23"};
{\ar@{-->}^{\epsilon} "23";"24"};
{\ar^{h^{U_1}} "11";"21"};
{\ar^{z_2} "12";"22"};
{\ar^{u_3} "13";"23"};
{\ar^{h^{U_1}} "14";"24"};
{\ar@{}|\car "11";"22"};
{\ar@{}|\car "12";"23"};
{\ar@{}|\car "13";"24"};
\endxy
\]

There exists the following commutative diagram
\[
\xy
(0,16)*+{U_1}="11";
(16,16)*+{U_2}="12";
(36,16)*+{U_3}="13";
(64,16)*+{\sigma {U_3}}="14";
(0,0)*+{\sigma U_1}="21";
(16,0)*+{\sigma {U_2}^{\pr}}="22";
(36,0)*+{{U_3}^{\pr}}="23";
(64,0)*+{\sigma {U_3}^{\pr}}="24";
(0,-16)*+{\sigma U_1}="31";
(16,-16)*+{I^{\sigma U_1}}="32";
(36,-16)*+{(\sigma U_1)\bracket{1}}="33";
(64,-16)*+{\Sigma (\sigma U_1)}="34";
{\ar^{a} "11";"12"};
{\ar^{b} "12";"13"};	
{\ar^{h^{U_3}} "13";"14"};	
{\ar^{x} "21";"22"};
{\ar^{y} "22";"23"};
{\ar^{h^{{U_3}^{\pr}}} "23";"24"};	
{\ar^{i^{\sigma U_1}} "31";"32"};
{\ar^{p^{\sigma U_1}} "32";"33"};
{\ar^{h^{(\sigma U_1)\bracket{1}}} "33";"34"};
{\ar^{h^{U_1}} "11";"21"};
{\ar@{=}^{} "21";"31"};
{\ar^{z_2} "12";"22"};
{\ar^{} "22";"32"};
{\ar^{u_3} "13";"23"};
{\ar^{c^x} "23";"33"};
{\ar^{\sigma(u_3)} "14";"24"};
{\ar^{\sigma(c^x)} "24";"34"};
{\ar@{}|\car "11";"22"};
{\ar@{}|\car "12";"23"};
{\ar@{}|\car "13";"24"};
{\ar@{}|\car "21";"32"};
{\ar@{}|\car "22";"33"};
{\ar@{}|\car "23";"34"};
\endxy
\]
and $\sigma U_1 \xrar{\ul{x}} \sigma {U_2}^{\pr} \xrar{\ul{h^{{U_3}^{\pr}} y}} \sigma {U_3}^{\pr} \xrar{\sigma(\ul{c^x})} \Sigma (\sigma U_1)$ is a right triangle.

Since we define $\sigma Z = Z$ for $Z \in \ZZ$, 
$\ul{x} = \sigma(\ul{x}) = \sigma(\ul{x h^{U_1}}) = \sigma(\ul{z_2 a})$ and
$\ul{h^{{U_3}^{\pr}} y} \sigma(\ul{z_2})
= \sigma(\ul{y})\sigma(\ul{z_2})= \sigma(\ul{u_3}) \sigma(\ul{b})$ from Remark \ref{rmk_sigma_of_h}.
Thus, the following diagram in $\ul{\CC}$ is commutative.
\[
\xy
(0,8)*+{\sigma U_1}="11";
(16,8)*+{\sigma U_2}="12";
(36,8)*+{\sigma U_3}="13";
(56,8)*+{\Sigma(\sigma U_1)}="14";
(0,-8)*+{\sigma U_1}="21";
(16,-8)*+{\sigma {U_2}^{\pr}}="22";
(36,-8)*+{\sigma {U_3}^{\pr}}="23";
(56,-8)*+{\Sigma (\sigma U_1)}="24";
{\ar^-{\sigma(\ul{a})} "11";"12"};
{\ar^-{\sigma(\ul{b})} "12";"13"};	
{\ar^-{\sigma(\ul{c^x u_3})} "13";"14"};
{\ar^-{\ul{x}} "21";"22"};
{\ar^-{\ul{h^{{U_3}^{\pr}} y}} "22";"23"};
{\ar^-{\sigma(\ul{c^x})} "23";"24"};
{\ar@{=}^{} "11";"21"};
{\ar^{\sigma(\ul{z_2})} "12";"22"};
{\ar^{\sigma(\ul{u_3})} "13";"23"};
{\ar@{=}^{} "14";"24"};
{\ar@{}|\car "11";"22"};
{\ar@{}|{\phantom{Xx}\car} "12";"23"};
{\ar@{}|{\phantom{Xx}\car} "13";"24"};
\endxy
\]

From Lemma \ref{isomorphism}, $\sigma(\ul{z_2})$ and $\sigma(\ul{u_3})$ are isomorphisms.
Therefore, the statement holds.
\end{proof}

\begin{prop} \label{proof_of_RT4}
$(\ul{\ZZ}, \Sigma, \nabla)$ satisfies (rTR4) in Remark \ref{axioms_of_RTri}.
\end{prop}
\begin{proof}
Let
\[ 
\begin{array}{c}
	X_1 \xrar{\,\ul{x_1}\,} X_2 \xrar{\,\ul{h^{C^{x_1} y_1}}\,} \sigma C^{x_1} \xrar{\sigma(\ul{c^{x_1}})} \Sigma X_1 \\
	X_1 \xrar{\,\ul{x_2}\,} X_3 \xrar{\,\ul{h^{C^{x_2} y_2}}\,} \sigma C^{x_2} \xrar{\sigma(\ul{c^{x_2}})} \Sigma X_2 \\
	X_2 \xrar{\,\ul{x_3}\,} X_3 \xrar{\,\ul{h^{C^{x_3} y_3}}\,} \sigma C^{x_3} \xrar{\sigma(\ul{c^{x_3}})} \Sigma X_3
\end{array}
\]
be right triangles where $\ul{x_2} = \ul{x_3 x_1}, X_i \in \ZZ$ for $1 \leq i \leq 3$. 
From Lemma \ref{previous_lem}, we may assume that $x_i$ is an $\fraks^{\II}$-inflation for $1 \leq i \leq 3$ 
and $x_2 = x_3 x_1$.
Indeed, there exists a morphism $j \colon I^{X_1} \to X_3$ which satisfies $x_2 = x_3 x_1 + j i^{X_1}$
since $i^{X_1}$ is a left $\II$-approximation.
Then $\wt{x_2} = \msize{0.8}{\begin{bmatrix} x_2 \\ i^{X_1} \end{bmatrix}}
= \msize{0.8}{\begin{bmatrix} x_3 & j \\ 0 & 1 \end{bmatrix}} \wt{x_1}$.
Thus, we may replace $x_1, x_2$ and $x_3$ by $\fraks^{\II}$-inflations $\wt{x_1}, \wt{\msize{0.8}{\begin{bmatrix} x_3 & j \\ 0 & 1 \end{bmatrix}}}\wt{x_1}$ and $\wt{\msize{0.8}{\begin{bmatrix} x_3 & j \\ 0 & 1 \end{bmatrix}}}$, respectively, up to $[\II]$.

In the rest of the proof, we denote $C^{x_i}$ by $U_i$. 
Then there exist $\fraks^{\II}_{\II}$-triangles
\[ 
\begin{array}{c}
	X_1 \xrar{x_1} X_2 \xrar{y_1} U_1 \xdrar{\delta_1} X_1 \\
	X_1 \xrar{x_2} X_3 \xrar{y_2} U_2 \xdrar{\delta_2} X_2 \\
	X_2 \xrar{x_3} X_3 \xrar{y_3} U_3 \xdrar{\delta_3} X_3
\end{array}
\]
with $x_2 = x_3 x_1$.
From (ET4), 
\begin{align}
\xy
(0,24)*+{}="11";
(16,24)*+{U_3}="12";
(32,24)*+{U_3}="13";
(48,24)*+{}="14";
(0,8)*+{X_1}="21";
(16,8)*+{X_2}="22";
(32,8)*+{U_1}="23";
(48,8)*+{X_1}="24";
(0,-8)*+{X_1}="31";
(16,-8)*+{X_3}="32";
(32,-8)*+{U_2}="33";
(48,-8)*+{X_1}="34";
(0,-24)*+{}="41";
(16,-24)*+{U_3}="42";
(32,-24)*+{U_3}="43";
(48,-24)*+{X_2}="44";
{\ar@{=} "12";"13"};
{\ar^{x_1} "21";"22"};
{\ar^{y_1} "22";"23"};
{\ar@{-->}^{\delta_1} "23";"24"};
{\ar^{x_2} "31";"32"};
{\ar^{y_2} "32";"33"};
{\ar@{-->}^{\delta_2} "33";"34"};
{\ar@{=} "42";"43"};
{\ar@{-->}^{\delta_3} "43";"44"};
{\ar@{=} "21";"31"};
{\ar@{-->}^{\delta_3} "12";"22"};
{\ar^{x_3} "22";"32"};
{\ar^{y_3} "32";"42"};
{\ar@{-->}^{\delta} "13";"23"};
{\ar^{a} "23";"33"};
{\ar^{b} "33";"43"};
{\ar@{=} "24";"34"};
{\ar^{x_1} "34";"44"};
{\ar@{}|\car "12";"23"};
{\ar@{}|\car "21";"32"};
{\ar@{}|\car "22";"33"};
{\ar@{}|\car "23";"34"};
{\ar@{}|\car "32";"43"};
{\ar@{}|\car "33";"44"};
\endxy \label{diag:proof_RT4-1}
\end{align}
and we obtain an $\fraks^{\II}_{\II}$-triangle $U_1 \xrar{a} U_2 \xrar{b} U_3 \lxdrar{\delta} U_1$. 
By Claim \ref{induced_right_tri},
\[
\xy
(0,8)*+{X_1}="11";
(16,8)*+{X_2}="12";
(32,8)*+{U_1}="13";
(48,8)*+{X_1}="14";
(0,-8)*+{X_1}="21";
(16,-8)*+{X_3}="22";
(32,-8)*+{U_2}="23";
(48,-8)*+{X_1}="24";
{\ar^{x_1} "11";"12"};
{\ar^{y_1} "12";"13"};	
{\ar@{-->}^{\delta_1} "13";"14"};
{\ar^{x_2} "21";"22"};
{\ar^{y_2} "22";"23"};
{\ar@{-->}^{\delta_2} "23";"24"};
{\ar@{=}^{} "11";"21"};
{\ar^{x_3} "12";"22"};
{\ar^{a} "13";"23"};
{\ar@{=}^{} "14";"24"};
{\ar@{}|\car "11";"22"};
{\ar@{}|\car "12";"23"};
{\ar@{}|\car "13";"24"};
\endxy
\]
induces the following commutative diagram in $\ul{\CC}$.
\begin{align}
\xy
(0,8)*+{X_1}="11";
(16,8)*+{X_2}="12";
(32,8)*+{\sigma U_1}="13";
(48,8)*+{\Sigma X_1}="14";
(0,-8)*+{X_1}="21";
(16,-8)*+{X_3}="22";
(32,-8)*+{\sigma U_2}="23";
(48,-8)*+{\Sigma X_1}="24";
(52,-8)*+{\vphantom{X}}="piri";
{\ar^{\ul{x_1}} "11";"12"};
{\ar^{\ul{h^{U_1} y_1}} "12";"13"};	
{\ar@{->}^{\sigma(\ul{c^{x_1}})} "13";"14"};
{\ar^{\ul{x_2}} "21";"22"};
{\ar^{\ul{h^{U_2} y_2}} "22";"23"};
{\ar@{->}^{\sigma(\ul{c^{x_2}})} "23";"24"};
{\ar@{=}^{} "11";"21"};
{\ar^{\ul{x_3}} "12";"22"};
{\ar^{\sigma(\ul{a})} "13";"23"};
{\ar@{=}^{} "14";"24"};
{\ar@{}|\car "11";"22"};
{\ar@{}|\car "12";"23"};
{\ar@{}|{\phantom{Xx}\car} "13";"24"};
\endxy
\label{diag:1}
\end{align}

From Lemma \ref{lem_for_RT4}, there exists the following commutative diagram
\begin{align}
\xy
(0,8)*+{U_1}="11";
(16,8)*+{U_2}="12";
(32,8)*+{U_3}="13";
(48,8)*+{U_1}="14";
(0,-8)*+{\sigma U_1}="21";
(16,-8)*+{Z_2}="22";
(32,-8)*+{{U_3}^{\pr}}="23";
(48,-8)*+{Z_1}="24";
{\ar^{a} "11";"12"};
{\ar^{b} "12";"13"};	
{\ar@{-->}^{\delta} "13";"14"};
{\ar^{x} "21";"22"};
{\ar^{y} "22";"23"};
{\ar@{-->}^{\epsilon} "23";"24"};
{\ar^{h^{U_1}} "11";"21"};
{\ar^{z_2} "12";"22"};
{\ar^{u_3} "13";"23"};
{\ar^{z_1} "14";"24"};
{\ar@{}|\car "11";"22"};
{\ar@{}|\car "12";"23"};
{\ar@{}|\car "13";"24"};
\endxy \label{diag:proof_RT4-2}
\end{align}
where $Z_2 \in \ZZ$ and ${U_3}^{\pr} \in \Cone_{\bbE^{\II}}(\ZZ, \ZZ)$.
Moreover, there exists the following commutative diagram from Claim \ref{induced_right_tri}.
Note that $\sigma(\ul{z_2})$ is an isomorphism from Lemma \ref{isomorphism}.
\begin{align}
\xy
(0,8)*+{\sigma U_1}="11";
(16,8)*+{\sigma U_2}="12";
(36,8)*+{\sigma U_3}="13";
(56,8)*+{\Sigma(\sigma U_1)}="14";
(0,-8)*+{\sigma U_1}="21";
(16,-8)*+{Z_2}="22";
(36,-8)*+{\sigma {U_3}^{\pr}}="23";
(56,-8)*+{\Sigma (\sigma U_1)}="24";
{\ar^-{\sigma(\ul{a})} "11";"12"};
{\ar^-{\sigma(\ul{b})} "12";"13"};	
{\ar^-{\sigma(\ul{c^x u_3})} "13";"14"};
{\ar^-{\ul{x}} "21";"22"};
{\ar^-{\ul{h^{{U_3}^{\pr}} y}} "22";"23"};
{\ar^-{\sigma(\ul{c^x})} "23";"24"};
{\ar@{=}^{} "11";"21"};
{\ar^{\sigma(\ul{z_2})} "12";"22"};
{\ar^{\sigma(\ul{u_3})} "13";"23"};
{\ar@{=}^{} "14";"24"};
{\ar@{}|\car "11";"22"};
{\ar@{}|\car "12";"23"};
{\ar@{}|{\phantom{Xx}\car} "13";"24"};
\endxy \label{diag:proof_RT4-3}
\end{align}
From (\ref{diag:proof_RT4-1}) and (\ref{diag:proof_RT4-2}), there exists the following commutative diagram in $\CC$.
\[
\xy
(0,8)*+{X_2}="11";
(16,8)*+{X_3}="12";
(32,8)*+{U_3}="13";
(48,8)*+{X_2}="14";
(0,-8)*+{\sigma U_1}="21";
(16,-8)*+{Z_2}="22";
(32,-8)*+{U_3^{\pr}}="23";
(48,-8)*+{\sigma U_1}="24";
{\ar^{x_3} "11";"12"};
{\ar^{y_3} "12";"13"};	
{\ar@{-->}^{\delta_3} "13";"14"};
{\ar^{x} "21";"22"};
{\ar^{y} "22";"23"};
{\ar@{-->}^{\epsilon} "23";"24"};
{\ar^{h^{U_1} y_1} "11";"21"};
{\ar^{z_2 y_2} "12";"22"};
{\ar^{u_3} "13";"23"};
{\ar^{h^{U_1} y_1} "14";"24"};
{\ar@{}|\car@<10pt> "11";"22"};
{\ar@{}|\car@<10pt> "12";"23"};
{\ar@{}|\car "13";"24"};
\endxy
\]
By Claim \ref{induced_right_tri} and (\ref{diag:proof_RT4-3}), we obtain the following commutative diagram.
\begin{align}
\xy
(0,16)*+{X_2}="11";
(16,16)*+{X_3}="12";
(32,16)*+{\sigma U_3}="13";
(52,16)*+{\Sigma X_2}="14";
(0,0)*+{\sigma U_1}="21";
(16,0)*+{Z_2}="22";
(32,0)*+{\sigma U_3^{\pr}}="23";
(52,0)*+{\Sigma (\sigma U_1)}="24";
(0,-16)*+{\sigma U_1}="31";
(16,-16)*+{\sigma U_2}="32";
(32,-16)*+{\sigma U_3}="33";
(52,-16)*+{\Sigma(\sigma U_1)}="34";
{\ar^{\ul{x_3}} "11";"12"};
{\ar^{\ul{h^{U_3} y_3}} "12";"13"};	
{\ar@{->}^{\sigma(\ul{c^{x_3}})} "13";"14"};
{\ar^{\ul{x}} "21";"22"};
{\ar^{\ul{h^{{U_3}^{\pr}} y}} "22";"23"};
{\ar@{->}^-{\sigma(\ul{c^x})} "23";"24"};
{\ar^{\sigma(\ul{a})} "31";"32"};
{\ar^{\sigma(\ul{b})} "32";"33"};
{\ar^{\sigma(\ul{c^x u_3})} "33";"34"};
{\ar^{\ul{h^{U_1} y_1}} "11";"21"};
{\ar^{\ul{z_2 y_2}} "12";"22"};
{\ar^{\sigma(\ul{u_3})} "13";"23"};
{\ar^{\Sigma \ul{h^{U_1} y_1}} "14";"24"};
{\ar@{=}^{} "21";"31"};
{\ar@{<-}^{\sigma(\ul{z_2})} "22";"32"};
{\ar@{<-}^{\sigma(\ul{u_3})} "23";"33"};
{\ar@{=}^{} "24";"34"};
{\ar@{}|\car@<10pt> "11";"22"};
{\ar@{}|\car@<10pt> "12";"23"};
{\ar@{}|{\phantom{Xx}\car} "13";"24"};
{\ar@{}|\car "21";"32"};
{\ar@{}|{\phantom{XX}\car} "22";"33"};
{\ar@{}|{\phantom{Xx}\car} "23";"34"};
\endxy
\label{diag:2}
\end{align}

Note that 
$\sigma(\ul{z_2})^{-1} \ul{z_2 y_2} = \sigma(\ul{z_2})^{-1} \sigma(\ul{z_2 y_2}) = \sigma(\ul{y_2}) = \ul{h^{U_2} y_2}$.
Thus, we obtain the following commutative diagram from \eqref{diag:1} and \eqref{diag:2}.
\[
\xy
(0,24)*+{X_1}="11";
(20,24)*+{X_2}="12";
(40,24)*+{\sigma U_1}="13";
(60,24)*+{\Sigma X_1}="14";
(0,8)*+{X_1}="21";
(20,8)*+{X_3}="22";
(40,8)*+{\sigma U_2}="23";
(60,8)*+{\Sigma X_1}="24";
(0,-8)*+{}="31";
(20,-8)*+{\sigma U_3}="32";
(40,-8)*+{\sigma U_3}="33";
(0,-24)*+{}="41";
(20,-24)*+{\Sigma X_2}="42";
(40,-24)*+{\Sigma(\sigma U_1)}="43";
(60,-24)*+{}="44";
{\ar^{\msize{0.6}{\Sigma(\ul{h^{U_1} y_1})}} "42";"43"};
{\ar^{\ul{x_1}} "11";"12"};
{\ar^{\ul{h^{U_1} y_1}} "12";"13"};
{\ar^{\sigma(\ul{c^{x_1}})} "13";"14"};
{\ar^{\ul{x_2}} "21";"22"};
{\ar^{\ul{h^{U_2} y_2}} "22";"23"};
{\ar^{\sigma(\ul{c^{x_2}})} "23";"24"};
{\ar@{=} "32";"33"};
{\ar@{=} "11";"21"};
{\ar^{\ul{x_3}} "12";"22"};
{\ar^{\ul{h^{U_3} y_3}} "22";"32"};
{\ar^{\sigma(\ul{c^{x_3}})} "32";"42"};
{\ar^{\sigma(\ul{a})} "13";"23"};
{\ar^{\sigma(\ul{b})} "23";"33"};
{\ar^{\sigma(\ul{c^x u_3})} "33";"43"};
{\ar@{=} "14";"24"};
%
{\ar@{}|\car "12";"23"};
{\ar@{}|\car "11";"22"};
{\ar@{}|{\phantom{XX}\car} "22";"33"};
{\ar@{}|{\phantom{XX}\car} "32";"43"};
{\ar@{}|\car "13";"24"};
\endxy
\]

So, it is enough to show that $\Sigma \ul{x_1} \sigma(\ul{c^{x_2}}) = \sigma(\ul{c^{x_3}})\sigma(\ul{b})$. 
From the following commutative diagram and (\ref{diag:proof_RT4-1}), $\lambda^{X_2} c^{x_3} b = \delta_3 b = x_1 \delta_2 = x_1 \lambda^{X_1} c^{x_2} = \lambda^{X_2} x_1\bracket{1} c^{x_2}$.
\[
\xy
(0,24)*+{X_1}="11";
(16,24)*+{X_3}="12";
(32,24)*+{U_2}="13";
(48,24)*+{X_1}="14";
(0,8)*+{X_1}="21";
(16,8)*+{I^{X_1}}="22";
(32,8)*+{X_1\bracket{1}}="23";
(48,8)*+{X_1}="24";
(0,-8)*+{X_2}="31";
(16,-8)*+{I^{X_2}}="32";
(32,-8)*+{X_2 \bracket{1}}="33";
(48,-8)*+{X_2}="34";
(0,-24)*+{X_2}="41";
(16,-24)*+{X_3}="42";
(32,-24)*+{U_3}="43";
(48,-24)*+{X_2}="44";
{\ar^{x_2} "11";"12"};
{\ar^{y_2} "12";"13"};	
{\ar@{-->}^{\delta_2} "13";"14"};
{\ar^{i^{X_1}} "21";"22"};
{\ar^{p^{X_1}} "22";"23"};
{\ar@{-->}^{\lambda^{X_1}} "23";"24"};
{\ar^{i^{X_2}} "31";"32"};
{\ar^{p^{X_2}} "32";"33"};
{\ar@{-->}^{\lambda^{X_2}} "33";"34"};
{\ar^{x_3} "41";"42"};
{\ar^{y_3} "42";"43"};
{\ar@{-->}^{\delta_3} "43";"44"};
{\ar@{=}^{} "11";"21"};
{\ar^{} "12";"22"};
{\ar^{c^{x_2}} "13";"23"};
{\ar@{=}^{} "14";"24"};
{\ar^{x_1} "21";"31"};
{\ar^{} "22";"32"};
{\ar^{x_1 \bracket{1}} "23";"33"};
{\ar^{x_1} "24";"34"};
{\ar@{=}^{} "31";"41"};
{\ar@{<-}^{} "32";"42"};
{\ar@{<-}^{c^{x_3}} "33";"43"};
{\ar@{=}^{} "34";"44"};
{\ar@{}|\car "11";"22"};
{\ar@{}|\car "12";"23"};
{\ar@{}|\car "13";"24"};
{\ar@{}|\car "21";"32"};
{\ar@{}|\car "22";"33"};
{\ar@{}|{\phantom{XX}\car} "23";"34"};
{\ar@{}|\car "31";"42"};
{\ar@{}|\car "32";"43"};
{\ar@{}|\car "33";"44"};
\endxy
\]
So, $\ul{x_1}\bracket{1} \ul{c^{x_2}} = \ul{c^{x_3} b}$. 
Applying $\sigma$, $\Sigma \ul{x_1} \sigma(\ul{c^{x_2}}) = \sigma(\ul{c^{x_3}})\sigma(\ul{b})$ holds.
\end{proof}

To sum up, we obtain the following result.

\begin{cor} \label{RT}
Let $(\SS, \ZZ)$ be a right mutation double.
Then $(\ul{\ZZ}, \Sigma, \nabla)$ is a right triangulated category.
\end{cor}
\begin{proof}
From Lemma \ref{proof_of_RT0-1}, \ref{proof_of_RT2}, \ref{proof_of_RT3} and Proposition \ref{proof_of_RT4}.
\end{proof}

Dually, we can construct a left triangulated structure from a left mutation double $(\ZZ, \VV)$.

First, we fix the following $\fraks_{\II}$-triangles to define $\bracket{-1}$ and $\omega$.

\begin{enumerate}
\item For $X \in \ZZ$, there exists the following $\fraks_{\II}$-triangle (and also an $\fraks^{\II}$-triangle) where $I_X \in \II$ and we fix it.
	\[
	X \xdrar{\lambda_X} X\bracket{-1} \xrar{i_X} I_X \xrar{p_X} X
	\]
	Then we define $i_X, p_X, \lambda_X$ by the above fixed $\fraks_{\II}$-triangle.
\item For $T \in \wt{\TT}$, there exists the following $\fraks_{\II}$-triangle (this is not an $\fraks^{\II}$-triangle in general)
	where $\omega T \in \ZZ, V_T \in \VV$ and we fix it.
	\[
	T \lxdrar{\rho_T} V_T \xrar{g_T} \omega T \xrar{h_T} T
	\]
	Then we define $h_T, g_T, \rho_T$ by the above fixed $\fraks_{\II}$-triangle.
	For $Z \in \ZZ$, we always take $\omega Z$ so that $\omega Z = Z$ and $h_Z = \id_Z$.
\end{enumerate}

\begin{nota} \label{notation_left_tri} 
	\begin{enumerate}
	\item For a morphism $b \colon Y \to Z$ in $\ZZ$, there exists the following $\fraks_{\II}$-triangle 
	(this is also an $\fraks^{\II}$-triangle ) from dual of Lemma \ref{inflation} and we fix it.
	\[
	Z \xdrar{\wh{\delta}} C_b \xrar{\wh{a}} Y \oplus I_Z 
	\xrar{\msize{0.6}{\begin{bmatrix} b & p_Z \end{bmatrix}}} Z
	\]
	Then we define $\wh{a}$ and $\wh{\delta}$ by the above $\fraks_{\II}^{\II}$-triangle. 
	We also define $\wh{b} = \msize{0.8}{\begin{bmatrix} b & p_Z \end{bmatrix}}$ and 
	$a$ as the composition of $C_b \xrar{\wh{a}} Y \oplus I_X \xrar{\msize{0.6}{\begin{bmatrix} 1 & 0 \end{bmatrix}}} Y$.
	\item There exists the following commutative diagram in $(\CC, \bbE^{\II}_{\II}, \fraks^{\II}_{\II})$ and we fix it.
	\begin{align}
	\xy
	(0,24)*+{}="11";
	(0,8)*+{Z}="21";
	(0,-8)*+{Z}="31";
	(0,-24)*+{}="41";
	(16,24)*+{Y}="12";
	(16,8)*+{Z \bracket{-1}}="22";
	(16,-8)*+{C_b}="32";
	(16,-24)*+{Y}="42";
	(32,24)*+{Y}="13";
	(32,8)*+{I_Z}="23";
	(32,-8)*+{Y \! \oplus \! I_Z}="33";
	(32,-24)*+{Y}="43";
	(48,24)*+{}="14";
	(48,8)*+{Z}="24";
	(48,-8)*+{Z}="34";
	(48,-24)*+{Z \bracket{-1}}="44";
	{\ar@{=} "12";"13"};
	{\ar@{-->}^-{\lambda_Z} "21";"22"};
	{\ar^-{i_Z} "22";"23"};
	{\ar@{->}^{p_Z} "23";"24"};
	{\ar@{-->}^{\wh{\delta}} "31";"32"};
	{\ar^-{\wh{a}} "32";"33"};
	{\ar@{->}^-{\wh{b}} "33";"34"};
	{\ar@{=} "42";"43"};
	{\ar@{-->}^-{-\gamma_b} "43";"44"};
	{\ar@{=} "21";"31"};
	{\ar@{-->}^{\gamma_b} "12";"22"};
	{\ar^{c_b} "22";"32"};
	{\ar@{->}^{a} "32";"42"};
	{\ar@{-->}^{0} "13";"23"};
	{\ar^{\msize{0.6}{\begin{bmatrix} 0 \\ 1 \\ \end{bmatrix}}} "23";"33"};
	{\ar@{->}^{\msize{0.6}{\begin{bmatrix} 1 & 0 \\ \end{bmatrix}}} "33";"43"};
	{\ar@{=} "24";"34"};
	{\ar@{-->}^-{\lambda_Z} "34";"44"};
	{\ar@{}|\car "12";"23"};
	{\ar@{}|\car "21";"32"};
	{\ar@{}|\car "22";"33"};
	{\ar@{}|\car "23";"34"};
	{\ar@{}|\car "32";"43"};
	{\ar@{}|\car "33";"44"};
	\endxy
	\label{THE_shifted_oct_2}
	\end{align}
	Then we define $c_b, \gamma_b$ by the above diagram. (We often drop ``$b$'' if there is no confusion.)
\end{enumerate}
There exists the following sequence in $\ul{\ZZ}$.
	 \begin{align}
	 \Omega Z \xrar{\omega(\ul{c})} \omega C_b \xrar{\,\ul{\wh{a} h_{C_b}}\,} Y \oplus I_Z \xrar{\,\ul{\wh{b}}\,} Z
	 \end{align}
\end{nota}

\begin{defi}
Let $\ul{b} \colon Y \to Z$ be a morphism in $\ul{\ZZ}$. Then there exists the following unique sequence in $\ul{\ZZ}$ up to isomorphisms. 
This sequence is called the \emph{standard left triangle} of $\ul{b}$.
\[
\Omega Z \xrar{\omega(\ul{c})} \omega C_b \xrar{\,\ul{a h_{C_b}}\,} Y \xrar{\,\ul{b}\,} Z
\]
We define
\[
\Delta = 
\left(
\begin{array}{ll}
\text{sequences} &\text{in } \ul{\ZZ} \text{ isomorphic to one in } \\
&\{ \Omega Z \xrar{\omega(\ul{c})} \omega C_b \xrar{\,\ul{a h_{C_b}}\,} Y \xrar{\,\ul{b}\,} Z \mid b \text{ is a morphism in }\ZZ \}
\end{array}
\right)
\]
and a sequence in $\Delta$ is called a \emph{left triangle} in $\ul{\ZZ}$.
\end{defi}

We omit the proof of the following statement, which is the dual of Corollary \ref{RT}.

\begin{prop}
Let $(\ZZ, \VV)$ be a left mutation double.
Then $(\ul{\ZZ}, \Omega, \Delta)$ is a left triangulated category.
\end{prop}

\subsection{Pretriangulated structures induced by premutation triples}
We assume that $(\SS, \ZZ, \VV)$ is a premutation triple.

\begin{thm} \label{main_thm1}
$(\ul{\ZZ}, \Sigma, \Omega, \nabla, \Delta)$ is a pretriangulated category, that is, 
for the following commutative diagrams in $\ul{\ZZ}$.
\[
\xy
(0,8)*+{X}="11";
(15,8)*+{Y}="12";
(30,8)*+{Z}="13";
(45,8)*+{\Sigma X}="14";
(0,-8)*+{\Omega Z^{\pr}}="21";
(15,-8)*+{X^{\pr}}="22";
(30,-8)*+{Y^{\pr}}="23";
(45,-8)*+{Z^{\pr}}="24";
{\ar^{\ul{a}} "11";"12"};
{\ar^{\ul{h^{C^a} b}} "12";"13"};	
{\ar^{\sigma(\ul{c})} "13";"14"};
{\ar^{\omega(\ul{c^{\pr}})} "21";"22"};
{\ar^{\ul{a^{\pr} h_{C_{b^{\pr}}}}} "22";"23"};
{\ar^{\ul{b^{\pr}}} "23";"24"};
{\ar^{\ul{s}} "11";"21"};
{\ar^{\ul{t}} "12";"22"};
{\ar^{\beta_{Z^{\pr}} \circ \Sigma \ul{s}} "14";"24"};
{\ar@{}|\car "11";"22"};
\endxy
\xy
(0,8)*+{X}="11";
(15,8)*+{Y}="12";
(30,8)*+{Z}="13";
(45,8)*+{\Sigma X}="14";
(0,-8)*+{\Omega Z^{\pr}}="21";
(15,-8)*+{X^{\pr}}="22";
(30,-8)*+{Y^{\pr}}="23";
(45,-8)*+{Z^{\pr}}="24";
{\ar^{\ul{a}} "11";"12"};
{\ar^{\ul{h^{C^a} b}} "12";"13"};	
{\ar^{\sigma(\ul{c})} "13";"14"};
{\ar^{\omega(\ul{c^{\pr}})} "21";"22"};
{\ar^{\ul{a^{\pr} h_{C_{b^{\pr}}}}} "22";"23"};
{\ar^{\ul{b^{\pr}}} "23";"24"};
{\ar_{\Omega \ul{u^{\pr}} \circ \alpha_X} "11";"21"};
{\ar_{\ul{t^{\pr}}} "13";"23"};
{\ar_{\ul{u^{\pr}}} "14";"24"};
{\ar@{}|\car "13";"24"};
\endxy
\]
where 
$Z = \sigma C^a, c = c^a, X^{\pr} = \omega C_{b^{\pr}}$ and $c^{\pr} = c_{b^{\pr}}$,
there exist morphisms $u \colon Z \to Y^{\pr}$ and $s^{\pr} \colon Y \to X^{\pr}$ 
which make the following diagrams commutative.
\[
\xy
(0,8)*+{X}="11";
(15,8)*+{Y}="12";
(30,8)*+{Z}="13";
(45,8)*+{\Sigma X}="14";
(0,-8)*+{\Omega Z^{\pr}}="21";
(15,-8)*+{X^{\pr}}="22";
(30,-8)*+{Y^{\pr}}="23";
(45,-8)*+{Z^{\pr}}="24";
{\ar^{\ul{a}} "11";"12"};
{\ar^{\ul{h^{C^a} b}} "12";"13"};	
{\ar^{\sigma(\ul{c})} "13";"14"};
{\ar^{\omega(\ul{c^{\pr}})} "21";"22"};
{\ar^{\ul{a^{\pr} h_{C_{b^{\pr}}}}} "22";"23"};
{\ar^{\ul{b^{\pr}}} "23";"24"};
{\ar^{\ul{s}} "11";"21"};
{\ar^{\ul{t}} "12";"22"};
{\ar^{\ul{u}} "13";"23"};
{\ar^{\beta_{Z^{\pr}} \circ \Sigma \ul{s}} "14";"24"};
{\ar@{}|\car "11";"22"};
{\ar@{}|\car "12";"23"};
{\ar@{}|\car "13";"24"};
\endxy
\xy
(0,8)*+{X}="11";
(15,8)*+{Y}="12";
(30,8)*+{Z}="13";
(45,8)*+{\Sigma X}="14";
(0,-8)*+{\Omega Z^{\pr}}="21";
(15,-8)*+{X^{\pr}}="22";
(30,-8)*+{Y^{\pr}}="23";
(45,-8)*+{Z^{\pr}}="24";
{\ar^{\ul{a}} "11";"12"};
{\ar^{\ul{h^{C^a} b}} "12";"13"};	
{\ar^{\sigma(\ul{c})} "13";"14"};
{\ar^{\omega(\ul{c^{\pr}})} "21";"22"};
{\ar^{\ul{a^{\pr} h_{C_{b^{\pr}}}}} "22";"23"};
{\ar^{\ul{b^{\pr}}} "23";"24"};
{\ar_{\Omega \ul{u^{\pr}} \circ \alpha_X} "11";"21"};
{\ar^{\ul{s^{\pr}}} "12";"22"};
{\ar^{\ul{t^{\pr}}} "13";"23"};
{\ar^{\ul{u^{\pr}}} "14";"24"};
{\ar@{}|\car "11";"22"};
{\ar@{}|\car "12";"23"};
{\ar@{}|\car "13";"24"};
\endxy
\]
\end{thm}
\begin{proof}
We only prove the existence of $u$ because one can show that of $s^{\pr}$ by dual argument.
We divide our proof into two steps.

\vspace{5pt} 
\noindent {\ul{\textbf{Step.1 Construct the morphism $u \colon Z \to Y^{\pr}$ }}}
\vspace{5pt} 

From $\ul{ta} = \omega(\ul{c^{\pr}}) \ul{s}$, there exists a morphism $j \colon I^X \to X^{\pr}$ where 
$ta + j i^X = \omega(c^{\pr})s$. By (ET3) of ($\CC, \bbE, \fraks$), 
we obtain a morphism $u^{\pr} \colon C^a \to Y^{\pr}$ which makes the following diagram commutative.
\begin{align}
\xy
(0,8)*+{X}="11";
(16,8)*+{Y \!\oplus\! I^X}="12";
(35,8)*+{C^a}="13";
(51,8)*+{X}="14";
(0,-8)*+{Z^{\pr}\bracket{-1}}="21";
(16,-8)*+{C_{b^{\pr}}}="22";
(35,-8)*+{Y^{\pr}}="23";
(51,-8)*+{Z^{\pr}\bracket{-1}}="24";
{\ar^-{\wt{a}} "11";"12"};
{\ar^-{\wt{b}} "12";"13"};	
{\ar@{-->}^-{\wt{\delta}} "13";"14"};
{\ar^-{c^{\pr}} "21";"22"};
{\ar^-{a^{\pr}} "22";"23"};
{\ar@{-->}^-{\gamma_{b^{\pr}}} "23";"24"};
{\ar_{h_{Z^{\pr}\bracket{-1}} s} "11";"21"};
{\ar^{h_{C_{b^{\pr}}} \msize{0.6}{\begin{bmatrix} t & j \end{bmatrix}}} "12";"22"};
{\ar^{u^{\pr}} "13";"23"};
{\ar^{h_{Z^{\pr}\bracket{-1}} s} "14";"24"};
{\ar@{}|\car "11";"22"};
{\ar@{}|{\phantom{XXXX}\car} "12";"23"};
{\ar@{}|\car "13";"24"};
\endxy \label{diag:step1}
\end{align}
Since we assumed that $\bbE^{\II}(\SS, \ZZ) = 0$, $u^{\pr}$ factors through 
the morphism $h^{C^a} \colon C^a \to Z$. 
Then there exists a morphism $u \colon Z \to Y^{\pr}$ where $u^{\pr} = u h^{C^a}$.
Thus, we obtain the following commutative diagram.
\[
\xy
(0,8)*+{Y \!\oplus\! I^X}="11";
(16,8)*+{C^a}="12";
(32,8)*+{Z}="13";
(0,-8)*+{C_{b^{\pr}}}="21";
(16,-8)*+{Y^{\pr}}="22";
{\ar^-{\wt{b}} "11";"12"};
{\ar^-{h^{C^a}} "12";"13"};	
{\ar^-{a^{\pr}} "21";"22"};
{\ar_{h_{C_{b^{\pr}}} \msize{0.6}{\begin{bmatrix} t & j \end{bmatrix}}} "11";"21"};
{\ar^{u^{\pr}} "12";"22"};
{\ar^{u} "13";"22"};
{\ar@{}|\car "11";"22"};
{\ar@{}_\car "13";"22"};
\endxy
\]

\vspace{5pt} 
\noindent {\ul{\textbf{Step.2 Checks for commutativity }}}
\vspace{5pt} 

Since $u h^{C^a} \wt{b} = a^{\pr} h_{C_{b^{\pr}}} \msize{0.8}{\begin{bmatrix} t & j \end{bmatrix}}$,  $\ul{u h^{C^a} b} = \ul{a^{\pr} h_{C_{b^{\pr}}} t}$. 
We only have to show that $\beta_{Z^{\pr}}$ $\Sigma(\ul{s})$ $\sigma(\ul{c})$ $= \ul{b^{\pr} u}$.
Recall that we obtain a morphism $s^{\pr} \colon \Sigma X \to Z^{\pr}$ where $\ul{s^{\pr}} = \beta_{Z^{\pr}} 
\Sigma(\ul{s})$ from Remark \ref{rmk_adjoint} and the following commutative diagram exists.
\[
\xy
(-10,-12)*+{}="10";
(0,8)*+{X}="11";
(16,8)*+{X\bracket{1}}="12";
(32,8)*+{\Sigma X}="13";
(-16,-8)*+{\Omega Z^{\pr}}="20";
(0,-8)*+{Z^{\pr} \bracket{-1}}="21";
(16,-8)*+{Z^{\pr}}="22";
(26,12)*+{}="23";
{\ar@{-->}_{\lambda^X} "12";"11"};
{\ar@{->}^{h^{X\bracket{1}}} "12";"13"};
{\ar@{->}_{h_{Z^{\pr}\bracket{-1}}} "20";"21"};
{\ar@{-->}^-{- \lambda_{Z^{\pr}}} "22";"21"};
{\ar_{s} "11";"20"};
{\ar^{} "11";"21"};
{\ar_{} "12";"22"};
{\ar^{s^{\pr}} "13";"22"};
{\ar@{}|\car "11";"10"};
{\ar@{}|\car "11";"22"};
{\ar@{}|\car "23";"22"};
\endxy
\]
Then
$\lambda_{Z^{\pr}} s^{\pr} \sigma(c) h^{C^a} = 
\lambda_{Z^{\pr}} s^{\pr} h^{X\bracket{1}} c =
- h_{Z^{\pr}\bracket{-1}} s \lambda^X c =
- h_{Z^{\pr}\bracket{-1}} s \wt{\delta} =
- \gamma_{b^{\pr}} u^{\pr} =
\lambda_{Z^{\pr}} \wh{b^{\pr}} u h^{C^a}$ from \eqref{THE_shifted_oct}, \eqref{THE_shifted_oct_2} and \eqref{diag:step1}.
Thus, $\lambda_{Z^{\pr}} (s^{\pr} \sigma(c) - \wh{b^{\pr}} u) h^{C^a} = 0$
and $(\beta_{Z^{\pr}} \Sigma(\ul{s}) \sigma(\ul{c}) - \ul{b^{\pr} u}) \ul{h^{C^a}} = 0$.
From Remark \ref{rmk_adj_of_sigma}(\ref{rmk_adj_of_sigma_1}), 
$\beta_{Z^{\pr}} \Sigma(\ul{s}) \sigma(\ul{c}) - \ul{{b^{\pr}} u} = 0$.
\end{proof}

\begin{cor} \label{cor_trictg}
Let $(\SS, \ZZ, \VV)$ be a premutation triple where $\Omega$ and $\Sigma$ are mutually quasi-inverse. 
Then $(\ul{\ZZ}, \Sigma, \nabla)$ is a triangulated category.
\end{cor}
\begin{proof}
By definition of pretriangulated categories.
\end{proof}

We consider sufficient conditions that $\Omega$ and $\Sigma$ are mutually quasi-inverse in the next section.

\section{Triangulated structures induced by mutation triples}\label{triangulated}
In this section, we finally achieve the aim of this paper, 
which is finding sufficient conditions for premutation triples to induce triangulated categories.
We fix an ET category $(\CC, \bbE, \fraks)$ in this section.

\subsection{Mutation triples and main theorem}

Next, we consider another condition (MT4) to induce triangulated structures for premutation triples in ET categories.
We assume that $(\SS, \ZZ, \VV)$ is a premutation triple.

To begin with, we construct new functors $(\, \cdot \,)^+$ and $(\, \cdot \,)^-$, 
which are analogies of those in Definition \ref{def_of_plus}.

\begin{defi}
	\begin{enumerate}
	\item We define $\wt{\UU}^-$ as $\CoCone_{\bbE^{\II}}(\II, \SS)$.
	Note that $\wt{\UU}^- \subset \wt{\UU}$ by definitions.
	\item We define $\wt{\TT}^+$ as $\Cone_{\bbE_{\II}}(\VV, \II)$.
	Note that $\wt{\TT}^+ \subset \wt{\TT}$ by definitions.
	\end{enumerate}
\end{defi}

\begin{rmk}
Let $((\SS, \TT), (\UU, \VV))$ be a concentric twin cotorsion pair of $\CC$.
Then $((\SS, \TT), (\SS, \TT))$ and $((\UU, \VV), (\UU, \VV))$ are also concentric twin cotorsion pairs of $\CC$.

$\wt{\UU}^-$ corresponds to $\CC^-$ in Definition \ref{defi_plus-minus_ccTCP} determined by $((\SS, \TT), (\SS, \TT))$ and 
$\wt{\TT}^+$ corresponds to $\CC^+$ in Definition \ref{defi_plus-minus_ccTCP} determined by $((\UU, \VV), (\UU, \VV))$.
\end{rmk}

We prepare the following lemma before we define those functors.

\begin{lem} \label{lem_sign_ext_vanish}
	\begin{enumerate}
	\item $\bbE^{\II}(\wt{\UU}^-, \ZZ\bracket{-1}) = 0$.
	\item $\bbE_{\II}(\ZZ\bracket{1}, \wt{\TT}^+) = 0$.
	\end{enumerate}
\end{lem}
\begin{proof}
We only prove (1).
Let $Z^{\pr} \in \ZZ\bracket{-1}, U^- \in \wt{\UU}^-$ and $\delta \in \bbE^{\II}(U^-, Z^{\pr})$.
Since $Z^{\pr} \in \ZZ\bracket{-1}$, there exists
an $\fraks^{\II}_{\II}$-triangle $Z^{\pr} \rar I \rar Z \lxdrar{\lambda} Z^{\pr}$ where $I \in \II$ and $Z \in \ZZ$.

Then $\delta$ factors through $\lambda$ and let $f \colon U^- \to Z$ be a morphism where $\delta = \lambda f$.
From $U^- \in \wt{\UU}^-$, there exists an $\fraks^{\II}$-triangle $U^- \xrar{h} I^{\pr} \rar S \drar U^-$
where $I^{\pr} \in \II, S \in \SS$.

Then $f$ factors through $h$ since $\bbE^{\II}(\SS, \ZZ) =0$, 
and let $g \colon I^{\pr} \to Z$ be a morphism where $f = gh$.

Thus, $\delta = \lambda g h = 0$ because $\lambda g \in \bbE^{\II}_{\II}(\II, \ZZ) = 0$.
\end{proof}

\begin{nota} \label{nota_+_and_-}
\begin{enumerate}
\item
	\begin{enumerate}
	\item For $U \in \wt{\UU}$, take the fixed triangles $U \xrar{h^U} Z^U \xrar{g^U} S^U \xdrar{\rho^U}$ and
	$Z^U\bracket{-1} \xrar{i_{Z^U}} I_{Z^U} \xrar{p_{Z^U}} Z^U \xdrar{\lambda_{Z^U}} Z^U\bracket{-1}$. 
	Then there exists the following commutative diagram (\ref{diag:-_1}) in $(\CC, \bbE^{\II}, \fraks^{\II})$.
	\begin{align}
	\xy
	(16,24)*+{S^U}="12";
	(32,24)*+{S^U}="13";
	(0,8)*+{Z^U\bracket{-1}}="21";
	(16,8)*+{U^-}="22";
	(32,8)*+{U}="23";
	(48,8)*+{Z^U\bracket{-1}}="24";
	(0,-8)*+{Z^U\bracket{-1}}="31";
	(16,-8)*+{I_{Z^U}}="32";
	(32,-8)*+{Z^U}="33";
	(48,-8)*+{Z^U\bracket{-1}}="34";
	(16,-24)*+{S^U}="42";
	(32,-24)*+{S^U}="43";
	{\ar@{=} "12";"13"};
	{\ar^-{r_U} "21";"22"};
	{\ar^-{s_U} "22";"23"};
	{\ar@{-->}^-{\chi_U} "23";"24"};
	{\ar^-{i_{Z^U}} "31";"32"};
	{\ar^-{p_{Z^U}} "32";"33"};
	{\ar@{-->}^-{\lambda_{Z^U}} "33";"34"};
	{\ar@{=} "42";"43"};
	{\ar@{=} "21";"31"};
	{\ar@{-->}^-{\pi^U} "12";"22"};
	{\ar^-{} "22";"32"};
	{\ar^-{} "32";"42"};
	{\ar@{-->}^-{\rho^U} "13";"23"};
	{\ar^-{h^U} "23";"33"};
	{\ar^-{g^U} "33";"43"};
	{\ar@{=} "24";"34"};
	{\ar@{}|\car "12";"23"};
	{\ar@{}|\car "21";"32"};
	{\ar@{}|\car "22";"33"};
	{\ar@{}|\car "23";"34"};
	{\ar@{}|\car "32";"43"};
	\endxy
	\label{diag:-_1}
	\end{align}
	We define $U^-$ and a morphism $s_U \colon U^- \to U$ by the $\fraks^{\II}_{\II}$-triangle.
	\[
	Z^U\bracket{-1} \xrar{r_U} U^- \xrar{s_U} U \lxdrar{\chi_U} Z^U\bracket{-1}
	\]
	\item For a morphism $u \colon U_1 \to U_2$ in $\wt{\UU}$, 
	there exists the following commutative diagram (\ref{diag:-_2}) in $\CC$ 
	and define a morphism $u^- \colon U^-_1 \to U^-_2$ by it.
	\begin{align}
	\xy
	(0,8)*+{U^-_1}="11";
	(16,8)*+{U_1}="12";
	(0,-8)*+{U^-_2}="21";
	(16,-8)*+{U_2}="22";
	{\ar^{s_{U_1}} "11";"12"};
	{\ar^{s_{U_2}} "21";"22"};
	{\ar^{u^-} "11";"21"};
	{\ar^{u} "12";"22"};
	{\ar@{}|\car "11";"22"};
	\endxy
	\label{diag:-_2}
	\end{align}
	\end{enumerate}
\item 
	\begin{enumerate}
	\item 
	For $T \in \wt{\TT}$, take the fixed triangles $V_T \xrar{g_T} Z_T \xrar{h_T} T \lxdrar{\rho_T}$ and
	$Z_T \xrar{i^{Z_T}} I^{Z_T} \xrar{p^{Z_T}} Z_T\bracket{1} \xdrar{\lambda^{Z_T}} Z_T$. 
	Then there exists the following commutative diagram (\ref{diag:+_1}) in $(\CC, \bbE_{\II}, \fraks_{\II})$.
	\begin{align}
	\xy
	(16,24)*+{Z_T\bracket{1}}="12";
	(32,24)*+{Z_T\bracket{1}}="13";
	(0,8)*+{V_T}="21";
	(16,8)*+{Z_T}="22";
	(32,8)*+{T}="23";
	(48,8)*+{V_T}="24";
	(0,-8)*+{V_T}="31";
	(16,-8)*+{I^{Z_T}}="32";
	(32,-8)*+{T^+}="33";
	(48,-8)*+{V_T}="34";
	(16,-24)*+{Z_T\bracket{1}}="42";
	(32,-24)*+{Z_T\bracket{1}}="43";
	{\ar@{=} "12";"13"};
	{\ar^-{g_T} "21";"22"};
	{\ar^-{h_T} "22";"23"};
	{\ar@{-->}^-{\rho_T} "23";"24"};
	{\ar^-{} "31";"32"};
	{\ar^-{} "32";"33"};
	{\ar@{-->}^-{\pi_T} "33";"34"};
	{\ar@{=} "42";"43"};
	{\ar@{=} "21";"31"};
	{\ar@{-->}^-{\lambda^{Z_T}} "12";"22"};
	{\ar^-{i^{Z_T}} "22";"32"};
	{\ar^-{p^{Z_T}} "32";"42"};
	{\ar@{-->}^-{\chi^T} "13";"23"};
	{\ar^-{s^T} "23";"33"};
	{\ar^-{r^T} "33";"43"};
	{\ar@{=} "24";"34"};
	{\ar@{}|{\phantom{Xx}\car} "12";"23"};
	{\ar@{}|\car "21";"32"};
	{\ar@{}|\car "22";"33"};
	{\ar@{}|\car "23";"34"};
	{\ar@{}|\car "32";"43"};
	\endxy
	\label{diag:+_1}
	\end{align}
	We define $T^+$ and a morphism $s^T \colon T \to T^+$ by the $\fraks^{\II}_{\II}$-triangle.
	\[
	T \xrar{s^T} T^+ \xrar{r^T} Z_T\bracket{1} \xdrar{\chi^T} T
	\]
	
	\item
	For a morphism $t \colon T_1 \to T_2$ in $\wt{\TT}$,
	there exists the following commutative diagram (\ref{diag:+_2}) in $\CC$ 
	and define a morphism $t^+ \colon T^+_1 \to T^+_2$ by it.
	\begin{align}
	\xy
	(0,8)*+{T_1}="11";
	(16,8)*+{T_1^+}="12";
	(0,-8)*+{T_2}="21";
	(16,-8)*+{T_2^+}="22";
	{\ar^{s^{T_1}} "11";"12"};
	{\ar^{s^{T_2}} "21";"22"};
	{\ar^{t} "11";"21"};
	{\ar^{t^+} "12";"22"};
	{\ar@{}|\car "11";"22"};
	\endxy
	\label{diag:+_2}
	\end{align}
	\end{enumerate}
\end{enumerate}
\end{nota}

\begin{prop} \label{prop_+_and_-}
We use the notations in Notation \ref{nota_+_and_-}.
\begin{enumerate} 
\item 
	\begin{enumerate}
	\item $\ul{s_U} \circ - \colon \ul{\wt{\UU}^-}(\wt{\UU}^-, U^-) \to \ul{\wt{\UU}}(\wt{\UU}^-, U)$ 
	is a natural isomorphism.
	\item The correspondence $(\, \cdot \,)^- \colon \ul{\wt{\UU}} \to \ul{\wt{\UU}^-}$ 
	is a right adjoint functor of the inclusion functor $\ul{\wt{\UU}^-} \to \ul{\wt{\UU}}$.
	\end{enumerate}
\item
	\begin{enumerate}
	\item $- \circ \ul{s^T} \colon \ul{\wt{\TT}^+}(T^+, \wt{\TT}^+) \to \ul{\wt{\TT}}(T, \wt{\TT}^+)$ 
	is a natural isomorphism.
	\item The correspondence $(\, \cdot \,)^+ \colon \ul{\wt{\TT}} \to \ul{\wt{\TT}^+}$ 
	is a left adjoint functor of the inclusion functor $\ul{\wt{\TT}^+} \to \ul{\wt{\TT}}$.
	\end{enumerate}
\end{enumerate}
\end{prop}
\begin{proof}
We only prove (1).
(i) From Lemma \ref{lem_sign_ext_vanish}, $\ul{s_U} \circ -$ is surjective, so we only have to show its injectivity.
Take a morphism $u^{\pr} \colon U^{\pr} \to U^-$ where $\ul{s_U u^{\pr}} = 0$ and $U^{\pr} \in \wt{\UU}^-$.
Let $i$ be a morphism contained in $[\II]$ where $s_U u^{\pr} = i$.
Since $\chi_U$ is an $\bbE^{\II}_{\II}$-extension, $i$ factors through $s_U$.
Thus, there exists a morphism $f \colon U^{\pr} \to Z^U \bracket{-1}$ where $\ul{r_U f} = \ul{u^{\pr}}$.
Take an $\fraks_{\II}$-triangle $U^{\pr} \rar I \rar S \lxdrar{\delta} U^{\pr}$ where $I \in \II$ and $S \in \SS$.
Since $\bbE_{\II}(\SS, \ZZ\bracket{-1}) =0$, $f \delta =0$.
Then $f$ factors through $I$, in particular, 
$\ul{u^{\pr}}=\ul{r_{U_2} f}=0$.

(ii) From (i), $u^-$ is uniquely determined by the right commutative diagram in (\ref{diag:-_2}) up to $[\II]$.
Then $(\, \cdot \,)^-$ induces a functor.
\end{proof}

We define the following condition which is named (MT4)$^{\pr}$ before (MT4).
In Remark \ref{MT4}, we define (MT4) and check that the two conditions are equivalent.

\begin{condi} \label{MT4_prime}
Let $(\SS, \ZZ, \VV)$ be a premutation triple and we consider the following conditions.
	\begin{itemize}[leftmargin=50pt]
	\item[(MT4)$^{\pr}$]
		\begin{enumerate}[label=(\roman*)]
		\item $\ul{\ZZ\bracket{1}}^- \subset \ul{\wt{\TT}}{}^+$.
		\item $\ul{\ZZ\bracket{-1}}^+ \subset \ul{\wt{\UU}}{}^-$.
		\end{enumerate}
	\end{itemize}
\end{condi}

\begin{defi}
Let $(\SS, \ZZ, \VV)$ be a premutation triple.
A triplet $(\SS, \ZZ, \VV)$ is a \emph{mutation triple} if it satisfies the condition (MT4)$^{\pr}$.
\end{defi}

\begin{thm} \label{main_thm2}
Let $(\SS, \ZZ, \VV)$ be a mutation triple. 
Then the quintuplet $(\ul{\ZZ},$ $\Sigma,$ $\Omega,$ $\nabla,$ $\Delta)$ in Theorem \ref{main_thm1} is a triangulated category.
\end{thm}
\begin{proof}
We only prove $\Sigma \Omega \cong \mathrm{Id}$. 
Let $Z \in \ZZ$ and take the following commutative diagram in $(\CC, \bbE_{\II}, \fraks_{\II})$ in Notation \ref{nota_+_and_-}.
Note that $h_{Z\bracket{-1}} \lambda^{\Omega Z}$ is an $\fraks^{\II}_{\II}$-triangle
since $\lambda^{\Omega Z}$ is an $\fraks^{\II}_{\II}$-triangle from Lemma \ref{bracket_equiv}.
\[
\xy
(0,24)*+{}="11";
(16,24)*+{(\Omega Z)\bracket{1}}="12";
(36,24)*+{(\Omega Z)\bracket{1}}="13";
(52,24)*+{}="14";
(0,8)*+{V}="21";
(16,8)*+{\Omega Z}="22";
(36,8)*+{Z\bracket{-1}}="23";
(52,8)*+{V}="24";
(0,-8)*+{V}="31";
(16,-8)*+{I^{\Omega Z}}="32";
(36,-8)*+{Z\bracket{-1}^+}="33";
(52,-8)*+{V}="34";
(0,-24)*+{}="41";
(16,-24)*+{(\Omega Z)\bracket{1}}="42";
(36,-24)*+{(\Omega Z)\bracket{1}}="43";
{\ar@{=} "12";"13"};
{\ar^{} "21";"22"};
{\ar^-{h_{Z\bracket{-1}}} "22";"23"};
{\ar@{-->}^-{\rho_{Z\bracket{-1}}} "23";"24"};
{\ar^{} "31";"32"};
{\ar^{} "32";"33"};
{\ar@{-->}^{} "33";"34"};
{\ar@{=} "42";"43"};
{\ar@{=} "21";"31"};
{\ar@{-->}^{\lambda^{\Omega Z}} "12";"22"};
{\ar^{} "22";"32"};
{\ar^{} "32";"42"};
{\ar@{-->}^{} "13";"23"};
{\ar^{f} "23";"33"};
{\ar^{g} "33";"43"};
{\ar@{=} "24";"34"};
{\ar@{}|\car "12";"23"};
{\ar@{}|\car "21";"32"};
{\ar@{}|\car "22";"33"};
{\ar@{}|\car "23";"34"};
{\ar@{}|\car "32";"43"};
\endxy
\]
From (MT4)$^{\pr}$, there exists an $\fraks^{\II}$-triangle 
$S \lxdrar{\delta} Z\bracket{-1}^+ \xrar{s} I^{\pr}_Z \xrar{t} S$ where $I^{\pr}_Z \in \II, S \in \SS$.
By (ET4), we obtain the following commutative diagram in $(\CC, \bbE^{\II}, \fraks^{\II})$
and define $\Psi(Z)$ by this diagram.
\begin{align}
\xy
(0,24)*+{}="11";
(16,24)*+{Z\bracket{-1}}="12";
(40,24)*+{Z\bracket{-1}}="13";
(56,24)*+{}="14";
(0,8)*+{S}="21";
(16,8)*+{Z\bracket{-1}^+}="22";
(40,8)*+{I^{\pr}_Z}="23";
(56,8)*+{S}="24";
(0,-8)*+{S}="31";
(16,-8)*+{(\Omega Z)\bracket{1}}="32";
(40,-8)*+{\Psi(Z)}="33";
(56,-8)*+{S}="34";
(0,-24)*+{}="41";
(16,-24)*+{Z\bracket{-1}}="42";
(40,-24)*+{Z\bracket{-1}}="43";
{\ar@{=} "12";"13"};
{\ar@{-->}^-{\delta} "21";"22"};
{\ar^-{s} "22";"23"};
{\ar@{->}^{t} "23";"24"};
{\ar@{-->}^{} "31";"32"};
{\ar^-{\hbar^{(\Omega Z)\bracket{1}}} "32";"33"};
{\ar@{->}^{} "33";"34"};
{\ar@{=} "42";"43"};
{\ar@{=} "21";"31"};
{\ar@{->}^{f} "12";"22"};
{\ar^{g} "22";"32"};
{\ar@{-->}_{h_{Z\bracket{-1}} \lambda^{\Omega Z}} "32";"42"};
{\ar@{->}^{i^{\pr}_Z} "13";"23"};
{\ar^{p^{\pr}_Z} "23";"33"};
{\ar@{-->}^{\epsilon_Z} "33";"43"};
{\ar@{=} "24";"34"};
{\ar@{}|\car "12";"23"};
{\ar@{}|\car "21";"32"};
{\ar@{}|\car "22";"33"};
{\ar@{}|\car "23";"34"};
{\ar@{}|\car "32";"43"};
\endxy
\label{diag:before-replace}
\end{align}
Thus, we obtain an $\fraks^{\II}$-triangle
$Z\bracket{-1} \xrar{i^{\pr}_Z} I^{\pr}_Z \xrar{p^{\pr}_Z} \Psi(Z) \lxdrar{\epsilon_Z} Z\bracket{-1}$.
Then we obtain the following commutative diagram in $(\CC, \bbE^{\II}, \fraks^{\II})$
and define $\psi_Z$ by this diagram.
\begin{align}
\xy
(0,8)*+{Z\bracket{-1}}="11";
(16,8)*+{I_Z}="12";
(32,8)*+{Z}="13";
(48,8)*+{Z\bracket{-1}}="14";
(0,-8)*+{Z\bracket{-1}}="21";
(16,-8)*+{I^{\pr}_Z}="22";
(32,-8)*+{\Psi(Z)}="23";
(48,-8)*+{Z\bracket{-1}}="24";
{\ar^{i_Z} "11";"12"};
{\ar^{p_Z} "12";"13"};	
{\ar@{-->}^-{\lambda_Z} "13";"14"};
{\ar^{i^{\pr}_Z} "21";"22"};
{\ar^{p^{\pr}_Z} "22";"23"};
{\ar@{-->}^-{\epsilon_Z} "23";"24"};
{\ar@{=}^{} "11";"21"};
{\ar^{} "12";"22"};
{\ar^{\psi_Z} "13";"23"};
{\ar@{=}^{} "14";"24"};
{\ar@{}|\car "11";"22"};
{\ar@{}|\car "12";"23"};
{\ar@{}|\car "13";"24"};
\endxy
\label{diag:5-1}
\end{align}
One can show that $\ul{\psi_Z}$ is an isomorphism, like as proof of Lemma \ref{def_of_bracket}
since both $i_Z$ and $i^{\pr}_Z$ are left $\II$-approximations.
Then $\Psi(Z)$ is isomorphic to an object in $\ul{\ZZ}$.
From Corollary \ref{cor_iso-in-ul}, there exists $I \in \II$ where $\Psi(Z) \oplus I \in \ZZ$.
We may assume that $\Psi(Z) \in \ZZ$ by replacing diagram \eqref{diag:before-replace} with the following one.

\[
\xy
(-4,24)*+{}="11";
(16,24)*+{Z\bracket{-1}}="12";
(40,24)*+{Z\bracket{-1}}="13";
(56,24)*+{}="14";
(-4,8)*+{S \oplus I}="21";
(16,8)*+{Z\bracket{-1}^+}="22";
(40,8)*+{I^{\pr}_Z \oplus I}="23";
(60,8)*+{S \oplus I}="24";
(-4,-8)*+{S \oplus I}="31";
(16,-8)*+{(\Omega Z)\bracket{1}}="32";
(40,-8)*+{\Psi(Z) \oplus I}="33";
(60,-8)*+{S \oplus I}="34";
(-4,-24)*+{}="41";
(16,-24)*+{Z\bracket{-1}}="42";
(40,-24)*+{Z\bracket{-1}}="43";
{\ar@{=} "12";"13"};
{\ar@{-->}^-{\msize{0.6}{\begin{bmatrix} \delta & 0 \end{bmatrix}}} "21";"22"};
{\ar^-{\msize{0.6}{\begin{bmatrix} s \\ 0 \end{bmatrix}}} "22";"23"};
{\ar@{->}^{\msize{0.6}{\begin{bmatrix} t & 0 \\ 0 & 1 \end{bmatrix}}} "23";"24"};
{\ar@{-->}^{} "31";"32"};
{\ar^-{\msize{0.6}{\begin{bmatrix} \hbar^{(\Omega Z)\bracket{1}} \\ 0 \end{bmatrix}}} "32";"33"};
{\ar@{->}^{} "33";"34"};
{\ar@{=} "42";"43"};
{\ar@{=} "21";"31"};
{\ar@{->}^{f} "12";"22"};
{\ar^{g} "22";"32"};
{\ar@{-->}_{h_{Z\bracket{-1}} \lambda^{\Omega Z}} "32";"42"};
{\ar@{->}^{\msize{0.6}{\begin{bmatrix} i^{\pr}_Z \\ 0 \end{bmatrix}}} "13";"23"};
{\ar^{\msize{0.6}{\begin{bmatrix} p^{\pr}_Z & 0 \\ 0 & 1 \end{bmatrix}}} "23";"33"};
{\ar@{-->}^{\msize{0.6}{\begin{bmatrix} \epsilon_Z & 0 \end{bmatrix}}} "33";"43"};
{\ar@{=} "24";"34"};
{\ar@{}|\car "12";"23"};
{\ar@{}|\car "21";"32"};
{\ar@{}|\car "22";"33"};
{\ar@{}|\car "23";"34"};
{\ar@{}|\car "32";"43"};
\endxy
\]

For a morphism $z \colon Z_1 \to Z_2$ in $\ZZ$, there exists the following commutative diagram in $(\CC, \bbE, \fraks)$ and define $\Psi(z)$ by this diagram.
\begin{align}
\xy
(0,8)*+{Z_1\bracket{-1}}="11";
(16,8)*+{I^{\pr}_{Z_1}}="12";
(32,8)*+{\Psi(Z_1)}="13";
(52,8)*+{Z_1\bracket{-1}}="14";
(0,-8)*+{Z_2\bracket{-1}}="21";
(16,-8)*+{I^{\pr}_{Z_2}}="22";
(32,-8)*+{\Psi(Z_2)}="23";
(52,-8)*+{Z_2\bracket{-1}}="24";
{\ar^-{i^{\pr}_{Z_1}} "11";"12"};
{\ar^-{p^{\pr}_{Z_1}} "12";"13"};	
{\ar@{-->}^-{\epsilon_{Z_1}} "13";"14"};
{\ar^-{i^{\pr}_{Z_2}} "21";"22"};
{\ar^-{p^{\pr}_{Z_2}} "22";"23"};
{\ar@{-->}^-{\epsilon_{Z_2}} "23";"24"};
{\ar^{z\bracket{-1}} "11";"21"};
{\ar^{} "12";"22"};
{\ar^{\Psi(z)} "13";"23"};
{\ar^{z\bracket{-1}} "14";"24"};
{\ar@{}|{\phantom{XXx}\car} "11";"22"};
{\ar@{}|\car "12";"23"};
{\ar@{}|\car "13";"24"};
\endxy
\label{diag:5-2}
\end{align}
Then we define a morphism $\Psi(\ul{z}) \colon \Psi(Z_1) \to \Psi(Z_2)$ as $\ul{\Psi(z)}$.
One can show that $\Psi \colon \ul{\ZZ} \to \ul{\ZZ}$ is an additive functor like $\bracket{1}$.
Moreover, $\ul{\psi_Z}$ induces a natural isomorphism $\ul{\psi} \colon \mathrm{Id} \Rightarrow \Psi$
from \eqref{diag:5-1} and \eqref{diag:5-2}.

Recall that $S \xdrar{g \delta} (\Omega Z)\bracket{1} \xrar{\hbar^{(\Omega Z)\bracket{1}}} \Psi(Z) \rar S$ is an $\fraks^{\II}$-triangle.
Then from Remark \ref{rmk_adj_of_sigma}(\ref{rmk_adj_of_sigma_3}) and $\Psi(Z) \in \ZZ$,
there exists a morphism $\varphi_Z \colon \Sigma\Omega Z \to \Psi(Z)$
which satisfies $\hbar^{(\Omega Z)\bracket{1}} = \varphi_Z h^{(\Omega Z)\bracket{1}}$
and $\varphi_Z$ is unique up to $[\II]$.
Because of this uniqueness, one can show that $\ul{\varphi} \colon \Psi \Rightarrow \Sigma\Omega$ is a
natural isomorphism like Remark \ref{rmk_adj_of_sigma}(\ref{rmk_adj_of_sigma_3}).

Therefore, $(\ul{\ZZ}, \Sigma,\Omega, \nabla, \Delta)$ is a triangulated category from Corollary \ref{cor_trictg}.
\end{proof}

\begin{rmk} \label{MT4}
	\begin{enumerate}
	\item From proof of Theorem \ref{main_thm2}, 
	$Z\bracket{-1}^+ \cong (\Omega Z)\bracket{1}^-$ in $\ul{\CC}$.
	Thus, (MT4)$^{\pr}$ is equivalent to the following condition.
		\begin{itemize}[leftmargin=44pt]
		\item[(MT4)] $\ul{\ZZ\bracket{1}}^- = \ul{\ZZ\bracket{-1}}^+$.
		\end{itemize}
	\item Since $\ul{\ZZ\bracket{1}}^- \subset \ul{\wt{\UU}}{}^-$ and $\ul{\ZZ\bracket{-1}}^+ \subset \ul{\wt{\TT}}{}^+$ always hold, the following condition is stronger than (MT4)$^{\pr}$.
		\begin{itemize}[leftmargin=50pt]
		\item[(MT4$^+$)] $\ul{\wt{\UU}}{}^- = \ul{\wt{\TT}}{}^+$.
		\end{itemize}
	\item
	We may also define the following conditions, which is the same conditions defined in Condition \ref{condi_MT4-} .
		\begin{itemize}[leftmargin=50pt]
		\item[(MT4$^-$)]
			\begin{enumerate}[label=(\roman*)]
			\item $\ul{\ZZ\bracket{1}}^- \subset \ul{\Cone_{\bbE}(\VV, \SS)} = 
			\ul{\Cone_{\bbE_{\II}}(\VV, \SS)}$.
			\item $\ul{\ZZ\bracket{-1}}^+ \subset \ul{\CoCone_{\bbE}(\VV, \SS)} = 
			\ul{\CoCone_{\bbE^{\II}}(\VV, \SS)}$.
			\end{enumerate}
		\end{itemize}
	By definition of (MT4)$^{\pr}$ and (MT4$^-$), (MT4)$^{\pr}$ is stronger than (MT4$^-$).
	To sum up,
	\[
	(\text{MT4}^+) \implies (\text{MT4})^{\pr} \iff (\text{MT4}) \implies (\text{MT4}^-).
	\]
	\end{enumerate}
\end{rmk}

\begin{ex} \label{ex_MT4}
	\begin{enumerate}
	\item \cite{NP19}
	Assume that $\CC$ is Frobenius with $\PP = \Proj_{\bbE} \CC$. 
	Then $(\PP, \CC, \PP)$ is a mutation triple satisfying (MT4$^+$). 
	More generally, for any strongly functorially finite  subcategory $\XX$ in $\CC$, 
	$(\XX, \CC, \XX)$ in $(\CC, \bbE^{\XX}_{\XX}, \fraks^{\II}_{\II})$ is also a mutation triple satisfying (MT4$^+$).
	\item \cite{IY08} \label{ex_MT4_IY}
	In the case of Example \ref{ex_MT}(\ref{ex_MT_IY}),
	then $(\II, \ZZ, \II)$ is a mutation triple.
	\item \cite{SP20} \label{ex_MT4_SP}
	In the case of Example \ref{ex_MT}(\ref{ex_MT_SP}),
	then $(\bracket{\MM[1]}, \ZZ, \bracket{\MM[-1]})$ is a mutation triple satisfying (MT4$^+$).
	\end{enumerate}
\end{ex}
\begin{proof}
(1) From $\ul{\wt{\UU}}{}^- = \ul{\XX} = 0$ and $\ul{\wt{\TT}}{}^+ = \ul{\XX} = 0$.
(2) Since $\ZZ = \ZZ\bracket{1} = \ZZ\bracket{-1}$, 
$\ul{\ZZ\bracket{-1}}^+ = \ul{\ZZ}^+ = 0$ and $\ul{\ZZ\bracket{1}}^- = \ul{\ZZ}^- = 0$.
(3) From $\wt{\UU}^- = \wt{\TT}^+ = \bracket{\MM}$.
\end{proof}

\begin{rmk}
We may generalize the example of the mutation triple in Example \ref{ex_MT4} (\ref{ex_MT4_IY}),
which does not assume $n$-rigidity of $\II$.
This is suggested by Professor Michael Wemyss and the readers can check it in \cite[Definition 2.4, Theorem 2.7]{IW18}.
\end{rmk}

\subsection{Comparisons with known results in concentric twin cotorsion pairs}
In this subsection, we recall the results in \cite{Nak18} 
and compare them with the previous subsection.

We assume that $\CC$ is a triangulated category and $((\SS,\TT), (\UU,\VV))$ is a concentric twin cotorsion pair.
See Appendix \ref{Twin cotorsion pairs which induces triangulated categories},
for heart-equivalent and Hovey.

\begin{rmk} \cite[Proposition 6.3]{Nak18} \label{rmk_HE}
	Let $((\SS, \TT), (\UU, \VV))$ be a heart-equivalent concentric twin cotorsion pair, 
	then $\ZZ$ and $\II$ satisfy Condition \ref{IY_condi}.
	Thus, from Example \ref{ex_MT4} (\ref{ex_MT4_IY}), $(\II, \ZZ, \II)$ is a mutation triple.
\end{rmk}

\begin{prop} \label{prop_hoveytwin cotorsion pair_MT}
	Assume that $((\SS, \TT), (\UU, \VV))$ is Hovey,
	then $(\SS, \ZZ, \VV)$ induced by $((\SS, \TT), (\UU, \VV))$ is a mutation triple.
\end{prop}
\begin{proof}
Take the $\fraks^{\II}$-triangle 
$T \xrar{} {Z^{\pr}}^- \xrar{} Z^{\pr} \lxdrar{\chi_{Z^{\pr}}} T$ for $Z^{\pr} \in \ZZ\bracket{1}$ where $T \in \ZZ\bracket{-1} \subset \TT$.
Since we assume that $((\SS, \TT), (\UU, \VV))$ is Hovey,
there exists an $\fraks_{\II}$-triangle (since $\bbE(\II, \VV)=0$)
$V \rar S \rar {Z^{\pr}}^- \lxdrar{\delta} V$ where $V \in \VV, S \in \SS$.
Thus, it is enough to show $S \in \II$.

By (ET4)$^{\op}$ and Proposition \ref{shifted_octahedrons}, 
there exist the following commutative diagrams in $(\CC, \bbE_{\II}, \fraks_{\II})$
where $I \in \II, Z \in \ZZ$.
\[
\xy
(16,24)*+{V}="12";
(32,24)*+{V}="13";
(0,8)*+{Z^{\pr}}="21";
(16,8)*+{T^{\pr}}="22";
(32,8)*+{S}="23";
(48,8)*+{Z^{\pr}}="24";
(0,-8)*+{Z^{\pr}}="31";
(16,-8)*+{T}="32";
(32,-8)*+{{Z^{\pr}}^-}="33";
(48,-8)*+{Z^{\pr}}="34";
(16,-24)*+{V}="42";
(32,-24)*+{V}="43";
{\ar@{=} "12";"13"};
{\ar@{-->}^{\chi^{\pr}} "21";"22"};
{\ar^{} "22";"23"};
{\ar@{->}^{} "23";"24"};
{\ar@{-->}^{\chi_{Z^{\pr}}} "31";"32"};
{\ar^{} "32";"33"};
{\ar@{->}^{} "33";"34"};
{\ar@{=} "42";"43"};
{\ar@{=} "21";"31"};
{\ar@{->}^{} "12";"22"};
{\ar^{} "22";"32"};
{\ar@{-->}^{} "32";"42"};
{\ar@{->}^{} "13";"23"};
{\ar^{} "23";"33"};
{\ar@{-->}^{\delta} "33";"43"};
{\ar@{=} "24";"34"};
{\ar@{}|\car "12";"23"};
{\ar@{}|\car "21";"32"};
{\ar@{}|\car "22";"33"};
{\ar@{}|\car "23";"34"};
{\ar@{}|\car "32";"43"};
\endxy
\quad
\xy
(16,24)*+{Z}="12";
(32,24)*+{Z}="13";
(0,8)*+{T^{\pr}}="21";
(16,8)*+{E}="22";
(32,8)*+{I}="23";
(48,8)*+{T^{\pr}}="24";
(0,-8)*+{T^{\pr}}="31";
(16,-8)*+{S}="32";
(32,-8)*+{Z^{\pr}}="33";
(48,-8)*+{T^{\pr}}="34";
(16,-24)*+{Z}="42";
(32,-24)*+{Z}="43";
{\ar@{=} "12";"13"};
{\ar^{} "21";"22"};
{\ar^{} "22";"23"};
{\ar@{-->}^{} "23";"24"};
{\ar^{} "31";"32"};
{\ar^{} "32";"33"};
{\ar@{-->}^{\chi^{\pr}} "33";"34"};
{\ar@{=} "42";"43"};
{\ar@{=} "21";"31"};
{\ar^{} "12";"22"};
{\ar^{} "22";"32"};
{\ar@{-->}^{} "32";"42"};
{\ar^{} "13";"23"};
{\ar^{} "23";"33"};
{\ar@{-->}^{\lambda} "33";"43"};
{\ar@{=} "24";"34"};
{\ar@{}|\car "12";"23"};
{\ar@{}|\car "21";"32"};
{\ar@{}|\car "22";"33"};
{\ar@{}|\car "23";"34"};
{\ar@{}|\car "32";"43"};
\endxy
\]
$T^{\pr} \in \TT$ from $V, T \in \TT$. Then $E \in \UU \cap \TT = \ZZ$.
Since $\bbE(\SS, \ZZ) =0$, $S$ is a direct summand of $E \in \ZZ$.
Thus, $S \in \SS \cap \ZZ = \II$.
\end{proof}

\begin{cor} (Theorem \ref{case_of_twin cotorsion pair})
If $((\SS,\TT), (\UU,\VV))$ is heart-equivalent or Hovey, then a pretriangulated category $(\ul{\ZZ}, \Sigma,\Omega, \nabla, \Delta)$ defined by premutation triple $(\SS,$ $\ZZ,$ $\VV)$ is a triangulated category.
\end{cor}

Since $\wt{\TT}^+ \subset \Cone_{\bbE_{\II}}(\VV, \SS)$ and 
$\wt{\UU}^- \subset \CoCone_{\bbE^{\II}}(\VV, \SS)$, (MT4) is stronger than 
the conditions in Theorem \ref{case_of_twin cotorsion pair} for concentric twin cotorsion pairs.

\section{Reducible triples} \label{redMT}
We fix an ET category $(\CC, \bbE, \fraks)$ in this section.
We modify the condition (MT3) and (MT4) so as to make mutation triples suitable for the reduction theory.
Then we want to show that taking $\ul{(\cdot) \cap \ZZ}$ preserves extension closed subcategories.

First, we consider the following condition which is a different version of (MT3).

\begin{condi} \label{condi_rMT3}
Let $(\SS, \ZZ, \VV)$ be a triplet of subcategories of $\CC$ satisfying (MT1) and (MT2).
	\begin{itemize}[leftmargin=40pt]
	\item[(RT3)]
		\begin{enumerate}[label=(\roman*)]
		\item $\Cone_{\bbE^{\II}}(\ZZ, \ZZ) \subset \CoCone_{\bbE^{\II}_{\II}}(\ZZ, \SS)$.
		\item $\CoCone_{\bbE_{\II}}(\ZZ, \ZZ) \subset \Cone_{\bbE^{\II}_{\II}}(\VV, \ZZ)$.
		\item $\SS$, $\ZZ$ and $\VV$ are closed under extensions in $(\CC, \bbE^{\II}_{\II}, \fraks^{\II}_{\II})$.
		\end{enumerate}
	\end{itemize}
\end{condi}

\begin{rmk}
	\begin{enumerate}
	\item We are not sure that all concentric twin cotorsion pairs satisfy (RT3).
	\item The conditions ($\mathrm{i}$) and ($\mathrm{ii}$) are stronger than those in (MT3) of Definition \ref{defi_first}, but
	the condition ($\mathrm{iii}$) is weaker than that in (MT3) of Definition \ref{defi_first}.
	So, it is not also sure that there exist implications between (MT3) and (RT3).
	\end{enumerate}
\end{rmk}

\begin{defi}
Let $\SS, \ZZ$ and $\VV$ be subcategories of $\CC$.
$(\SS, \ZZ, \VV)$ is a \emph{prereducible triple} if it satisfies (MT1), (MT2) and (RT3).
\end{defi}

From (RT3), we may define $\bracket{\pm 1}$, $\sigma$ and $\omega$ by $\fraks^{\II}_{\II}$-triangles.
Then the proofs in section \ref{pretri} for premutation triple also work even if $(\SS, \ZZ, \VV)$ is a prereducible triple.
That is because all $\fraks^{\II}$-triangles and $\fraks_{\II}$-triangles can be replaced by $\fraks^{\II}_{\II}$-triangles.

\begin{cor}
Let $(\SS, \ZZ, \VV)$ be a prereducible triple.
We define $\Sigma, \Omega, \nabla$ and $\Delta$ same as premutation triples.
Then $(\ul{\ZZ}, \Sigma, \Omega, \nabla, \Delta)$ is a pretriangulated category.
\end{cor}

\begin{ex} \label{ex_redpreMT}
Let $(\SS, \ZZ, \VV)$ be a triplet of subcategories satisfying (MT1) and (MT2).
	\begin{enumerate}
	\item \label{ex_redpreMT_rigid}
	If $(\SS, \ZZ, \VV)$ satisfies the following conditions, then $(\SS, \ZZ, \VV)$ satisfies (RT3).
		\begin{enumerate}
		\item $\Cone_{\bbE^{\II}}(\ZZ, \ZZ) = \CoCone_{\bbE_{\II}}(\ZZ, \ZZ) = \ZZ$.
		\item $\II = \SS = \VV$.
		\item $\ZZ$ is closed under extensions in $(\CC, \bbE^{\II}_{\II}, \fraks^{\II}_{\II})$.
		\end{enumerate}
	The examples which satisfy these condition are as follows.
		\begin{enumerate}[label=(\alph*)]
		\item \cite[Lemma 4.3]{IY08}
		$(\II, \ZZ, \II)$ in Example \ref{ex_MT}(\ref{ex_MT_IY}).
		\item \cite[Lemma 4.3]{IY08}, \cite[Lemma 3.5]{IY18}
		The premutation triple 
		$\big( (\thick \II)_{\geq 0},$ $\rpp{\II[<\!0]} \cap \lpp{\II[>\!0]},$ $(\thick \II)_{\leq 0} \big)$
		which is defined by the Hovey twin cotorsion pair 
		in Example \ref{ex_HoveyTCP}(\ref{ex_HoveyTCP-presilt}).
		\end{enumerate}
	\item \label{ex_redpreMT_orth}
	If $(\SS, \ZZ, \VV)$ satisfies the following conditions, then $(\SS, \ZZ, \VV)$ satisfies (RT3).
		\begin{enumerate}
		\item $\II \subset \Proj \CC \cap \Inj \CC$.
		\item (MT3).
		\end{enumerate}
	The examples which satisfy these condition are as follows.
		\begin{enumerate}[label=(\alph*)]
		\item \cite[Theorem 4.1]{SP20}
		$(\bracket{\MM[1]}, \ZZ, \bracket{\MM[-1]})$ in Example \ref{ex_MT}(\ref{ex_MT_SP}).
		\item \cite[Proposition 3.2]{Jin23}
		The premutation triple
		$\big( (\thick \MM)^{<0},$ $\rpp{\MM[\geq\!\!0]} \cap \lpp{\MM[\leq\!0]},$ $(\thick \MM)^{>0} \big)$ 
		which is defined by the Hovey twin cotorsion pair 
		in Example \ref{ex_HoveyTCP}(\ref{ex_HoveyTCP-preSMC}).
		\end{enumerate}
	\end{enumerate}
\end{ex}

The following condition is an analogy of (MT4$^+$).

\begin{condi}
Let $(\SS, \ZZ, \VV)$ be a prereducible triple.
We consider the following conditions.
	\begin{itemize}[leftmargin=40pt]
	\item[(RT4)]
		\begin{enumerate}[label=(\roman*)]
		\item $\II$ is strongly contravariantly finite in $\SS$.
		\item $\II$ is strongly covariantly finite in $\VV$.
		\item $\SS\bracket{-1} = \VV\bracket{1}$, denoted by $\RR$.
		\end{enumerate}
	\end{itemize}
\end{condi}

\begin{rmk}
Let $(\SS, \ZZ, \VV)$ be a triplet of subcategories in $\CC$ satisfying (MT1) and (MT2).
	\begin{enumerate}
	\item From (MT2), 
	\begin{align*}
	\SS\bracket{-1} &= \CoCone_{\bbE_{\II}}(\II, \SS) = \CoCone_{\bbE^{\II}_{\II}}(\II, \SS) \\
	\VV\bracket{1} &= \Cone_{\bbE^{\II}}(\VV, \II) = \Cone_{\bbE^{\II}_{\II}}(\VV, \II).
	\end{align*}
	\item Even if $\SS\bracket{-1} = \VV\bracket{1}$ holds,
	$\wt{\UU}^- = \CoCone_{\bbE^{\II}}(\II, \SS)$ does not coincide with 
	$\wt{\TT}^+ = \Cone_{\bbE_{\II}}(\VV, \II)$ in general. 
	For example, let $\CC = \Db(k A_4)$ where $A_4$ is the linearly oriented Dynkin quiver of type $A_4$.
	We define $\II^{\pr} = \{I_1, I_2\}$ in the following pictures and denote $\add \II^{\pr}$ by $\II$.
	Then $\II$ induces a rigid mutation triple from Lemma \ref{lem_redMT}(\ref{lem_redMT_presilt}) 
	and Example \ref{ex_redMT}(\ref{ex_redMT_rigid}).
	The left \resp{right} picture is the AR quiver of $(\CC, \bbE^{\II}, \fraks^{\II})$ 
	\resp{$(\CC, \bbE_{\II}, \fraks_{\II})$} \cite[Proposition 5.10]{INP24}.
	Then $X \in \CoCone_{\bbE^{\II}}(\II, \II)$ but $X \notin \Cone_{\bbE_{\II}}(\II, \II)$, and  
	$Y \in \Cone_{\bbE_{\II}}(\II, \II)$ but $Y \notin \CoCone_{\bbE^{\II}}(\II, \II)$.
	On the other hand, $\II\bracket{1} = \II\bracket{-1} = \II$.
\[
\begin{tikzpicture}
	\foreach \x in {0,1,...,4}
		\foreach \y in {1,3}
		{
		\fill (\x,\y) circle[radius=0.04cm];
		\fill (\x+0.5,\y-1) circle[radius=0.04cm];
		\draw[-{stealth[scale=3]}] (\x,\y) -- (\x+0.48,\y-0.96);
		\draw[-{stealth[scale=3]}] (\x,1) -- (\x+0.48,1.96);
		\draw[-{stealth[scale=3]}, dashed] (\x+0.98,3) -- (\x+0.02,3);
		\draw[-{stealth[scale=3]}] (\x+0.5,\y-1) -- (\x+0.98,\y-0.04);
		\draw[-{stealth[scale=3]}] (\x+0.5,2) -- (\x+0.98,1.04);
		\ifnum \x<4
			\draw[-{stealth[scale=3]}, dashed] (\x+1.48,2) -- (\x+0.52,2);
		\fi
		\ifnum \x<3
			\draw[-{stealth[scale=3]}, dashed] (\x+0.98,1) -- (\x+0.02,1);
		\fi
		\ifnum \x>3
			\draw[-{stealth[scale=3]}, dashed] (\x+0.98,1) -- (\x+0.02,1);
		\fi
		\ifnum \x<2
			\draw[-{stealth[scale=3]}, dashed] (\x+1.48,0) -- (\x+0.52,0);
		\fi
		\ifnum 2<\x
			\ifnum \x<4
				\draw[-{stealth[scale=3]}, dashed] (\x+1.48,0) -- (\x+0.52,0);
			\fi
		\fi
		\draw[dashed] (0.5,2) -- (0,2);
		\draw[dashed] (0.5,0) -- (0,0);
		\draw[-{stealth[scale=3]}, dashed] (5,2) -- (4.52,2);
		\draw[-{stealth[scale=3]}, dashed] (5,0) -- (4.52,0);
		}
	\fill [black]
	(2.5,0) circle[radius=0.08cm]
	(3,1) circle[radius=0.08cm]
	(5,3) circle[radius=0.04cm]
	(5,1) circle[radius=0.04cm]
	(1,3) circle[radius=0.08cm];
	\node at (2.53,0) [below] {\msize{0.7}{I_1}};
	\node at (3,0.9) [below] {\msize{0.7}{I_2}};
	\node at (1,2.9) [below] {\msize{0.7}{X}};
\end{tikzpicture}
\qquad
\begin{tikzpicture}
	\foreach \x in {0,1,...,4}
		\foreach \y in {1,3}
		{
		\fill (\x,\y) circle[radius=0.04cm];
		\fill (\x+0.5,\y-1) circle[radius=0.04cm];
		\draw[-{stealth[scale=3]}] (\x,\y) -- (\x+0.48,\y-0.96);
		\draw[-{stealth[scale=3]}] (\x,1) -- (\x+0.48,1.96);
		\draw[-{stealth[scale=3]}, dashed] (\x+0.98,3) -- (\x+0.02,3);
		\draw[-{stealth[scale=3]}] (\x+0.5,\y-1) -- (\x+0.98,\y-0.04);
		\draw[-{stealth[scale=3]}] (\x+0.5,2) -- (\x+0.98,1.04);
		\ifnum \x<4
			\draw[-{stealth[scale=3]}, dashed] (\x+1.48,2) -- (\x+0.52,2);
		\fi
		\ifnum \x<2
			\draw[-{stealth[scale=3]}, dashed] (\x+0.98,1) -- (\x+0.02,1);
		\fi
		\ifnum \x>2
			\draw[-{stealth[scale=3]}, dashed] (\x+0.98,1) -- (\x+0.02,1);
		\fi
		\ifnum \x<1
			\draw[-{stealth[scale=3]}, dashed] (\x+1.48,0) -- (\x+0.52,0);
		\fi
		\ifnum 1<\x
			\ifnum \x<4
				\draw[-{stealth[scale=3]}, dashed] (\x+1.48,0) -- (\x+0.52,0);
			\fi
		\fi
		\draw[dashed] (0.5,2) -- (0,2);
		\draw[dashed] (0.5,0) -- (0,0);
		\draw[-{stealth[scale=3]}, dashed] (5,2) -- (4.52,2);
		\draw[-{stealth[scale=3]}, dashed] (5,0) -- (4.52,0);
		}
	\fill [black]
	(2.5,0) circle[radius=0.08cm]
	(3,1) circle[radius=0.08cm]
	(5,3) circle[radius=0.04cm]
	(5,1) circle[radius=0.04cm]
	(3.5,0) circle[radius=0.08cm];
	\node at (2.53,0) [below] {\msize{0.7}{I_1}};
	\node at (3,0.9) [below] {\msize{0.7}{I_2}};
	\node at (3.5,0) [below] {\msize{0.7}{Y}};
\end{tikzpicture}
\]
	\end{enumerate}
\end{rmk}

\begin{defi}
Let $(\SS, \ZZ, \VV)$ be a prereducible triple.
$(\SS, \ZZ, \VV)$ is called a \emph{reducible triple} if it satisfies (RT4).
\end{defi}

We are not sure reducible triples are always mutation triples.
However, the following statement also holds for reducible triples.

\begin{cor}
Let $(\SS, \ZZ, \VV)$ be a reducible triple.
Then the quintuplet $(\ul{\ZZ},$ $\Sigma,$ $\Omega,$ $\nabla,$ $\Delta)$ in Theorem \ref{main_thm1} is a triangulated category.
\end{cor}
\begin{proof}
From (RT3), we may assume that all diagrams in the proof of Theorem \ref{main_thm2} are in $(\CC, \bbE^{\II}_{\II}, \fraks^{\II}_{\II})$.
So, the proof also works under (RT4).
\end{proof}

Before we introduce some examples, we prepare the following lemma.
We use the notations in Notations \ref{nota_thick}.
See Appendix \ref{defi_rigid} and \ref{defi_orthogonal}, for the definitions of presilting and pre-simple-minded collections.

\begin{lem} \label{lem_redMT}
Assume that $\CC$ is a triangulated category.
	\begin{enumerate}
	\item \label{lem_redMT_presilt}
	Let $\II$ be a presilting subcategory. We assume that 
	$\big( ((\thick \II)_{\geq 0},$ $\rpp{\II[<\!\!0]}),$ $(\lpp{\II[>\!0]},$ $(\thick \II)_{\leq 0}) \big)$
	is a Hovey twin cotorsion pair and let $\ZZ = \rpp{\II[<\!0]} \cap \lpp{\II[>\!0]}$.
	Then $(\II, \ZZ, \II)$ satisfies (MT1), (MT2) and the conditions in Example \ref{ex_redpreMT}(\ref{ex_redpreMT_rigid}).
	In particular, it is a prereducible triple.
	\item \label{lem_redMT_preSMC}
	Let $\MM$ be a pre-simple-minded collection. We assume that 
	$\big( ((\thick \MM)^{<0},$ $\rpp{\MM[\geq\!0]}),$ $(\lpp{\MM[\leq \!0]},$ $(\thick \MM)^{>0}) \big)$
	is a Hovey twin cotorsion pair and let $\ZZ = \rpp{\MM[\geq\!0]} \cap \lpp{\MM[\leq\!0]}$.
	Then $(\bracket{\MM[1]}, \ZZ, \bracket{\MM[-1]})$ satisfies (MT1), (MT2) and the conditions in Example \ref{ex_redpreMT}(\ref{ex_redpreMT_orth}).
	In particular, it is a prereducible triple.
	\end{enumerate}
\end{lem}
\begin{proof}
(1)
Since $\big( (\thick \II)_{\geq 0},$ $\rpp{\II[<\!0]} \cap \lpp{\II[>\!0]},$ $(\thick \II)_{\leq 0} \big)$ 
is a premutation triple, (MT1) and (MT2) are clear.
$\SS = \VV = \II$ and $\II, \ZZ$ are closed under extensions in $(\CC, \bbE^{\II}_{\II}, \fraks^{\II}_{\II})$.
Because $\Cone_{\bbE^{\II}}(\ZZ, \ZZ) = \CoCone_{\bbE_{\II}}(\ZZ, \ZZ) = \ZZ$,
the statement holds.

(2)
Similar to (1), (MT1) and (MT2) hold.
Since $\II=0$, it is enough to show that (MT3) holds.
Take an $\fraks^{\II}$-triangle $Z_1 \rar Z_2 \rar C \rar Z_1[1]$ where $Z_1, Z_2 \in \ZZ$.
From the long exact sequence, $C \in \rpp{\MM[\geq\! 1]} \cap \lpp{\MM[\leq\!0]}$.
Since $((\thick \MM)^{\leq0}, \rpp{\MM[\geq0]})$ is a torsion pair,
there exists a triangle $C^0 \rar C \rar Z^C \rar C^0[1]$ where 
$C^0 \in (\thick \MM)^{\leq0}$ and $Z^C \in \rpp{\MM[\geq0]}$.
Note that $C^0 \in \rpp{\MM[\geq 1]}$ because $Z^C[-1], C \in \rpp{\MM[\geq 1]}$.
Then $C^0 \in \bracket{\MM}$ follows from $C^0 \in (\thick \MM)^{\leq0}$.
Since $C, C^0[1] \in \lpp{\MM[\leq 0]}$, then $Z^C \in \lpp{\MM[\leq 0]}$. 
Thus, $Z^C \in \ZZ$ and $\Cone_{\bbE}(\ZZ, \ZZ) \subset \CoCone_{\bbE}(\ZZ, \bracket{\MM[1]})$.
Dually,  $\CoCone_{\bbE}(\ZZ, \ZZ) \subset \Cone_{\bbE}(\bracket{\MM[-1]}, \ZZ)$ holds.
Because $\bracket{\MM[1]}, \bracket{\MM[-1]}$ and $\ZZ$ are closed under extensions,
(MT3) holds.
\end{proof}

\begin{ex} \label{ex_redMT}
Let $(\SS, \ZZ, \VV)$ be a prereducible triple.
	\begin{enumerate} 
	\item \label{ex_redMT_rigid}
	If $\SS = \VV = \II$, then (RT4) holds. The followings are these examples.
		\begin{enumerate}
		\item \cite[Lemma 4.3]{IY08}
		The prereducible triple $(\II, \ZZ, \II)$ in Example \ref{ex_MT}(\ref{ex_MT_IY}).
		\item \cite[Lemma 4.3]{IY08}, \cite[Lemma 3.5]{IY18}
		The prereducible triple $(\II, \ZZ, \II)$ in Lemma \ref{lem_redMT}(\ref{lem_redMT_presilt}).
		\end{enumerate}
	\item \label{ex_redMT_orth}
	The followings are examples satisfying (RT4) and $\II \subset \Proj\CC \cap \Inj\CC$.
		\begin{enumerate}
		\item \cite[Theorem 4.1]{SP20}
		The prereducible triple
		$(\bracket{\MM[1]}, \ZZ, \bracket{\MM[-1]})$ in Example \ref{ex_MT}(\ref{ex_MT_SP}).
		\item \!\!\cite[Proposition 3.2]{Jin23} 
		The prereducible triple
		$(\bracket{\MM[1]},$ $\ZZ,$ $\bracket{\MM[-1]} )$ in Lemma \ref{lem_redMT}(\ref{lem_redMT_preSMC}).
		\end{enumerate}
	\end{enumerate}
\end{ex}

From the observations in Example \ref{ex_redpreMT} and \ref{ex_redMT}, 
we define the following special class of reducible triples.

\begin{defi}
Let $(\SS, \ZZ, \VV)$ be a triplet of subcategories which satisfies (MT1) and (MT2).
	\begin{enumerate}
	\item $(\SS, \ZZ, \VV)$ is a \emph{rigid mutation triple} if it satisfies the following conditions, named (rigMT).
		\begin{enumerate}
		\item $\Cone_{\bbE^{\II}}(\ZZ, \ZZ) = \CoCone_{\bbE_{\II}}(\ZZ, \ZZ) = \ZZ$.
		\item $\II = \SS = \VV$.
		\item $\ZZ$ is closed under extensions in $(\CC, \bbE^{\II}_{\II}, \fraks^{\II}_{\II})$.
		\end{enumerate}
	\item $(\SS, \ZZ, \VV)$ is an \emph{orthogonal mutation triple} if it satisfies the following conditions, named (ortMT).
		\begin{enumerate}
		\item $\II \subset \Proj \CC \cap \Inj \CC$.
		\item (MT3) and (RT4).
		\end{enumerate}
	\end{enumerate}
\end{defi}

\begin{rmk}
If $(\SS, \ZZ, \VV)$ is a rigid mutation triple, 
(MT4) follows from $\ZZ\bracket{1}^-,$ $\ZZ\bracket{-1}^+$ $\subset \II$.
If $(\SS, \ZZ, \VV)$ is an orthogonal mutation triple,
(MT4$^+$) directly follows from (RT4) since $\bbE = \bbE^{\II}_{\II}$.
Therefore, both rigid mutation triples and orthogonal mutation triples are mutation triples.
\end{rmk}

\begin{rmk}
We illustrate the hierarchy of conditions for triplets of subcategories $(\SS, \ZZ, \VV)$ satisfying (MT1) and (MT2)
and concentric twin cotorsion pairs $((\SS, \TT), (\UU, \VV))$.
In the following picture, we abbreviate ``premutation triple'' to ``preMT'', ``prereducible triple'' to ``preRT'' and ``concentric twin cotorsion pair'' to ``cTCP''.
The condition (I)+(I\hspace{-1.2pt}I) is defined in \cite[Condition 5.2]{Nak18}.
(HE) and (Ho) are defined in Definition \ref{defi_HEandHo}.
(Nak) is defined in Condition \ref{condi_MT4-}.

Note that a heart-equivalent concentric twin cotorsion pair $((\SS, \TT), (\UU, \VV))$ induces 
a rigid mutation triple $(\II, \ZZ, \II)$ from Remark \ref{rmk_HE}.
Any concentric twin cotorsion pair $((\SS, \TT), (\UU, \VV))$ induces a premutation triple $(\SS, \ZZ, \VV)$.
This correspondence induces mutation triples from Hovey twin cotorsion pairs and the condition (MT4$^-$) from the condition (Nak).
Remark that all of the above $(\SS, \ZZ, \VV)$ and $((\SS, \TT), (\UU, \VV))$ induce triangulated categories \emph{except for} premutation triples satisfying only (MT4$^-$).
\[
\begin{tikzpicture}
	\draw[rounded corners] (-1, 5.4) -- (-1,-0.5) -- (3,-0.5) -- (3,5.8) -- (0.4,5.8);
	\draw[rounded corners] (5, 5.3) -- (5,-0.6) -- (1,-0.6) -- (1,5.7) -- (4,5.7);
	\draw[rounded corners] (6.9, 5.3) -- (6.9,-0.6) -- (10.7,-0.6) -- (10.7,5.7) -- (7.9,5.7);
	\node (r4) at (0,2) [centered] {\text{(RT4)}};
	\node (ortMT) at (2,5) [centered] {\text{(ortMT)}};
	\node (rigMT) at (2,3) [centered] {\text{(rigMT)}};
	\node (4+) at (4,4) [centered] {\text{(MT4$^+$)}};
	\node (4) at (4,2) [centered] {\text{(MT4)}};
	\node (4-) at (4,0) [centered] {\text{(MT4$^-$)}};
	\node (redMT) at (-0.4,5.8) [centered] {\text{preRT}};
	\node (preMT) at (4.8,5.7) [centered] {\text{preMT}};
	\node (cTCP) at (7.1,5.7) [centered] {\text{cTCP}};
	\node (HE) at (10.1,3) [centered] {\text{(HE)}};
	\node (Ho) at (8.3,2) [centered] {\text{(Ho)}};
	\node (Ho-l) at (8.1,2) [centered] {\phantom{\text{(Ho)}}};
	\node (Ho-r) at (8.5,2) [centered] {\phantom{\text{(Ho)}}};
	\node (1+2) at (9.1,1) [centered] {(I)+(I\hspace{-1.2pt}I)};
	\node (Nak) at (9.1,0) [centered] {\text{(Nak)}};
	\node (silt-l) at (8.1,4) [centered] {\phantom{\text{Lemma \ref{lem_redMT}(\ref{lem_redMT_presilt})}}};
	\node (smc-r) at (8.5,5) [centered] {\phantom{\text{Lemma \ref{lem_redMT}(\ref{lem_redMT_preSMC})}}};
	\node (smc) at (8.3,5) [centered, fill=white] {\text{Lemma \ref{lem_redMT}(\ref{lem_redMT_preSMC})}};
	\draw[-implies,double equal sign distance] (ortMT) -- (r4);
	\draw[-implies,double equal sign distance] (rigMT) -- (r4);
	\draw[-implies,double equal sign distance] (ortMT) -- (4+);
	\draw[-implies,double equal sign distance] (rigMT) -- (4);
	\draw[-implies,double equal sign distance] (4+) -- (4);
	\draw[-implies,double equal sign distance] (4) -- (4-);
	\draw[-implies,double equal sign distance] (HE) -- (1+2);
	\draw[-implies,double equal sign distance] (Ho) -- (1+2);
	\draw[-implies,double equal sign distance] (1+2) -- (Nak);
	\draw[-implies, double equal sign distance] (silt-l) -- (Ho-l);
	\draw[-implies, double equal sign distance] (smc-r) -- (Ho-r);
	\node (silt) at (8.3,4) [centered, fill=white] {\text{Lemma \ref{lem_redMT}(\ref{lem_redMT_presilt})}};
	\draw[dashed, ->] (cTCP) -- node[above] {$\msize{0.6}{(\SS, \ZZ, \VV)}$} (preMT);
	\draw[dashed, ->] (Ho) -- node[above, sloped] {$\msize{0.6}{(\SS, \ZZ, \VV)}$} (4);
	\draw[dashed, ->] (Nak) -- node[above, fill=white, near end] {$\msize{0.6}{(\SS, \ZZ, \VV)}$} (4-);
	\draw[dashed, ->] (HE) -- node[above, sloped] {$\msize{0.6}{(\II, \ZZ, \II)}$} (rigMT);
	\draw[dashed, ->] (silt) -- node[above, sloped, near start] {$\msize{0.6}{(\II, \ZZ, \II)}$} (rigMT);
	\draw[dashed, ->] (smc) -- node[above, sloped, fill=white] {$\msize{0.6}{(\bracket{\MM[1]}, \ZZ, \bracket{\MM[-1]})}$} (ortMT);
\end{tikzpicture}
\]
\end{rmk}

Now, we try to relate mutation functors of reducible triples to already known mutations.
We denote the extension closure in $(\CC, \bbE^{\II}_{\II}, \fraks^{\II}_{\II})$ by $\bracket{\cdot}^{\II}_{\II}$.

\begin{defi}
Let $(\SS, \ZZ, \VV)$ be a reducible triple.
$\RR^{\pr} = \{ R_i \}_{i \in I} \subset \Ob(\CC)$ is called a \emph{reducible collection} of $(\SS, \ZZ, \VV)$ if 
it satisfies the following conditions.
	\begin{enumerate}[label=(\roman*)]
	\item $\RR = \bracket{\RR^{\pr}}^{\II}_{\II}$ and 
	$\RR \neq \bracket{\RR^{\pr} \setminus R_i}^{\II}_{\II}$ for any $i \in I$.
	\item $R_i \cong R_j$ if and only if $i=j$.
	\item $R_i$ is indecomposable for any $i \in I$.
	\end{enumerate}
\end{defi}

\begin{ex} \phantom{X}\label{ex_rMT-red-coll}
	\begin{enumerate}
	\item \label{ex_rMT-red-coll_1}
	Assume that $\CC$ is a Krull-Schmidt category and let $(\II, \ZZ, \II)$ be a rigid mutation triple.
	We denote the set of isomorphism class of indecomposable objects in $\II$ by $\II^{\pr}$.
	Then a reducible collection of $(\II, \ZZ, \II)$ is $\II^{\pr}$.
	\item \label{ex_rMT-red-coll_2}
	Let $(\bracket{\MM[1]}, \ZZ, \bracket{\MM[-1]})$ be an orthogonal mutation triple 
	in Example \ref{ex_MT}(\ref{ex_MT_SP}) or Lemma \ref{lem_redMT}(\ref{lem_redMT_preSMC}).
	Then a reducible collection of $(\bracket{\MM[1]}, \ZZ, \bracket{\MM[-1]})$ is $\MM$.
	\end{enumerate}
\end{ex}

Now, we fix a reducible triple $(\SS, \ZZ, \VV)$.
We check that the definition below is a simultaneous generalization of silting mutation, 
cluster-tilting mutation, mutations of simple-minded collections and mutations of simple-minded systems.

\begin{defi} \label{defi_mu-rMT}
Assume that $(\SS, \ZZ, \VV)$ has a reducible collection $\RR^{\pr}$ and 
$\XX \subset \Ob(\ZZ)$ be a collection containing $\RR^{\pr}$.
We denote $\XX \setminus \RR^{\pr}$ by $\XX_{\RR^{\pr}}$.
	\begin{enumerate}
	\item We define \emph{right $\RR^{\pr}$-mutation} of $\XX$ as 
	$\RR^{\pr} \cup \Sigma \XX_{\RR^{\pr}}$, which is denoted by $\mu^-_{\RR^{\pr}}(\XX)$.
	\item We define \emph{left $\RR^{\pr}$-mutation} of $\XX$ as 
	$\RR^{\pr} \cup \Omega \XX_{\RR^{\pr}}$, which is denoted by $\mu^+_{\RR^{\pr}}(\XX)$.
	\end{enumerate}
\end{defi}

Next, we consider restricting mutations to extension closed subcategories such as 2-term silting in \cite{AIR14}.

\begin{thm} \label{main_thm3}
Let $(\SS, \ZZ, \VV)$ be a reducible triple.
Let $\EE$ be an extension closed subcategory in $(\CC, \bbE^{\II}_{\II}, \fraks^{\II}_{\II})$ containing $\RR$.
Then $\ul{\EE \cap \ZZ}$ is an extension closed subcategory in $\ul{\ZZ}$.
\end{thm}
\begin{proof}
Take a triangle in $\ul{\ZZ}$
\begin{align}
E \xrar{\ul{a}} X \xrar{\ul{b^{\pr}}} E^{\pr} \xrar{\ul{c^{\pr}}} \Sigma E \label{tri_in_Z}
\end{align}
where $E, E^{\pr} \in \ul{\EE}$.
From Corollary \ref{cor_iso-in-ul}, we may assume $E, E^{\pr} \in \EE$ by replacing the triangle \eqref{tri_in_Z}
to isomorphic triangle in $\ul{\ZZ}$. 
There exists an $\fraks^{\II}_{\II}$-triangle 
$E \xrar{\msize{0.6}{\begin{bmatrix} a \\ i^E  \end{bmatrix}}} X\oplus I^E \xrar{\wt{b}} C^a \xdrar{\wt{\delta}} E$
and this $\fraks^{\II}_{\II}$-triangle determines the standard right triangle of $a$, 
\begin{align}
E \xrar{\ul{a}} X \xrar{\ul{h^{C^a} b}} \sigma C^a \xrar{\sigma(\ul{c})} \Sigma E. \label{standard_tri_in_Z}
\end{align}
By definition of triangles in $\ul{\ZZ}$, the triangle \eqref{tri_in_Z} is isomorphic to the standard right triangle \eqref{standard_tri_in_Z}, in particular, $\sigma C^a$ is isomorphic to $E^{\pr}$ in $\ul{\CC}$.
In particular, there exists $I \in \II$ where $\sigma C^a \oplus I \in \EE$ from Corollary \ref{cor_iso-in-ul}. 
By (RT3), there exists an $\fraks^{\II}_{\II}$-triangle 
$C^a \xrar{h^{C^a}} \sigma C^a \rar S \xdrar{\rho^{C^a}} C^a$ 
where $S \in \SS$.
From (RT4), there exists an $\fraks^{\II}_{\II}$-triangle
$S^{\pr} \xrar{i_S} I_S \xrar{p_S} S \lxdrar{\lambda_S} S^{\pr}$ where $I_S \in \II$.
Then there exists the following commutative diagrams in $(\CC, \bbE^{\II}_{\II}, \fraks^{\II}_{\II})$.
\[
\xy
(16,24)*+{C^a}="12";
(32,24)*+{C^a}="13";
(0,8)*+{S^{\pr}}="21";
(16,8)*+{Y}="22";
(32,8)*+{\sigma C^a}="23";
(48,8)*+{S^{\pr}}="24";
(0,-8)*+{S^{\pr}}="31";
(16,-8)*+{I_S}="32";
(32,-8)*+{S}="33";
(48,-8)*+{S^{\pr}}="34";
(16,-24)*+{C^a}="42";
(32,-24)*+{C^a}="43";
{\ar@{=} "12";"13"};
{\ar^{} "21";"22"};
{\ar^{} "22";"23"};
{\ar@{-->}^{\epsilon} "23";"24"};
{\ar^{i_S} "31";"32"};
{\ar^{p_S} "32";"33"};
{\ar@{-->}^{\lambda_S} "33";"34"};
{\ar@{=} "42";"43"};
{\ar@{=} "21";"31"};
{\ar^{} "12";"22"};
{\ar^{} "22";"32"};
{\ar@{-->}^{0} "32";"42"};
{\ar^{h^{C^a}} "13";"23"};
{\ar^{} "23";"33"};
{\ar@{-->}^{\rho^{C^a}} "33";"43"};
{\ar@{=} "24";"34"};
{\ar@{}|\car "12";"23"};
{\ar@{}|\car "21";"32"};
{\ar@{}|\car "22";"33"};
{\ar@{}|\car "23";"34"};
{\ar@{}|\car "32";"43"};
\endxy
\qquad
\xy
(16,24)*+{I}="12";
(32,24)*+{I}="13";
(0,8)*+{S^{\pr}}="21";
(16,8)*+{Y^{\pr}}="22";
(32,8)*+{\sigma C^a \oplus I}="23";
(48,8)*+{S^{\pr}}="24";
(0,-8)*+{S^{\pr}}="31";
(16,-8)*+{Y}="32";
(32,-8)*+{\sigma C^a}="33";
(48,-8)*+{S^{\pr}}="34";
(16,-24)*+{I}="42";
(32,-24)*+{I}="43";
{\ar@{=} "12";"13"};
{\ar^{} "21";"22"};
{\ar^{} "22";"23"};
{\ar@{-->}^{} "23";"24"};
{\ar^-{} "31";"32"};
{\ar^-{} "32";"33"};
{\ar@{-->}^{\epsilon} "33";"34"};
{\ar@{=} "42";"43"};
{\ar@{=} "21";"31"};
{\ar^{} "12";"22"};
{\ar^{} "22";"32"};
{\ar@{-->}^{0} "32";"42"};
{\ar^{} "13";"23"};
{\ar^{} "23";"33"};
{\ar@{-->}^{0} "33";"43"};
{\ar@{=} "24";"34"};
{\ar@{}|\car "12";"23"};
{\ar@{}|\car "21";"32"};
{\ar@{}|\car "22";"33"};
{\ar@{}|\car "23";"34"};
{\ar@{}|\car "32";"43"};
\endxy
\]
Since $S^{\pr} \in \SS\bracket{-1} = \RR \subset \EE$ and $\sigma C^a \oplus I \in \EE$, $Y^{\pr} \in \EE$.
Thus, $C^a \cong Y \cong Y^{\pr} \in \ul{\EE}$.
Note that $Y \cong C^a \oplus I_S$ and $Y^{\pr} \cong Y \oplus I$ because the above diagram is considered in 
$(\CC, \bbE^{\II}_{\II}, \fraks^{\II}_{\II})$.
From Corollary \ref{cor_iso-in-ul} again, there exists $I^{\pr} \in \II$ where $C^a \oplus I^{\pr} \in \EE$
and there exists the commutative diagram \eqref{diag:extension-closed} in $(\CC, \bbE^{\II}_{\II}, \fraks^{\II}_{\II})$.

Because $\EE$ is extension closed in $(\CC, \bbE^{\II}_{\II}, \fraks^{\II}_{\II})$, 
$X^{\pr} \in \EE$ and $X \cong X^{\pr} \in \ul{\EE}$ holds.

\begin{align}
\xy
(16,24)*+{I^{\pr}}="12";
(32,24)*+{I^{\pr}}="13";
(0,8)*+{E}="21";
(16,8)*+{X^{\pr}}="22";
(32,8)*+{C^a \oplus I^{\pr}}="23";
(48,8)*+{E}="24";
(0,-8)*+{E}="31";
(16,-8)*+{X\oplus I^E}="32";
(32,-8)*+{C^a}="33";
(48,-8)*+{E}="34";
(16,-24)*+{I^{\pr}}="42";
(32,-24)*+{I^{\pr}}="43";
{\ar@{=} "12";"13"};
{\ar^{} "21";"22"};
{\ar^{} "22";"23"};
{\ar@{-->}^{} "23";"24"};
{\ar^-{\wt{a}} "31";"32"};
{\ar^-{\wt{b}} "32";"33"};
{\ar@{-->}^{\wt{\delta}} "33";"34"};
{\ar@{=} "42";"43"};
{\ar@{=} "21";"31"};
{\ar^{} "12";"22"};
{\ar^{} "22";"32"};
{\ar@{-->}^{0} "32";"42"};
{\ar^{} "13";"23"};
{\ar^{} "23";"33"};
{\ar@{-->}^{0} "33";"43"};
{\ar@{=} "24";"34"};
{\ar@{}|\car "12";"23"};
{\ar@{}|\car "21";"32"};
{\ar@{}|\car "22";"33"};
{\ar@{}|\car "23";"34"};
{\ar@{}|\car "32";"43"};
\endxy \label{diag:extension-closed}
\end{align}
\end{proof}

\begin{cor}
In Theorem \ref{main_thm3}, $\ul{\EE \cap \ZZ}$ has an ET structure as an extension closed subcategory 
in the triangulated category $\ul{\ZZ}$.
\end{cor}

Now, we define the restriction of mutations on extension closed subcategories.

\begin{defi}
In Theorem \ref{main_thm3}, we assume that $(\SS, \ZZ, \VV)$ has a reducible collection $\RR^{\pr}$.
Let $\XX \subset \Ob(\ZZ \cap \EE)$ be a collection containing $\RR^{\pr}$.
We denote $\XX \setminus \RR^{\pr}$ by $\XX_{\RR^{\pr}}$.
	\begin{enumerate}
	\item If $\Sigma \XX_{\RR^{\pr}} \subset \EE$,
	we define $\mu^-_{\RR^{\pr}}(\XX) = \RR^{\pr} \cup \Sigma \XX_{\RR^{\pr}}$,
	which is called a \emph{right $\RR^{\pr}$-mutation} of $\XX$ on $\EE$.
	\item If $\Omega \XX_{\RR^{\pr}} \subset \EE$,
	we define $\mu^+_{\RR^{\pr}}(\XX) = \RR^{\pr} \cup \Omega \XX_{\RR^{\pr}}$ ,
	which is called a \emph{left $\RR^{\pr}$-mutation} of $\XX$ on $\EE$.
	\end{enumerate}
\end{defi}

\begin{ex} \cite{AIR14} 
The following picture is the AR quiver of $\CC = \Db(k A_4)$ 
where $A_4$ is the linearly oriented Dynkin quiver of type $A_4$ (we omitted AR translations).
Let $\II^{\pr} = \{I_1, I_2\}$ which are indicated in the following picture.
An extension closed subcategory $\EE = k A_4 \ast k A_4[1]$ is shaded light gray and
a presilting subcategory $\II = \add \II^{\pr}$ is shaded dark gray.
Then $\II$ induces a rigid mutation triple $(\II, \ZZ, \II)$ with a reducible collection $\II^{\pr}$.
Take $\XX = \{X_1, X_2 \} \cup \II^{\pr}$ which are indicated in the following picture.
Then $\mu^-_{\II^{\pr}} \circ \mu^-_{\II^{\pr}}(\XX) \setminus \II^{\pr}$, 
$\mu^-_{\II^{\pr}}(\XX) \setminus \II^{\pr}$, 
$\mu^+_{\II}(\XX) \setminus \II^{\pr}$ and 
$\mu^+_{\II^{\pr}} \circ \mu^+_{\II^{\pr}}(\XX) \setminus \II^{\pr}$ 
are indicated as follows.
A mutation on $\EE$ is only $\mu^-_{\II^{\pr}}(\XX \cup \II^{\pr})$.
\[
\begin{tikzpicture}
	\path[rounded corners,fill=gray,opacity=0.2] (2.75, -0.5) -- (4.75, 3.5) -- (6.25, 3.5) -- (8.25, -0.5) --cycle;
	\fill[rounded corners,gray,opacity=0.6,rotate around={63.3:(3.39,0)}] (3.14,-0.5) rectangle ++(1.7,0.8);
	\foreach \x in {0,1,...,11}
		\foreach \y in {1,3}
		{
		\fill (\x,\y) circle[radius=0.04cm];
		\fill (\x+0.5,\y-1) circle[radius=0.04cm];
		\draw[-{stealth[scale=3]}] (\x,\y) -- (\x+0.48,\y-0.96);
		\draw[-{stealth[scale=3]}] (\x,1) -- (\x+0.48,1.96);
		\ifnum \x<11
			\draw[-{stealth[scale=3]}] (\x+0.5,\y-1) -- (\x+0.98,\y-0.04);
			\draw[-{stealth[scale=3]}] (\x+0.5,2) -- (\x+0.98,1.04);
		\fi
		}
	\fill [red]
	(1, 1) circle[radius=0.08cm]
	(1.5, 0) circle[radius=0.08cm];
	\fill [red!50!white]
	(3, 3) circle[radius=0.08cm]
	(3.5, 2) circle[radius=0.08cm];
	\fill [black] 
	(3.5,0) circle[radius=0.08cm]
	(4,1) circle[radius=0.08cm]
	(4.5,2) circle[radius=0.08cm]
	(5,3) circle[radius=0.08cm];
	\fill [blue!50!white] 
	(7, 1) circle[radius=0.08cm] 
	(7.5, 0) circle[radius=0.08cm];
	\fill [blue] 
	(9.5,2) circle[radius=0.08cm] 
	(10,3) circle[radius=0.08cm];
	\node [text=red] at (1,1) [right] {\msize{0.7}{\Omega^2 X_1}};
	\node [text=red] at (1.5,0) [right] {\msize{0.7}{\Omega^2 X_2}};
	\node [text=red!50!white] at (3,3) [right] {\msize{0.7}{\Omega X_1}};
	\node [text=red!50!white] at (3.5,2) [right] {\msize{0.7}{\Omega X_2}};
	\node [text=black] at (4.5,2) [right] {\msize{0.7}{X_1}};
	\node [text=black] at (5,3) [right] {\msize{0.7}{X_2}};
	\node [text=blue!50!white] at (7,1) [right] {\msize{0.7}{\Sigma X_1}};
	\node [text=blue!50!white] at (7.5,0) [right] {\msize{0.7}{\Sigma X_2}};
	\node [text=blue] at (9.5,2) [right] {\msize{0.7}{\Sigma^2 X_1}};
	\node [text=blue] at (10,3) [right] {\msize{0.7}{\Sigma^2 X_2}};
	\node (D) at (3.8,0.4) [right] {\msize{0.7}{\II}};
	\node at (3.5,0) [right] {\msize{0.7}{I_1}};
	\node at (4,1) [right] {\msize{0.7}{I_2}};
\end{tikzpicture}
\]
\end{ex}

\appendix
\section{(Pre)triangulated structures induced by concentric twin cotorsion pairs} \label{ccTCP}
In this section, let $(\CC, \bbE, \fraks)$ be an ET category. 
We assume all subcategories are closed under direct summands.

\subsection{Properties of concentric twin cotorsion pairs}
Before we define cotorsion pairs in ET categories, we recall torsion pairs in triangulated categories.

\begin{defi} \cite[Definition 2.2]{IY08}
Assume that $\CC$ is triangulated. A pair of subcategories of $\CC$, $(\UU, \VV)$ is a \emph{torsion pair} of $\CC$ 
if it satisfies the following conditions.
\begin{enumerate}[label=(\roman*)]
	\item $\CC(\UU,\VV)=0$.
	\item $\CC = \UU \ast \VV$.
\end{enumerate}
\end{defi}

\begin{defi}\cite[Definition 2.1]{LN19}
Let $(\UU,\VV)$ be a pair of subcategories of $\CC$. 
$(\UU,\VV)$ is a \emph{cotorsion pair} of $\CC$ if it satisfies the following conditions.
\begin{enumerate}[label=(\roman*)]
	\item $\bbE(\UU,\VV)=0$.
	\item $\CC = \Cone(\VV,\UU)$.
	\item $\CC = \CoCone(\VV,\UU)$.
\end{enumerate}
\end{defi}

\begin{rmk}
\begin{enumerate}
	\item If $\CC$ is triangulated, there exists the following bijection.
	\[
	\{(\UU, \VV) \mid \mathrm{cotorsion \ pair}\} \longleftrightarrow \{(\UU[-1], \VV) \mid \mathrm{torsion\ pair}\}.
	\]
	\item  By definition of cotorsion pair, if $(\UU,\VV)$ is cotorsion pair, then
	\[
	\UU = \{ X \in \CC \mid \bbE(X, \VV)=0 \} \text{ and } \VV = \{X \in \CC \mid \bbE(\UU, X)=0\}.
	\]
	In particular, $\UU,\VV$ are closed under extensions.
	\item Let $(\UU, \VV)$ be a cotorsion pair. Then $\UU$ is strongly contravariantly finite in $\CC$ and $\VV$ is strongly covariantly finite  in $\CC$.
\end{enumerate}
\end{rmk}

We can consider ``resolutions'' of objects in $\CC$ with each cotorsion pair. 
Thus, we can deal with different kinds of resolutions at the same time when there exists a pair of cotorsion pairs. 
Next, we study special pairs of cotorsion pairs.

\begin{defi} \label{def_cctwin cotorsion pair} \cite[Definition 2.3]{LN19} \cite[Definition 3.3]{Nak18}
	\begin{enumerate}
	\item Let $(\SS, \TT)$ and $(\UU, \VV)$ be cotorsion pairs of $\CC$. 
	A pair of cotorsion pair $((\SS, \TT),$ $(\UU, \VV))$ is called a \emph{twin cotorsion pair} 
	if $\SS \subset \UU$.
	\item Let $((\SS, \TT), (\UU, \VV))$ be a twin cotorsion pair. 
	$((\SS, \TT), (\UU, \VV))$ is called \emph{concentric} if $\SS \cap \TT = \UU \cap \VV$.
	In this section, we denote $\SS \cap \TT$ by $\II$ 
	for a concentric twin cotorsion pair $((\SS, \TT), (\UU, \VV))$.
	\end{enumerate}
\end{defi}

\begin{ex} \label{ex_of_twin cotorsion pair}
In this example, we assume that all subcategories are closed under direct summands.
	\begin{enumerate}[listparindent=11pt] 
 	\item \label{ex_of_twin cotorsion pair_same}
	For any cotorsion pair $(\UU, \VV)$, then $((\UU, \VV), (\UU, \VV))$ is a trivial concentric twin cotorsion pair.
	\item Assume that $\CC$ has enough projectives and enough injectives. Then $((\Proj \CC, \CC), $
	$(\CC, \Inj \CC))$ is a twin cotorsion pair, which is concentric if and only if $\CC$ is Frobenius.
	\item \cite{IY08} \label{ex_of_twin cotorsion pair_IY}
	Assume that $\CC$ is a Krull-Schmidt triangulated category. 
	If there exists a strongly functorially finite rigid subcategory $\DD$ in $\CC$
	(so $\DD$ is closed under direct summands), 
	then $((\DD,$ $\rpp{\DD[-1]}),$ $(\lpp{\DD[1}],$ $\DD))$ is a concentric twin cotorsion pair.
	\item \cite{IY18, Jin23} \label{ex_of_twin cotorsion pair_thick}
	Assume that $\CC$ is a Krull-Schmidt triangulated category. 
	If there exists a strongly functorially finite thick subcategory $\NN$ in $\CC$ 
	and a cotorsion pair of $\NN$ $(\SS, \VV)$, 
	then $((\SS, {\SS[-1]}^{\perp_{\CC}}), ({}^{\perp_{\CC}}\VV[1], \VV))$ is a twin cotorsion pair.
\end{enumerate}
\end{ex}

We define some important subcategories.
\begin{defi} \cite[Definition 2.5,2.6]{LN19} \label{defi_plus-minus_ccTCP}
Let $((\SS, \TT), (\UU, \VV))$ be a twin cotorsion pair in $\CC$. 
	\begin{enumerate}
	\item $\ZZ = \TT \cap \UU$ is called a \emph{core} of the twin cotorsion pair.
	\item $\CC^{+} = \Cone(\VV, \ZZ)$.
	\item $\CC^{-} = \CoCone(\ZZ, \SS)$.
	\end{enumerate}
\end{defi}

\begin{ex}
\begin{enumerate}
	\item Let $A$ be a ring. 
	The derived category of $A$-modules $\sfD(A)$ 
	has the standard $t$-structure $(\sfD^{\leq0}(A), \sfD^{> 0}(A) )$. 
	
	This induces a cotorsion pair $(\sfD^{<0}(A), \sfD^{> 0}(A) )$. 
	We regard this cotorsion pair as a twin cotorsion pair like Example \ref{ex_of_twin cotorsion pair}(\ref{ex_of_twin cotorsion pair_same}).
	\begin{enumerate}[leftmargin=18pt]
		\item $\ZZ = \sfD^{> 0}(A) \cap \sfD^{< 0}(A) = 0$
		\item $\CC^{+} = \Cone(\sfD^{> 0}(A), 0) = \sfD^{\geq 0}(A)$
		\item $\CC^{-} = \CoCone(0, \sfD^{<0}(A)) = \sfD^{\leq 0}(A)$ (aisle of the $t$-structure)
	\end{enumerate}
	\item Let $A$ be a ring. The perfect derived category of $A$ in $\sf{D}(A)$, denoted by $\per(A)$, 
	has the standard co-$t$-structure $(\per_{\geq 0}(A), \per_{<0}(A))$. 
	
	This induces a cotorsion pair $(\per_{\geq 0}(A), \per_{\leq 0}(A))$. 
	We regard this cotorsion pair as a twin cotorsion pair like Example \ref{ex_of_twin cotorsion pair}(\ref{ex_of_twin cotorsion pair_same}).
	\begin{enumerate}[leftmargin=18pt]
		\item $\ZZ = \per_{\geq 0}(A) \cap \per_{\leq 0}(A) = \add A$ 
		(co-heart of the co-$t$-structure)
		\item $\CC^{+} = \Cone(\per_{\leq 0}(A), \add A) = \per_{\leq 0}(A)$ 
		\item $\CC^{-} = \CoCone(\add A, \per_{\geq 0}(A)) = \per_{\geq 0}(A)$ 
		(aisle of the co-$t$-structure)
	\end{enumerate}
\end{enumerate}
\end{ex}

In the rest of this section, we denote $\XX/[\II]$ by $\ul{\XX}$ for a subcategory $\XX$ containing $\II$.

\begin{lem} \label{lem_ccTCP_functors}
Let $((\SS, \TT), (\UU, \VV))$ be a concentric twin cotorsion pair.
	\begin{enumerate}
	\item \cite[Definition 4.1]{Nak18}
	$\II$ is strongly functorially finite in $\ZZ$. 
	Then we may define the following functors from Lemma \ref{def_of_bracket}.
		\[
		\bracket{1} \colon \ul{\ZZ} \to \ul{\UU}, \ \bracket{-1} \colon \ul{\ZZ} \to \ul{\TT}.
		\]
	Moreover, there exists the following bifunctorial isomorphism for $Z_1, Z_2 \in \ZZ$.
	\[
	\ul{\UU}(Z_1\bracket{1}, Z_2) \xrar{\sim} \ul{\TT}(Z_1, Z_2\bracket{-1})
	\]
	\item \cite[Corollary 4.11]{NP19} \label{lem_ccTCP_functors_2}
	Let $C \in \CC$.
		\begin{enumerate}
		\item \label{lem_ccTCP_functors_2-1}
		Take an $\fraks$-triangle (this is also an $\fraks_{\II}$-triangle) $V_C \rar U_C \xrar{u_C} C \drar V_C$ 
		where $U_C \in \UU, V_C \in \VV$.
		Then $\omega_{\msize{0.6}{\UU}} \colon \ul{\CC} \to \ul{\UU}$ defined by $C \mapsto U_C$ 
		is a right adjoint functor of the inclusion functor $\ul{\UU} \to \ul{\CC}$ and 
		$\varepsilon_{\msize{0.6}{\UU}} = \{u_C\}_{C \in \Ob(\ul{\CC})}$ gives the counit of this adjoint pair.
		\item \label{lem_ccTCP_functors_2-2}
		Take an $\fraks$-triangle (this is also an $\fraks^{\II}$-triangle) $S^C \drar C \xrar{t^C} T^C \rar S^C$ 
		where $S^C \in \SS, T^C \in \TT$.
		Then $\sigma_{\msize{0.6}{\TT}} \colon \ul{\CC} \to \ul{\TT}$ defined by $C \mapsto T^C$ 
		is a left adjoint functor of the inclusion functor $\ul{\TT} \to \ul{\CC}$ and 
		$\eta_{\msize{0.6}{\hspace{0.5pt}\TT}} = \{t^C\}_{C \in \Ob(\ul{\CC})}$ gives the unit of this adjoint pair.
		\end{enumerate}
	\item \cite[Remark 4.13]{NP19}
		\begin{enumerate}
		\item In (\ref{lem_ccTCP_functors_2-1}), 
		the restriction of $\omega_{\msize{0.6}{\UU}}$ gives 
		a right adjoint functor $\omega \colon \ul{\TT} \to \ul{\ZZ}$ of the inclusion functor $\ul{\ZZ} \to \ul{\TT}$.
		\item In (\ref{lem_ccTCP_functors_2-2}), 
		the restriction of $\sigma_{\msize{0.6}{\TT}}$ gives 
		a left adjoint functor $\sigma \colon \ul{\UU} \to \ul{\ZZ}$ of the inclusion functor $\ul{\ZZ} \to \ul{\UU}$.
		\end{enumerate}
	\end{enumerate}
\end{lem}

In section \ref{pretri}, 
We check that $\sigma$ and $\omega$ in Definition \ref{defi_sigma_omega_apdx} and
those in Lemma \ref{lem_ccTCP_functors} are examples of those in Lemma \ref{adjoint_pair}.

Now, we obtain two endofunctors of $\ul{\ZZ}$, which are also examples of right mutation functors and left mutation functors in Definition \ref{def_mutations}.

\begin{defi} \cite[Definition 4.2]{Nak18} \label{defi_sigma-omega_CTP}
Let $((\SS, \TT), (\UU, \VV))$ be a concentric twin cotorsion pair.
We define $\Sigma = \sigma \circ \bracket{1}$ and $\Omega = \omega \circ \bracket{-1}$, called 
a \emph{right mutation functor} and a \emph{left mutation functor}, respectively.
\end{defi}

\begin{rmk}
In \cite{Nak18}, $\Sigma$ and $\Omega$ do not have specific names.
\end{rmk}

Before we introduce examples of concentric twin cotorsion pairs, 
we define some special subcategories and collections.

\begin{defi} \cite[Section 3]{IY08} \cite[Section 2.3]{IY18} \label{defi_rigid}
Let $\II$ be a subcategory of $\CC$.
	\begin{enumerate}
	\item For $n > 1$, $\II$ is \emph{$n$-rigid} if $\CC(\II, \II[i]) =0$ for $0<i<n$.
	\item $\II$ is \emph{presilting}, or \emph{$\infty$-rigid} if $\CC(\II, \II[i]) =0$ for $i>0$.
	\end{enumerate}
\end{defi}

\begin{rmk} \cite[Section 3]{IY08}
We often use \emph{rigid} instead of 2-rigid.
\end{rmk}

\begin{ex} \cite{AI12, BMRRT06}
	\begin{enumerate}
	\item Let $A$ be a finite dimensional $k$-algebra and $\CC = \Db(A)$. Then $\add A$ is a presilting  subcategory (in this case, this subcategory is also a silting subcategory \cite{AI12, IY18}.)
	\item Let $Q$ be a Dynkin quiver and $H = kQ$. Let $\CC$ be the cluster category of $H$ (for details, see \cite{BMRRT06, IY08}).  Then $\add H$ is a rigid subcategory of $\CC$ (in this case, this subcategory is also a cluster-tilting subcategory).
	\end{enumerate}
\end{ex}

\begin{defi}\cite[Definition 2.1]{SP20} \cite[Definition 1.1]{Asa20} \cite[Definition 3.2]{KY14} \label{defi_orthogonal}

Let $\MM = \{M_i \mid i \in I\} \subset \Ob(\CC)$ where $M_i \neq M_j$ for $i \neq j$, and $n \geq 1$.
	\begin{enumerate}
	\item
	$
	\MM \text{ is a \emph{semibrick} if } 
	\CC(M_i, M_j) = 
	\left\{
		\begin{array}{ll}
		\mathrm{division \ ring} & (i = j)\\
		0 & (\mathrm{else})
		\end{array} \quad (M_i, M_j \in \MM)
	\right.
	$
	\item $\MM$ is \emph{$n$-orthogonal} if $\MM$ is a semibrick and $\CC(\MM[i], \MM) =0$ for $0<i<n$.
	\item $\MM$ is a \emph{pre-simple-minded collection}, or \emph{$\infty$-orthogonal} if
	$\MM$ is $n$-orthogonal for $n \geq 1$.
	\end{enumerate}
\end{defi}

\begin{rmk}
We often use \emph{orthogonal} instead of 2-orthogonal. 
Note that ``orthogonal'' means $1$-orthogonal in \cite{SP20}.
\end{rmk}

\begin{ex}
	\begin{enumerate}
	\item \cite{KY14} 
	Let $A$ be a finite dimensional $k$-algebra. Then $\MM = \{$simple $A$-modules$\}$ is a pre-simple-minded collection in $\Db(A)$ (in this case, this is also a simple-minded collection.)
	\item \cite{Jin20, Jin23} 
	Let $H$ be a path algebra of Dynkin diagram and $A = H \oplus DH[1-d]$ be the trivial extension dg algebra for $d \geq 2$. Let $\MM$ be a complete system of representatives of simple $H$-modules with respect to isomorphism class. 
	Then $\MM$ is $d$-orthogonal (in this case, this is also a $d$-simple-minded system) in $\mathsf{D}_{\mathrm{sg}}(A)$.
	Since there exists a triangle equivalence $\Db(H)/\nu[d] \to \mathsf{D}_{\mathrm{sg}}(A)$ induced by 
	the canonical morphism $\Db(H) \to \Db(A)$,
	 $\MM$ is $d$-orthogonal (in this case, this is also a $d$-simple-minded system) in $(-d)$-cluster category $\Db(H) / \nu [d]$ where $\nu$ is a Nakayama functor of $\Db(H)$.
	\end{enumerate}
\end{ex}

\begin{nota} \label{nota_thick}
	\begin{enumerate}
	\item For an $\infty$-rigid subcategory $\II$, we use the following notation.
		\begin{align*}
		(\thick \II)_{\geq i} = \bcup{l \geq i} \, \II[-l] \ast \cdots \ast \II[-i-1] \ast \II[-i] \\
		(\thick \II)_{\leq i} = \bcup{l \geq -i} \II[-i] \ast \II[1-i] \ast \cdots \ast \II[l]
		\end{align*}
	\item For an $\infty$-orthogonal collection $\MM$, we use the following notation.
		\begin{align*}
		(\thick \MM)^{\leq i} = \bcup{l \leq i} \, \bracket{\MM[-l]} \ast \cdots \ast \bracket{\MM[1-i]} \ast \bracket{\MM[-i]} \\
		(\thick \MM)^{\geq i} = \bcup{l \leq -i} \bracket{\MM[-i]} \ast \bracket{\MM[-i-1]} \ast \cdots \ast \bracket{\MM[l]}
		\end{align*}
	\end{enumerate}
\end{nota}

\begin{ex} \label{ex_HoveyTCP_ARquiv}
Let $H$ be the path algebra of  the linearly oriented Dynkin quiver of type $A_4$ over $k$ and
$\CC = \Db(H)$.
\begin{enumerate}
\item \label{ex_HoveyTCP_ARquiv_1}
Let $\II$ be the presilting subcategory which is shaded both red and blue.
We define the following subcategories of $\CC$.
\begin{gather*}
(\thick \II)_{\geq 0} \ (\text{shaded blue}), \ (\thick \II)_{\leq 0}  \ (\text{shaded red})\\
\rpp{\II[<\!0]} = \bcap{i>0}\, \rpp{\II[-i]} \ (\text{outlined in red}), \
\lpp{\II[>\!0]} = \bcap{i>0} \lpp{\II[i]} \ (\text{outlined in blue})
\end{gather*}
Then $\big(((\thick \II)_{\geq 0}, \rpp{\II[<\!\!0]}), (\lpp{\II[>\!\!0]}, (\thick \II)_{\leq 0}) \big)$ is a concentric twin cotorsion pair from \cite[Proposition 3.2]{IY18}.
$\Omega^2 X, \Omega X, X, \Sigma X$ and $\Sigma^2 X$ are indicated in the following picture.
\[
\begin{tikzpicture}
	\path[rounded corners,fill=red,opacity=0.3] (4.08,-0.2) -- (5, 1.64) -- (5.92, -0.2) --cycle;
	\path[rounded corners,fill=red,opacity=0.3] (6.58,3.2) -- (7.5, 1.36) -- (8.42, 3.2) --cycle;
	\path[rounded corners,fill=red,opacity=0.3] (9.08,-0.2) -- (10, 1.64) -- (10.92, -0.2) --cycle;
	\path[rounded corners,fill=blue,opacity=0.3] (1.58,3.2) -- (2.5, 1.36) -- (3.42, 3.2) --cycle;
	\fill[rounded corners=2mm,blue,fill=blue,opacity=0.3,rotate around={63.44:(4.5,0)}] (4.3,-0.2) rectangle ++(1.55,0.4);
	\fill[rounded corners=2mm,blue,fill=blue,opacity=0.3,rotate around={116.56:(0.5,0)}] (0.3,-0.2) rectangle ++(1.55,0.4);
	\draw[red,thick] (0,3) circle[radius=2mm];
	\draw[rounded corners,red,thick] (1.08,-0.2) -- (2, 1.64) -- (2.92, -0.2) --cycle;
	\draw[rounded corners,red,thick] (3.5,3.3) -- (4.65,1) -- (3.95, -0.4) -- (11.85,-0.4) -- (11.85,3.3) --cycle;
	\draw[rounded corners,blue,thick] (7.08,-0.2) -- (8, 1.64) -- (8.92, -0.2) --cycle;
	\draw[rounded corners,blue,thick] (9.58,3.2) -- (10.5, 1.36) -- (11.42, 3.2) --cycle;
	\draw[rounded corners,blue,thick] (-0.35,3.4) -- (6.55, 3.4) -- (4.8, -0.3) -- (-0.35,-0.3) --cycle;
	\foreach \x in {0,1,2,3,4,5,6,9,11}
		\foreach \y in {1,3}
		{
		\fill (\x,\y) circle[radius=0.04cm];
		\fill (\x+0.5,\y-1) circle[radius=0.04cm];
		\draw[-{stealth[scale=3]}] (\x,\y) -- (\x+0.48,\y-0.96);
		\draw[-{stealth[scale=3]}] (\x,1) -- (\x+0.48,1.96);
		\ifnum \x<11
			\draw[-{stealth[scale=3]}] (\x+0.5,\y-1) -- (\x+0.98,\y-0.04);
			\draw[-{stealth[scale=3]}] (\x+0.5,2) -- (\x+0.98,1.04);
		\fi
		}
	\foreach \x in {7,8,10}
		\foreach \y in {1,3}
		{
		\fill (\x,\y) circle[radius=0.04cm];
		\fill (\x+0.5,\y-1) circle[radius=0.04cm];
		}
		\draw[-{stealth[scale=3]}] (7,3) -- (7+0.48,2.04);
		\draw[-{stealth[scale=3]}] (7,1) -- (7+0.48,0.04);
		\draw[-{stealth[scale=3]}] (7,1) -- (7+0.48,1.96);
		\draw[-{stealth[scale=3]}] (7+0.5,2) -- (7+0.98,2.96);
		\draw[-{stealth[scale=3]}] (7+0.5,2) -- (7+0.98,1.04);
		\draw (7+0.5,0) -- (7+0.85,0.7);%
		\draw[-{stealth[scale=3]}] (8,3) -- (8+0.48,2.04);
		\draw[-{stealth[scale=3]}] (8+0.15,0.7) -- (8+0.48,0.04);%
		\draw[-{stealth[scale=3]}] (8,1) -- (8+0.48,1.96);
		\draw[-{stealth[scale=3]}] (8+0.5,2) -- (8+0.98,2.96);
		\draw[-{stealth[scale=3]}] (8+0.5,2) -- (8+0.98,1.04);
		\draw[-{stealth[scale=3]}] (8+0.5,0) -- (8+0.98,0.96);
		\draw (10,3) -- (10+0.35,2.3);%
		\draw[-{stealth[scale=3]}] (10,1) -- (10+0.48,0.04);
		\draw[-{stealth[scale=3]}] (10,1) -- (10+0.48,1.96);
		\draw[-{stealth[scale=3]}] (10+0.65,2.3) -- (10+0.98,2.96);%
		\draw[-{stealth[scale=3]}] (10+0.5,2) -- (10+0.98,1.04);
		\draw[-{stealth[scale=3]}] (10+0.5,0) -- (10+0.98,0.96);
	\fill (1.5, 0) circle[radius=0.08cm] 
	(4, 3) circle[radius=0.08cm] 
	(5.5, 2) circle[radius=0.08cm] 
	(8, 1) circle[radius=0.08cm]
	(10.5, 2) circle[radius=0.08cm]
	(3, 3) circle[radius=0.08cm]
	(4.5,0) circle[radius=0.08cm]
	(5,1) circle[radius=0.08cm]
	;
	\node (X_-2) at (1.5,0) [right] {\msize{0.6}{\Omega^2 X}};
	\node (X_-1) at (4,3) [right] {\msize{0.6}{\Omega X}};
	\node (X_0) at (5.5,2) [right] {\msize{0.6}{X}};
	\node (X_1) at (8,1) [below] {\msize{0.6}{\Sigma X}};
	\node (X_2) at (10.5,2) [above] {\msize{0.6}{\Sigma^2 X}};
	\node (Y) at (3,3) [right] {\msize{0.6}{Y}};
	\node (I_1) at (4.5,0) [right] {\msize{0.6}{I_1}};
	\node (I_2) at (5,1) [right] {\msize{0.6}{I_2}};
	\draw [-{stealth[scale=3]}, very thick] (3,3) -- (4.48,0.04);
	\draw [-{stealth[scale=3]}, very thick] (4.5,0) -- (4.98,0.96);
	\node (f) at (3.75,1.5) [right] {\msize{0.6}{f}};
	\node (f) at (4.75,0.5) [right] {\msize{0.6}{g}};
\end{tikzpicture}
\]
\item \label{ex_HoveyTCP_ARquiv_2}
Let $\MM$ be the pre-simple-minded collection which is marked gray and $\bracket{\MM}$ is shaded gray.
We define the following subcategories of $\CC$.
\begin{gather*}
(\thick \MM)^{<0} \ (\text{shaded blue}),\ (\thick \MM)^{>0} \ (\text{shaded red})\\
\rpp{\MM[\geq\!0]} = \bcap{i \geq 0} \, \rpp{\MM[i]} \ (\text{outlined in red}), \
\lpp{\MM[\leq\!0]} = \bcap{i \geq 0} \lpp{\MM[-i]} \ (\text{outlined in blue})
\end{gather*}
Then $\big(((\thick \MM)^{<0}, \rpp{\MM[\geq\!\!0]}), (\lpp{\MM[\leq\!\!0]}, (\thick \MM)^{>0}) \big)$ is a concentric twin cotorsion pair from \cite[Proposition 3.2]{Jin23}.
$\Omega X, X, X[1], \Sigma X, \Sigma^2 X$ and $\Sigma^3 X$ are indicated in the following picture.
\[
\begin{tikzpicture}
	\path[rounded corners,fill=gray,opacity=0.3] (4.08,-0.2) -- (5, 1.64) -- (5.92, -0.2) --cycle;
	\path[rounded corners,fill=blue,opacity=0.3] (6.58,3.2) -- (7.5, 1.36) -- (8.42, 3.2) --cycle;
	\path[rounded corners,fill=blue,opacity=0.3] (9.08,-0.2) -- (10, 1.64) -- (10.92, -0.2) --cycle;
	\fill[rounded corners=2mm,red,opacity=0.3,rotate around={116.56:(0.5,0)}] (0.3,-0.2) rectangle ++(1.55,0.4);
	\path[rounded corners,fill=red,opacity=0.3] (1.58,3.2) -- (2.5, 1.36) -- (3.42, 3.2) --cycle;
	\draw[rounded corners,red,thick] (6.08,-0.2) -- (7, 1.64) -- (7.92, -0.2) --cycle;
	\draw[rounded corners,red,thick] (8.58,3.2) -- (9.5, 1.36) -- (10.42, 3.2) --cycle;
	\draw[rounded corners,red,thick] (-0.4,3.4) -- (5.6, 3.4) -- (3.85, -0.3) -- (-0.4,-0.3) --cycle;
	\draw[red,thick] (11.5,0) circle[radius=1.2mm];
	\draw[rounded corners,blue,thick] (2.08,-0.2) -- (3, 1.64) -- (3.92, -0.2) --cycle;
	\draw[rounded corners,blue,thick] (-0.42,3.2) -- (0.5, 1.36) -- (1.42, 3.2) --cycle;
	\draw[rounded corners,blue,thick] (4.4,3.3) -- (6.25, -0.4) -- (11.85,-0.4) -- (11.85,3.3) --cycle;
	\foreach \x in {1,4,5,6,7,8,9,10,11}
		\foreach \y in {1,3}
		{
		\fill (\x,\y) circle[radius=0.04cm];
		\fill (\x+0.5,\y-1) circle[radius=0.04cm];
		\draw[-{stealth[scale=3]}] (\x,\y) -- (\x+0.48,\y-0.96);
		\draw[-{stealth[scale=3]}] (\x,1) -- (\x+0.48,1.96);
		\ifnum \x<11
			\draw[-{stealth[scale=3]}] (\x+0.5,\y-1) -- (\x+0.98,\y-0.04);
			\draw[-{stealth[scale=3]}] (\x+0.5,2) -- (\x+0.98,1.04);
		\fi
		}
		\foreach \x in {0,2,3}
		\foreach \y in {1,3}
		{
		\fill (\x,\y) circle[radius=0.04cm];
		\fill (\x+0.5,\y-1) circle[radius=0.04cm];
		}
		\draw[-{stealth[scale=3]}] (2,3) -- (2+0.48,2.04);
		\draw[-{stealth[scale=3]}] (2,1) -- (2+0.48,0.04);
		\draw[-{stealth[scale=3]}] (2,1) -- (2+0.48,1.96);
		\draw[-{stealth[scale=3]}] (2+0.5,2) -- (2+0.98,2.96);
		\draw[-{stealth[scale=3]}] (2+0.5,2) -- (2+0.98,1.04);
		\draw (2+0.5,0) -- (2+0.83,0.66);%
		\draw[-{stealth[scale=3]}] (3,3) -- (3+0.48,2.04);
		\draw[-{stealth[scale=3]}] (3+0.17,0.66) -- (3+0.48,0.04);%
		\draw[-{stealth[scale=3]}] (3,1) -- (3+0.48,1.96);
		\draw[-{stealth[scale=3]}] (3+0.5,2) -- (3+0.98,2.96);
		\draw[-{stealth[scale=3]}] (3+0.5,2) -- (3+0.98,1.04);
		\draw[-{stealth[scale=3]}] (3+0.5,0) -- (3+0.98,0.96);
		\draw (0,3) -- (0+0.33,2.34);%
		\draw[-{stealth[scale=3]}] (0,1) -- (0+0.48,0.04);
		\draw[-{stealth[scale=3]}] (0,1) -- (0+0.48,1.96);
		\draw[-{stealth[scale=3]}] (0+0.67,2.34) -- (0+0.98,2.96);%
		\draw[-{stealth[scale=3]}] (0+0.5,2) -- (0+0.98,1.04);
		\draw[-{stealth[scale=3]}] (0+0.5,0) -- (0+0.98,0.96);
	\fill (0.5, 2) circle[radius=0.08cm] 
	(3, 1) circle[radius=0.08cm] 
	(5.5, 2) circle[radius=0.08cm] 
	(6.5, 0) circle[radius=0.08cm] 
	(9, 3) circle[radius=0.08cm]
	(11.5, 0) circle[radius=0.08cm]
	;
	\draw [fill=gray] (4.5, 0) circle[radius=0.08cm] (5.5, 0) circle[radius=0.08cm];
	\node (X_-1) at (0.5,2) [above] {\msize{0.7}{\Omega X}};
	\node (X_0) at (3,1) [below] {\msize{0.7}{X}};
	\node (X_0.5) at (5.5,2) [right] {\msize{0.7}{X[1]}};
	\node (X_1) at (6.5,0) [right] {\msize{0.7}{\Sigma X}};
	\node (X_2) at (9,3) [right] {\msize{0.7}{\Sigma^2 X}};
	\node (X_3) at (11.5,0) [below] {\msize{0.7}{\Sigma^3 X}};
\end{tikzpicture}
\]
\end{enumerate}
\end{ex}

\begin{defi} \cite[Definition 4.4]{Nak18}
Let $((\SS, \TT), (\UU, \VV))$ be a concentric twin cotorsion pair. Let $A \xrar{x} B \xrar{y} C \lxdrar{\delta} A$ be an $\fraks$-triangle.
	\begin{enumerate}
	\item $A \xrar{x} B \xrar{y} C \lxdrar{\delta} A$ is called \emph{$\UU$-conic} 
	if $A,B \in \ZZ$ and $C \in \UU$.
	\item $A \xrar{x} B \xrar{y} C \lxdrar{\delta} A$ is called \emph{$\TT$-coconic} 
	if $B,C \in \ZZ$ and $A \in \TT$.
	\end{enumerate}
Note that $\UU$-conic is defined in triangulated categories in \cite[Definition 4.4]{Nak18}.
\end{defi}

The following statement is used in Definition \ref{def_right-left_tri} and Example \ref{ex_MT}.

\begin{lem} \label{conic}
Let $((\SS, \TT), (\UU, \VV))$ be a concentric twin cotorsion pair. 
Let $A \xrar{x} B \xrar{y} C \lxdrar{\delta} A$ be an $\fraks$-triangle.
	\begin{enumerate}
	\item Assume $A,B \in \ZZ$, then
	\[
	A \xrar{x} B \xrar{y} C \lxdrar{\delta} A \text{ is } \UU \text{-conic.} \iff x \text{ is } \II \text{-monic.}
	\]
	\item Assume $B,C \in \ZZ$, then
	\[
	A \xrar{x} B \xrar{y} C \lxdrar{\delta} A \text{ is } \TT \text{-coconic.} \iff y \text{ is } \II \text{-epic.}
	\]
	\end{enumerate}
\end{lem}
\begin{proof}
We only prove $(1)$.
Assume that $A \xrar{x} B \xrar{y} C \lxdrar{\delta} A$ is $\UU$-conic. 
Then $\bbE(C, \II) = 0$, so $x$ is $\II$-monic. 
On the other hand, assume that $x$ is $\II$-monic. 
From $\bbE(\ZZ, \II) = 0$ and the long exact sequence, $\bbE(C, \II) =0$.
Take a conflation $V \rar U \rar C \lxdrar{\delta^{\pr\pr}} V$ where $U \in \UU$ and $V \in \VV$.
Then we obtain the following diagram.
	\[
	\xy
	(0,8)*+{A}="01";
	(0,-8)*+{A}="02";
	(16,24)*+{V}="10";
	(16,8)*+{U^{\pr}}="11";
	(16,-8)*+{B}="12";
	(16,-24)*+{V}="13";
	(32,24)*+{V}="20";
	(32,8)*+{U}="21";
	(32,-8)*+{C}="22";
	(32,-24)*+{V}="23";
	(48,8)*+{A}="31";
	(48,-8)*+{A}="32";
	{\ar^{} "01";"11"};
	{\ar^{} "11";"21"};
	{\ar^{} "11";"12"};
	{\ar@{-->}^{\delta^{\pr}} "12";"13"};
	{\ar^{x} "02";"12"};
	{\ar^{y} "12";"22"};
	{\ar^{} "21";"22"};
	{\ar@{-->}^{\delta^{\pr\pr}} "22";"23"};
	{\ar@{=} "01";"02"};
	{\ar@{=} "10";"20"};
	{\ar@{=} "13";"23"};
	{\ar@{=} "31";"32"};
	{\ar@{->}^{} "10";"11"};
	{\ar@{->}^{} "20";"21"};
	{\ar@{-->}^{} "21";"31"};
	{\ar@{-->}^{\delta} "22";"32"};
	{\ar@{}|\car "01";"12"};
	{\ar@{}|\car "10";"21"};
	{\ar@{}|\car "11";"22"};
	{\ar@{}|\car "12";"23"};
	{\ar@{}|\car "21";"32"};
	\endxy
	\]
Note that $U^{\pr} \in \UU$ since $A, U \in \UU$ and $\UU$ is closed under extensions.
From $\delta^{\pr} \in \bbE(\ZZ, \VV) = 0$, $V \in \UU \cap \VV = \II$. Thus, $\delta^{\pr\pr} \in \bbE(C, \II) = 0$.
Therefore, $C \in \UU$.
\end{proof}

Note that $U$ \resp{$T$} in the following definition are contained in $\UU$ \resp{$\TT$} 
from Lemma \ref{inflation} and \ref{conic}.

\begin{defi} \cite[Definition 4.5]{Nak18} \label{def_right-left_tri}
	\begin{enumerate}
	\item Let $f \colon X \to Y$ be a morphism in $\ZZ$.
	Then there exists the following commutative diagram in $\CC$.
	\[
	\xy
	(0,8)*+{X}="21";
	(16,8)*+{I^X}="22";
	(32,8)*+{X\bracket{1}}="23";
	(48,8)*+{X}="24";
	(0,-8)*+{Y}="31";
	(16,-8)*+{U}="32";
	(32,-8)*+{X\bracket{1}}="33";
	(48,-8)*+{Y}="34";
	{\ar^{i^X} "21";"22"};
	{\ar^{p^X} "22";"23"};
	{\ar@{-->}^-{\lambda^X} "23";"24"};
	{\ar^{g} "31";"32"};
	{\ar^{h} "32";"33"};
	{\ar@{-->}^-{\delta} "33";"34"};
	{\ar^{f} "21";"31"};
	{\ar^{} "22";"32"};
	{\ar@{=}^{} "23";"33"};
	{\ar^{f} "24";"34"};
	{\ar@{}|\car "21";"32"};
	{\ar@{}|\car "22";"33"};
	{\ar@{}|\car "23";"34"};
	\endxy
	\]
	Then we define the \emph{standard right triangle} of $f$ as 
	$X \xrar{\ul{f}} Y \xrar{\ul{t^U \circ g}} \sigma U \xrar{\sigma(\ul{h})} \Sigma X$ in $\ul{\ZZ}$ and
	\[
	\nabla = 
	\left(
	\begin{array}{ll}
	\text{sequences} &\text{in } \ul{\ZZ} \text{ isomorphic to one in } \\
	&\{ X \xrar{\ul{f}} Y \xrar{\ul{t^U \circ g}} \sigma U \xrar{\sigma(\ul{h})} \Sigma X \mid 
	f \text{ is a morphism in } \ZZ \}
	\end{array}
	\right).
	\]
	\item Let $h^{\pr} \colon Y \to Z$ be a morphism in $\ZZ$.
	Then there exists the following commutative diagram in $\CC$.
	\[
	\xy
	(0,8)*+{Y}="21";
	(16,8)*+{Z\bracket{-1}}="22";
	(32,8)*+{T}="23";
	(48,8)*+{Y}="24";
	(0,-8)*+{Z}="31";
	(16,-8)*+{Z\bracket{-1}}="32";
	(32,-8)*+{I_Z}="33";
	(48,-8)*+{Z}="34";
	{\ar@{-->}^-{\delta^{\pr}} "21";"22"};
	{\ar^{f^{\pr}} "22";"23"};
	{\ar^-{g^{\pr}} "23";"24"};
	{\ar@{-->}^-{\lambda_Z} "31";"32"};
	{\ar^{i_Z} "32";"33"};
	{\ar^-{p_Z} "33";"34"};
	{\ar^{h^{\pr}} "21";"31"};
	{\ar@{=}^{} "22";"32"};
	{\ar^{} "23";"33"};
	{\ar^{h^{\pr}} "24";"34"};
	{\ar@{}|\car "21";"32"};
	{\ar@{}|\car "22";"33"};
	{\ar@{}|\car "23";"34"};
	\endxy
	\]
	Then we define the \emph{standard left triangle} of $h^{\pr}$ as 
	$\Omega Z \xrar{\omega(\ul{f^{\pr}})} \omega T \xrar{\ul{g^{\pr} \circ u_T}} Y \xrar{\ul{h^{\pr}}} Z$ in $\ul{\ZZ}$ and
	\[
	\Delta = 
	\left(
	\begin{array}{ll}
	\text{sequences} &\text{in } \ul{\ZZ} \text{ isomorphic to one in } \\
	&\{ \Omega Z \xrar{\omega(\ul{f^{\pr}})} \omega T \xrar{\ul{g^{\pr} \circ u_T}} Y \xrar{\ul{h^{\pr}}} Z \mid 
	h^{\pr} \text{ is a morphism in } \ZZ \}
	\end{array}
	\right).
	\]
	\end{enumerate}
\end{defi}

\begin{thm} \cite[Theorem 4.15]{Nak18} \label{thm_pretri-ccTCP}
Let $((\SS, \TT), (\UU, \VV))$ be a concentric twin cotorsion pair.
Then $(\ul{\ZZ}, \Sigma, \Omega, \nabla, \Delta)$ is a pretriangulated category in Definition \ref{defi_pretri-in-BR}.
\end{thm}

\begin{proof}
If $\CC$ is triangulated, see \cite[Theorem 4.15]{Nak18}.
One can similarly show this statement even if $\CC$ is an ET category 
(or we prove this statement in Theorem \ref{main_thm1} under more general situation.)
\end{proof}

\subsection{Twin cotorsion pairs which induces triangulated categories}
\label{Twin cotorsion pairs which induces triangulated categories}

Next, we introduce two types of twin cotorsion pairs which induce triangulated categories.
\begin{defi} \label{defi_HEandHo}
\cite[Condition 6.1, Definition 6.6]{Nak18}
\cite[Definition 3.10, Proposition 3.12]{LN19}
Let $((\SS, \TT), (\UU, \VV))$ be a concentric twin cotorsion pair. 
We denote $\Cone(\VV, \SS)$ by $\NN^i$ and $\CoCone(\VV, \SS)$ by $\NN^f$.

We consider the following conditions.
\begin{itemize}[leftmargin=40pt]
\item[(Ho)] $\NN^i = \NN^f$.
\item[(HE)] $\add (\SS \ast \TT) = \add (\UU \ast \VV)$
\end{itemize}

\begin{enumerate}
\item $((\SS, \TT), (\UU, \VV))$ is called \emph{Hovey} if it satisfies the condition (Ho).
\item $((\SS, \TT), (\UU, \VV))$ is called \emph{heart-equivalent} if it satisfies the condition (HE).
\end{enumerate}
\end{defi}

\begin{ex} \label{ex_HoveyTCP}
Assume that $\CC$ is a Krull-Schmidt triangulated category.
	\begin{enumerate}
	\item \cite{IY18} \label{ex_HoveyTCP-presilt}
	Assume that $\CC$ has a silting subcategory $\XX$.
	Let $\II$ be a presilting subcategory which is functorially finite  in $\CC$. 
	Then $((\SS,\TT),(\UU,\VV))$ defined as follows is Hovey.
		\begin{enumerate}[leftmargin=18pt]
			\item $\SS = \bcup{i \geq 0}\, \II[-i] \ast \cdots \ast \II$,
			$\TT = \bcap{i>0}\, \rpp{\II[-i]}$.
			\item $\UU= \bcap{i>0}\, \lpp{\II[i]}$,
			$\VV = \bcup{i \geq 0}\, \II \ast \cdots \ast \II[i]$.
		\end{enumerate}
	In particular, $\big(((\thick \II)_{\geq 0}, \rpp{\II[<\!0]}), (\lpp{\II[>\!0]}, (\thick \II)_{\leq 0}) \big)$ 
	in Example \ref{ex_HoveyTCP_ARquiv}(\ref{ex_HoveyTCP_ARquiv_1}) is Hovey.
	\item \cite{Jin23} \label{ex_HoveyTCP-preSMC}
	Assume that $\CC$ has a simple-minded collection $\XX$ and let $\MM$ be a subcollection of $\MM$.
	Assume that $\bracket{\MM}$ is functorially finite  in $\CC$, 
	then $((\SS,\TT),(\UU,\VV))$ defined as follows is Hovey.
		\begin{enumerate}[leftmargin=18pt]
			\item $\SS = \bcup{i > 0}\, \bracket{\MM[i]} \ast \cdots \ast \bracket{\MM[1]}$,
			$\TT = \bcap{i \geq 0}\, \rpp{\MM[i]}$.
			\item $\UU = \bcap{i \geq 0}\, \lpp{\MM[-i]}$,
			$\VV = \bcup{i > 0}\, \bracket{\MM[-1]} \ast \cdots \ast \bracket{\MM[-i]}$.
		\end{enumerate}
	In particular, $\big(((\thick \MM)^{<0}, \rpp{\MM[\geq\!0]}), (\lpp{\MM[\leq\!0]}, (\thick \MM)^{>0}) \big)$ 
	in Example \ref{ex_HoveyTCP_ARquiv}(\ref{ex_HoveyTCP_ARquiv_2}) is Hovey.
	\item \cite{IY08}
	We additionally assume $\CC$ is a Hom-finite triangulated $k$-category
	 and has a Serre functor $\bbS$.
	Let $\II$ be a functorially finite rigid subcategory in $\CC$. 
	If $\II$ is $\bbS_2$-stable, $((\II, \II[-1]^{\perp}), ({}^{\perp}\II[1], \II))$ is heart-equivalent.
\end{enumerate}
\end{ex}

The following definition is well-defined from \cite{LN19}.
We will define similar functors under more general situations in Definition \ref{prop_+_and_-}.

\begin{defi} \cite[Definition 2.18]{LN19} \label{def_of_plus}
Let  $((\SS,\TT), (\UU,\VV))$ be a concentric twin cotorsion pair and $X \in \CC$.
\begin{enumerate}
	\item There exist $X^- \in \CC^-$ and the following commutative diagrams in $\CC$
	where $S \in \SS, T \in \TT, Z \in \ZZ$ and $V \in \VV$.
\[
\xy
(0,24)*+{}="11";
(16,24)*+{S}="12";
(32,24)*+{S}="13";
(48,24)*+{}="14";
(0,8)*+{V}="21";
(16,8)*+{X^-}="22";
(32,8)*+{X}="23";
(48,8)*+{V}="24";
(0,-8)*+{V}="31";
(16,-8)*+{Z}="32";
(32,-8)*+{T}="33";
(48,-8)*+{V}="34";
(0,-24)*+{}="41";
(16,-24)*+{S}="42";
(32,-24)*+{S}="43";
{\ar@{=} "12";"13"};
{\ar^{r_X} "21";"22"};
{\ar^{s_X} "22";"23"};
{\ar@{-->}^{\rho_X} "23";"24"};
{\ar^{} "31";"32"};
{\ar^{} "32";"33"};
{\ar@{-->}^{} "33";"34"};
{\ar@{=} "42";"43"};
{\ar@{=} "21";"31"};
{\ar@{-->}^{} "12";"22"};
{\ar^{} "22";"32"};
{\ar^{} "32";"42"};
{\ar@{-->}^{} "13";"23"};
{\ar^{} "23";"33"};
{\ar^{} "33";"43"};
{\ar@{=} "24";"34"};
{\ar@{}|\car "12";"23"};
{\ar@{}|\car "21";"32"};
{\ar@{}|\car "22";"33"};
{\ar@{}|\car "23";"34"};
{\ar@{}|\car "32";"43"};
\endxy
\]
Moreover, $(\cdot)^- \colon \ul{\CC} \to \ul{\CC^-}$ defined by $X \mapsto X^-$ is a right adjoint functor of the inclusion functor $\ul{\CC^-} \to \ul{\CC}$.
	\item There exist $X^+ \in \CC^+$ and the following commutative diagram in $\CC$ 
	where $S \in \SS, U \in \UU, Z \in \ZZ$ and $V \in \VV$.
\[
\xy
(0,24)*+{}="11";
(16,24)*+{S}="12";
(32,24)*+{S}="13";
(48,24)*+{}="14";
(0,8)*+{V}="21";
(16,8)*+{U}="22";
(32,8)*+{X}="23";
(48,8)*+{V}="24";
(0,-8)*+{V}="31";
(16,-8)*+{Z}="32";
(32,-8)*+{X^+}="33";
(48,-8)*+{V}="34";
(0,-24)*+{}="41";
(16,-24)*+{S}="42";
(32,-24)*+{S}="43";
{\ar@{=} "12";"13"};
{\ar^{} "21";"22"};
{\ar^{} "22";"23"};
{\ar@{-->}^{} "23";"24"};
{\ar^{} "31";"32"};
{\ar^{} "32";"33"};
{\ar@{-->}^{} "33";"34"};
{\ar@{=} "42";"43"};
{\ar@{=} "21";"31"};
{\ar@{-->}^{} "12";"22"};
{\ar^{} "22";"32"};
{\ar^{} "32";"42"};
{\ar@{-->}^{\rho^X} "13";"23"};
{\ar^{r^X} "23";"33"};
{\ar^{s^X} "33";"43"};
{\ar@{=} "24";"34"};
{\ar@{}|\car "12";"23"};
{\ar@{}|\car "21";"32"};
{\ar@{}|\car "22";"33"};
{\ar@{}|\car "23";"34"};
{\ar@{}|\car "32";"43"};
\endxy
\]
Moreover, $(\cdot)^+ \colon \ul{\CC} \to \ul{\CC^+}$ defined by $X \mapsto X^+$ is a left adjoint functor of the inclusion functor $\ul{\CC^+} \to \ul{\CC}$.
\end{enumerate}
\end{defi}

\begin{thm} \label{case_of_twin cotorsion pair} \cite[Theorem 4.15, Remark 5.8]{Nak18}
Let $((\SS,\TT), (\UU,\VV))$ be a concentric twin cotorsion pair.
We assume that $\CC$ is a triangulated category.
We denote the image of $\UU$ \resp{$\TT$} by $(\cdot)^-$ \resp{$(\cdot)^+$} as $\UU^-$ \resp{$\TT^+$}.
\begin{enumerate}
	\item We additionally assume the following conditions. \label{case_of_twin cotorsion pair_1}
		\begin{enumerate}[label=(\roman*)]
		\item $\ul{\UU^-} \subset \ul{\NN^i}$.
		\item $\ul{\TT^+} \subset \ul{\NN^f}$.
		\end{enumerate}
		Then the pretriangulated category $(\ul{\ZZ}, \Sigma, \Omega, \nabla, \Delta)$ 
		in Theorem \ref{thm_pretri-ccTCP} 
		becomes a triangulated category.
	\item If $((\SS,\TT), (\UU,\VV))$ is Hovey or heart-equivalent, then it satisfies the conditions in (1).
	In particular, $\ul{\ZZ}$ is a triangulated category.
\end{enumerate}
\end{thm}
\begin{proof}
(1) From \cite[Remark 5.8]{Nak18}.

(2) From \cite[Corollary 5.11, 6.12, Proposition 6.2]{Nak18}.
\end{proof}

One can check that the poof of Theorem \ref{case_of_twin cotorsion pair} in \cite{Nak18} also works under the following slightly looser conditions.
\begin{condi} \label{condi_MT4-}
Let $((\SS, \TT), (\UU, \VV))$ be a concentric twin cotorsion pairs.
	\begin{itemize}[leftmargin=40pt]
		\item[(Nak1)] $\ul{\ZZ\bracket{1}^-} \subset \ul{\NN^i}$.
		\item[(Nak2)] $\ul{\ZZ\bracket{-1}^+} \subset \ul{\NN^f}$.
	\end{itemize}
\end{condi}
If a concentric twin cotorsion pair $((\SS, \TT), (\UU, \VV))$ satisfies (Nak1) and (Nak2), 
then we say that $((\SS, \TT), (\UU, \VV))$ satisfies (Nak).

\section{Mutation theory} \label{tri_str_by_MP}
We assume that $\CC$ is a triangulated category and all subcategories are closed under direct summands in this section.

\subsection{Rigid mutation pairs} \label{Rigid mutation pairs}
The aim of this subsection is to introduce rigid mutation pairs. For details, see \cite{IY08, AI12}.

\begin{lem} \cite[Definition 2.5, Proposition 2.6]{IY08}, \cite[Lemma 3.5]{Ara24} \label{lem_bracket_rigidver}.
We define the following functors.
	\begin{enumerate}
	\item Assume that $\II$ is covariantly finite in $\XX$.
		\begin{enumerate}[leftmargin=15pt]
		\item For $X \in \XX$, there exists a left $\II$-approximation $i^X \colon X \to I^X$.
		Then we obtain the following triangle in $\CC$.
		\[
		X \xrar{i^X} I^X \xrar{p^X} X\bracket{1} \xrar{l^X} X[1]
		\]
		\item For a morphism $x \colon X \to X^{\pr}$ in $\XX$, we define $x\bracket{1} \colon X\bracket{1} \to X^{\pr}\bracket{1}$ as a morphism in $\CC$ which makes the following diagram in $\CC$ commutative.
		\[
		\xy
		(0,8)*+{X \bracket{1}}="1";
		(0,-8)*+{X^{\pr} \bracket{1}}="2";
		(16,8)*+{X[1]}="3";
		(16,-8)*+{X^{\pr}[1]}="4";
		{\ar^-{l^X} "1";"3"};		
		{\ar^-{l^{X^{\pr}}} "2";"4"};
		{\ar^{x[1]} "3";"4"};
		{\ar_{x \bracket{1}} "1";"2"};
		{\ar@{}|\car "1";"4"};
		\endxy
		\]
		\end{enumerate}
	Then $\bracket{1}$ induces an additive functor $\bracket{1} \colon \XX/[\II] \to \CC/[\II]$.
	Moreover, $\bracket{1}$ is unique up to natural isomorphisms.
	\item Assume that $\II$ is contravariantly finite in $\YY$.
		\begin{enumerate}[leftmargin=15pt]
		\item For $Y \in \YY$, there exists a right $\II$-approximation $p_Y \colon I_Y \to X$.
		Then we obtain the following triangle in $\CC$.
		\[
		Y\bracket{-1} \xrar{i_Y} I_Y \xrar{p_Y} Y \xrar{l_Y} Y\bracket{-1}[1]
		\]
		\item For a morphism $y \colon Y \to Y^{\pr}$ in $\YY$, 
		we define $y\bracket{-1} \colon Y\bracket{-1} \to Y^{\pr}\bracket{-1}$ as a morphism in $\CC$ 
		which makes the following diagram commutative.
		\[
		\xy
		(24,8)*+{Y \bracket{-1}}="3";
		(24,-8)*+{Y^{\pr} \bracket{-1}}="4";
		(0,8)*+{Y[-1]}="1";
		(0,-8)*+{Y^{\pr}[-1]}="2";
		{\ar^-{l_Y[-1]} "1";"3"};		
		{\ar^-{l_{Y^{\pr}}[-1]} "2";"4"};
		{\ar^{y  \bracket{-1}} "3";"4"};
		{\ar_{y[-1]} "1";"2"};
		{\ar@{}|\car "1";"4"};
		\endxy
		\]
		\end{enumerate}
	Then $\bracket{-1}$ induces an additive functor $\bracket{-1} \colon \YY/[\II] \to \CC/[\II]$.
	Moreover, $\bracket{-1}$ is unique up to natural isomorphisms.
	\end{enumerate}
\end{lem}
\begin{proof}
See \cite[Definition 2.5, Proposition 2.6]{IY08} (or Lemma \ref{def_of_bracket}).
\end{proof}

\begin{ex}
The following picture is the AR quiver of $\CC = \Db(k A_4)/\nu[-4]$ 
where $\nu$ is the Nakayama functor and $A_4$ is the linearly oriented Dynkin quiver of type $A_4$
(we omitted AR translations).

The fundamental domain is shaded light gray and $\II$ is shaded dark gray.
Then $\II$ is a functorially finite 4-rigid subcategory in $\CC$.
Take $X \in \CC$ indicated in the following picture.
Then $X\bracket{1}, X\bracket{2}$, $X\bracket{3}$ and $X\bracket{4}$ are as follows.
Note that there exists a nonzero morphism $f$ which is a left $\II$-approximation of $X\bracket{1}$.
\[
\begin{tikzpicture}
	\path[rounded corners,fill=gray,opacity=0.2] (0.75, -0.5) -- (2.75, 3.5) -- (9.25, 3.5) -- (11.25, -0.5) --cycle;
	\draw[rounded corners,gray,fill=gray,opacity=0.6,rotate around={63.44:(4.5,0)}] (4.25,-0.25) rectangle ++(1.63,0.5);
	\foreach \x in {0,1,...,11}
		\foreach \y in {1,3}
		{
		\fill (\x,\y) circle[radius=0.04cm];
		\fill (\x+0.5,\y-1) circle[radius=0.04cm];
		\draw[-{stealth[scale=3]}] (\x,\y) -- (\x+0.48,\y-0.96);
		\draw[-{stealth[scale=3]}] (\x,1) -- (\x+0.48,1.96);
		\ifnum \x<11
			\draw[-{stealth[scale=3]}] (\x+0.5,\y-1) -- (\x+0.98,\y-0.04);
			\draw[-{stealth[scale=3]}] (\x+0.5,2) -- (\x+0.98,1.04);
		\fi
		}
	\fill (1.5, 0) circle[radius=0.08cm] 
	(4, 3) circle[radius=0.08cm] 
	(5.5, 2) circle[radius=0.08cm] 
	(8, 1) circle[radius=0.08cm]
	(2, 1) circle[radius=0.08cm]
	;
	\node (X_0) at (1.5,0) [right] {\msize{0.7}{X}};
	\node (X_1) at (4,3) [right] {\msize{0.7}{X\bracket{1}}};
	\node (X_2) at (5.5,2) [right] {\msize{0.7}{X\bracket{2}}};
	\node (X_3) at (8,1) [right] {\msize{0.7}{X\bracket{3}}};
	\node (X_4-2) at (2,1) [right] {\msize{0.7}{X\bracket{4}}};
	\node (D) at (4.9,0.3) [right] {\msize{0.7}{\II}};
	\draw [-{stealth[scale=3]}, blue, very thick] (4,3) -- (5,1);
	\node (morph) [text=blue] at (4.5,2) [right] {\msize{0.7}{f}};
\end{tikzpicture}
\]
\end{ex}

We use the following notation.
\[
\XX\bracket{1} =
\left\{
Z \in \CC \ \middle| \ 
	\begin{aligned}
	\text{there exists a triangle }& X \xrar{f} I \rar Z \rar X[1] \\
	&\text{ where $f$ is a left $\II$-approximation.}
	\end{aligned}
\right\}
\]
That is because $\ul{\XX\bracket{1}}$ equals to
the isomorphic closure of $\{ X\bracket{1} \mid X \in \XX \}$ in $\ul{\CC}$.
Dually, we define $\YY\bracket{-1}$.

\begin{lem} \cite[Definition 2.5]{IY08}, \cite[Definition 2.30]{AI12} \label{lem_mu-of-rigidMP}
Let $\II$ be a functorially finite rigid subcategory in $\CC$. 
Let $\XX,\YY$ be subcategories containing $\II$.
	\begin{enumerate}
	\item $(\II \ast \XX[1]) \cap {}^{\perp}\II[1] = \XX\bracket{1}$.
	\item $(\YY[-1] \ast \II) \cap \II[-1]^{\perp} = \YY\bracket{-1}$.
	\end{enumerate} 
\end{lem}
\begin{proof}
We only show (1).
Take $Z \in (\II \ast \XX[1]) \cap {}^{\perp}\II[1]$. Since $Z \in \II \ast \XX[1]$,
there exists a triangle $X \xrar{f} I \xrar{g} Z \xrar{h} X[1]$ where $X \in \XX, I \in \II$.
From $Z[-1] \in \lpp{\II}$, $f$ is a left $\II$-approximation.
Let $Z \in \XX\bracket{1}$, then $Z \in \II \ast \XX[1]$.
From the long exact sequence, $Z \in \lpp{\II[1]}$.
\end{proof}

\begin{prop} \cite[Proposition 2.6, 2.7]{IY08} \label{prop_rigid-mutation-pair}
Assume that $\II$ is a functorially finite rigid subcategory in $\CC$.
	\begin{enumerate}
	\item $\bracket{1} \colon \rpp{\II[-1]}/[\II] \to \lpp{\II[1]}/[\II]$ and 
	$\bracket{-1} \colon \lpp{\II[1]}/[\II] \to \rpp{\II[-1]}/[\II]$ are mutually quasi-inverse.
	\item The following correspondence of subcategories is bijective.
	\begin{align*}
	\bracket{1} \colon \{ \XX \colon \text{subcategory} \mid \, \II \subset \XX \subset \rpp{\II[-1]} \}  
	&\rightleftarrows \\
	\{\YY \colon \text{subcategory} \mid &\ \II \subset \YY \subset \lpp{\II[1]} \} \colon \bracket{-1}
	\end{align*}
	\end{enumerate}
\end{prop}
\begin{proof}
See \cite[Proposition 2.6, 2.7]{IY08}.
\end{proof}

\begin{defi} \label{rigid_MP} \cite[Definition 2.5]{IY08}
Let $\II$ be a functorially finite rigid subcategory in $\CC$.
Let $\XX$ and $\YY$ be subcategories of $\CC$ containing $\II$.
	\begin{enumerate}
	\item We call a pair of subcategories $(\XX, \YY)$ a \emph{rigid $\II$-mutation pair} 
	if $\XX \subset \rpp{\II[-1]}$ and $\YY = \XX\bracket{1}$.
	\item Let $(\XX, \YY)$ be a rigid $\II$-mutation pair.
	Then an additive functor $\bracket{1}$ \resp{$\bracket{-1}$} defined in Lemma \ref{lem_bracket_rigidver}
	is called a \emph{right \resp{left} mutation functor}.
	\end{enumerate}
\end{defi}

\begin{rmk}
We reversed the names ``right mutation'' and ``left mutation'' from \cite{IY08, AI12}.
Then right mutation functors induce right triangulated structures.
\end{rmk}

\begin{ex} \cite{IY08}
Let $\II$ be a functorially finite rigid subcategory in $\CC$.
The following pairs of subcategories are elementary examples of rigid mutation pairs.
	\begin{enumerate}
	\item $(\II, \II)$ is a rigid $\II$-mutation pair.
	\item $(\rpp{\II[-1]}, \lpp{\II[1]})$ is a rigid $\II$-mutation pair.
	\item For any subcategory $\XX$, $(\XX, \XX[1])$ is a rigid $0$-mutation pair.
	\end{enumerate}
\end{ex}

\subsection{Triangulated structures induced by rigid mutation pairs}
\label{Triangulated structures induced by rigid mutation pairs}

The following statement is \cite[Theorem 4.2]{IY08}.

 \begin{condi} \cite[Section 4]{IY08} \label{IY_condi}
 Let $\II \subset \ZZ$ be subcategories of $\CC$. We consider the following conditions.
 	\begin{itemize}[leftmargin=40pt]
	\item[(IY1)] $\II$ is functorially finite in $\ZZ$.
	\item[(IY2)] $(\ZZ, \ZZ)$ is a rigid $\II$-mutation pair.
	\item[(IY3)] $\ZZ$ is closed under extensions.
	\end{itemize}
 \end{condi}
 
 \begin{thm} \cite[Theorem 4.2]{IY08}
 We assume (IY1), (IY2) and (IY3) in Condition \ref{IY_condi}.
 
 Then $(\ZZ/[\II], \bracket{1}, \triangle)$ is a triangulated category, 
 where distinguished triangles are defined as follows.
 
 For a morphism $f \colon X \to Y$ in $\ZZ$, we define a distinguished triangle 
 $X \xrar{f} Y \xrar{g} Z \xrar{h} X\bracket{1}$ 
 in $\ZZ/[\II]$ by the following diagram in $\CC$
 \[
 \xy
(0,8)*+{X}="11";
(16,8)*+{I^X}="12";
(32,8)*+{X\bracket{1}}="13";
(48,8)*+{X[1]}="14";
(0,-8)*+{Y}="21";
(16,-8)*+{Z}="22";
(32,-8)*+{X\bracket{1}}="23";
(48,-8)*+{Y[1]}="24";
{\ar^{i^X} "11";"12"};
{\ar^{p^X} "12";"13"};	
{\ar^{} "13";"14"};
{\ar^{g} "21";"22"};
{\ar^{h} "22";"23"};
{\ar^{} "23";"24"};
{\ar^{f} "11";"21"};
{\ar^{} "12";"22"};
{\ar@{=}^{} "13";"23"};
{\ar^{f[1]} "14";"24"};
{\ar@{}|\car "11";"22"};
{\ar@{}|\car "12";"23"};
{\ar@{}|\car "13";"24"};
\endxy
 \]
and
\[
\triangle = 
\left(
\begin{array}{ll}
\text{sequences in } \ZZ/[\II] & \text{isomorphic to one in } \\
&\{ X \xrar{f} Y \xrar{g} Z \xrar{h} X\bracket{1} \mid f \text{ is a morphism in } \ZZ \}
\end{array}
\right).
\]
\end{thm}
\begin{proof}
From \cite[Theorem 4.2]{IY08} (or, Theorem \ref{main_thm2} and Example \ref{ex_MT4}(\ref{ex_MT4_IY})).
\end{proof}
 
In the last of this section, we introduce important examples of rigid mutation pairs.
We assume that $\CC$ is a Hom-finite $k$-category which has a Serre functor $\bbS$.

\begin{defi}  \cite[Section 3]{IY08} \cite[Definition 2.5]{SP20} \label{defi_stable}
Let $\XX$ be a subcategory of $\CC$ and $n$ be an arbitrary integer.
	\begin{enumerate}
	\item We denote $\bbS \circ [-n]$ by $\bbS_n$.
	\item $\XX$ is $\bbS_n$-stable if $\XX \cong \bbS_n \XX$.
	\end{enumerate}
\end{defi}
 
 \begin{ex} \cite{IY08} \label{ex_rigid_mutation_pair}
Let $n \geq 2$ be an integer and $\II$ be a functorially finite $n$-rigid $\bbS_{n}$-stable subcategory in $\CC$.
 
 Then $\ZZ = \bcap{0 < i < n} \lpp{\II[i]}$ satisfies (IY1), (IY2) and (IY3).
 \end{ex}

\subsection{Orthogonal mutation pairs}
\label{Orthogonal mutation pairs}
We collect some information about mutations of orthogonal collections in this section. 
For details, see \cite{Dug15, SP20}. 

Let $\MM \subset \mathrm{Ob}(\CC)$ and fix it.

\begin{condi} \cite[Setup 3.5]{SP20} \label{SP_condi1} We consider the following conditions.
	\begin{itemize}[leftmargin=40pt]
	\item[(SP1)]
		\begin{enumerate}[label=(\roman*)]
		\item $\MM$ is a semibrick and $\bracket{\MM}$ is functorially finite in $\CC$.
		\item $\MM$ is orthogonal or $\bbS_{-1}$-stable.
		\end{enumerate}
	\end{itemize}
\end{condi}

\begin{defi} \cite[Definition 3.1, Lemma 3.3]{SP20} \cite[Section 4]{Dug15} \label{defi_sigma_omega_apdx}
We assume (SP1)(i).
Let $X \in \CC$.
	\begin{enumerate}
	\item $\omega X$ is defined (uniquely up to isomorphisms) by the following triangle.
		\begin{align*}
		\omega X \xrar{f^X} X \xrar{g^X} M^X &\xrar{h^X} \omega X[1]  \\
		&\text{where } g^X \text{ is a left  minimal } \bracket{\MM} \text{-approximation}.
		\end{align*}
	\item $\sigma X$ is defined (uniquely up to isomorphisms) by the following triangle.
		\begin{align*}
		M_X \xrar{f_X} X \xrar{g_X} \sigma X &\xrar{h_X} M_X[1]  \\
		&\text{where } f_X \text{ is a right  minimal } \bracket{\MM} \text{-approximation}.
		\end{align*}
	\end{enumerate}
\end{defi}

\begin{rmk}
From Wakamatsu's lemma, $\omega X \in \lpp{\MM}$ and $\sigma X \in \rpp{\MM}$. 
\end{rmk}

\begin{lem} \cite[Lemma 2.6(1)]{SP20} \label{append_omega_sigma}
We assume that $\MM$ is a semibrick.
Let $X \in \CC$.
	\begin{enumerate}
	\item If $X \in \lpp{\MM[-1]}$ and there exists a left minimal $\bracket{\MM}$-approximation of $X$,
	then $\omega X \in \lpp{\MM[-1]} \cap \lpp{\MM}$. 
	\item If $X \in \rpp{\MM[1]}$ and there exists a right minimal $\bracket{\MM}$-approximation of $X$,
	then $\sigma X \in \rpp{\MM[1]} \cap \rpp{\MM}$.
	\end{enumerate}
\end{lem}
\begin{proof}
See \cite[Lemma 2.6(1)]{SP20}.
\end{proof}

\begin{cor}
We assume (SP1)(i).
	\begin{enumerate}
	\item $\omega \colon \lpp{\MM[-1]} \to \lpp{\MM[-1]} \cap \lpp{\MM}$ is an additive functor.
	\item $\sigma \colon \rpp{\MM[1]} \to \rpp{\MM[1]} \cap \rpp{\MM}$ is an additive functor.
	\end{enumerate}
\end{cor}
\begin{proof}
We only prove (1).
Take a morphism $a \colon X \rar Y$. $a$ induces another morphism $b \colon \omega X \to \omega Y$ 
since $X \xrar{g^X} M^X$ is a left $\bracket{\MM}$-approximation and the axioms of triangulated category. 
Uniqueness of $b$ follows from $\omega X \in \lpp{\MM[-1]}$.
\end{proof}

\begin{lem} \cite[Lemma 2.6(3)]{SP20}
We assume (SP1).
	\begin{enumerate}
	\item If $X \in \rpp{\MM}$, then $\omega X \in \rpp{\MM}$. 
	\item If $X \in \lpp{\MM}$, then $\sigma X \in \lpp{\MM}$. 
	\end{enumerate}
\end{lem}
\begin{proof}
We only prove (2).
If $\MM$ is orthogonal, this is clear. If $\MM$ is $\bbS_{-1}$-stable, this follows from \cite[Lemma 2.6(3)]{SP20}.
\end{proof}

\begin{cor}
We denote $\lpp{\MM} \cap \rpp{\MM}$ by $\dpp{\MM}$. 
Assume (SP1).
	\begin{enumerate}
	\item $\omega \colon \lpp{\MM[-1]} \cap \rpp{\MM}  \to \lpp{\MM[-1]} \cap \dpp{\MM}$ is an additive functor.
	\item $\sigma \colon \rpp{\MM[1]} \cap \lpp{\MM} \to \rpp{\MM[1]} \cap \dpp{\MM}$ is an additive functor.
	\end{enumerate}
\end{cor}

\begin{defi} \cite[Definition 3.1]{SP20} \label{sigma_and_omega_inSP}
Assume (SP1). We define the following functors.
	\begin{enumerate}
	\item $\Omega = \omega \circ [-1] \colon \lpp{\MM} \cap \rpp{\MM[1]} \to \lpp{\MM[-1]} \cap \dpp{\MM}$, called a \emph{left mutation functor}.
	\item $\Sigma = \sigma \circ [1] \colon \rpp{\MM} \cap \lpp{\MM[-1]} \to \rpp{\MM[1]} \cap \dpp{\MM}$, called a \emph{right mutation functor}.
	\end{enumerate}
\end{defi}

\begin{ex}
The following picture is the AR quiver of $\CC = \Db(k A_4)/\nu [2]$ 
where $\nu$ is the Nakayama functor and $A_4$ is the linearly oriented Dynkin quiver of type $A_4$
(we omitted AR translations).

The fundamental domain is shaded light gray, 
$\MM$ is marked by gray circle and $\bracket{\MM}$ is shaded dark gray.
Then $\MM$ is a 2-orthogonal collection and $\bracket{\MM}$ is functorially finite in $\CC$.
Take $X \in \CC$ indicated in the following picture.
Then $X[1], \Sigma X, \Sigma^2 X$ and $\Sigma^3 X$ are as follows.
Note that there exists a nonzero morphism $f$ which is a minimal right $\bracket{\MM}$-approximation of $X[1]$.
\[
\begin{tikzpicture}
	\path[rounded corners,fill=gray,opacity=0.2] (0.75, -0.5) -- (2.75, 3.5) -- (7.25, 3.5) -- (9.25, -0.5) --cycle;
	\path[rounded corners,fill=gray,opacity=0.6] (4,-0.3) -- (5, 1.7) -- (6, -0.3) --cycle;
	\foreach \x in {0,1,...,11}
		\foreach \y in {1,3}
		{
		\fill (\x,\y) circle[radius=0.04cm];
		\fill (\x+0.5,\y-1) circle[radius=0.04cm];
		\draw[-{stealth[scale=3]}] (\x,\y) -- (\x+0.48,\y-0.96);
		\draw[-{stealth[scale=3]}] (\x,1) -- (\x+0.48,1.96);
		\ifnum \x<11
			\draw[-{stealth[scale=3]}] (\x+0.5,\y-1) -- (\x+0.98,\y-0.04);
			\draw[-{stealth[scale=3]}] (\x+0.5,2) -- (\x+0.98,1.04);
		\fi
		}
	\fill (3, 1) circle[radius=0.08cm] 
	(5.5, 2) circle[radius=0.08cm] 
	(6.5, 0) circle[radius=0.08cm] 
	(2.5, 0) circle[radius=0.08cm]
	(5, 3) circle[radius=0.08cm]
	;
	\draw [fill=gray] (4.5, 0) circle[radius=0.08cm] (5.5, 0) circle[radius=0.08cm];
	\node (X_0) at (3,1) [right] {\msize{0.7}{X}};
	\node (X_1) at (5.5,2) [right] {\msize{0.7}{X[1]}};
	\node (X_2) at (6.5,0) [right] {\msize{0.7}{\Sigma X}};
	\node (X_3) at (2.5,0) [right] {\msize{0.7}{\Sigma^2 X}};
	\node (X_4) at (5,3) [right] {\msize{0.7}{\Sigma^3 X}};
	\node (D) at (5,-0.3) [below] {\msize{0.7}{\bracket{\MM}}};
	\draw [-{stealth[scale=3]}, blue, very thick] (5,1) -- (5.48,1.96);
	\node (morph) [text=blue] at (5.25,1.5) [right] {\msize{0.7}{f}};
\end{tikzpicture}
\]
\end{ex}

\begin{prop} \cite[Lemma 3.6]{SP20} \label{prop_SM_mutation}
Assume (SP1).
	\begin{enumerate}
	\item $\Sigma$ and $\Omega$ are mutually quasi-inverse.
	\item (1) induces the following one-to-one correspondence of subcategories.
		\begin{align*}
		\Sigma \colon \{ \XX \colon \text{subcategory in } \CC \mid \ \XX \subset \dpp{\MM} \cap \lpp{\MM[-1]} \} 
		&\rightleftarrows \\
		\{\YY \colon \text{subcategory in } \CC \mid &\ \YY \subset \dpp{\MM} \cap \rpp{\MM[1]} \} \colon \Omega \\
		\text{where }\Sigma\XX =
		\XX[1] \ast \bracket{\MM[1]} \cap \rpp{\MM[1]} \cap \dpp{\MM} \text{ and} \\
		\Omega\YY = 
		\bracket{\MM[-1]} \ast \YY[-1] &\cap \lpp{\MM[-1]} \cap \dpp{\MM}.
		\end{align*}
	\end{enumerate}
\end{prop}
\begin{proof}
For (1), see \cite[Lemma 3.6]{SP20}. 
We only prove $\Sigma \XX = \XX[1] \ast \bracket{\MM[1]} \cap \rpp{\MM[1]} \cap \dpp{\MM}$. 

By definition, $\Sigma \XX \subset \XX[1] \ast \bracket{\MM[1]} \cap \rpp{\MM[1]} \cap \dpp{\MM}$. 
Take $Z \in \XX[1] \ast \bracket{\MM[1]} \cap \rpp{\MM[1]} \cap \dpp{\MM}$. 
Then there exists a triangle $Z[-1] \rar M \xrar{f} X[1] \rar Z$ where $M \in \bracket{\MM}, X \in \XX$.
From $Z \in \rpp{\MM}$, $f$ is a right $\bracket{\MM}$-approximation. 
Assume that $f$ is not right minimal.
Then $Z[-1]$ and $M$ has a nonzero common direct summand.
However, this contradicts $Z[-1] \in \rpp{\MM}$. 
Thus, $f$ is right minimal.
Therefore, $Z \cong \sigma(X[1]) = \Sigma X$.
\end{proof}

\begin{defi} \cite[Definition 3.2]{SP20} \label{SM_MP}
Let $\XX, \YY$ be subcategories of $\CC$ and assume (SP1).
A pair of subcategories $(\XX, \YY)$ is an \emph{orthogonal $\bracket{\MM}$-mutation pair} 
if $\XX \subset \lpp{\MM[-1]} \cap \dpp{\MM}$ and $\Sigma \XX = \YY$.
\end{defi}

\begin{rmk}
In \cite{SP20}, they define (orthogonal) $\bracket{\MM}$-mutation pairs for a collection $\MM$ in $\Ob(\CC)$
where $\bracket{\MM}$ is functorially finite in $\CC$.
From Proposition \ref{prop_SM_mutation}, the definition of orthogonal mutation pairs in \cite{SP20} is equivalent to that in Definition \ref{SM_MP}.
\end{rmk}

\subsection{Triangulated structures induced by orthogonal mutation pairs}
\label{Triangulated structures induced by orthogonal mutation pairs}

The following statement is \cite[Theorem 5.1]{SP20}. 
We start from introducing some conditions for orthogonal $\bracket{\MM}$-mutation pairs satisfying (SP1).

\begin{condi} \cite[Theorem 4.1]{SP20} \label{SP_condi2}
We assume that $\MM$ satisfies (SP1). Let $\ZZ$ be a subcategory of $\CC$.
	\begin{itemize}[leftmargin=40pt]
	\item[(SP2)] $(\ZZ, \ZZ)$ is an orthogonal $\bracket{\MM}$-mutation pair.
	\item[(SP3)] 
		\begin{enumerate}[label=(\roman*)]
		\item $\Cone(\ZZ, \ZZ) \subset \bracket{\MM} \ast \ZZ$.
		\item $\CoCone(\ZZ, \ZZ) \subset \ZZ \ast \bracket{\MM}$.
		\item $\ZZ$ is closed under extensions.
		\end{enumerate}
	\end{itemize}
\end{condi}

\begin{thm} \cite[Theorem 4.1,5.1]{SP20} \label{thm_tri_SM}
We assume (SP1), (SP2) and (SP3) in Condition \ref{SP_condi1} and \ref{SP_condi2}.

Then $(\ZZ, \Sigma, \triangle)$ is a triangulated category, where distinguished triangles are defined as follows.

For a morphism $f \colon X \to Y$ in $\ZZ$, we define a distinguished triangle $X \xrar{f} Y \xrar{g} Z \xrar{h} \Sigma X$ in $\ZZ$ by the following diagram in $\CC$
\[
\xy
(0,16)*+{X}="1";
(16,16)*+{Y}="3";
(32,16)*+{Z^{\pr}}="5";
(32,0)*+{Z}="6";
(48,16)*+{X[1]}="7";
(48,0)*+{\Sigma X}="8";
{\ar^{f} "1";"3"};
{\ar^{g^{\pr}} "3";"5"};
{\ar^{h^{\pr}} "5";"7"};
{\ar^{g_{Z^{\pr}}} "5";"6"};
{\ar^{g_{X[1]}} "7";"8"};
{\ar^{h} "6";"8"};
{\ar_{g}^{\car} "3";"6"};
{\ar@{}|\car "5";"8"};
\endxy
\]
where
$Z = \sigma Z^{\pr}$ and
$X \xrar{f} Y \xrar{g^{\pr}} Z^{\pr} \xrar{h^{\pr}} X[1]$ is a triangle.
Then 
\[
\triangle = 
\left(
\begin{array}{ll}
\text{sequences} &\text{in } \ZZ  \text{ isomorphic to one in } \\
&\{ X \xrar{f} Y \xrar{g} Z \xrar{h} \Sigma X \mid f \text{ is a morphism in } \ZZ \}
\end{array}
\right).
\]
\end{thm}
\begin{proof}
From \cite[Theorem 4.1,5.1]{SP20} (or Theorem \ref{main_thm2} and Example \ref{ex_MT4}(\ref{ex_MT4_SP})).
\end{proof}

\begin{ex} \cite[Lemma 6.3]{SP20} \label{ex_tri_SM}
Let $\MM$ be an $n$-orthogonal collection for $n \geq 1$.
We assume that $\bracket{\MM}$ is $\bbS_{-n}$-stable.
Let $\ZZ = \bcap{0 \leq i \leq n} \rpp{\MM[i]}$. Then $\ZZ$ satisfies (SP1), (SP2) and (SP3).
\end{ex}

\section{Reduction theory} \label{Reduction theory}
We assume that $\CC$ is a Hom-finite Krull-Schmidt triangulated $k$-category 
and has a Serre functor $\bbS$.
We also assume that all subcategories are closed under direct summands in this section.
The aim of this section is to organize the known results in the reduction theory by reducible triples.
For a subcategory $\XX$, we denote the smallest thick subcategory containing $\XX$ by $\thick \XX$.

\subsection{Characterizations for reducible collections}
In this subsection, we characterize reducible collections;
silting subcategory, $n$-cluster tilting subcategory, simple-minded collections and $n$-simple-minded systems, 
as a preliminary to the following subsections.

\begin{defi} \cite[5.1]{BR07} \cite[Section 3]{IY08} \cite[Definition 2.1]{AI12}
Let $\XX$ be a subcategory of $\CC$.
	\begin{enumerate}
	\item For $1 < n <\infty$, $\XX$ is \emph{$n$-cluster tilting} if
	 $\XX$ is functorially finite in $\CC$ and $\XX = \bcap{0<i<n} \lpp{\XX[i]} = \bcap{0<i<n} \rpp{\XX[-i]}$.
	\item $\XX$ is \emph{silting} if $\XX$ is $\infty$-rigid and $\CC = \thick \XX$.
	\end{enumerate}
\end{defi}

\begin{lem} \cite[Proposition 2.1, 2.4]{IY08} \label{lem_ZZ_funct-fin_RG}
Let $\II$ be an $n$-rigid subcategory for $1<n<\infty$. 
We assume that $\II$ is functorially finite in $\CC$.
	\begin{enumerate}
	\item $\Big(\II \ast \II[1] \ast \cdots \ast \II[i], \bcap{0 \leq l \leq i} \rpp{\II[l]} \Big)$ is a torsion pair of $\CC$ 
	for $0 \leq i \leq n-2$.
	\item $\Big( \, \bcap{0 \leq l \leq i} \lpp{\II[-l]}, \, \II[-i] \ast \cdots \ast \II[-1] \ast \II \Big)$ is a torsion pair of $\CC$ 
	for $0 \leq i \leq n-2$.
	\end{enumerate}
\end{lem}
\begin{proof}
See \cite[Proposition 2.1, 2.4]{IY08}.
\end{proof}

\begin{cor} \cite[Proposition 3.5]{IY08} \label{cor_characterize_CTL}
Let $\II$ be an $n$-rigid subcategory for $1<n<\infty$. Then the following are equivalent.
	\begin{enumerate}
	\item $\II$ is functorially finite in $\CC$ and $\II = \bcap{0 < i < n} \lpp{\II[i]} = \bcap{0<i<n} \rpp{\II[-i]}$.
	\item $\II$ is functorially finite in $\CC$, $\bbS_n$-stable and 
	$\II = \bcap{0 < i < n} \lpp{\II[i]}$.
	\item $\CC = \II \ast \II[1] \ast \cdots \ast \II[n-1]$.
	\end{enumerate}
\end{cor}
\begin{proof}
(1)$\Leftrightarrow$(2) is proved in \cite[Proposition 3.5]{IY08}.
(1)$\Leftrightarrow$(3) follows from Lemma \ref{lem_ZZ_funct-fin_RG} and 
$\Big(\, \bcap{0<i<n} \, \rpp{\II[-i]} \Big)[n-1] = \bcap{0 \leq i \leq n-2} \, \rpp{\II[i]}$.
\end{proof}

\begin{lem} \cite[Proposition 3.2]{IY18} \label{lem_torsion_pair_silt}
Let $\II$ be an $\infty$-rigid subcategory.
We assume the following conditions.
	\begin{enumerate}[label=(\roman*), leftmargin=20pt]
	\item $\II$ is covariantly finite in $\bcap{i>0} \lpp{\II[i]}$.
	\item $\II$ is contravariantly finite in $\bcap{i>0} \, \rpp{\II[-i]}$.
	\item $\CC(\II[-i], C) = \CC(C, \II[i]) =0$ for $C \in \CC, i \gg 0$.
	\end{enumerate}
Then the following statements hold.
	\begin{enumerate}
	\item $\Big((\thick \II)_{\geq 0}, \bcap{i \geq 0}\, \rpp{\II[-i]} \Big)$ is a torsion pair of $\CC$.
	\item $\Big(\, \bcap{i \geq 0} \lpp{\II[i]}, (\thick \II)_{\leq 0} \Big)$ is a torsion pair of $\CC$.
	\item \label{lem_torsion_pair_silt_3} 
	 $\bigg( \Big((\thick \II)_{\geq 0}, \bcap{i > 0} \, \rpp{\II[-i]} \Big), 
	\Big(\, \bcap{i > 0}\, \lpp{\II[i]}, (\thick \II)_{\leq 0} \Big) \bigg)$ is a Hovey twin cotorsion pair.
	\end{enumerate}
\end{lem}
\begin{proof}
(1), (2) See \cite[Proposition 3.2]{IY18}. (3) This follows from (1), (2) and definition of Hovey twin cotorsion pairs.
\end{proof}

\begin{cor} \label{cor_characterize_silt}
Let $\II$ be an $\infty$-rigid subcategory.
Then the following are equivalent.
	\begin{enumerate}
	\item $\CC = \thick \II$.
	\item $\CC = \bcup{i \geq 0} \ \II[-i] \ast \II[1-i] \ast \cdots \ast \II[i-1] \ast \II[i]$.
	\item 
		\begin{enumerate}
		\item $\II = \bcap{i>0} (\lpp{\II[i]} \cap \rpp{\II[-i]})$.
		\item $\CC(\II[-i], C) = \CC(C, \II[i]) =0$ for $C \in \CC, i \gg 0$.
		\item $\II$ is covariantly finite in $\bcap{i>0} \lpp{\II[i]}$.
		\item $\II$ is contravariantly finite in $\bcap{i>0} \, \rpp{\II[-i]}$.
		\end{enumerate}
	\end{enumerate}
\end{cor}
\begin{proof}
(1)$\Leftrightarrow$(2) follows from \cite[Proposition 2.23]{AI12}.

Assume (2). Then $\bcap{i > 0} \ \rpp{\II[-i]} = (\thick \II)_{\leq 0}$ and $\bcap{i > 0} \lpp{\II[i]} = (\thick \II)_{\geq 0}$. 
Thus, (i), (iii) and (iv) hold.
Let $C \in \CC$. 
Then there exists $l \geq 0$ which satisfies $C \in \II[-l] \ast \II[1-l] \ast \cdots \ast \II[l-1] \ast \II[l]$.
So, $\CC(\II[-i], C) = \CC(C, \II[i]) =0$ for $i \gg 0$. Therefore, (3) holds.

Assume (3). From Lemma \ref{lem_torsion_pair_silt}(\ref{lem_torsion_pair_silt_3}), 
\begin{align*}
\CC &= (\thick \II)_{\geq 0}[-1] \ast \bcap{i>0} (\lpp{\II[i]} \cap \rpp{\II[-i]}) \ast (\thick \II)_{\leq 0}[1] \\
&= \bcup{i \geq 0} \ \II[-i] \ast \II[1-i] \ast \cdots \ast \II[-1] \ast \II \ast \II[1] \ast \cdots \ast \II[i-1] \ast \II[i].
\end{align*}

Thus, (2) holds.
\end{proof}

\begin{defi} \cite[Definition 3.2]{KY14} \label{defi_simple-minded collection}
Let $\MM = \{M_i \mid i \in I\} \subset \Ob(\CC)$ be an $n$-orthogonal collection for $0<n < \infty$.
	\begin{enumerate}
	\item $\MM$ is an \emph{$n$-simple-minded system}
	if $\bracket{\MM}$ is functorially finite in $\CC$ and
	$\bracket{\MM} = \bcap{0< i < n} \lpp{\MM[-i]} = \bcap{0< i < n} \rpp{\MM[i]}$.
	\item $\MM$ is a \emph{simple-minded collection} 
	if $\MM$ is a pre-simple-minded collection and $\CC = \thick \MM$.
	\end{enumerate}
\end{defi}

One can check the definition of $n$-simple-minded system and that in \cite{SP20} are same by the following statement,
which is an orthogonal version of Lemma \ref{lem_ZZ_funct-fin_RG}. 

\begin{lem} \cite[Lemma 2.10]{SP20} \label{lem_ZZ_funct-fin_SM} 
Let $\MM$ be $n$-orthogonal for $0<n< \infty$. We assume that $\bracket{\MM}$ is functorially finite in $\CC$.
	\begin{enumerate}
	\item $\Big( \add(\bracket{\MM[i]} \ast \cdots \ast \bracket{\MM[1]} \ast \bracket{\MM}), 
	\bcap{0 \leq l \leq i}\, \rpp{\MM[l]} \Big)$ 
	is a torsion pair of $\CC$ for $0 \leq i \leq n$.
	\item $\Big(\, \bcap{0 \leq l \leq i} \lpp{\MM[-l]},  
	\, \add(\bracket{\MM} \ast \bracket{\MM[-1]} \ast \cdots \ast \bracket{\MM[-i]}) \Big)$ 
	is a torsion pair of $\CC$ for $0 \leq i \leq n$.
	\end{enumerate}
\end{lem}
\begin{proof}
(1) We prove this by induction on $i$.
If $i=0$, the statement follows from Wakamatsu's lemma.
Let $i>0$. We define $\XX_i = \add(\bracket{\MM} \ast \cdots \ast \bracket{\MM[1-i]} \ast \bracket{\MM[-i]})$ and
$\YY_i = \bcap{0 \leq l \leq i} \rpp{\MM[-l]}$. 
From the induction hypothesis, $(\XX_{i-1}, \YY_{i-1})$ is a torsion pair.
Let $C \in \CC$ and take a triangle $X_{i-1} \rar C \rar Y_{i-1} \rar X_{i-1}[1]$ 
where $X_{i-1} \in \XX_{i-1}$ and $Y_{i-1} \in \YY_{i-1}$.
We also take a triangle $M[-i] \xrar{m_i} Y_{i-1} \rar Y_i \rar M[1-i]$ 
where $m_i$ is a minimal right $\bracket{\MM[-i]}$-approximation.
Then there exists the following diagram in $\CC$.
\[
\xy
(16,24)*+{Y_i[-1]}="12";
(32,24)*+{Y_i[-1]}="13";
(0,8)*+{X_{i-1}}="21";
(16,8)*+{X_i}="22";
(32,8)*+{M[-i]}="23";
(48,8)*+{X_{i-1}[1]}="24";
(0,-8)*+{X_{i-1}}="31";
(16,-8)*+{C}="32";
(32,-8)*+{Y_{i-1}}="33";
(48,-8)*+{X_{i-1}[1]}="34";
(16,-24)*+{Y_i}="42";
(32,-24)*+{Y_i}="43";
{\ar@{=} "12";"13"};
{\ar^{} "21";"22"};
{\ar^{} "22";"23"};
{\ar^{} "23";"24"};
{\ar^{} "31";"32"};
{\ar^{} "32";"33"};
{\ar^{} "33";"34"};
{\ar@{=} "42";"43"};
{\ar@{=} "21";"31"};
{\ar^{} "12";"22"};
{\ar^{} "22";"32"};
{\ar^{} "32";"42"};
{\ar^{} "13";"23"};
{\ar^{m_i} "23";"33"};
{\ar^{} "33";"43"};
{\ar@{=} "24";"34"};
{\ar@{}|\car "12";"23"};
{\ar@{}|\car "21";"32"};
{\ar@{}|\car "22";"33"};
{\ar@{}|\car "23";"34"};
{\ar@{}|\car "32";"43"};
\endxy
\]
Since $Y_{i-1} \in \rpp{\MM[1-i]}$, $Y_i \in \rpp{\MM[1-i]} \cap \rpp{\MM[-i]}$ from Lemma \ref{append_omega_sigma}.
From the long exact sequence, 
$Y_i \in \bcap{0 \leq l \leq i-2} \rpp{\MM[-l]}$.
Thus, $X_i \in \XX_i$ and $Y_i \in \YY_i$.

(2) Dual of (1).
\end{proof}

\begin{rmk} \cite[Proposition 2.1]{IY08} \label{rmk_direct-summand_SM}
Let $\XX$ and $\YY$ be subcategories of $\CC$ which is closed under direct summands.
If $\CC(\XX, \YY) =0$, then $\XX \ast \YY$ is closed under direct summands from \cite[Proposition 2.1]{IY08}.
Thus, for $0 \leq i < n$,
\[
\add(\bracket{\MM[i]} \ast \cdots \ast \bracket{\MM[1]} \ast \bracket{\MM}) =
\bracket{\MM[i]} \ast \cdots \ast \bracket{\MM[1]} \ast \bracket{\MM}.
\]
\end{rmk}

\begin{cor} \label{cor_characterize_simple-minded system}
Let $\MM$ be an $n$-orthogonal collection for $n < \infty$. Then the following are equivalent.
	\begin{enumerate}
	\item $\bracket{\MM}$ is functorially finite in $\CC$ and 
	$\bracket{\MM} = \bcap{0<i<n} \lpp{\MM[-i]} = \bcap{0<i<n} \rpp{\MM[i]}$.
	\item $\bracket{\MM}$ is functorially finite in $\CC$, $\bbS_{-n}$-stable and
	$\bracket{\MM} = \bcap{0<i<n} \rpp{\MM[i]}$.
	\item $\CC = \bracket{\MM[n-1]} \ast \cdots \ast \bracket{\MM[1]} \ast \bracket{\MM}$.
	\end{enumerate}
\end{cor}
\begin{proof}
Clearly, (2)$\Rightarrow$(1) holds.
(1)$\Rightarrow$(2) follows from the equation below.
\[
\bbS_{-n}\bracket{\MM} = \bbS \Big(\, \bcap{0<i<n} \lpp{\MM[n-i]} \Big) = \bcap{0<i<n} \!\rpp{\MM[n-i]}
= \bracket{\MM}
\]
(1)$\Leftrightarrow$(3) follows from Lemma \ref{lem_ZZ_funct-fin_SM} and Remark \ref{rmk_direct-summand_SM}.
\end{proof}

\begin{lem} \label{lem_torsion_pair_SMC}
Let $\MM$ be an $\infty$-orthogonal collection. 
We assume the following conditions.
	\begin{enumerate}[label=(\roman*), leftmargin=30pt]
	\item $\bracket{\MM}$ is covariantly finite in $\bcap{i > 0} \lpp{\MM[-i]}$.
	\item $\bracket{\MM}$ is contravariantly finite in $\bcap{i>0}\, \rpp{\MM[i]}$.
	\item $\CC(\MM[i], C) = \CC(C, \MM[-i]) =0$ for $C \in \CC, i \gg 0$.
	\end{enumerate}
Then the following statements hold.
	\begin{enumerate}
	\item \label{lem_torsion_pair_SMC_1}
	$\Big( (\thick \MM)^{\leq 0}, \, \bcap{i \geq 0}\, \rpp{\MM[i]} \Big)$ 
	is a torsion pair of $\CC$.
	\item \label{lem_torsion_pair_SMC_2}
	$\Big( \, \bcap{i \geq 0} \lpp{\MM[-i]}, \, (\thick \MM)^{\geq 0} \Big)$ 
	is a torsion pair of $\CC$.
	\item \label{lem_torsion_pair_SMC_3}
	$\Big(\Big( (\thick \MM)^{<0}, \, \bcap{i \geq 0}\, \rpp{\MM[i]} \Big), 
	\Big( \, \bcap{i \geq 0} \lpp{\MM[-i]}, \, (\thick \MM)^{>0} \Big)\Big)$
	is a Hovey twin cotorsion pair of $\CC$.
	\end{enumerate}
\end{lem}
\begin{proof}
From \cite[Proposition 3.2]{Jin23} and by definition of Hovey twin cotorsion pairs.
\end{proof}

\begin{cor} \label{cor_characterize_simple-minded collection}
Let $\MM$ be an $\infty$-orthogonal collection. 
Then the following are equivalent.
	\begin{enumerate}
	\item $\CC = \thick{\MM}$.
	\item $\CC = \bcup{i \geq 0} \bracket{\MM[i]} \ast \bracket{\MM[i-1]} \ast \cdots \ast \bracket{\MM[1-i]} \ast \bracket{\MM[-i]}$.
	\item 
		\begin{enumerate}
		\item $\bracket{\MM} = \bcap{i > 0} (\lpp{\MM[-i]} \cap \rpp{\MM[i]})$.
		\item $\bracket{\MM}$ is covariantly finite in $\bcap{i > 0} \lpp{\MM[-i]}$.
		\item $\bracket{\MM}$ is contravariantly finite in $\bcap{i>0}\, \rpp{\MM[i]}$.
		\item $\CC(\MM[i], C) = \CC(C, \MM[-i]) =0$ for $C \in \CC, i \gg 0$.
		\end{enumerate}
	\end{enumerate}
\end{cor}
\begin{proof}
(1) $\Leftrightarrow$ (2) follows from \cite[Proposition 2.6]{Jin23}.

Assume (2). Then
\begin{align*}
\Big( \bcup{i \geq 0} \bracket{\MM[i]} \ast \cdots \ast \bracket{\MM[1]} \ast \bracket{\MM}, 
\,\bcup{i > 0} \bracket{\MM[-1]} \ast \cdots \ast \bracket{\MM[-i]} \Big) \\
\Big( \bcup{i > 0} \bracket{\MM[i]} \ast \cdots \ast \bracket{\MM[1]}, 
\,\bcup{i \geq 0} \bracket{\MM} \ast \bracket{\MM[-1]} \ast \cdots \ast \bracket{\MM[-i]} \Big)
\end{align*}
are torsion pairs of $\CC$. Thus,
\begin{align*}
\bcap{i>0} \lpp{\MM[-i]} &= \bcup{i \geq 0} \bracket{\MM[i]} \ast \cdots \ast \bracket{\MM[1]} \ast \bracket{\MM} \\
\bcap{i>0} \, \rpp{\MM[i]} &= \bcup{i \geq 0} \bracket{\MM} \ast \bracket{\MM[-1]} \ast \cdots \ast \bracket{\MM[-i]}.
\end{align*}
So, $(\mathrm{i}), (\mathrm{ii})$ and $(\mathrm{iii})$ hold.
Since $\MM$ is $\infty$-orthogonal, 
$\CC(\MM[i], C) = \CC(C, \MM[-i]) =0$ for $C \in \CC, i \gg 0$.
Thus, (2)$\Rightarrow$(3) holds.

Assume (3). Then
\[
\bigg(\Big( (\thick \MM)^{< -1}, \, \bcap{i > 0}\, \rpp{\MM[i]} \Big),
\Big( \, \bcap{i > 0} \lpp{\MM[-i]}, \, (\thick \MM)^{> 1} \Big)\bigg)
\]
is a concentric twin cotorsion pair. Thus, 
\begin{align*}
\CC &= (\thick \MM)^{< -1}[-1] \ast \bcap{i > 0} ( \lpp{\MM[-i]} \cap \rpp{\MM[i]})
\ast (\thick \MM)^{> 1}[1] \\
&= (\thick \MM)^{< 0} \ast \bcap{i > 0} ( \lpp{\MM[-i]} \cap \rpp{\MM[i]})
\ast (\thick \MM)^{> 0} \\
&= \bcup{i \geq 0} \bracket{\MM[i]} \ast \cdots \ast \bracket{\MM[1]} \ast \bracket{\MM}
\ast \bracket{\MM[-1]} \ast \cdots \ast \bracket{\MM[-i]}.
\end{align*}

Therefore, (3)$\Rightarrow$(2) holds.
\end{proof}

\subsection{Reduction theory for rigid subcategories} \label{reduction_of_rigid}
All results and proofs in this subsection are already known.
However, we have organized this subsection to explicitly contract with the next subsection, which is about simple-minded systems and collections.

In this subsection, let $\II$ be an $n$-rigid subcategory in $\CC$ for $1 < n \leq \infty$.

We assume that
	\begin{enumerate}[leftmargin=30pt, label=(\roman*)]
	\item For $1<n<\infty$, $\II$ is functorially finite in $\CC$.
	\item For $n = \infty$, $\II$ is covariantly finite in $\bcap{i>0} \lpp{\II[i]}$ and 
	$\II$ is contravariantly finite in $\bcap{i>0} \, \rpp{\II[-i]}$.
	\end{enumerate}
Let $\ZZ$ be an extension closed subcategory in $\CC$ which contains $\II$.
We denote the set of the isomorphism classes of indecomposable objects in $\II$ by $\II^{\pr}$.
Note that $\II = \add \II^{\pr} = \bracket{\II^{\pr}}^{\II}_{\II}$.
We consider the following conditions.

\begin{condi}
\phantom{X}
	\begin{itemize}[leftmargin=50pt]
	\item[($n$-PS)] Either of the following conditions holds.
		\begin{enumerate}[label=(\roman*)]
		\item $n < \infty$ and $\II$ is $\bbS_n$-stable.
		\item $n = \infty$ and $\CC(\II[-i], C) = \CC(C, \II[i]) =0$ for $C \in \CC, i \gg 0$.
		\end{enumerate}
	\item[($n$-wRG)] $\ZZ \subset \bcap{0<i<n} \lpp{\II[i]}$ or $\ZZ \subset \bcap{0<i<n} \rpp{\II[-i]}$ holds.
	\item[($n$-RG)] $\ZZ = \bcap{0<i<n} (\lpp{\II[i]} \cap \rpp{\II[-i]})$.
	\end{itemize}
\end{condi}

\begin{lem} \label{rmk_ex-of-rigid-mutation-tri}
	\begin{enumerate}
	\item If ($n$-PS) and ($n$-RG) hold with $n < \infty$, 
	then $(\II, \ZZ, \II)$ is a rigid mutation triple.
	\item \label{lem_ex-of-rigid-mutation-tri_2}
	If ($\infty$-PS) and ($\infty$-RG) holds, then 
	$(\II, \ZZ, \II)$ is a rigid mutation triple.
	\end{enumerate}
\end{lem}
\begin{proof}
(1) From Example \ref{ex_rigid_mutation_pair}, \ref{ex_redpreMT}(\ref{ex_redpreMT_rigid}) and \ref{ex_redMT}(\ref{ex_redMT_rigid}).

(2) From Example \ref{ex_redMT}(\ref{ex_redMT_rigid}),
Lemma \ref{lem_redMT}(\ref{lem_redMT_presilt}) and \ref{lem_torsion_pair_silt}(\ref{lem_torsion_pair_silt_3}).
\end{proof}

In the rest of this subsection, we assume that $(\ul{\ZZ}, \bracket{1}, \nabla)$ is a triangulated category 
which is induced by a rigid mutation triple $(\II, \ZZ, \II)$.
Note that $\ul{\ZZ}$ is Hom-finite and Krull-Schmidt since so is $\CC$.

Our main interest of this subsection is to consider the following correspondence of subcategories.
\begin{gather}
\ul{(\cdot)} \colon \{ \XX \mid \II \subset \XX \text{ : subcategory in $\ZZ$}\} \to 
\{ \ul{\XX} \text{ : subcategory of } \ul{\ZZ} \} \label{corres_underbar}
\end{gather}

We prepare some lemmas as follows.

\begin{lem} \label{lem_hom-vanish_RG} \cite[Lemma 4.8]{IY18}
Assume ($n$-wRG). 
	\begin{enumerate}
	\item \label{lem_hom-vanish_RG_1}
	Let $0 < i < n$ and $Z_1,Z_2 \in \ul{\ZZ}$.
		\begin{enumerate}
		\item $\ul{\ZZ}(Z_1, Z_2\bracket{i}) \cong \CC(Z_1, Z_2[i])$.
		\item $\ul{\ZZ}(Z_1\bracket{-i}, Z_2) \cong \CC(Z_1[-i], Z_2)$.
		\end{enumerate}
	\item \label{lem_hom-vanish_RG_2}
	Let $\XX$ be a subcategory which satisfies $\II \subset \XX \subset \ZZ$. 
	We define ${}^{\perp_{\CC}} \XX[i] = \{C \in \CC \mid \CC(C, \XX[i]) =0 \}$ and 
	${}^{\perp_{\ul{\ZZ}}} \ul{\XX}\bracket{i} = \{Z \in \ul{\ZZ} \mid \ul{\ZZ}(Z, \ul{\XX}\bracket{i}) =0 \}$. Then 
	\[
	\ul{\ZZ \cap \Big(\, \bcap{0<i<n} {}^{\perp_{\CC}}\XX[i]} \Big)
	= \bcap{0<i<n} {}^{\perp_{\ul{\ZZ}}} \ul{\XX}\bracket{i},
	\]
	\[
	\ul{\ZZ \cap \Big( \, \bcap{0<i<n} {\XX[-i]}^{\perp_{\CC}}} \Big)
	= \bcap{0<i<n} {\ul{\XX}\bracket{-i}}^{\perp_{\ul{\ZZ}}}.
	\]
	In particular,
	\[
	\ul{\ZZ \cap \Big(\, \bcap{0<i<n} ({}^{\perp_{\CC}}\XX[i] \cap {\XX[-i]}^{\perp_{\CC}}) \Big)} 
	= \bcap{0<i<n} ({}^{\perp_{\ul{\ZZ}}} \ul{\XX}\bracket{i} \cap {\ul{\XX}\bracket{-i}}^{\perp_{\ul{\ZZ}}}).
	\]
	Moreover, if ($n$-RG) and ($n$-PS) holds, we may drop ``$\ZZ \, \cap\,$'' in the left-hand side.
	\end{enumerate}
\end{lem}
\begin{proof}
(1) From \cite[Lemma 4.8]{IY08}. 
(2) Since $\II \subset \XX$, this is direct from (1).
Note that $\ZZ = \bcap{0<i<n} \, \lpp{\II[i]} = \bcap{0<i<n} \, \rpp{\II[-i]}$ hold under ($n$-RG) and ($n$-PS).
\end{proof}

\begin{cor} \label{cor_n-rigid_preserve}
Assume ($n$-wRG).
	\begin{enumerate}
	\item \label{cor_n-rigid_preserve_1} \cite[Theorem 4.9]{IY08}
	For a subcategory $\XX$ which satisfies $\II \subset \XX \subset \ZZ$,
	\[
	\XX \text{ is $n$-rigid in }\CC \iff \ul{\XX} \text{ is $n$-rigid in } \ul{\ZZ}.
	\]
	\item \label{cor_n-rigid_preserve_2}
	Assume ($\infty$-PS) and ($\infty$-RG). 
	Let $\XX$ be an $\infty$-rigid subcategory containing $\II$.
	Then $\XX \subset \ZZ$ and
	\begin{align*}
		\begin{gathered}
		\CC(\XX[-i], C) = \CC(C, \XX[i]) =0 \\
		\text{ for } C \in \CC, i \gg 0.
		\end{gathered}
		\iff
		\begin{gathered}
		\ul{\ZZ}(\ul{\XX}\bracket{-i}, Z) = \ZZ(Z, \ul{\ZZ}\bracket{i}) =0 \\
		\text{ for } Z \in \ul{\ZZ}, i \gg 0.
		\end{gathered}
	\end{align*}
	\end{enumerate}
\end{cor}
\begin{proof}
(1) From Lemma \ref{lem_hom-vanish_RG}(1).

(2) The left-hand side clearly implies the right-hand side from Lemma \ref{lem_hom-vanish_RG}(1).
Let $C \in \CC$. From Lemma \ref{lem_torsion_pair_silt}(3), $\CC = (\thick \II)_{>0} \ast \ZZ \ast (\thick \II)_{<0}$.
Then there exists triangles $I_{>0} \rar C \rar C^{\pr} \rar I_{>0}[1]$ and 
$I_{<0}[-1] \rar Z_C \rar C^{\pr} \rar I_{<0}$
where $I_{>0} \in (\thick \II)_{>0}, I_{<0} \in (\thick \II)_{<0}$ and $Z_C \in \ZZ$.
Thus, for $i \gg 0$, 
$\CC(C, \XX[i]) \cong \CC(Z_C, \XX[i]) \cong \ul{\ZZ}(Z_C, \ul{\XX}\bracket{i}) =0$.
In the same way, one can show $\CC(\XX[-i], C) =0$ for $i \gg 0$.
\end{proof}

\begin{lem} \cite[Theorem 4.7, 4.9]{IY08} \label{lem_iroiro_CTL}
Let $1 < n <\infty$. Assume ($n$-PS) and ($n$-RG).
	\begin{enumerate}
	\item \label{lem_iroiro_CTL_1}
	$\bbS^{\pr} = \bracket{n} \circ \bbS_n$ is a Serre functor of $\ul{\ZZ}$.
	\item \label{lem_iroiro_CTL_2}
	Let $\XX$ be a subcategory which satisfies $\II \subset \XX \subset \ZZ$.
	Then
	\begin{gather*}
	\XX \text{ is } \bbS_n \text{-stable in }\CC \iff
	\ul{\XX} \text{ is } {\bbS^{\pr}}\!_n \text{-stable in }\ul{\ZZ}. \\
	\XX \text{ is functorially finite in } \ZZ \iff 
	\ul{\XX} \text{ is functorially finite in } \ul{\ZZ}.
	\end{gather*}
	\end{enumerate}
\end{lem}
\begin{proof}
(1) From \cite[Theorem 4.7]{IY08}.

(2) $\bbS_n$-stability is direct from (1).
Clearly, $\ul{\XX}$ is functorially finite in $\ul{\ZZ}$ if $\XX$ is functorially finite in $\ZZ$.
Assume that $\ul{\XX}$ is functorially finite in $\ul{\ZZ}$.
Let $Z \in \ZZ$ and take a right $\ul{\XX}$-approximation $g \colon X \to Z$.
Since $\II$ is also functorially finite in $\ZZ$, there exists a right $\II$-approximation $i \colon I \to Z$.
Then one can show that $\msize{0.8}{\begin{bmatrix} g & i  \end{bmatrix}} \colon X \oplus I \to Z$
is a right $\XX$-approximation.
\end{proof}

\begin{rmk}
The proof of Lemma \ref{lem_iroiro_CTL}(\ref{lem_iroiro_CTL_1}) is not as easy as that of Lemma \ref{lem_Serre_SM}.
That is because Lemma \ref{lem_hom-vanish_SM}(\ref{lem_hom-vanish_SM_1}) holds for $0 \leq i \leq n$ 
but Lemma \ref{lem_hom-vanish_RG}(\ref{lem_hom-vanish_RG_1}) holds for $0 < i < n$.
\end{rmk}

\begin{thm} \cite[Theorem 4.9]{IY08} \label{thm_reduction_nCT}
Assume ($n$-PS) and ($n$-RG). Let $1 < n <\infty$.
Then the following correspondence of subcategories induced by (\ref{corres_underbar}) is bijective.
	\[
	\{ \XX : n\text{-cluster tilting in }\CC \mid \II \subset \XX\} \to
	\{ \ul{\XX} : n\text{-cluster tilting in }\ul{\ZZ}\}
	\]
\end{thm}
\begin{proof}
From Corollary \ref{cor_characterize_CTL}, \ref{cor_n-rigid_preserve}(1), 
Lemma \ref{lem_hom-vanish_RG}(\ref{lem_hom-vanish_RG_2}) and \ref{lem_iroiro_CTL}(2).
\end{proof}

\begin{lem} \label{lem_cov-cont-finiteness_silt}
Assume ($\infty$-PS) and ($\infty$-RG).
Let $\XX$ be a subcategory which satisfies $\II \subset \XX \subset \ZZ$. Then
\begin{align*}
\XX &\text{ is covariantly finite in } \bcap{i>0} {}^{\perp_{\CC}}\XX[i] \\
&\iff \ul{\XX} \text{ is covariantly finite in } \bcap{i>0} {}^{\perp_{\ul{\ZZ}}} \ul{\XX}\bracket{i}. \\
\XX &\text{ is contravariantly finite in } \bcap{i>0}\, {\XX[-i]}^{\perp_{\CC}} \\
&\iff \ul{\XX} \text{ is contravariantly finite in } \bcap{i>0}\, {\ul{\XX}\bracket{-i}}^{\perp_{\ul{\ZZ}}}.
\end{align*}
\end{lem}
\begin{proof}
This can be shown in a similar way to the proof of Lemma \ref{lem_iroiro_CTL}(\ref{lem_iroiro_CTL_2}).
\end{proof}

\begin{thm} \cite[Theorem 3.7]{IY18} \label{thm_reduction_silt}
Assume ($\infty$-PS) and ($\infty$-RG). 
Then the following correspondence of subcategories induced by (\ref{corres_underbar}) is bijective.
	\[
	\{ \XX : \text{silting subcategory in }\CC \mid \II \subset \XX\} \to
	\{ \ul{\XX} : \text{silting subcategory in }\ul{\ZZ}\}
	\]
\end{thm}
\begin{proof}
From Corollary \ref{cor_characterize_silt}, \ref{cor_n-rigid_preserve}, Lemma \ref{lem_hom-vanish_RG}(\ref{lem_hom-vanish_RG_2}) and \ref{lem_cov-cont-finiteness_silt}.
\end{proof}

\begin{rmk} \cite[Definition 2.5]{IY08}, \cite[Definition 2.30]{AI12} \label{rmk_mutations-generalization}
Let $\XX$ be an $n$-cluster tilting subcategory or a silting subcategory in $\CC$ and 
$\II$ be a subcategory which is contained in $\XX$ and satisfies ($n$-PS).
Then $\XX$ and $\II$ determine a rigid mutation triple $(\II, \ZZ, \II)$ so that $\ZZ$ and $\II$ satisfy ($n$-RG)
from Lemma \ref{rmk_ex-of-rigid-mutation-tri}.
We denote the set of the isomorphism classes of indecomposable objects in $\XX$ \resp{$\II$} by
$\XX^{\pr}$ \resp{$\II^{\pr}$}.

$\add(\mu^-_{\II^{\pr}}(\XX^{\pr}))$ \resp{$\add(\mu^+_{\II^{\pr}}(\XX^{\pr}))$} in Definition \ref{defi_mu-rMT} 
is exactly $\mu^{-1}(\XX; \II)$ \resp{$\mu(\XX; \II)$} in \cite{IY08} if $\XX$ is an $n$-cluster tilting subcategory, 
and exactly $\mu^+(\XX; \II)$ \resp{$\mu^-(\XX; \II)$} in \cite{AI12} if $\XX$ is a silting subcategory.
\end{rmk}

\begin{cor} \cite[Theorem 5.1]{IY08}, \cite[Theorem 2.31]{AI12}
Let $\XX$ be an $n$-cluster tilting subcategory \resp{a silting subcategory}.
In Remark \ref{rmk_mutations-generalization}, 
both $\add(\mu_{\II^{\pr}}^+(\XX^{\pr}))$ and $\add(\mu_{\II^{\pr}}^-(\XX^{\pr}))$ are $n$-cluster tilting subcategories \resp{silting subcategories}.
\end{cor}
\begin{proof}
Since $\Sigma$ \resp{$\Omega$} is the shift functor \resp{the inverse of the shift functor} in $\ul{\ZZ}$,
this statement directly follows from Theorem \ref{thm_reduction_nCT} and \ref{thm_reduction_silt}.
\end{proof}

\begin{ex}
\begin{enumerate}
\item
The following picture is the AR quiver of $\CC = \Db(k A_4)/\nu[-4]$ 
where $\nu$ is the Nakayama functor and $A_4$ is the linearly oriented Dynkin quiver of type $A_4$
(we omitted AR translations).
Let $\II^{\pr} = \{I_1, I_2\}$.
The fundamental domain is shaded light gray and a functorially finite 4-rigid subcategory
$\II = \add \II^{\pr}$ in $\CC$ is shaded dark gray.

Take a 4-cluster tilting subcategory $\XX = \add \XX^{\pr}$ where $\XX^{\pr} = \{X_1, X_2\} \cup \II^{\pr}$.
Then $\mu^-_{\II^{\pr}} \circ \mu^-_{\II^{\pr}}(\XX^{\pr}) \setminus \II^{\pr} $, 
$\mu^-_{\II^{\pr}}(\XX^{\pr})\setminus \II^{\pr}$ and 
$\mu^+_{\II^{\pr}}(\XX^{\pr})\setminus \II^{\pr}$ are indicated in 
dark blue, light blue and red in the following picture, respectively.
\[
\begin{tikzpicture}
	\path[rounded corners,fill=gray,opacity=0.2] (0.75, -0.5) -- (2.75, 3.5) -- (9.25, 3.5) -- (11.25, -0.5) --cycle;
	\fill[rounded corners,gray,opacity=0.6,rotate around={63.44:(4.5,0)}] (4.25,-0.5) rectangle ++(1.7,0.8);
	\foreach \x in {0,1,...,11}
		\foreach \y in {1,3}
		{
		\fill (\x,\y) circle[radius=0.04cm];
		\fill (\x+0.5,\y-1) circle[radius=0.04cm];
		\draw[-{stealth[scale=3]}] (\x,\y) -- (\x+0.48,\y-0.96);
		\draw[-{stealth[scale=3]}] (\x,1) -- (\x+0.48,1.96);
		\ifnum \x<11
			\draw[-{stealth[scale=3]}] (\x+0.5,\y-1) -- (\x+0.98,\y-0.04);
			\draw[-{stealth[scale=3]}] (\x+0.5,2) -- (\x+0.98,1.04);
		\fi
		}
	\fill [red]
	(4, 3) circle[radius=0.08cm]
	(4.5, 2) circle[radius=0.08cm];
	\fill [black] 
	(4.5,0) circle[radius=0.08cm]
	(5,1) circle[radius=0.08cm]
	(5.5,2) circle[radius=0.08cm]
	(6,3) circle[radius=0.08cm];
	\fill [blue!50!white] 
	(8, 1) circle[radius=0.08cm] 
	(8.5, 0) circle[radius=0.08cm];
	\fill [blue]
	(2, 1) circle[radius=0.08cm]
	(2.5, 0) circle[radius=0.08cm];
	\node [text=red] at (4,3) [right] {\msize{0.7}{\Omega X_1}};
	\node [text=red] at (4.5,2) [right] {\msize{0.7}{\Omega X_2}};
	\node [text=black] at (5.5,2) [right] {\msize{0.7}{X_1}};
	\node [text=black] at (6,3) [right] {\msize{0.7}{X_2}};
	\node [text=blue!50!white] at (8,1) [right] {\msize{0.7}{\Sigma X_1}};
	\node [text=blue!50!white] at (8.5,0) [right] {\msize{0.7}{\Sigma X_2}};
	\node [text=blue] at (2,1) [right] {\msize{0.7}{\Sigma^2 X_1}};
	\node [text=blue] at (2.5,0) [right] {\msize{0.7}{\Sigma^2 X_2}};
	\node (D) at (4.8,0.4) [right] {\msize{0.7}{\II}};
	\node at (4.5,0) [right] {\msize{0.7}{I_1}};
	\node at (5,1) [right] {\msize{0.7}{I_2}};
\end{tikzpicture}
\]
From \cite[Proposition 5.10]{INP24}, the AR quiver of triangulated category $\ul{\ZZ} = \bcap{0<i<4}\, \ul{\lpp{\II[i]}}$ is illustrated as follows (we omitted AR translations).
\[
\begin{tikzpicture}
	\path[rounded corners,fill=gray,opacity=0.2] (0.75, -0.5) -- (2.75, 3.5) -- (9.25, 3.5) -- (11.25, -0.5) --cycle;
	\foreach \x in {-1,0,5}
	{
	\fill 
	(5+\x,3) circle[radius=0.04cm]
	(5.5+\x,2) circle[radius=0.04cm]
	(6+\x,3) circle[radius=0.04cm];
	\draw[-{stealth[scale=3]}] (5+\x,3) -- (5.48+\x,2.04);
	\draw[-{stealth[scale=3]}] (5.5+\x,2) -- (5.98+\x,2.96);
	}
	\foreach \x in {0,6}
	{
	\fill
	(1.5+\x,0) circle[radius=0.04cm]
	(2+\x,1) circle[radius=0.04cm]
	(2.5+\x,0) circle[radius=0.04cm];
	\draw[-{stealth[scale=3]}] (1.5+\x,0) -- (2+\x,0.96);
	\draw[-{stealth[scale=3]}] (2+\x,1) -- (2.5+\x,0.04);
	\draw[-{stealth[scale=3]}] (2.5+\x,0) -- (4+\x,2.96);
	\draw[-{stealth[scale=3]}] (0+\x,3) -- (1.5+\x,0.04);
	}
	\fill (0,3) circle[radius=0.04cm];
	\draw[-{stealth[scale=3]}] (11,3) -- (11.5,2.04);
	\fill [red]
	(4, 3) circle[radius=0.08cm]
	(4.5, 2) circle[radius=0.08cm];
	\fill [black] 
	(5.5,2) circle[radius=0.08cm]
	(6,3) circle[radius=0.08cm];
	\fill [blue!50!white] 
	(8, 1) circle[radius=0.08cm] 
	(8.5, 0) circle[radius=0.08cm];
	\fill [blue]
	(2, 1) circle[radius=0.08cm]
	(2.5, 0) circle[radius=0.08cm];
	\node [text=red] at (4,3) [right] {\msize{0.7}{\Omega X_1}};
	\node [text=red] at (4.5,2) [right] {\msize{0.7}{\Omega X_2}};
	\node [text=black] at (5.5,2) [right] {\msize{0.7}{X_1}};
	\node [text=black] at (6,3) [right] {\msize{0.7}{X_2}};
	\node [text=blue!50!white] at (8,1) [right] {\msize{0.7}{\Sigma X_1}};
	\node [text=blue!50!white] at (8.5,0) [right] {\msize{0.7}{\Sigma X_2}};
	\node [text=blue] at (2,1) [right] {\msize{0.7}{\Sigma^2 X_1}};
	\node [text=blue] at (2.5,0) [right] {\msize{0.7}{\Sigma^2 X_2}};
	\node at (11.5,0) [above] {};
	\node at (0,3) [below] {};
\end{tikzpicture}
\]
\item
The following picture is the AR quiver of $\CC = \Db(k A_4)$ 
where $A_4$ is the linearly oriented Dynkin quiver of type $A_4$
(we omitted AR translations).
Let $\II^{\pr} = \{I_1, I_2\}$.
A presilting subcategory $\II = \add \II^{\pr}$ is shaded gray and 
take a silting subcategory $\XX = \add \XX^{\pr}$ where $\XX^{\pr} =  \{ X_1, X_2 \} \cup \II^{\pr}$.
Then $\mu^-_{\II^{\pr}} \circ \mu^-_{\II^{\pr}}(\XX^{\pr}) \setminus \II^{\pr} $, 
$\mu^-_{\II^{\pr}}(\XX^{\pr})\setminus \II^{\pr}$, 
$\mu^+_{\II^{\pr}}(\XX^{\pr})\setminus \II^{\pr}$ and 
$\mu^+_{\II^{\pr}} \circ \mu^+_{\II^{\pr}}(\XX^{\pr})\setminus \II^{\pr}$ 
are indicated in dark blue, light blue, light red and dark red in the following picture, respectively.
\[
\begin{tikzpicture}
	\fill[rounded corners,gray,opacity=0.6,rotate around={63.44:(4.5,0)}] (4.25,-0.5) rectangle ++(1.7,0.8);
	\foreach \x in {0,1,...,11}
		\foreach \y in {1,3}
		{
		\fill (\x,\y) circle[radius=0.04cm];
		\fill (\x+0.5,\y-1) circle[radius=0.04cm];
		\draw[-{stealth[scale=3]}] (\x,\y) -- (\x+0.48,\y-0.96);
		\draw[-{stealth[scale=3]}] (\x,1) -- (\x+0.48,1.96);
		\ifnum \x<11
			\draw[-{stealth[scale=3]}] (\x+0.5,\y-1) -- (\x+0.98,\y-0.04);
			\draw[-{stealth[scale=3]}] (\x+0.5,2) -- (\x+0.98,1.04);
		\fi
		}
	\fill [red]
	(2, 1) circle[radius=0.08cm]
	(1.5, 0) circle[radius=0.08cm];
	\fill [red!50!white]
	(4, 3) circle[radius=0.08cm]
	(4.5, 2) circle[radius=0.08cm];
	\fill [black] 
	(4.5,0) circle[radius=0.08cm]
	(5,1) circle[radius=0.08cm]
	(5.5,2) circle[radius=0.08cm]
	(6,3) circle[radius=0.08cm];
	\fill [blue!50!white] 
	(8, 1) circle[radius=0.08cm] 
	(8.5, 0) circle[radius=0.08cm];
	\fill [blue] 
	(10.5,2) circle[radius=0.08cm] 
	(11,3) circle[radius=0.08cm];
	\node [text=red] at (2,1) [right] {\msize{0.7}{\Omega^2 X_2}};
	\node [text=red] at (1.5,0) [right] {\msize{0.7}{\Omega^2 X_1}};
	\node [text=red!50!white] at (4,3) [right] {\msize{0.7}{\Omega X_1}};
	\node [text=red!50!white] at (4.5,2) [right] {\msize{0.7}{\Omega X_2}};
	\node [text=black] at (5.5,2) [right] {\msize{0.7}{X_1}};
	\node [text=black] at (6,3) [right] {\msize{0.7}{X_2}};
	\node [text=blue!50!white] at (8,1) [right] {\msize{0.7}{\Sigma X_1}};
	\node [text=blue!50!white] at (8.5,0) [right] {\msize{0.7}{\Sigma X_2}};
	\node [text=blue] at (10.5,2) [right] {\msize{0.7}{\Sigma^2 X_1}};
	\node [text=blue] at (11,3) [right] {\msize{0.7}{\Sigma^2 X_2}};
	\node (D) at (4.8,0.4) [right] {\msize{0.7}{\II}};
	\node at (4.5,0) [right] {\msize{0.7}{I_1}};
	\node at (5,1) [right] {\msize{0.7}{I_2}};
\end{tikzpicture}
\]
From \cite[Proposition 5.10]{INP24}, the AR quiver of triangulated category 
$\ul{\ZZ} = \bcap{i>0}\, \ul{\big( \lpp{\II[i]} \cap \rpp{\II[-i]} \big)}$ is illustrated as follows
(we omitted AR translations).
\[
\begin{tikzpicture}
	\foreach \x in {-1,0,5}
	{
	\fill 
	(5+\x,3) circle[radius=0.04cm]
	(5.5+\x,2) circle[radius=0.04cm]
	(6+\x,3) circle[radius=0.04cm];
	\draw[-{stealth[scale=3]}] (5+\x,3) -- (5.48+\x,2.04);
	\draw[-{stealth[scale=3]}] (5.5+\x,2) -- (5.98+\x,2.96);
	}
	\foreach \x in {0,6}
	{
	\fill
	(1.5+\x,0) circle[radius=0.04cm]
	(2+\x,1) circle[radius=0.04cm]
	(2.5+\x,0) circle[radius=0.04cm];
	\draw[-{stealth[scale=3]}] (1.5+\x,0) -- (2+\x,0.96);
	\draw[-{stealth[scale=3]}] (2+\x,1) -- (2.5+\x,0.04);
	\draw[-{stealth[scale=3]}] (2.5+\x,0) -- (4+\x,2.96);
	\draw[-{stealth[scale=3]}] (0+\x,3) -- (1.5+\x,0.04);
	}
	\fill (0,3) circle[radius=0.04cm];
	\draw[-{stealth[scale=3]}] (11,3) -- (11.5,2.04);
	\fill [red]
	(2, 1) circle[radius=0.08cm]
	(1.5, 0) circle[radius=0.08cm];
	\fill [red!50!white]
	(4, 3) circle[radius=0.08cm]
	(4.5, 2) circle[radius=0.08cm];
	\fill [black]
	(5.5,2) circle[radius=0.08cm]
	(6,3) circle[radius=0.08cm];
	\fill [blue!50!white] 
	(8, 1) circle[radius=0.08cm] 
	(8.5, 0) circle[radius=0.08cm];
	\fill [blue] 
	(10.5,2) circle[radius=0.08cm] 
	(11,3) circle[radius=0.08cm];
	\node [text=red] at (2,1) [right] {\msize{0.7}{\Omega^2 X_2}};
	\node [text=red] at (1.5,0) [right] {\msize{0.7}{\Omega^2 X_1}};
	\node [text=red!50!white] at (4,3) [right] {\msize{0.7}{\Omega X_1}};
	\node [text=red!50!white] at (4.5,2) [right] {\msize{0.7}{\Omega X_2}};
	\node [text=black] at (5.5,2) [right] {\msize{0.7}{X_1}};
	\node [text=black] at (6,3) [right] {\msize{0.7}{X_2}};
	\node [text=blue!50!white] at (8,1) [right] {\msize{0.7}{\Sigma X_1}};
	\node [text=blue!50!white] at (8.5,0) [right] {\msize{0.7}{\Sigma X_2}};
	\node [text=blue] at (10.5,2) [right] {\msize{0.7}{\Sigma^2 X_1}};
	\node [text=blue] at (11,3) [right] {\msize{0.7}{\Sigma^2 X_2}};
	\node at (11.5,0) [above] {};
	\node at (0,3) [below] {};
\end{tikzpicture}
\]
\end{enumerate}
\end{ex}

\subsection{Reduction theory for orthogonal collections} \label{reduction_of_SM}

In this subsection, we consider the reduction theory for simple-minded systems and simple-minded collections.
Although all results in this subsection are also already known, we composed another proof of Theorem 6.6 in \cite{SP20}.

In this subsection, let $\MM \subset \Ob(\CC)$ be an $n$-orthogonal collection for $0 < n \leq \infty$.

We assume that
	\begin{enumerate}[leftmargin=30pt, label=(\roman*)]
	\item For $1 < n < \infty$, $\bracket{\MM}$ is functorially finite in $\CC$.
	\item For $n = \infty$, $\bracket{\MM}$ is covariantly finite in $\bcap{i>0} \lpp{\MM[-i]}$ and 
	$\bracket{\MM}$ is contravariantly finite in $\bcap{i>0} \, \rpp{\MM[i]}$.
	\end{enumerate}
Let $\ZZ$ be an extension closed subcategory of $\CC$.
We consider the following conditions.

\begin{condi}
	\begin{itemize}[leftmargin=50pt]
	\item[]
	\item[($n$-NS)] Either of the following conditions holds.
		\begin{enumerate}[leftmargin=20pt, label=(\roman*)]
		\item $n < \infty$ and $\bracket{\MM}$ is $\bbS_{-n}$-stable.
		\item $n = \infty$ and $\CC(\MM[i], C) = \CC(C, \MM[-i]) =0$ for $C \in \CC, i \gg 0$.
		\end{enumerate}
	\item[($n$-wSM)] $\ZZ \subset \bcap{0 \leq i < n+1} \rpp{\MM[i]}$ or 
	$\ZZ \subset \bcap{0 \leq i < n+1} \lpp{\MM[-i]}$ holds.
	\item[($n$-SM)] 
	$\ZZ = \bcap{0 \leq i < n+1} (\lpp{\MM[-i]} \cap \rpp{\MM[i]})$.
	\end{itemize}
\end{condi}

\begin{lem} \label{lem_SMMT}
	\begin{enumerate}
	\item If both ($n$-NS) and ($n$-SM) hold with $n < \infty$, 
	then $(\bracket{\MM[1]}, \ZZ,$ $\bracket{\MM[-1]})$ is an orthogonal mutation triple.
	\item \label{lem_SMMT_2}
	If both ($\infty$-NS) and ($\infty$-SM) holds, then $(\bracket{\MM[1]}, \ZZ, \bracket{\MM[-1]})$ 
	is an orthogonal mutation triple.
	\end{enumerate}
\end{lem}
\begin{proof}
(1)
From Example \ref{ex_tri_SM}, \ref{ex_redpreMT}(\ref{ex_redpreMT_orth}) and \ref{ex_redMT}(\ref{ex_redMT_orth}).

(2) From Example \ref{ex_redMT}(\ref{ex_redMT_orth}),
Lemma \ref{lem_redMT}(\ref{lem_redMT_preSMC}) and \ref{lem_torsion_pair_SMC}(\ref{lem_torsion_pair_SMC_3}).
\end{proof}

In the rest of this subsection, we assume that $(\ZZ, \Sigma, \nabla)$ is a triangulated category 
which is induced by an orthogonal mutation triple $(\bracket{\MM[1]}, \ZZ, \bracket{\MM[-1]})$.
Note that $\ZZ$ is Hom-finite and Krull-Schmidt since so is $\CC$.

Our main interest of this subsection is to consider the following correspondence of subcategories.
\begin{gather}
(\cdot) \cap \ZZ \colon \{ \XX \mid \MM \subset \XX \subset \Ob(\CC)\} \to 
\{ \XX^{\pr} \mid \XX^{\pr} \subset \Ob(\ZZ) \} \label{corres_cap_ZZ} \\
(\cdot) \cup \MM \colon \{ \XX^{\pr} \mid \XX^{\pr} \subset \Ob(\ZZ) \} \to 
\{ \XX \mid \MM \subset \XX \subset \Ob(\CC)\} \label{corres_cup_MM}
\end{gather}

We prepare some lemmas as follows.

\begin{lem} \label{lem_hom-vanish_SM}
Assume ($n$-wSM). 
	\begin{enumerate}
	\item \label{lem_hom-vanish_SM_1}
	Let $0 \leq i \leq n$ and $Z_1,Z_2 \in \ZZ$.
		\begin{enumerate}
		\item $\ZZ(\Sigma^i Z_1, Z_2) \cong \CC(Z_1[i], Z_2)$.
		\item $\ZZ(Z_1, \Sigma^{-i} Z_2) \cong \CC(Z_1, Z_2[-i])$.
		\end{enumerate}
	\item \label{lem_hom-vanish_SM_2}
	Let $\XX^{\pr} \subset \ZZ$. 
	We define ${\XX^{\pr}[i]}^{\perp_{\CC}} = \{C \in \CC \mid \CC(\XX^{\pr}[i], C) =0\}$ and 
	$\Sigma^i {\XX^{\pr}}^{\perp_{\ZZ}} = \{Z \in \ZZ \mid \ZZ(\Sigma^i \XX^{\pr}, Z) =0\}$. Then 
	\[
	\ZZ \cap \Bigl( \ \bcap{0<i<n} {\XX^{\pr}[i]}^{\perp_{\CC}} \Bigr) = \bcap{0<i<n} \Sigma^i {\XX^{\pr}}^{\perp_{\ZZ}},
	\]
	\[
	\ZZ \cap \Bigl( \ \bcap{0<i<n} {}^{\perp_{\CC}}\XX^{\pr}[-i] \Bigr) = \bcap{0<i<n} {}^{\perp_{\ZZ}} \Sigma^{-i} \XX^{\pr}.
	\]
	In particular,
	\[
	\ZZ \cap \Bigl( \ \bcap{0<i<n} ({}^{\perp_{\CC}}\XX^{\pr}[-i] \cap {\XX^{\pr}[i]}^{\perp_{\CC}}) \Bigr) 
	= \bcap{0<i<n} ({}^{\perp_{\ZZ}} \Sigma^{-i}\XX^{\pr} \cap \Sigma^i {\XX^{\pr}}^{\perp_{\ZZ}}).
	\]
	\end{enumerate}
\end{lem}
\begin{proof}
(1) Assume that $\ZZ \subset \bcap{0 \leq i \leq n} \rpp{\MM[i]}$. 
By definition of $\Sigma$, there exists the following triangle in $\CC$
where $M \in \bracket{\MM}, Z \in \ZZ$ for $0 \leq l < i$.
\[
M[l] \rar (\Sigma^{i-(l+1)}Z)[l+1] \rar (\Sigma^{i-l}Z)[l] \rar M[l+1]
\]
So, for $0 \leq l < n$ and $l < i$,
\[
\CC((\Sigma^{i-(l+1)}Z)[l+1], \ZZ) \cong \CC((\Sigma^{i-l}Z)[l], \ZZ).
\]
Thus, for $0 \leq i \leq n$,
\begin{align*}
	\ZZ(\Sigma^i Z, \ZZ) &\cong \CC((\Sigma^{i-1}Z)[1], \ZZ) \\
					&\cong \CC((\Sigma^{i-2}Z)[2], \ZZ) \\
					&\cdots \\
					&\cong \CC(Z[i], \ZZ).
\end{align*}
If $\ZZ \subset \bcap{0 \leq i \leq n} \lpp{\MM[-i]}$ holds, one can show $\ZZ(Z_1, \Omega^i Z_2) \cong \CC(Z_1, Z_2[-i])$ similarly.

(2) Direct from (1).
\end{proof}

\begin{cor} \label{cor_n-orth_preserve}
Assume ($n$-wSM).
	\begin{enumerate}
	\item \label{cor_n-orth_preserve_1}
	The correspondence (\ref{corres_cap_ZZ}) preserves $n$-orthogonality, that is, 
	\[
	\XX \text{ is $n$-orthogonal in }\CC \implies \XX \cap \ZZ \text{ is $n$-orthogonal in } \ZZ.
	\]
	\item \label{cor_n-orth_preserve_2} 
	Let $\XX$ be an $\infty$-orthogonal collection satisfying 
	$\CC(\XX[i], C) = \CC(C, \XX[-i]) =0$ for $C \in \CC, i \gg 0$. 
	Then $\XX \cap \ZZ$ is also an $\infty$-orthogonal collection satisfying 
	$\ZZ(\Sigma^i(\XX \cap \ZZ), Z) = \ZZ(Z, \Sigma^{-i}(\XX \cap \ZZ)) =0$ for $Z \in \ZZ, i \gg 0$.
	\end{enumerate}
\end{cor}

\begin{lem} \label{lem_Serre_SM}
Let $0 < n <\infty$ and assume ($n$-wSM).
Then if $\ZZ$ is $\bbS_{-n}$-stable, $\bbS^{\pr\pr} = \Sigma^{-n} \circ \bbS_{-n}$ is a Serre functor of $\ZZ$.
\end{lem}
\begin{proof}
Let $X, Y \in \ZZ$. Note that $\bbS_{-n}X \in \ZZ$ from assumptions.
We denote $k$-dual by $D$.
\begin{align*}
	\ZZ(X, Y) &\cong \CC(X[n], Y[n]) 
			\cong D \hspace{1pt} \CC(Y[n], \bbS_{-n}X) \\
			&\cong D\ZZ(\Sigma^n Y, \bbS_{-n}X) 
			\cong D\ZZ(Y, \Sigma^{-n}(\bbS_{-n}X))
\end{align*}
\end{proof}

The following technical lemma is \cite[Theorem 2.11]{SP20}, or \cite[Theorem 3.3]{Dug15}. 
(We may assume the right minimality of $f$ though it was not stated in \cite{SP20}.)

\begin{lem} \label{lem_extension_resolving} \cite[Theorem 2.11]{SP20}, \cite[Theorem 3.3]{Dug15}
Let $\XX$ be a semibrick containing $\MM$.
For $m \geq 1$, we define $\bracket{\XX}_m$ inductively by 
$\bracket{\XX}_1 = \XX$ and $\bracket{\XX}_m = \bracket{\XX}_{m-1} \ast \XX$.

Assume that $X \in \bracket{\XX}_m$ and $m$ is small as possible.
Then there exists a triangle $M \xrar{f} X \rar X^{\pr} \rar M[1]$ which satisfies the following conditions.
\begin{enumerate}[label=(\roman*), leftmargin=40pt]
\item $f$ is a minimal right $\bracket{\MM}$-approximation.
\item $X^{\pr} \in \bracket{\XX}_{m^{\pr}} \cap \rpp{\MM}$ where $m^{\pr} \leq m$ with equality if and only if $X \cong X^{\pr}$.
\end{enumerate}
In particular, $X^{\pr} \cong \sigma X$ if $X \in \bracket{\MM} \ast \ZZ = \wt{\UU}$.
\end{lem}

\begin{proof}
From Theorem 2.11 in \cite{SP20}, there exists a triangle $M \xrar{f} X \rar X^{\pr} \rar M[1]$ 
which satisfies (i) and (ii) except for right minimality of $f$.
We may assume that
$f = \msize{0.8}{\begin{bmatrix} f^{\pr} & 0 \end{bmatrix}} \colon M_1 \oplus M_2 \to X$
where $f^{\pr}$ is minimal right $\bracket{\MM}$-approximation.
Then the triangle $M \xrar{f} X \rar X^{\pr} \rar M[1]$ is a direct sum of 
$M_1 \xrar{f^{\pr}} X \rar X^{\pr\pr} \rar M_1[1]$ and $M_2 \rar 0 \rar M_2[1] \rar M_2[1]$. 
Take an integer $m^{\pr\pr}$ so that $X^{\pr\pr} \in \bracket{\XX}_{m^{\pr\pr}}$ and 
$m^{\pr\pr}$ is small as possible.
Then $m^{\pr\pr} \leq m^{\pr} \leq m$.
If $m^{\pr\pr} = m$, then $m^{\pr\pr} = m^{\pr}$. So, $M_2 = 0$ and $X^{\pr} \cong X^{\pr\pr}$.
From assumption of $f$ and $m=m^{\pr}$, $X \cong X^{\pr}$.
Therefore, equality if and only if $X \cong X^{\pr\pr}$.

Take a triangle $M^{\pr} \xrar{g} X \rar \sigma X \rar M^{\pr}$ in $\CC$.
Since $\CC(\bracket{\MM}, \ZZ) \cong \bbE(\bracket{\MM[1]}, \ZZ) = 0$,
$g$ is a right $\bracket{\MM}$-approximation.
Then we may assume that $g = \msize{0.8}{\begin{bmatrix} g^{\pr} & 0 \end{bmatrix}}$
where $g^{\pr}$ is right minimal.
Because $\bracket{\MM}$ is closed under direct summands, there exists $M^{\pr\pr} \in \bracket{\MM}$ which
satisfies $\sigma X \cong X^{\pr} \oplus M^{\pr\pr}[1]$.
However, $\ZZ \cap \bracket{\MM[1]} = 0$, $X^{\pr} \cong \sigma X$ follows.
\end{proof}

In the rest of this subsection, we denote the extension closure in $(\ZZ, \Sigma, \nabla)$ by $\{ \cdot \}$.
We define $\{\cdot\}_m$ for $m \geq 1$ same as $\bracket{\cdot}_m$.

\begin{lem} \label{lem_cap_ZZ}
Assume ($n$-NS) and ($n$-SM). 
Let $\XX$ be an $n$-orthogonal collection containing $\MM$.
	\begin{enumerate}
	\item \label{lem_cap_ZZ_1}
	$\XX \cap \ZZ = \XX \setminus \MM$.
	\item \label{lem_cap_ZZ_2}
	$\bracket{\XX} \cap \ZZ = \{ \XX \setminus \MM \} = \{ \XX \cap \ZZ \}$.
	\end{enumerate}
\end{lem}
\begin{proof}
(1) From $\ZZ \subset \dpp{\MM}$, $\XX \cap \ZZ \subset \XX \setminus \MM$. 
On the other hand, since $\XX$ is $n$-orthogonal, $\XX \setminus \MM \subset \bcap{0\leq i<n} (\lpp{\MM[-i]} \cap \rpp{\MM[i]})$.
If $n=\infty$, $\XX \setminus \MM \subset \ZZ$.
If $n < \infty$, $\XX \setminus \MM \subset \lpp{\MM} = \rpp{\MM[n]}$ from ($n$-NS). 
Thus, $\XX \setminus \MM \subset \bcap{0\leq i \leq n} \rpp{\MM[i]} = \ZZ$.

(2) Let $X \in \{ \XX \setminus \MM \}_m$. 
We show that $X \in \bracket{\XX} \cap \ZZ$ by induction on $m$.
If $m =1$, this is direct from (1).
If $m > 1$, we may assume that there exists a nonzero morphism
 $f \colon X_1 \to X$ where $X_1 \in \XX \setminus \MM$ and take a triangle $X_1 \xrar{f} X \rar X^{\pr} \rar \Sigma X_1$.
Then from \cite[Lemma 2.6]{Dug15}, $X^{\pr} \in \{ \XX \setminus \MM \}_{m-1} \subset \bracket{\XX} \cap \ZZ$.
Note that $X_1 \in \XX \setminus \MM = \XX \cap \ZZ$ since $f$ is nonzero.
$f$ is a morphism in $\ZZ$.
From Lemma \ref{lem_SMMT} and the definition of standard right triangles, $\sigma C^f \cong X^{\pr}$.
Thus, $C^f \in \bracket{\XX}$ and so is $X$.

Let $X \in \bracket{\XX}_m \cap \ZZ$. We show $X \subset \{ \XX \setminus \MM \}$ by induction on $m$.
If $m =1$, this is direct from (1).
If $m > 1$, we may assume that there exists a nonzero morphism
$f \colon X_1 \to X$ in $\ZZ$ where $X_1 \in \XX$ and take a triangle $X_1 \xrar{f} X \rar X^{\pr} \rar X_1[1]$.
Note that $X^{\pr} \in \bracket{\XX}_{m-1}$.
This triangle induces $X_1 \xrar{f} X \rar \sigma X^{\pr} \rar \Sigma X_1$.
From Lemma \ref{lem_extension_resolving}, $\sigma X^{\pr} \in \bracket{\XX}_{m^{\pr}}$ where $m^{\pr} < m$.
Thus, $\sigma X^{\pr} \in \{ \XX \setminus \MM \}$ and so is $X$. 
\end{proof}

\begin{prop} \label{prop_cap_ZZ_1}
Assume ($n$-NS) and ($n$-SM) and $0 < n < \infty$. 
	\begin{enumerate}
	\item \label{prop_cap_ZZ_1-1}
	(\ref{corres_cap_ZZ}) preserves functorially finiteness of extension closures, that is, 
	\[
	\bracket{\XX} \text{ is functorially finite in } \CC \implies 
	\{ \XX \cap \ZZ \} \text{ is functorially finite in }\ZZ.
	\]
	\item \label{prop_cap_ZZ_1-2}
	(\ref{corres_cap_ZZ}) preserves $\bbS_{-n}$-stability of extension closures, that is, 
	\[
	\bracket{\XX} \text{ is $\bbS_{-n}$-stable} \implies \{ \XX \cap \ZZ \} \text{ is $\bbS^{\pr\pr}_{-n}$-stable}.
	\]
	\item \label{prop_cap_ZZ_1-3}
	(\ref{corres_cap_ZZ}) preserves $n$-simple-minded systems, that is, 
	\begin{align*}
	\XX \text{ is an } &\text{$n$-simple-minded system in $\CC$ with } [1] \implies \\
	&\XX \cap \ZZ \text{ is an $n$-simple-minded system in $\ZZ$ with } \Sigma.
	\end{align*}
	\end{enumerate}
\end{prop}
\begin{proof}
(1) Let $Z \in \ZZ$. Take a triangle $Z^{\pr} \rar Z \xrar{g} X \rar Z^{\pr}[1]$ 
where $g$ is a minimal left $\bracket{\XX}$-approximation.
From Lemma \ref{append_omega_sigma}, $Z^{\pr} \in \lpp{\MM} \cap \lpp{\MM[-1]}$ and
from the long exact sequence, $Z^{\pr} \in \bcap{1 < i \leq n} \lpp{\MM[-i]}$.
Thus, $Z^{\pr} \in \ZZ$. In particular, $X \in \bracket{\MM} \ast \ZZ$. 
Then from Lemma \ref{lem_extension_resolving}, $\sigma X \in \bracket{\XX} \cap \ZZ = \{ \XX \cap \ZZ\}$.
Since $Z^{\pr} \in \lpp{\XX} \subset \lpp{(\XX \cap \ZZ)}$, $h^X \circ g \colon Z \to \sigma X$ is a left $\{\XX \cap \ZZ\}$-approximation.

(2) From Lemma \ref{lem_Serre_SM} and \ref{lem_cap_ZZ}(\ref{lem_cap_ZZ_2}),
	\begin{align*}
	\bbS_{-n}^{\pr\pr}\{ \XX \cap \ZZ \} &= \bbS^{\pr\pr} \circ \Sigma^n \{ \XX \cap \ZZ \} \cong \Sigma^n \circ 	\bbS^{\pr\pr} \{ \XX \cap \ZZ \} \\
	&\cong \bbS_{-n} (\bracket{\XX} \cap \ZZ) \cong \bracket{\XX} \cap \ZZ = \{ \XX \cap \ZZ\}.
	\end{align*}

(3) From (1), (2), Corollary \ref{cor_characterize_simple-minded system} and \ref{cor_n-orth_preserve}(\ref{cor_n-orth_preserve_1}).
\end{proof}

\begin{prop} \label{prop_cap_ZZ_2}
Assume ($\infty$-NS) and ($\infty$-SM).
Let $\XX$ be an $\infty$-orthogonal collection in $\CC$ and we denote $\XX \cap \ZZ$ by $\XX^{\pr}$.
	\begin{enumerate}
	\item
	(\ref{corres_cap_ZZ}) preserves covariantly finiteness of extension closures, that is, 
	\begin{align*}
	\bracket{\XX} \text{ is covari} &\text{antly finite in } \bcap{i>0} {}^{\perp_{\CC}}{\XX[-i]} \implies \\
	\{ \XX^{\pr} \} \text{ is} &\text{ covariantly finite in } \bcap{i>0} {}^{\perp_{\ZZ}}{\Sigma^{-i} \XX^{\pr}}.
	\end{align*}
	\item
	(\ref{corres_cap_ZZ}) preserves contravariantly finiteness of extension closures, that is, 
	\begin{align*}
	\bracket{\XX} \text{ is contra} &\text{variantly finite in } \bcap{i>0} {\XX[i]}^{\perp_{\CC}} \implies \\
	\{ \XX^{\pr} \} \text{ is} &\text{ contravariantly finite in } \bcap{i>0} {\Sigma^{i} \XX^{\pr}}^{\perp_{\ZZ}}.
	\end{align*}
	\item \label{prop_cap_ZZ_2-3}
	(\ref{corres_cap_ZZ}) preserves simple-minded collections, that is, 
	\begin{align*}
	\XX \text{ is a } &\text{simple-minded collection in $\CC$} \implies \\
	&\XX \cap \ZZ \text{ is a simple-minded collection in $\ZZ$}.
	\end{align*}
	\end{enumerate}
\end{prop}
\begin{proof}
(1) Note that 
\[
\bcap{i>0} {}^{\perp_{\ZZ}}{{\Sigma^{-i} \XX^{\pr}}} 
= \ZZ \cap \Big( \, \bcap{i>0} {}^{\perp_{\CC}}{\XX^{\pr}[-i]} \Big) 
= \ZZ \cap \Big( \, \bcap{i>0} {}^{\perp_{\CC}}{\XX[-i]} \Big)
\]
from Lemma \ref{lem_hom-vanish_SM}(\ref{lem_hom-vanish_SM_2}) and \ref{lem_cap_ZZ}(\ref{lem_cap_ZZ_1}).
Let $Z \in {}^{\perp_{\ZZ}}{{\Sigma^{-i} \XX^{\pr}}}$.
Take a triangle $Z^{\pr} \rar Z \xrar{g} X \rar Z^{\pr}[1]$ where $g$ is a minimal left $\bracket{\XX}$-approximation.
Just like the proof of Proposition \ref{prop_cap_ZZ_1}(\ref{prop_cap_ZZ_1-1}), one can show $Z^{\pr} \in \bcap{i \geq 0} {}^{\perp_{\CC}} \MM[-i]$. 
Since $X[-1] \in \bcap{i \geq 0}\, \MM[i]^{\perp_{\CC}}$ and
$Z \in \ZZ \subset \bcap{i \geq 0}\, \MM[i]^{\perp_{\CC}}$,
$Z^{\pr} \in \bcap{i \geq 0} \MM[i]^{\perp_{\CC}}$. Thus, $Z^{\pr} \in \ZZ$.
Therefore, there exists a morphism $h^X \circ g \colon Z \to \sigma X$ and
this is a left $\{\XX^{\pr}\}$-approximation since $Z^{\pr} \in \lpp{\XX}$.

(2) Dual of (1).

(3) From (1), (2), Corollary \ref{cor_characterize_simple-minded collection}, \ref{cor_n-orth_preserve},
 Lemma \ref{lem_hom-vanish_SM}(\ref{lem_hom-vanish_SM_2}) and \ref{lem_cap_ZZ}(\ref{lem_cap_ZZ_2}).
\end{proof}

Next, we consider the correspondence (\ref{corres_cup_MM}).

\begin{lem} \label{lem_n-orth_lifting}
Assume ($n$-SM).
	\begin{enumerate}
	\item \label{lem_n-orth_lifting_1}
	(\ref{corres_cup_MM}) preserves $n$-orthogonality, that is, 
	\[
	\XX^{\pr} \text{ is $n$-orthogonal in } \ZZ \implies \XX^{\pr} \cup \MM \text{ is $n$-orthogonal in } \CC.
	\]
	\item \label{lem_n-orth_lifting_2} 
	We assume ($\infty$-NS).
	Let $\XX^{\pr}$ be an $\infty$-orthogonal collection with 
	$\ZZ(\Sigma^i \XX^{\pr}, Z) = \ZZ(Z, \Sigma^{-i} \XX^{\pr}) =0$ for $Z \in \ZZ, i \gg 0$.
	Then $\XX = \XX^{\pr} \cup \MM$ is also an $\infty$-orthogonal collection with
	$\CC(\XX[i], C) = \CC(C, \XX[-i]) =0$ for $C \in \CC, i \gg 0$.
	\end{enumerate}
\end{lem}
\begin{proof}
We denote $\XX^{\pr} \cup \MM$ by $\XX$.

(1)
From Lemma \ref{lem_hom-vanish_SM}(\ref{lem_hom-vanish_SM_1}), $\CC(\XX^{\pr}[i], \XX^{\pr}) \cong \ZZ(\Sigma^i \XX^{\pr}, \XX^{\pr}) =0$ for $0 < i < n$.
From ($n$-SM), $\CC(\XX^{\pr}[i], \MM) \cong \CC(\XX^{\pr}, \MM[-i]) =0$ for $0 < i < n$
and $\CC(\MM[i], \XX^{\pr}) =0$ for $0 < i < n$.
Since $\MM$ is $n$-orthogonal, $\CC(\XX[i], \XX) =0$ for $0 < i < n$.

(2)
From ($\infty$-NS) and Lemma \ref{lem_torsion_pair_SMC}(\ref{lem_torsion_pair_SMC_3}), 
$\CC = (\thick \MM)^{\leq 0} \ast \ZZ \ast (\thick \MM)^{\geq 0}$.
Let $C \in \CC$. Then there exist triangles 
$M^{\leq 0} \rar C \rar C^{\pr} \rar M^{\leq 0}[1]$ and $Z_C \rar C^{\pr} \rar M^{\geq 0} \rar Z_C[1]$ where 
$M^{\leq 0} \in (\thick \MM)^{\leq 0}, M^{\geq 0} \in (\thick \MM)^{\geq 0}$ and $Z_C \in \ZZ$.

For $i \gg 0$, 
\begin{align*}
\CC(\XX[i], C) \cong \CC(\XX[i], Z_C) = \CC(\XX^{\pr}[i], Z_C) \cong \ZZ(\Sigma^i \XX^{\pr}, Z_C) =0.
\end{align*}

In the same way, $\CC(C, \XX[-i]) =0$ for $i \gg 0$. 
\end{proof}

\begin{lem} \label{lem_lems-for-last-thm}
Assume ($n$-NS) and ($n$-SM). 
Let $\XX^{\pr} \subset \Ob(\ZZ)$ be an $n$-orthogonal collection with respect to $\Sigma$.
We denote $\XX^{\pr} \cup \MM$ by $\XX$.
	\begin{enumerate}
	\item $\bcap{0 < i < n} \big(\lpp{\MM[-i]} \cap \rpp{\MM[i]} \big) 
	= \bracket{\MM} \ast \ZZ \ast \bracket{\MM}$.
	\item $\bracket{\XX} = \bracket{\MM} \ast \{ \XX^{\pr} \} \ast \bracket{\MM}$.
	\label{lem_lems-for-last-thm_2}
	\item \label{lem_lems-for-last-thm_3}
	$
	\bcap{0 < i < n} \big({}^{\perp_{\CC}} \XX[-i]  \cap \XX[i]^{\perp_{\CC}} \big)
	= \bracket{\MM} \ast \Big( \ \bcap{0 < i < n} \big({}^{\perp_{\ZZ}} \Sigma^{-i}\XX^{\pr} \cap 
	\Sigma^i {\XX^{\pr}}^{\perp_{\ZZ}}\big) \Big) \ast \bracket{\MM}$
	
	$\phantom{\bcap{0 < i < n} \big({}^{\perp_{\CC}} \XX[-i]  \cap \XX[i] x} 
	= \bigg\langle \Big( \ \bcap{0 < i < n} \big({}^{\perp_{\ZZ}} \Sigma^{-i}\XX^{\pr} \cap 
	\Sigma^i {\XX^{\pr}}^{\perp_{\ZZ}}\big) \Big) \cup \MM \bigg\rangle$.
	\end{enumerate}
\end{lem}
\begin{proof}
(1) Since $\bracket{\MM}, \ZZ \subset \bcap{0 < i < n} (\lpp{\MM[-i]} \cap \rpp{\MM[i]})$,
so is $\bracket{\MM} \ast \ZZ \ast \bracket{\MM}$.

On the other hand, let $X \in \bcap{0 < i < n} (\lpp{\MM[-i]} \cap \rpp{\MM[i]})$.
Take a triangle $M_1 \xrar{f} X \rar X^{\pr} \rar M_1[1]$ where $f$ is a minimal right $\bracket{\MM}$-approximation. 
From Lemma \ref{append_omega_sigma} and the long exact sequence, 
$X^{\pr} \in$ $\Big( \, \bcap{0 < i < n-1} \lpp{\MM[-i]} \Big)$ $\cap$ 
$\Big( \, \bcap{0 \leq i < n}  \rpp{\MM[i]} \Big)$.
In particular, if $n < \infty$,
\begin{align*}
X^{\pr} \in \Big( \, \bcap{0 < i < n-1} \lpp{\MM[-i]} \Big) \cap 
\Big( \, \bcap{0 < i \leq n}  \lpp{\MM[-i]} \Big) = \bcap{0 < i \leq n} \lpp{\MM[-i]}
\end{align*}
from ($n$-NS) and if $n= \infty$,
\begin{align*}
X^{\pr} \in \Big( \, \bcap{i>0} \lpp{\MM[-i]} \Big) \cap \Big( \, \bcap{i \geq 0}\, \rpp{\MM[i]} \Big).
\end{align*}

Dually, take a triangle $M_2[-1] \rar Z_X \rar X^{\pr} \xrar{g} M_2$ where $g$ is a minimal left $\bracket{\MM}$-approximation and 
one can show that $Z_X \in \ZZ$ similarly.

(2) From Lemma \ref{lem_cap_ZZ} and \ref{lem_n-orth_lifting}, 
$\{ \XX^{\pr} \} = \bracket{\XX} \cap \ZZ$. 
Thus, $\bracket{\XX} \supset \bracket{\MM} \ast \{ \XX^{\pr} \} \ast \bracket{\MM}$.
Let $X \in \bracket{\XX}$. From Lemma \ref{lem_extension_resolving} and (1), 
$X \in \bracket{\MM} \ast (\bracket{\XX} \cap \ZZ) \ast \bracket{\MM} 
= \bracket{\MM} \ast \{ \XX^{\pr} \} \ast \bracket{\MM}$.

(3) From Lemma \ref{lem_hom-vanish_SM}(\ref{lem_hom-vanish_SM_2}),
\begin{align*}
\bcap{0 < i < n} \big({}^{\perp_{\ZZ}} \Sigma^{-i}\XX^{\pr} \cap \Sigma^i {\XX^{\pr}}^{\perp_{\ZZ}}\big)
&=  \ZZ \cap \bcap{0 < i < n} \big({}^{\perp_{\CC}} \XX^{\pr}[-i] \cap \XX^{\pr}[i]^{\perp_{\CC}} \big) \\
&= \ZZ \cap \bcap{0 < i < n} \big({}^{\perp_{\CC}} \XX[-i] \cap \XX[i]^{\perp_{\CC}} \big) \\
&\subset \bcap{0 < i < n} \big({}^{\perp_{\CC}} \XX[-i] \cap \XX[i]^{\perp_{\CC}} \big).
\end{align*}

Since $\bracket{\MM} \subset \bcap{0 < i < n} \big({}^{\perp_{\CC}} \XX[-i] \cap \XX[i]^{\perp_{\CC}} \big)$,
\[
\bigg\langle \Big( \ \bcap{0 < i < n} \big({}^{\perp_{\ZZ}} \Sigma^{-i}\XX^{\pr} \cap 
\Sigma^i {\XX^{\pr}}^{\perp_{\ZZ}}\big) \Big) \cup \MM \bigg\rangle
\subset \bcap{0 < i < n} \big({}^{\perp_{\CC}} \XX[-i]  \cap \XX[i]^{\perp_{\CC}} \big).
\]

Let $X \in \bcap{0 < i < n} \big({}^{\perp_{\CC}} \XX[-i] \cap \XX[i]^{\perp_{\CC}} \big)$.
From (1), there exists the following diagram
where $M_1, M_2 \in \bracket{\MM}$ and $Z_X \in \ZZ$.
\[
\xy
(16,24)*+{M_1}="12";
(32,24)*+{M_1}="13";
(0,8)*+{M_2[-1]}="21";
(16,8)*+{Y}="22";
(32,8)*+{X}="23";
(48,8)*+{M_2}="24";
(0,-8)*+{M_2[-1]}="31";
(16,-8)*+{Z_X}="32";
(32,-8)*+{X^{\pr}}="33";
(48,-8)*+{M_2}="34";
(16,-24)*+{M_1[1]}="42";
(32,-24)*+{M_1[1]}="43";
{\ar@{=} "12";"13"};
{\ar^{} "21";"22"};
{\ar^{} "22";"23"};
{\ar^{} "23";"24"};
{\ar^{} "31";"32"};
{\ar^{} "32";"33"};
{\ar^{} "33";"34"};
{\ar@{=} "42";"43"};
{\ar@{=} "21";"31"};
{\ar^{} "12";"22"};
{\ar^{} "22";"32"};
{\ar^{} "32";"42"};
{\ar^{} "13";"23"};
{\ar^{} "23";"33"};
{\ar^{} "33";"43"};
{\ar@{=} "24";"34"};
{\ar@{}|\car "12";"23"};
{\ar@{}|\car "21";"32"};
{\ar@{}|\car "22";"33"};
{\ar@{}|\car "23";"34"};
{\ar@{}|\car "32";"43"};
\endxy \label{diag_lems-for-last-thm}
\]
Since $\XX$ is $n$-orthogonal, $\MM \subset \bcap{0 \leq i < n} \big( {}^{\perp_{\CC}} \XX^{\pr}[-i] \cap \XX^{\pr}[i]^{\perp_{\CC}} \big)$.
If $n = \infty$, then $M_1[1], M_2[-1] \in \bcap{i>0} \big( {}^{\perp_{\CC}} \XX^{\pr}[-i] \cap \XX^{\pr}[i]^{\perp_{\CC}} \big)$.
If $n < \infty$,
$\XX^{\pr} \subset {}^{\perp_{\CC}}\!\MM^{\perp_{\CC}} = {}^{\perp_{\CC}}\MM[-n] \cap \MM[n]^{\perp_{\CC}}$ from ($n$-NS).
Then $\MM \subset \bcap{0 \leq i \leq n} \big( {}^{\perp_{\CC}} \XX^{\pr}[-i] \cap \XX^{\pr}[i]^{\perp_{\CC}} \big)$. 
In particular, $M_1[1],$ $M_2[-1] \in \bcap{0 < i < n} \big( {}^{\perp_{\CC}} \XX^{\pr}[-i] \cap \XX^{\pr}[i]^{\perp_{\CC}} \big)$.
Thus, $Z_X \in \ZZ \cap \bcap{0 < i < n} \big( {}^{\perp_{\CC}} \XX^{\pr}[-i] \cap \XX^{\pr}[i]^{\perp_{\CC}} \big)$ and
$X \in \bracket{\MM} \ast \Big( \ \bcap{0 < i < n} \big({}^{\perp_{\ZZ}} \Sigma^{-i}\XX^{\pr} \cap 
\Sigma^i {\XX^{\pr}}^{\perp_{\ZZ}}\big) \Big) \ast \bracket{\MM}$.
Because
\begin{align*}
\bracket{\MM} \ast \Big( \ \bcap{0 < i < n} &\big({}^{\perp_{\ZZ}} \Sigma^{-i}\XX^{\pr} \cap 
\Sigma^i {\XX^{\pr}}^{\perp_{\ZZ}}\big) \Big) \ast \bracket{\MM} \\
&\subset \bigg\langle \Big( \ \bcap{0 < i < n} \big({}^{\perp_{\ZZ}} \Sigma^{-i}\XX^{\pr} \cap 
\Sigma^i {\XX^{\pr}}^{\perp_{\ZZ}}\big) \Big) \cup \MM \bigg\rangle
\end{align*}
is clear, the statement holds.
\end{proof}

\begin{prop} \label{prop_cup_MM_1}
Assume ($n$-NS), ($n$-SM) and $0<n<\infty$.
	\begin{enumerate}
	\item \label{prop_cup_MM_1-1}
	(\ref{corres_cup_MM}) preserves functorially finiteness of extension closures, that is, 
	\[
	\{ \XX^{\pr} \} \text{ is functorially finite in } \ZZ \implies \bracket{\XX^{\pr} \cup \MM} \text{ is functorially finite in } \CC.
	\]
	\item \label{prop_cup_MM_1-2}
	(\ref{corres_cup_MM}) preserves $\bbS^{\pr\pr}_{-n}$-stability of extension closures, that is, 
	\[
	\{ \XX^{\pr} \} \text{ is $\bbS^{\pr\pr}_{-n}$-stable} \implies \bracket{\XX^{\pr} \cup \MM} \text{ is $\bbS_{-n}$-stable}.
	\]
	\item \label{prop_cup_MM_1-3}
	(\ref{corres_cup_MM}) preserves $n$-simple-minded systems, that is, 
	\begin{align*}
	\XX^{\pr} \text{ is an } &\text{$n$-simple-minded system in $\ZZ$} \implies \\
	&\XX^{\pr} \cup \MM \text{ is an $n$-simple-minded system in $\CC$}.
	\end{align*}
	\end{enumerate}
\end{prop}
\begin{proof}
We denote $\XX^{\pr} \cup \MM$ by $\XX$.

(1) From Lemma \ref{lem_ZZ_funct-fin_SM}, $\ZZ$ is functorially finite in $\CC$. 
Then $\{\XX^{\pr} \}$ is also functorially finite in $\CC$. 
From Lemma \ref{lem_lems-for-last-thm}(\ref{lem_lems-for-last-thm_2}) and \cite[Theorem 1.3]{Che09},
$\bracket{\XX} = \bracket{\MM} \ast \{ \XX^{\pr} \} \ast \bracket{\MM}$ is functorially finite in $\CC$.

(2) From Lemma \ref{lem_Serre_SM} and \ref{lem_lems-for-last-thm}(\ref{lem_lems-for-last-thm_2}),
\begin{align*}
\bbS_{-n}(\bracket{\XX}) 
&= \bbS_{-n}(\bracket{\MM} \ast \{\XX^{\pr}\} \ast \bracket{\MM}) \\
&= \bracket{\MM} \ast (\Sigma^n \circ \bbS^{\pr\pr}) \{\XX^{\pr}\} \ast \bracket{\MM} \\
&\cong \bracket{\MM} \ast \bbS^{\pr\pr}_{-n} \{\XX^{\pr}\} \ast \bracket{\MM} \\
&\cong \bracket{\MM} \ast \{\XX^{\pr}\} \ast \bracket{\MM} = \bracket{\XX}.
\end{align*}

(3) Let $\XX^{\pr}$ be an $n$-simple-minded system in $\ZZ$.
From (1), (2), Corollary \ref{cor_characterize_simple-minded system} and Lemma \ref{lem_n-orth_lifting}(\ref{lem_n-orth_lifting_1}),
we only have to show $\bracket{\XX} = \bcap{0 < i < n} \rpp{\XX[i]}$.
Note that $\bracket{\XX} = \bracket{ \{ \XX^{\pr} \} \cup \MM }$.
That is because $\{ \XX^{\pr} \} \subset \bracket{\XX}$.
From Lemma \ref{lem_lems-for-last-thm}(\ref{lem_lems-for-last-thm_3}),
\[
\bracket{\XX}
= \bracket{ \{ \XX^{\pr} \} \cup \MM }
= \Bigl\langle \Big( \, \bcap{0 < i < n} \Sigma^i {\XX^{\pr}}^{\perp_{\ZZ}} \Big) 
	\cup \MM \Bigr\rangle
=\bcap{0 < i < n} \rpp{\XX[i]}.
\] 
Thus, $\XX$ is an $n$-simple-minded system in $\CC$.
\end{proof}

\begin{prop} \label{prop_cup_MM_2}
Assume ($\infty$-NS) and ($\infty$-SM).
Let $\XX^{\pr}$ be an $\infty$-orthogonal collection in $\ZZ$.
Then (\ref{corres_cup_MM}) preserves simple-minded collections, that is, 
\begin{align*}
\XX^{\pr} \text{ is a } &\text{simple-minded collection in $\ZZ$} \implies \\
&\XX^{\pr} \cup \MM \text{ is a simple-minded collection in $\CC$}.
\end{align*}
\end{prop}

\begin{proof}
We denote $\XX^{\pr} \cup \MM$ by $\XX$.

Let $\XX^{\pr}$ be a simple-minded collection in $\ZZ$.
From Lemma \ref{lem_n-orth_lifting}(\ref{lem_n-orth_lifting_1}), $\XX$ is an $\infty$-orthogonal collection in $\CC$.
By Lemma \ref{lem_lems-for-last-thm}(\ref{lem_lems-for-last-thm_3}),
\begin{align*}
\bracket{\XX} = \bracket{ \{ \XX^{\pr} \} \cup \MM} 
&= \Bigl\langle \Big( \, \bcap{i > 0} \big( {}^{\perp_{\ZZ}} \Sigma^{-i} \XX^{\pr} \cap \Sigma^i {\XX^{\pr}}^{\perp_{\ZZ}} \big)\Big) \cup \MM \Bigr\rangle \\
&= \bcap{i > 0} \big( {}^{\perp_{\CC}} \XX[-i] \cap \XX[i]^{\perp_{\CC}} \big).
\end{align*}

From Corollary \ref{cor_characterize_simple-minded collection} and Lemma \ref{lem_n-orth_lifting}(\ref{lem_n-orth_lifting_2}), 
we only have to show $\bracket{\XX}$ is covariantly finite in $\bcap{i>0} \lpp{\XX[-i]}$ and 
contravariantly finite in $\bcap{i>0} \rpp{\XX[i]}$.

By definition of $\Sigma$, $\Sigma \XX^{\pr} \subset \XX^{\pr}[1] \ast \bracket{\MM[1]}$ and 
$\Sigma^i \XX^{\pr} \subset \XX^{\pr}[i] \ast \bracket{\MM[i]} \ast \cdots \ast \bracket{\MM[1]}$ for $i \geq 1$ from induction on $i$.
Considering the extensions in $\ZZ$ and \cite[Lemma 2.8]{SP20},
\begin{align*}
\bcup{i \geq 0} \{ \Sigma^i \XX^{\pr} \} \star \cdots \star \{ \XX^{\pr} \} 
&\subset \bcup{i \geq 0} \{ \Sigma^i \XX^{\pr} \cup \cdots \cup \Sigma\XX^{\pr} \cup \XX^{\pr} \} \\
&\subset \bcup{i \geq 0} \bracket{\Sigma^i \XX^{\pr} \cup \cdots \cup \Sigma\XX^{\pr} \cup \XX^{\pr} \cup \MM} \\
&\subset \bcup{i \geq 0} \bracket{ \XX^{\pr}[i] \cup \cdots \XX^{\pr}[1] \cup \XX^{\pr} \cup \MM[i] \cup \cdots \cup \MM[1] \cup \MM} \\
&\subset \bcup{i \geq 0} \bracket{\XX[i] \cup \cdots \XX[1] \cup \XX} \\
&= \bcup{i \geq 0} \bracket{\XX[i]} \ast \cdots \ast \bracket{\XX}
\end{align*}
where $\star$ stands for $\ast$ in $\ZZ$.
Let $C \in \bcap{i > 0} {}^{\perp_{\CC}} \XX[-i]$.
From Lemma \ref{lem_torsion_pair_SMC}(\ref{lem_torsion_pair_SMC_3}),
there exist triangles $M^{\leq0} \rar C \rar C^{\pr} \rar M^{\leq0}[1]$ and $M[-1] \rar Z_C \rar C^{\pr} \rar M$
where $M^{\leq 0} \in (\thick \MM)^{\leq0}, Z_C \in \ZZ$ and $M \in (\thick \MM)^{\geq0}$.
We claim that $Z_C \in \bcap{i > 0} {}^{\perp_{\CC}} \XX[-i]$.
Since $C, M^{\leq0}[1] \in \bcap{i > 0} {}^{\perp_{\CC}} \XX[-i]$, so is $C^{\pr}$.
Then $M \in \bracket{\MM}$ from $Z_C[1] \in \bcap{i > 0} {}^{\perp_{\CC}} \MM[-i]$.
Because $C^{\pr}, M[-1] \in \bcap{i > 0} {}^{\perp_{\CC}} \XX^{\pr}[-i]$, so is $Z_C$.
Thus, $Z_C \in \ZZ \cap \Big( \bcap{i > 0} {}^{\perp_{\CC}} \XX^{\pr}[-i] \Big) = \bcap{i > 0} {}^{\perp_{\CC}} \XX[-i]$.
To sum up,
\[
\bcap{i > 0} {}^{\perp_{\CC}} \XX[-i] 
\subset (\thick \MM)^{\leq 0} \ast \Big( \ZZ \cap \bcap{i>0} {}^{\perp_{\CC}} \XX[-i] \Big) \ast \bracket{\MM}.
\]
Because $\{ \XX^{\pr} \}$ is a simple-minded collection in $\ZZ$, 
\begin{align*}
\bcap{i>0} {}^{\perp_{\ZZ}}\Sigma^{-i}\XX^{\pr}
&= \bcup{i\geq0} \{\Sigma^i \XX^{\pr} \} \star \cdots \star \{\Sigma \XX^{\pr} \} \star \{ \XX^{\pr} \} \\
&\subset \bcup{i \geq 0} \bracket{\XX[i]} \ast \cdots \ast \bracket{\XX}. \\
\end{align*}
Therefore, from \cite[Lemma 2.8]{SP20},
\begin{align*}
\bcap{i > 0} {}^{\perp_{\CC}} \XX[-i]
&\subset (\thick \MM)^{\leq 0} \ast \Big( \ZZ \cap \bcap{i>0} {}^{\perp_{\CC}} \XX[-i] \Big) \ast \bracket{\MM} \\
&\subset \Big( \, \bcup{i \geq 0} \bracket{\MM[i]} \ast \cdots \ast \bracket{\MM} \Big) \ast \Big( \, \bcap{i>0} {}^{\perp_{\ZZ}}\Sigma^{-i}\XX^{\pr} \Big) \ast \bracket{\MM} \\
&\subset \Big( \, \bcup{i \geq 0} \bracket{\MM[i]} \ast \cdots \ast \bracket{\MM} \Big) \ast 
		\Big( \, \bcup{i \geq 0} \bracket{\XX[i]} \ast \cdots \ast \bracket{\XX} \Big) \ast \bracket{\MM} \\
&= \bcup{i \geq 0} \bracket{\XX[i]} \ast \cdots \ast \bracket{\XX} \\
&\subset \bcap{i > 0} {}^{\perp_{\CC}} \XX[-i].
\end{align*}
In particular, 
\[
\bcup{i \geq 0} \bracket{\XX[i]} \ast \cdots \ast \bracket{\XX} = \bcap{i > 0} {}^{\perp_{\CC}} \XX[-i].
\]
So, $\bracket{\XX}$ is covariantly finite in $\bcap{i > 0} {}^{\perp_{\CC}} \XX[-i]$.
One can show that $\bracket{\XX}$ is contravariantly finite in $\bcap{i > 0}  {\XX[i]}^{\perp_{\CC}}$, dually.
\end{proof}

\begin{thm} \label{thm_reduction-n-SMS}
Assume ($n$-NS) and ($n$-SM) where $n < \infty$. 
Then the following correspondence of collections is bijective.
	\begin{align*}
	\{ \XX : n\text{-simple-}&\text{minded system in }\CC \mid \MM \subset \XX\} \to \\
	&\{ \XX\setminus \MM \ : n\text{-simple-minded system in } \ZZ\}
	\end{align*}
\end{thm}
\begin{proof}
From Lemma \ref{lem_cap_ZZ}(1), Proposition \ref{prop_cap_ZZ_1}(\ref{prop_cap_ZZ_1-3}) and \ref{prop_cup_MM_1}(\ref{prop_cup_MM_1-3}), the correspondence (\ref{corres_cap_ZZ}) and (\ref{corres_cup_MM}) are bijection which preserve $n$-simple-minded systems. 
\end{proof}

\begin{thm} \label{thm_reduction-SMC}
Assume ($\infty$-NS) and ($\infty$-SM). Then the following correspondence of collections is bijective.
	\begin{align*}
	\{ \XX : \text{simple-}&\text{minded collection in }\CC \mid \MM \subset \XX\} \to \\
	&\{ \XX\setminus \MM : \text{simple-minded collection in }\ZZ\}
	\end{align*}
\end{thm}
\begin{proof}
From Lemma \ref{lem_cap_ZZ}(1) and Proposition \ref{prop_cap_ZZ_2}(\ref{prop_cap_ZZ_2-3}), \ref{prop_cup_MM_2}, the correspondence (\ref{corres_cap_ZZ}) and (\ref{corres_cup_MM}) are bijections which preserve simple-minded collections. 
\end{proof}

\begin{rmk} \cite[Definition 4.1]{Dug15}, \cite[Definition 3.1]{SP20} \label{rmk_mutation-of-orth-coll}
Let $\XX$ be an $n$-simple-minded system or a simple-minded collection in $\CC$ and 
$\MM$ be a collection which is contained in $\XX$ and satisfies ($n$-NS).
Then $\XX$ and $\MM$ determine an orthogonal mutation triple $(\bracket{\MM[1]}, \ZZ, \bracket{\MM[-1]})$
so that $\ZZ$ and $\MM$ satisfy ($n$-SM) from Lemma \ref{lem_SMMT}.

$\mu^-_{\MM}(\XX)$ \resp{$\mu^+_{\MM}(\XX)$} in Definition \ref{defi_mu-rMT} is exactly 
$\mu^-_{\MM}(\XX)$ \resp{$\mu^+_{\MM}(\XX)$} in \cite{Dug15} if $\XX$ is an $n$-simple-minded system, and
$\mu^-_{\MM}(\XX)[1]$ \resp{$\mu^+_{\MM}(\XX)[-1]$} in \cite{KY14} if $\XX = \{X_1, \ldots, X_m \}$ is a finite set and $\MM = \{X_i\} \subset \XX$.
\end{rmk}

\begin{cor} \cite[Proposition 7.6]{KY14}, \cite[Theorem 6.6]{SP20}
Let $\XX$ be an $n$-simple-minded system \resp{a simple-minded collection}.
In Remark \ref{rmk_mutation-of-orth-coll}, 
both $\mu_{\MM}^-(\XX)$ and $\mu_{\MM}^+(\XX)$ are $n$-simple-minded systems \resp{simple-minded collections} 
if $\XX$ is an $n$-simple-minded system \resp{a simple-minded collection}.
\end{cor}
\begin{proof}
Since $\Sigma$ \resp{$\Omega$} is the shift functor \resp{the inverse of the shift functor} in $\ZZ$,
this statement directly follows from Theorem \ref{thm_reduction-n-SMS} and \ref{thm_reduction-SMC}.
\end{proof}

\begin{ex}
\begin{enumerate}
\item
The following picture is the AR quiver of $\CC = \Db(k A_4)/\nu [2]$ 
where $\nu$ is the Nakayama functor and $A_4$ is the linearly oriented Dynkin quiver of type $A_4$
(we omitted AR translations).

Let $\MM = \{M_1, M_2\}$. Note that $\MM$ is a 2-orthogonal collection.
The fundamental domain is shaded light gray and a functorially finite subcategory $\bracket{\MM}$ in $\CC$ is shaded dark gray.

Take $X_1, X_2 \in \CC$ indicated in the following picture where $\XX = \{ X_1, X_2,$ $M_1,$ $M_2\}$ is a 2-simple-minded system.
Then $\mu^-_{\MM}(\XX) \setminus \MM$, $\mu^+_{\MM}(\XX) \setminus \MM$ are indicated in blue and red, respectively.
\[
\begin{tikzpicture}
	\path[rounded corners,fill=gray,opacity=0.2] (0.75, -0.5) -- (2.75, 3.5) -- (7.25, 3.5) -- (9.25, -0.5) --cycle;
	\path[rounded corners,fill=gray,opacity=0.6] (4,-0.3) -- (5, 1.7) -- (6, -0.3) --cycle;
	\foreach \x in {0,1,...,11}
		\foreach \y in {1,3}
		{
		\fill (\x,\y) circle[radius=0.04cm];
		\fill (\x+0.5,\y-1) circle[radius=0.04cm];
		\draw[-{stealth[scale=3]}] (\x,\y) -- (\x+0.48,\y-0.96);
		\draw[-{stealth[scale=3]}] (\x,1) -- (\x+0.48,1.96);
		\ifnum \x<11
			\draw[-{stealth[scale=3]}] (\x+0.5,\y-1) -- (\x+0.98,\y-0.04);
			\draw[-{stealth[scale=3]}] (\x+0.5,2) -- (\x+0.98,1.04);
		\fi
		}
	\fill [red]
	(7.5,0) circle[radius=0.08cm]
	(3,1) circle[radius=0.08cm];
	\fill 
	(3.5,0) circle[radius=0.08cm]
	(4.5, 0) circle[radius=0.08cm] 
	(5.5, 0) circle[radius=0.08cm]
	(6.5, 0) circle[radius=0.08cm];
	\fill [blue]
	(7,1) circle[radius=0.08cm]
	(2.5,0) circle[radius=0.08cm];
	\node [text=red] at (7.5,0) [right] {\msize{0.7}{\Omega X_1}};
	\node [text=red] at (3,1) [right] {\msize{0.7}{\Omega X_2}};
	\node at (3.5,0) [right] {\msize{0.7}{X_1}};
	\node at (4.5,0) [right] {\msize{0.7}{M_1}};
	\node at (5.5,0) [right] {\msize{0.7}{M_2}};
	\node at (6.5,0) [right] {\msize{0.7}{X_2}};
	\node [text=blue] at (7,1) [right] {\msize{0.7}{\Sigma X_1}};
	\node [text=blue] at (2.5,0) [right] {\msize{0.7}{\Sigma X_2}};
	\node (D) at (5,-0.3) [below] {\msize{0.7}{\bracket{\MM}}};
\end{tikzpicture}
\]
From \cite[Proposition 5.10]{INP24}, the AR quiver of triangulated category $\ZZ = \bcap{0\leq i\leq 2}\, \rpp{\MM[i]}$ is illustrated as follows
(we omitted AR translations).
\[
\begin{tikzpicture}
	\path[rounded corners,fill=gray,opacity=0.2] (0.75, -0.5) -- (2.75, 3.5) -- (7.25, 3.5) -- (9.25, -0.5) --cycle;
	\foreach \x in {0,4}
	{
	\fill (2.5+\x,0) circle[radius=0.04cm]
	(3+\x,1) circle[radius=0.04cm]
	(3.5+\x,0) circle[radius=0.04cm];
	\draw[-{stealth[scale=3]}] (2.5+\x,0) -- (2.98+\x,0.96);
	\draw[-{stealth[scale=3]}] (3+\x,1) -- (3.48+\x,0.04);
	\draw[-{stealth[scale=3]}] (3.5+\x,0) -- (4.98+\x,2.96);
	\draw[-{stealth[scale=3]}] (1+\x,3) -- (2.48+\x,0.04);
	}
	\foreach \x in {0,9}
	{
	\fill (0+\x,3) circle[radius=0.04cm]
	(0.5+\x,2) circle[radius=0.04cm]
	(1+\x,3) circle[radius=0.04cm];
	\draw[-{stealth[scale=3]}] (0+\x,3) -- (0.48+\x,2.04);
	\draw[-{stealth[scale=3]}] (0.5+\x,2) -- (0.98+\x,2.96);
	}
	\fill (5,3) circle[radius=0.04cm]
	(11.5,0) circle[radius=0.04cm];
	\draw[-{stealth[scale=3]}] (10,3) -- (11.48,0.04);
	\fill [red]
	(7.5,0) circle[radius=0.08cm]
	(3,1) circle[radius=0.08cm];
	\fill 
	(3.5,0) circle[radius=0.08cm]
	(6.5, 0) circle[radius=0.08cm];
	\fill [blue]
	(7,1) circle[radius=0.08cm]
	(2.5,0) circle[radius=0.08cm];
	\node [text=red] at (7.5,0) [right] {\msize{0.7}{\Omega X_1}};
	\node [text=red] at (3,1) [right] {\msize{0.7}{\Omega X_2}};
	\node at (3.5,0) [right] {\msize{0.7}{X_1}};
	\node at (6.5,0) [right] {\msize{0.7}{X_2}};
	\node [text=blue] at (7,1) [right] {\msize{0.7}{\Sigma X_1}};
	\node [text=blue] at (2.5,0) [right] {\msize{0.7}{\Sigma X_2}};
\end{tikzpicture}
\]
\item
The following picture is the AR quiver of $\CC = \Db(k A_4)$ 
where $\nu$ is the Nakayama functor and $A_4$ is the linearly oriented Dynkin quiver of type $A_4$
(we omitted AR translations).

Let $\MM = \{M_1, M_2\}$. Note that $\MM$ is a pre-simple-minded collection.
The fundamental domain is shaded light gray and a functorially finite subcategory $\bracket{\MM}$ in $\CC$ is shaded gray.

Take $X_1, X_2 \in \CC$ indicated in the following picture where $\XX = \{ X_1, X_2,$ $M_1,$ $M_2\}$ is a simple-minded collection.
Then $\mu^-_{\MM}(\XX) \setminus \MM$, $\mu^+_{\MM}(\XX) \setminus \MM$ are indicated in blue and red, respectively.
\[
\begin{tikzpicture}
	\path[rounded corners,fill=gray,opacity=0.6] (4,-0.3) -- (5, 1.7) -- (6, -0.3) --cycle;
	\foreach \x in {0,1,...,11}
		\foreach \y in {1,3}
		{
		\fill (\x,\y) circle[radius=0.04cm];
		\fill (\x+0.5,\y-1) circle[radius=0.04cm];
		\draw[-{stealth[scale=3]}] (\x,\y) -- (\x+0.48,\y-0.96);
		\draw[-{stealth[scale=3]}] (\x,1) -- (\x+0.48,1.96);
		\ifnum \x<11
			\draw[-{stealth[scale=3]}] (\x+0.5,\y-1) -- (\x+0.98,\y-0.04);
			\draw[-{stealth[scale=3]}] (\x+0.5,2) -- (\x+0.98,1.04);
		\fi
		}
	\fill [red]
	(1,3) circle[radius=0.08cm]
	(3,1) circle[radius=0.08cm];
	\fill 
	(3.5,0) circle[radius=0.08cm]
	(4.5, 0) circle[radius=0.08cm] 
	(5.5, 0) circle[radius=0.08cm]
	(6.5, 0) circle[radius=0.08cm];
	\fill [blue]
	(7,1) circle[radius=0.08cm]
	(9,3) circle[radius=0.08cm];
	\node [text=red] at (1,3) [right] {\msize{0.7}{\Omega X_1}};
	\node [text=red] at (3,1) [right] {\msize{0.7}{\Omega X_2}};
	\node at (3.5,0) [right] {\msize{0.7}{X_1}};
	\node at (6.5,0) [right] {\msize{0.7}{X_2}};
	\node [text=blue] at (7,1) [right] {\msize{0.7}{\Sigma X_1}};
	\node [text=blue] at (9,3) [right] {\msize{0.7}{\Sigma X_2}};
	\node (D) at (5,-0.3) [below] {\msize{0.7}{\bracket{\MM}}};
\end{tikzpicture}
\]
From \cite[Proposition 5.10]{INP24}, the AR quiver of triangulated category $\ZZ = \bcap{i \geq 0}\, \big( \lpp{\MM[-i]} \cap \rpp{\MM[i]} \big)$ is illustrated as follows
(we omitted AR translations).
\[
\begin{tikzpicture}
	\foreach \x in {0,4}
	{
	\fill (2.5+\x,0) circle[radius=0.04cm]
	(3+\x,1) circle[radius=0.04cm]
	(3.5+\x,0) circle[radius=0.04cm];
	\draw[-{stealth[scale=3]}] (2.5+\x,0) -- (2.98+\x,0.96);
	\draw[-{stealth[scale=3]}] (3+\x,1) -- (3.48+\x,0.04);
	\draw[-{stealth[scale=3]}] (3.5+\x,0) -- (4.98+\x,2.96);
	\draw[-{stealth[scale=3]}] (1+\x,3) -- (2.48+\x,0.04);
	}
	\foreach \x in {0,9}
	{
	\fill (0+\x,3) circle[radius=0.04cm]
	(0.5+\x,2) circle[radius=0.04cm]
	(1+\x,3) circle[radius=0.04cm];
	\draw[-{stealth[scale=3]}] (0+\x,3) -- (0.48+\x,2.04);
	\draw[-{stealth[scale=3]}] (0.5+\x,2) -- (0.98+\x,2.96);
	}
	\fill (5,3) circle[radius=0.04cm]
	(11.5,0) circle[radius=0.04cm];
	\draw[-{stealth[scale=3]}] (10,3) -- (11.48,0.04);
	\fill [red]
	(1,3) circle[radius=0.08cm]
	(3,1) circle[radius=0.08cm];
	\fill 
	(3.5,0) circle[radius=0.08cm]
	(6.5, 0) circle[radius=0.08cm];
	\fill [blue]
	(7,1) circle[radius=0.08cm]
	(9,3) circle[radius=0.08cm];
	\node [text=red] at (1,3) [right] {\msize{0.7}{\Omega X_1}};
	\node [text=red] at (3,1) [right] {\msize{0.7}{\Omega X_2}};
	\node at (3.5,0) [right] {\msize{0.7}{X_1}};
	\node at (6.5,0) [right] {\msize{0.7}{X_2}};
	\node [text=blue] at (7,1) [right] {\msize{0.7}{\Sigma X_1}};
	\node [text=blue] at (9,3) [right] {\msize{0.7}{\Sigma X_2}};
\end{tikzpicture}
\]
\end{enumerate}
\end{ex}

\section{Localization and Hovey mutation triples}
Let $(\SS, \ZZ, \VV)$ be a mutation triple.
In this subsection, we assume that $\CC$ is a Krull-Schmidt triangulated category and the following condition, which is stronger than (MT2).

\begin{itemize}[leftmargin=50pt]
\item[(MT2$^{+}$)] 
	\begin{enumerate}
	\item $\bbE(\SS, \ZZ) = 0$ and $\bbE(\SS, \ZZ\bracket{-1}) =0$.
	\item $\bbE(\ZZ, \VV) = 0$ and $\bbE(\ZZ\bracket{1}, \VV)=0$.
	\end{enumerate}
\end{itemize}

\begin{defi}
We say that $(\SS, \ZZ, \VV)$ is a \emph{Hovey mutation triple} if it satisfies the following conditions.
	\begin{enumerate}[leftmargin=30pt, label=(\roman*)]
	\item $\SS, \ZZ$ and $\VV$ are closed under direct summands.
	\item $\SS \cap \VV \subset \II$.
	\item $\CC = \SS[-1] \ast \ZZ \ast \VV[1]$.
	\item $\NN$ defined by $\SS \ast \VV[1]$ is a thick subcategory of $\CC$.
	\end{enumerate}
\end{defi}

\begin{ex}
Let $((\SS, \TT), (\UU, \VV))$ be a Hovey twin cotorsion pair.
Then $(\SS, \TT \cap \UU, \VV)$ is a Hovey mutation triple.
\end{ex}

The following statement and its proof are based on \cite[Proposition 6.10, Corollary 6.13]{Nak18}

\begin{thm}
Let $(\SS, \ZZ, \VV)$ be a Hovey mutation triple. Then the inclusion functor $\ZZ \to \CC$ induces a triangle equivalence $\Phi \colon \ul{\ZZ} \to \CC/\NN$.
\end{thm}
\begin{proof}
Let $X \in \CC$. Then there exists the following diagram where 
$S_X \in \SS, V_X \in \VV$ and $Z_X \in \ZZ$.
\[
\xy
(0,8)*+{V_X}="21";
(16,8)*+{U_X}="22";
(32,8)*+{X}="23";
(0,-8)*+{V_X}="31";
(16,-8)*+{Z_X}="32";
(32,-8)*+{T_X}="33";
(16,-24)*+{S_X}="42";
(32,-24)*+{S_X}="43";
{\ar^{} "21";"22"};
{\ar^{} "22";"23"};
{\ar^{} "31";"32"};
{\ar^{z} "32";"33"};
{\ar@{=} "42";"43"};
{\ar@{=} "21";"31"};
{\ar^{} "22";"32"};
{\ar^{} "32";"42"};
{\ar^{t} "23";"33"};
{\ar^{} "33";"43"};
{\ar@{}|\car "21";"32"};
{\ar@{}|\car "22";"33"};
{\ar@{}|\car "32";"43"};
\endxy
\]
Since $z \colon Z_X \to T_X$ and $t \colon X \to T_X$ are quasi-isomorphisms, in other words, morphisms whose cones are contained in $\NN$, $\Phi$ is essentially surjective.

Let $f \in \ZZ(X,Y)$ where $\Phi(\ul{f})=0$. Then there exists a quasi-isomorphism $s \colon Y \to C$ such that $sf=0$.
Because the cocone of $s$ is in $\NN$, $f$ factors through $\NN = \SS[-1] \ast \VV$.
Moreover, $f$ factors through $\VV$ since $\CC(\SS[-1], Y) \subset \CC(\SS[-1], \ZZ) =0$.
Finally, $f$ factors through $\II$ because $\CC(X\bracket{1}[-1], \VV) \cong \CC(X\bracket{1}, \VV[1]) =0$.
Thus, $\ul{f} =0$.

Let $f/s \in \CC/\NN(X,Y)$ be a morphism which can be expressed as a roof
\[
\xy
(0,0)*+{X}="1";
(12,6)*+{C}="2";
(24,0)*+{Y}="3";
{\ar^{f} "1";"2"};
{\ar_{s} "3";"2"};
\endxy
\]
where $s \colon Y \to C$ is a quasi-isomorphism and $X, Y \in \ZZ, C \in \CC$.
Take triangles $C \xrar{t} T_C \rar S_C \rar C[1]$ and $V_C \rar Z_C \xrar{z} T_C \xrar{v} V_C[1]$
where $S_C \in \SS, Z_C \in \ZZ$ and $V_C \in \VV$.
Then there exist morphisms $f^{\pr} \colon X \to Z_C$ and $s^{\pr} \colon Y \to Z_C$ in $\CC$ where $tf = zf^{\pr}$ and $ts = zs^{\pr}$
since $\CC(\ZZ, \VV[1]) = 0$.
Note that $s^{\pr}$ is a quasi-isomorphism, since $z$ and $ts$ are quasi-isomorphisms (recall that $\NN$ is thick).
Because $t$ and $z$ are quasi-isomorphisms $f/s = tf/ts = zf^{\pr}/zs^{\pr} = f^{\pr}/s^{\pr}$. 
Take triangles $Y \xrar{s^{\pr}} Z_C \rar N \rar Y[1]$ where $N \in \NN$ and $V \rar S \rar N \rar V[1]$.
Then there exists the following diagram.
\[
\xy
(0,24)*+{}="11";
(16,24)*+{V}="12";
(32,24)*+{V}="13";
(0,8)*+{Y}="21";
(16,8)*+{W}="22";
(32,8)*+{S}="23";
(48,8)*+{Y[1]}="24";
(0,-8)*+{Y}="31";
(16,-8)*+{Z_C}="32";
(32,-8)*+{N}="33";
(48,-8)*+{Y[1]}="34";
(16,-24)*+{V[1]}="42";
(32,-24)*+{V[1]}="43";
{\ar@{=} "12";"13"};
{\ar^{s_1} "21";"22"};
{\ar^{} "22";"23"};
{\ar^{0} "23";"24"};
{\ar^{s^{\pr}} "31";"32"};
{\ar^{} "32";"33"};
{\ar^{} "33";"34"};
{\ar@{=} "42";"43"};
{\ar@{=} "21";"31"};
{\ar^{} "12";"22"};
{\ar^{s_2} "22";"32"};
{\ar^{0} "32";"42"};
{\ar^{} "13";"23"};
{\ar^{} "23";"33"};
{\ar^{} "33";"43"};
{\ar@{=} "24";"34"};
{\ar@{}|\car "12";"23"};
{\ar@{}|\car "21";"32"};
{\ar@{}|\car "22";"33"};
\endxy
\]
Then $S$ is a direct summand of $V \oplus Z_C$.
Since $\CC$ is Krull-Schmidt, there exist $S_1, S_2 \in \CC$ where $S \cong S_1 \oplus S_2$ and $S_1$ \resp{$S_2$} is a direct summand of $V$ \resp{$Z_C$}.
From assumption, $S_1 \in \SS \cap \VV \subset \II$ and $S_2 \in \SS \cap \ZZ = \II$.
So, $S \in \II$.
In the same way, one can show that $V \in \II$.
Thus, $\ul{s_1}$ and $\ul{s_2}$ are isomorphisms in $\ul{\CC}$ and $\Phi(\ul{{s^{\pr}}^{-1} f^{\pr}}) = f/s$.

We define a natural isomorphism $\phi \colon [1] \circ \Phi \Rightarrow \Phi \circ \Sigma$ by 
$\phi_X = h^{X\bracket{1}} \circ {(l^X)}^{-1}$ where $h^{X\bracket{1}}$ and $l^X$ are defined as follows.
\[
\xy
(16,24)*+{S^X[-1]}="12";
(32,24)*+{S^X[-1]}="13";
(48,24)*+{}="14";
(0,8)*+{X}="21";
(16,8)*+{I^X}="22";
(32,8)*+{X\bracket{1}}="23";
(48,8)*+{X[1]}="24";
(0,-8)*+{X}="31";
(16,-8)*+{W^{\pr}}="32";
(32,-8)*+{\Sigma X}="33";
(48,-8)*+{X[1]}="34";
(0,-24)*+{}="41";
(16,-24)*+{S^X}="42";
(32,-24)*+{S^X}="43";
{\ar@{=} "12";"13"};
{\ar^{i^X} "21";"22"};
{\ar^{p^X} "22";"23"};
{\ar^{l^X} "23";"24"};
{\ar^{} "31";"32"};
{\ar^{} "32";"33"};
{\ar^{} "33";"34"};
{\ar@{=} "42";"43"};
{\ar@{=} "21";"31"};
{\ar^{} "12";"22"};
{\ar^{} "22";"32"};
{\ar^{} "32";"42"};
{\ar^{} "13";"23"};
{\ar^{h^{X\bracket{1}}} "23";"33"};
{\ar^{g^{X\bracket{1}}} "33";"43"};
{\ar@{=} "24";"34"};
{\ar@{}|\car "12";"23"};
{\ar@{}|\car "21";"32"};
{\ar@{}|\car "22";"33"};
{\ar@{}|{\phantom{XXX}\car} "23";"34"};
{\ar@{}|\car "32";"43"};
\endxy
\]

Then we only have to show $\Phi$ preserves triangles. It is enough to show that 
$\Phi$ sends a standard right triangle $X \xrar{\ul{a}} Y \xrar{\ul{h^{C^a} b}} \sigma C^a \xrar{\sigma (\ul{c^a})} \Sigma X$, 
which is induced by a triangle $X \xrar{\wt{a}} Y\oplus I^X \xrar{\wt{b}} C^a \xrar{\wt{c}} X[1]$ in $\CC$,
to a triangle in $\CC/\NN$.
Note that the morphism $\wt{c}$ makes the following diagram in $\CC$ commutative.
\[
\xy
(0,8)*+{X}="11";
(16,8)*+{Y \oplus I^X}="12";
(32,8)*+{C^a}="13";
(48,8)*+{X[1]}="14";
(0,-8)*+{X}="21";
(16,-8)*+{I^X}="22";
(32,-8)*+{X\bracket{1}}="23";
(48,-8)*+{X1]}="24";
{\ar^-{\wt{a}} "11";"12"};
{\ar^-{\wt{b}} "12";"13"};	
{\ar^-{\wt{c}} "13";"14"};
{\ar^{i^X} "21";"22"};
{\ar^{p^X} "22";"23"};
{\ar^{l^X} "23";"24"};
{\ar@{=}^{} "11";"21"};
{\ar^{\msize{0.6}{\begin{bmatrix} 0 & 1 \end{bmatrix}}} "12";"22"};
{\ar^{c^a} "13";"23"};
{\ar@{=}^{} "14";"24"};
{\ar@{}|\car "11";"22"};
{\ar@{}|\car "12";"23"};
{\ar@{}|\car "13";"24"};
\endxy
\]
Recall that there exists the following commutative diagram in $\CC$.
\[
\xy
(0,8)*+{C^a}="11";
(16,8)*+{\sigma C^a}="12";
(0,-8)*+{X\bracket{1}}="21";
(16,-8)*+{\Sigma X}="22";
{\ar^{h^{C^a}} "11";"12"};
{\ar^{h^{X\bracket{1}}} "21";"22"};
{\ar^{c^a} "11";"21"};
{\ar^{\sigma(c^a)} "12";"22"};
{\ar@{}|\car "11";"22"};
\endxy
\]
Thus, there exists the following commutative diagram in $\CC/\NN$.
\[
\xy
(0,8)*+{X}="11";
(16,8)*+{Y\oplus I^X}="12";
(32,8)*+{C^a}="13";
(48,8)*+{X[1]}="14";
(0,-8)*+{X}="21";
(16,-8)*+{Y}="22";
(32,-8)*+{\sigma C^a}="23";
(48,-8)*+{\Sigma X}="24";
{\ar^-{\wt{a}} "11";"12"};
{\ar^-{\wt{b}} "12";"13"};	
{\ar^{\wt{c}} "13";"14"};
{\ar^{\Phi(\ul{a})} "21";"22"};
{\ar^{\Phi(\ul{h^{C^a}b})} "22";"23"};
{\ar^{\Phi(\sigma(\ul{c^a}))} "23";"24"};
{\ar@{=}^{} "11";"21"};
{\ar_{\vsim}^{\msize{0.6}{\begin{bmatrix} 1 & 0 \end{bmatrix}}} "12";"22"};
{\ar^{h^{C^a}} "13";"23"};
{\ar^{h^{X\bracket{1}}{(l^X)}^{-1}} "14";"24"};
{\ar@{}|\car "11";"22"};
{\ar@{}|\car "12";"23"};
{\ar@{}|{\phantom{Xx}\car} "13";"24"};
\endxy
\]
\end{proof}

\newpage
\bibliographystyle{plain}
\bibliography{myrefs} 

\end{document}